\documentclass[reqno,english]{amsart}

\usepackage[colorlinks=true]{hyperref}
\usepackage[columns=1]{idxlayout}
\usepackage{graphicx}

\usepackage[margin=1.2in]{geometry}

\hypersetup{urlcolor=blue, linkcolor=blue, citecolor=red, anchorcolor=blue}

\newtheorem{theorem}{Theorem}[section]
\newtheorem{proposition}[theorem]{Proposition}
\newtheorem{lemma}[theorem]{Lemma}

\theoremstyle{remark}
\newtheorem{remark}[theorem]{Remark}

\theoremstyle{definition}

\numberwithin{equation}{section}

\newcommand{\nn}{\nabla}
\newcommand{\R}{\mathbb{R}}

\newcommand{\onn}{\quad \hbox{on }}
\newcommand{\C}{\mathbb{C}}
\newcommand{\N}{\mathbb{N}}
\newcommand{\Z}{\mathbb{Z}}
\newcommand{\norm}{\|}

\newcommand{\be}{\begin{equation}}
\newcommand{\ee}{\end{equation}}

\newcommand{\abs}[2][]{#1\lvert #2 #1\rvert}
\newcommand{\equ}[1]{(\ref{#1})}

\newcommand{\A}{1}
\newcommand{\Ae}{}

\newcommand{\vortfun}{\gamma}

\let\Re\relax
\let\Im\relax

\DeclareMathOperator{\Re}{Re}
\DeclareMathOperator{\Im}{Im}

\DeclareMathOperator{\supp}{supp}

\makeindex

\begin{document}

\title{Overhanging solitary water waves}

\author[J.~D\'avila]{Juan D\'avila}
\address{\noindent J.~D\'avila: Department of Mathematical Sciences University of Bath, Bath BA2 7AY, United Kingdom}
\email{jddb22@bath.ac.uk}

\author[M.~del Pino]{Manuel del Pino}
\address{\noindent M.~del Pino: Department of Mathematical Sciences University of Bath, Bath BA2 7AY, United Kingdom}
\email{m.delpino@bath.ac.uk}

\author[M.~Musso]{Monica Musso}
\address{\noindent M.~Musso: Department of Mathematical Sciences University of Bath, Bath BA2 7AY, United Kingdom.}
\email{m.musso@bath.ac.uk}

\author[M.~H.~Wheeler]{Miles H.~Wheeler}
\address{\noindent M.~H.~Wheeler: Department of Mathematical Sciences University of Bath, Bath BA2 7AY, United Kingdom.}
\email{mw2319@bath.ac.uk}

\begin{abstract}
We provide the first construction of overhanging gravity water waves having the approximate form of a disk joined to a strip by a thin neck. The waves are solitary with constant vorticity, and exist when an appropriate dimensionless gravitational constant $g>0$ is sufficiently small. Our construction involves combining three explicit solutions to related problems: a disk of fluid in rigid rotation, a linear shear flow in a strip, and a rescaled version of an exceptional domain discovered by Hauswirth, H\'elein, and Pacard \cite{hauswirth-helein-pacard}. The method developed here is related to the construction of constant mean curvature surfaces through gluing.
\end{abstract}

\maketitle

\section{Introduction}

\subsection{Statement of the problem}
In this paper, we examine the classical water wave problem for an incompressible inviscid fluid occupying a time-dependent domain $\Omega(t) \subset \mathbb{R}^2$, whose boundary consists of a fixed horizontal bed represented by
$$\mathcal B = \{ (x_1,x_2)  \ |\  x_2=-A\}, \quad A>0, $$
together with an unknown free boundary $\mathcal{S}(t)$ separating the fluid from the air outside the confining region; see Figure~\ref{fig:setup}.
Inside the fluid, the velocity ${\bf u} = (u_1(x_1,x_2,t),u_2(x_1,x_2,t))$ and pressure $p(x_1,x_2,t)$ obey the incompressible Euler equations
\begin{align}
   \label{euler}
    \frac{\partial \bf u}{\partial t} + ({\bf u} \cdot \nabla) {\bf u} = -\nabla p + (0, -g), \quad \operatorname{div} {\bf u} = 0 \quad \text{in} \ \Omega(t),
\end{align}
where we have assumed that the fluid has constant density $1$ and that gravity acts as a constant vertical force $(0, -g)$ with $g > 0$. The so-called \emph{kinematic boundary condition} requires that there is no flow across either boundary component or that fluid particles on the boundary remain there at all times. Finally, the \emph{dynamic boundary condition} requires that the pressure be constant along $\mathcal S(t)$, say $p = P_{\mathrm{atm}}$ where $P_{\mathrm{atm}}$ represents the constant atmospheric pressure.  

In particular, we are interested in
traveling waves, solutions that appear steady in a reference frame moving horizontally at some constant speed $c>0$. In this case, the velocity field 
and the pressure take the form 
 $${\bf u}(x_1,x_2,t) = U(x_1 - ct, x_2) +c\, (1,0), \quad p(x_1,x_2,t) = P(x_1 - ct, x_2) $$ while  the domain and its free boundary are given by
$$\Omega(t) = \Omega + ct \, (1,0), \quad \mathcal{S}(t) = \mathcal{S} + ct \, (1,0),$$
where $\Omega \subset \mathbb{R}^2$ is a fixed domain and $\mathcal{S} \subset \mathbb{R}^2$ is a fixed curve.  With this ansatz, \eqref{euler} reduces to the standard stationary Euler problem for $U(x_1,x_2)$ and $P(x_1,x_2)$,
\begin{align}
   \label{euler1}
    ( U \cdot \nabla) {U} = -\nabla P + (0, -g), \quad \operatorname{div} {U} = 0 \quad \text{in} \ \Omega.
\end{align}
Incompressibility is equivalent to the existence of a stream function $\Psi(x_1,x_2)$, determined up to a constant by
\begin{align}
\nonumber
    \nabla^\perp\Psi := (\partial_{x_2}\Psi,-\partial_{x_1}\Psi) = U .
\end{align}
The velocity field is then tangent to level sets of the stream function $\Psi$, called streamlines.
The kinematic boundary condition is therefore equivalent to $\Psi$ being constant respectively on $\mathcal{B}$ and $\mathcal{S}$ and we henceforth normalize $\Psi$ by requiring $\Psi=0$ on $\mathcal{S}$.
Similarly, $\Psi = \beta $ for some constant $\beta$ on 
$\mathcal B$.  We will later give the constant $\beta$ a convenient fixed value. 

Equations \eqref{euler1} can be written as 
\[
\nabla\left[ \tfrac{1}{2}|U|^2+gx_2+P \right]-\omega U^\perp=0\quad\text{in }\Omega
\]
where 
\[
\omega := \partial_{x_1}U_2-\partial_{x_2}U_1 = -\Delta \Psi
\] 
is the vorticity of the fluid.
This leads to Bernoulli's law, which says that $\tfrac{1}{2}|U|^2+gx_2+P $ is constant along streamlines.
In particular, the dynamic boundary condition
$P=P_{\mathrm{atm}}$ on $\mathcal S$ is equivalent to
\begin{align}
\label{C}
\tfrac 12 |\nabla\Psi|^2 +   gx_2  = C  \onn \mathcal S 
\end{align}
for some constant $C$.

Taking curl in both sides of \eqref{euler1} we arrive at the equation 
\begin{equation} \label{bracket}
\nn ^\perp \Psi\cdot \nn \omega = 0 
\end{equation}
which is satisfied whenever $\omega(x) = \vortfun(\Psi(x))$ for some fixed function $\vortfun$. 
Moreover, if \eqref{bracket} holds, then \eqref{euler1} can be uniquely solved for the pressure $P$, up to an additive constant.
In summary, any solution $\Psi(x)$ of an overdetermined elliptic problem of the form
\begin{align}
    \label{wv}
    \left \{
    \begin{aligned}
    -\Delta \Psi  &= \vortfun(\Psi(x)) &&\text{in } \Omega \\
    \Psi &=0 &&\text{on } \mathcal S \\
    \tfrac 12 |\nabla\Psi|^2 +   g x_2 &=    C &&\text{on }\mathcal S\\
    \Psi &=  \beta &&\text{on } \mathcal B
    \end{aligned}\right.
\end{align}
leads to a solution of the traveling water wave problem, where here $C,\beta$ and $g>0$ are constants and $\vortfun$ is a fixed function.

For the rest of the paper, we restrict attention to the case where the function $\vortfun$ in \eqref{wv}, and hence the vorticity $\omega = \vortfun(\Psi)$, are constant. We further specialize to \emph{solitary wave} solutions of \eqref{wv} for which, as $x_1 \to \pm \infty$, the region $\Omega$ becomes asymptotic to the strip 
\begin{align}
    \label{strip1}
    \Omega^S = \{ (x_1,x_2) \ | \  {-A}< x_2 < 0 \} 
\end{align}
and correspondingly, the stream function $\Psi(x_1,x_2)$ approaches a solution in the strip, which we assume to be 
\begin{align}
\label{psiS1}
    \Psi^S(x_2) =  -cx_2 - \tfrac{1}{2} \omega x_2^2 . 
\end{align}
Thus $A$ is the asymptotic depth of the fluid, while $c$ is the speed of the wave as measured in the (unique) reference frame where the fluid velocity vanishes along the surface at infinity. Sending $x_1 \to \pm\infty$ in \eqref{wv}, the dynamic boundary condition \eqref{C} yields $C = \tfrac 12 c^2$, while the kinematic condition on $\mathcal B$ becomes $\beta =  \Psi^S(-A) =  A  -\tfrac{1}{2} \omega A^2$.

Finally, we non-dimensionalize the problem by choosing $c$ as a velocity scale and $A$ as a length scale. In these units $c=1$ and $A=1$, and our problem becomes
\begin{align}
    \label{overd00}
    \left\{
    \begin{aligned}
    -\Delta \Psi &= \omega && \text{in }\Omega
    \\
    \Psi &= 0 &&\text{on } \mathcal{S}
    \\
    \tfrac{1}{2} |\nabla \Psi|^2 + g   x_2 &= \tfrac{1}{2} && \text{on } \mathcal{S} 
    \\
    \Psi &= \beta && \text{on } \mathcal{B}
    \\
    \Psi(x_1,x_2) &\to\Psi^S(x_2)   &&  \text{as } x_1\to\pm\infty ,
    \end{aligned}
    \right.
\end{align}
where $g > 0$ and $\omega \in \R$ are now dimensionless parameters and $\beta$ is given by
\begin{align}
    \label{def:B}
    \beta =  1  -\tfrac{1}{2} \omega.
\end{align}

\begin{figure}
    \centering
    \includegraphics[scale=1,page=1]{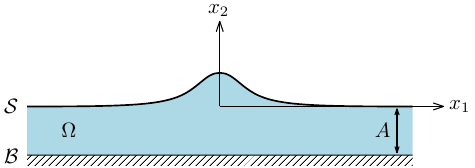}
    \caption{Fluid region $\Omega$ for a solitary wave. There are two boundary conditions on the unknown free surface $\mathcal S$, and only one on the fixed bed $\mathcal B$.}
    \label{fig:setup}
\end{figure}

We are interested in the existence of \emph{overhanging waves}, that is, solutions of \eqref{overd00} where $\mathcal{S}$ is not the graph of a function; see Figure~\ref{fig:overhanging}.
In our main theorem we will show that if $\omega>0$, overhanging waves with thin necks of the type shown in Figure~\ref{fig:overhanging}(c) exist for $g>0$ sufficiently small.

\subsection{Historical discussion}

The steady water wave problem \equ {wv} is classical, with an enormous literature spanning more than two centuries and including contributions from mathematicians such as Laplace, Lagrange, Cauchy, Poisson, and Levi-Civita. 
See \cite{darrigol-2003,craik-2004,craik-2005} for modern historical accounts and \cite{toland-1996,groves-2004,strauss-2010,2022-haziot-hur-strauss-toland-wahlen-walsh-wheeler-survey} for surveys of more recent developments. We note that \cite{2022-haziot-hur-strauss-toland-wahlen-walsh-wheeler-survey} includes 270 references.
The overwhelming majority of works concern the \emph{irrotational} case $\omega=0$ and waves which are periodic in the $x_1$-variable, known as \emph{Stokes waves}.
Stokes~\cite{stokes-1847} considered the problem in infinite depth, and conjectured the existence of a branch of nontrivial periodic waves bifurcating from a half-space and limiting to an extreme ``wave of greatest height''. 
This is a wave which saturates the bound $gx_2 \le \frac 12$ implied by \eqref{C} at its crest, which is therefore a \emph{stagnation point} where $\nabla \Psi = 0$. 
Stokes further conjectured that such a wave would have a singularity at its crest, namely a corner with an interior angle of 120 degrees.
The existence of a wave of greatest height was established by Toland \cite{toland-1978}, and the 120 degree angle was proved by Amick, Fraenkel and Toland \cite{amick-fraenkel-toland} and Plotnikov \cite{plotnikov-1982,plotnikov-1983}. See also \cite{amick-toland-1981a,amick-toland-1981b} for the solitary case. 
These constructions rely heavily on the assumption of zero vorticity, which allows for deep reformulations of the problem using complex variables including the Nekrasov~\cite{kuznetsov-2021} and Babenko~\cite{babenko-1987,longuet-higgins-1978} equations.

The surfaces $\mathcal S$ of the irrotational waves in the above constructions are all graphs. Moreover, the horizontal velocity $U_1 = \partial_{x_2}\Psi$ is non-vanishing inside the fluid, which implies that the other level curves of $\Psi$ (called \emph{streamlines}) are also graphs and that there are no interior stagnation points. 
Remarkably, Spielvogel showed that this holds not only for these waves but for arbitrary solutions satisfying some additional hypotheses~\cite{spielvogel,toland-1996}. The same result for 
waves with ``favorable'' constant vorticity $\omega < 0$ was proved by Constantin, Strauss, and V\u{a}rv\u{a}ruc\u{a}~\cite{constantin-strauss-varvaruca-2021}. In the irrotational case, Amick further showed that the slope of any streamline cannot exceed $31.15$ degrees~\cite{amick-1987}; 
see \cite{strauss-wheeler-2016,strauss-so-2018} for related but weaker results when $\omega < 0$ or $0 < \omega \ll 1$.

\begin{figure}
    \centering
    \includegraphics[scale=1,page=1]{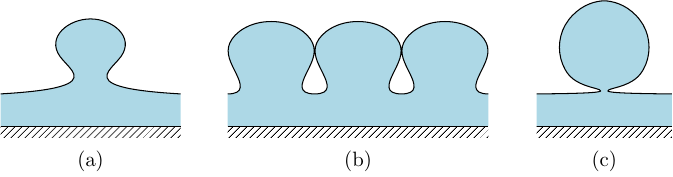}
    \caption{(a) An overhanging wave whose surface $\mathcal S$ is not a graph. (b) A ``touching wave'' whose surface self-intersects tangentially above each trough. (c) A wave with a thin neck connecting an approximate disk and an approximate strip.}
    \label{fig:overhanging}
\end{figure}

For large ``adverse'' vorticity $\omega > 0$, the situation changes dramatically. Numerical calculations for periodic waves in infinite depth by Simmen and Saffman \cite{simmen-saffman} show branches of solutions which terminate at ``touching waves'' whose surfaces $\mathcal S$ self-intersect tangentially above the trough as in Figure~\ref{fig:overhanging}(b). This is a new type of extreme wave which does not involve a corner singularity. A few years later, Teles da Silva and Pergrine \cite{teles-peregrine} calculated solutions in finite depth, and conjectured a third and extremely singular type of extreme wave in the limit $g \to 0$: the union of a disk and a strip. A nearly-limiting wave is sketched in Figure~\ref{fig:overhanging}(c); we will call this a wave with a thin neck. Subsequent numerics by Vanden-Broeck \cite{vandenbroeck-1994} for solitary waves indeed found such waves with small $g>0$, including even more exotic versions with multiple disks \cite{vandenbroeck-1995}. More recently, Dyachenko and Hur \cite{dyachenko-hur-sapm,dyachenko-hur-jfm} have computed periodic thin-necked waves for small $g$ involving as many as five disks per period. 

Rigorously establishing the existence of such overhanging waves is extremely difficult as they are not small-amplitude perturbations of a strip or half-space. 
Global bifurcation branches of periodic solutions emanating from a strip have been constructed by Constantin, Strauss, and V\u{a}rv\u{a}ruc\u{a} \cite{constantin-strauss-varvaruca-2016}, using a formulation which allows for overhanging solutions.
While it seems extremely likely that some of these branches contain overhanging waves, proving this remains an important open problem. See \cite{haziot-wheeler-2023} for a related result for solitary waves and \cite{wahlen-weber-2024} for periodic waves with general vorticity functions $\omega = \gamma(\psi)$.
We also mention that small-amplitude non-overhanging waves with interior stagnation points have been constructed, first in the periodic case \cite{wahlen-2009} and more recently in the solitary case \cite{kozlov-kuznetsov-lokharu-2020} under the additional assumption that $0 < g \ll 1$.

An alternative approach has opened up in recent years thanks to the discovery of a family of explicit overhanging periodic waves with zero gravity in infinite depth \cite{dyachenko-hur-sapm,dyachenko-hur-jfm,hur-vanden-broeck-2020,hur-wheeler-jfm}. The surfaces of these waves are the same as for Crapper's famous irrotational capillary waves \cite{crapper} (although the stream function $\Psi$ is entirely different), so that in particular they limit to a ``touching wave'' as in Figure~\ref{fig:overhanging}(b).
By perturbing these explicit solutions to $0 < g \ll 1$, Hur and Wheeler \cite{hur-wheeler-jde} gave the first rigorous construction of overhanging solutions to \eqref{wv}, and this argument has subsequently been refined and extended by Carvalho~\cite{carvalho}.
Crowdy \cite{crowdy} provides an elegant theory which simultaneously explains the explicit zero-gravity solutions of \cite{hur-wheeler-jfm} as well as solutions involving point vortices~\cite{crowdy-roenby-2014}. 
While the zero-gravity point vortex problem has many additional explicit solutions, including thin-necked waves resembling Figure~\ref{fig:overhanging}(c), related formulas seem to yield explicit solutions to \eqref{wv} only in the Crapper-type case covered by \cite{hur-wheeler-jfm}. 
It therefore appears that new ideas are needed in order to rigorously construct thin-necked waves with constant vorticity.

\subsection{Statement of the main result}
The numerical studies~\cite{vandenbroeck-1994,vandenbroeck-1995} suggest that the problem \eqref{overd00} has solutions $(\Omega_g,\Psi_g)$ with small necks as in Figure~\ref{fig:overhanging}(c), and that the asymptotic shape of $\Omega_g$ as $g\to 0$ is the union of the strip $\Omega^S$ and a disk 
\begin{align}
    \label{OmegaD0}
    \Omega^D = \big\{ \, (x_1,x_2) \in \R^2 \, | \, x_1^2+(x_2-R)^2=R^2\,\big\}
\end{align}
with radius $R>0$ whose boundary passes through the origin.
Let 
\begin{align*}
    \Psi^D(x) = \frac{\omega}{4}(R^2-|x-(0,R)|^2 ) .
\end{align*}
If $\omega R=2$, then $\Psi^D$ satisfies
\begin{align*}
    \left\{
    \begin{aligned}
    -\Delta \Psi^D &= \omega && \text{in }\Omega^D
    \\
    \Psi^D &=0 && \text{on }\partial\Omega^D
    \\
    |\nabla\Psi^D|^2 &=1 && \text{on }\partial\Omega^D .
    \end{aligned}
    \right.
\end{align*}
Thus the pair $(\Omega_*,\Psi_*)$, where
\begin{align*}
    \Omega_*=\Omega^S\cup\Omega^D,\quad
    \Psi_*(x)=\begin{cases}
    \Psi^D(x) & x \in \Omega^D
    \\
    \Psi^S(x) & x\in\Omega^S
    \end{cases}
    ,
\end{align*}
is formally a singular solution to the problem \eqref{overd00} when $g=0$.

The shape of the neck of $\Omega_g$ for small $g$ can be inferred from the following consideration. Let $\varepsilon>0$ be the width of the neck.
After possibly a small shift so that the neck is at the center of coordinates, let us assume that $\varepsilon^{-1} \Omega_g$ approaches a domain $\Omega^H$ and that 
\begin{align}
    \nonumber
    \lim_{g\to0} \frac{1}{\varepsilon}\Psi_g(\varepsilon y)= \psi^H(y) \quad\text{exists}.
\end{align}
Then $\psi^H$ should satisfy
\begin{align}
    \label{hairpin00}
    \left\{
    \begin{aligned}
    \Delta\psi^H&=0&&\text{in }\Omega^H
    \\
    \psi^H&=0&&\text{on }\partial \Omega^H
    \\
    |\nabla\psi^H|&=1&&\text{on }\partial \Omega^H .
    \end{aligned}
    \right.
\end{align}
A domain $\Omega^H$ and a function $\psi^H$ satisfying \eqref{hairpin00} and with the correct shape were already found by Hauswirth, H\'elein, and Pacard \cite{hauswirth-helein-pacard}. The domain
\begin{align}
    \label{def:hairpin0}
    \Omega^H = F(S)
\end{align}
is the image of the strip
\begin{align*}
    S = \Bigl\{ \, w\in \C \  \Big| \  \lvert \Re(w) \rvert  < \frac{\pi}{2} \, \Bigr\} 
\end{align*}
under the conformal mapping
\begin{align*}
    F(w)=w+\sin(w),
\end{align*}
while $\psi^H$ is given by
\begin{align}
    \label{psiH-0}
    \psi^H(y) = \Re (\cos (w)), \quad y = F(w).
\end{align}
Hauswirth, H\'elein, and Pacard called $\Omega^H$ an exceptional domain if there is $\psi^H$ satisfying \eqref{hairpin00} and $\psi^H>0$. A classification of exceptional domains in the plane whose boundary has a finite number of components was obtained by Traizet \cite{traizet}. Using this classification, Jerison and Kamburov \cite{jerison-kamburov} analyzed the structure of one-phase free boundaries in the plane. Their result, though probably requiring non-trivial adaptations, could potentially justify the previous formal limit as $g\to 0$, provided the solution is known to exist.

Our main result is the following.
\begin{theorem}
\label{thm}
Assume that  $\omega R=2$ and $\omega \Ae<1$.
For $g>0$ sufficiently small, problem \eqref{overd00} has a solution $(\Omega_g$, $\Psi_g)$ such that for any $\delta>0$, $\Omega_g$ approaches $\Omega_*$ and $\Psi_g \to \Psi_*$ in $\R^2\setminus B_\delta(0)$ as $g\to 0$. Moreover there is $c_0>0$ and 
$\varepsilon = c_0 g(1+o(1))$ as $g\to0$ such that $\frac1\varepsilon \Omega_g \, \cap \, K $ approaches $\Omega^H \cap K$ for any compact set $K$   and 
\begin{align*}
\lim_{g\to0} \frac{1}{\varepsilon}\Psi_g(\varepsilon y)= \psi^H(y),
\end{align*}
uniformly on compact sets of $\Omega^H$.
Both $\Omega_g$ and $\Psi_g$ are even with respect to $x_1$.
\end{theorem}

The method developed in this paper is constructive and does not rely on conformal maps, although complex variables do play an important role in the analysis of some building blocks.

There are a variety of directions in which Theorem~\ref{thm} might be generalized. 
For instance, one could consider the analogous problem for non-constant vorticity functions $\omega=\vortfun(\Psi)$ associated with more complicated shear flows $\Psi^S$ and radial vortices $\Psi^D$.
Rather than considering a single disk, another possibility is to consider periodic solutions with an infinite horizontal array of disks. 
Yet another possibility suggested by the numerics \cite{vandenbroeck-1995,dyachenko-hur-jfm} is to glue disks to the crests of smooth zero-gravity waves, for instance the explicit solutions in \cite{hur-wheeler-jfm}.
It seems likely that scheme presented in this paper can be applied in all of the above cases.
Treating solutions with multiple disks stacked vertically, which are also seen numerically \cite{vandenbroeck-1995,dyachenko-hur-jfm}, appears more challenging, and our approach would require significant adaptations.

\subsection{Analogy with constant mean-curvature surfaces}
We are inspired by a deep connection between solutions of overdetermined problems like \eqref{overd00} and the theory of constant mean-curvature (CMC) surfaces. This connection is already present when comparing the result of Alexandrov \cite{alexandrov}, which establishes that compact embedded CMC hypersurfaces must be spheres, with Serrin's \cite{serrin} symmetry result. More precisely, Serrin considered the problem
\begin{align}
    \label{overdCMC}
    \left\{
    \begin{aligned}
    -\Delta \Psi &= \gamma (\Psi) && \text{in }\Omega
    \\
    \Psi&>0 && \text{in }\Omega
    \\
    \Psi &= 0 &&\text{on } \partial\Omega
    \\
    |\nabla \Psi|&= 1 && \text{on } \partial\Omega
    \end{aligned}
    \right.
\end{align}
where $\gamma$ is a given Lipschitz function. He proved that if $\Omega$ is bounded and connected, and there exists $\Psi$ satisfying \eqref{overdCMC}, then $\Omega$ must be a ball. It is interesting to note that Serrin's result was partly motivated by fluid mechanics. 

Extensions of \cite{serrin} to unbounded domains were obtained in \cite{reichel,aftalion-busca,berestycki-caffarelli-nirenberg}. This led Berestycki, Caffarelli, and Nirenberg \cite{berestycki-caffarelli-nirenberg} to conjecture that for any solution of \eqref{overdCMC} with $\Psi$ bounded and the complement of $\Omega$ connected, $\Omega$ must be a half space, ball, complement of a ball, or a circular cylinder domain of the form $\R^j \times B$, with $B$ a ball.
Traizet \cite{traizet} obtained a classification for the case $n=2$ and $\gamma =0$, that is, 
\begin{align}
    \label{overdMS}
    \left\{
    \begin{aligned}
    \Delta \Psi &= 0 && \text{in }\Omega
    \\
    \Psi&>0 && \text{in }\Omega
    \\
    \Psi &= 0 &&\text{on } \partial\Omega
    \\
    |\nabla \Psi| &= 1 && \text{on } \partial\Omega,
    \end{aligned}
    \right.
\end{align}
assuming that $\partial\Omega$ has a finite number of components, or that $\Omega$ is periodic in one direction. He obtained the classification by establishing a correspondence between solutions of \eqref{overdMS} and a class of minimal surfaces in $\R^3$. Under this correspondence the hairpin solution $(\Omega^H,\psi^H)$ in \eqref{def:hairpin0} is mapped to the catenoid. In particular, Traizet  rediscovered a family of periodic exceptional domains found in~\cite{baker-saffman-sheffield}.

Counterexamples to the conjecture in \cite{berestycki-caffarelli-nirenberg} were obtained by Sicbaldi \cite{sicbaldi} for $\gamma(\Psi)=\lambda\Psi$ and and Fall, Minlend, and Weth \cite{fall-minlend-weth} for $\gamma \equiv\text{const}\not=0$. 
In \cite{delpino-pacard-wei-2015} a counterexample was found for a corresponding conjecture for epigraphs raised in \cite{berestycki-caffarelli-nirenberg} in large dimensions. 

The Delaunay unduloids \cite{Delaunay1841} are CMC surfaces in $\R^3$  which come parametrized as a smooth family $\{\mathcal{D}_\tau\}_{\tau\in[0,1)}$ such that $\mathcal{D}_0$ corresponds to a cylinder and, as $\tau$ approaches 1, $\mathcal{D}_\tau$ converges to an infinite union of touching spheres arranged periodically along an axis, say the ``vertical direction''.
Under this analogy, the constructions in \cite{sicbaldi,fall-minlend-weth} correspond to Delaunay surfaces $\mathcal{D}_\tau$ with $\tau>0$ small, in the sense that the domains obtained are perturbations of the straight solid cylinder in $\R^j\times B$ which are rotationally symmetric and periodic in the vertical direction. A rigorous correspondence between general nondegenerate CMC surfaces (including the unduloid) and the overdetermined problem for a class of  Lipschitz nonlinearities $\gamma(\Psi)$, was also established in \cite{delpino-pacard-wei-2015}.

So far, the analogy between CMC surfaces and solutions to overdetermined problems includes symmetry results, Traizet's rigorous correspondence, and the existence of unduloid-type domains by bifurcation from cylindrical ones. In this paper we explore a new aspect of this analogy: constructing new solutions using gluing techniques.
The behavior of the numerical solutions $(\Omega_g,\Psi_g)$ of \eqref{overd00} is illustrated in Figure~\ref{fig:overhanging}(c),
where one can observe a singular limit and a profile upon zooming in, evoking the Costa--Hoffman--Meeks surfaces when the genus tends to infinity \cite{hoffman-meeks-1990}. This analysis, in turn, is related to the construction of new CMC surfaces by desingularization, 
through the ``gluing'' of building blocks, such as spheres and pieces of Delaunay surfaces, by Kapouleas \cite{kapouleas-annals-1990, kapouleas-1997}, Mazzeo, Pacard, and Pollack \cite{mazzeo-pacard-pollack}, and many others \cite{davila-delpino-nguyen,haskins-kapouleas,pacard-scherk,kapouleas-doubling}. 
Some similarities and differences with \cite{mazzeo-pacard-pollack} are discussed later on in Remark~\ref{rem:pacard}.

In this work we develop a related desingularization strategy in the context of the overdetermined problem \eqref{overd00} and also \eqref{overdCMC}, which consists in gluing some known building blocks. 
This method allows us to construct the domain in Theorem~\ref{thm} and also to find Delaunay-type domains analogous to $\mathcal D_\tau$  satisfying \eqref{overdCMC} for all $\tau$ close to 1. Let $R>0$ and
\begin{align*}
\Omega_*&=
\bigcup_{k\in\Z} B_R(0,2Rk)\subset\R^2,
\\
\Psi_*(x) &= \frac{1}{2R}(R^2-|x|^2),\quad x\in B_R(0),
\end{align*}
extended periodically to $\Omega_*$.
Using the same strategy as in Theorem~\ref{thm}, we obtain:
\begin{theorem}
\label{thm2}
Let $\gamma(\Psi)=\omega$ where $\omega$ is a positive constant and $R=\frac{2}{\omega}$.
For all $\varepsilon>0$ small there exist a connected open set  $\Omega_\varepsilon\subset\R^2$ and $\Psi_\varepsilon$ satisfying \eqref{overdCMC} such that:
\begin{itemize}
\item
$\Omega_\varepsilon$, $\Psi_\varepsilon$ are periodic in $x_2$ with period $T_\varepsilon = 2R+O(\varepsilon\lvert\log\varepsilon|)$, and they are even in $x_1$.
\item
For any $\delta>0$, $\Omega_\varepsilon$ approaches $\Omega_*$ and $\Psi_\varepsilon \to \Psi_*$ in $\R^2\setminus  \bigcup_{k\in\Z} B_\delta(0,T_\varepsilon/2+T_\varepsilon k)$ as $\varepsilon\to 0$.
\item
We have
\begin{align*}
\lim_{\varepsilon\to0} \frac{1}{\varepsilon}\Psi_\varepsilon((0,T_\varepsilon/2)+\varepsilon y)= \psi^H(y),
\end{align*}
uniformly for $y$ on compact sets.
\end{itemize}
\end{theorem}

The strategy of the proof of Theorems~\ref{thm} and \ref{thm2} consists in building an approximate solution to \eqref{overd00} and then perturbing it into a true solution. All the sections below are focused towards a proof of Theorem~\ref{thm}. The proof of Theorem~\ref{thm2} requires only minor modifications.

\subsection{Outline of the paper}
In Section~\ref{sectFormal}, we discuss the formal linearization of \eqref{wv}, showing it to be the Poisson equation with a Robin boundary condition. 
Calculations of this general type appear in various forms throughout the water wave literature, especially for irrotational perturbations of strips and half-spaces; see for instance \cite{stoker-1957,mclean-1982,francius-kharif-2017,kozlov-lokharu-2023}.
We briefly discuss this linearization around some simple solutions to \eqref{wv}, which are later used as the building blocks of the construction. 
In Section~\ref{sec:approx}, we construct the approximate solution $(\Omega_0,\psi_0)$ to \eqref{wv} in the scaled variable $y=\frac{x}{\varepsilon}$, where the neck size is $O(1)$ and $\varepsilon>0$ is a small parameter. This approximation is formed by combining, with cut-off functions, a scaled version of the strip $\Omega^S$ \eqref{strip1}, a scaled disk $\Omega^D$ \eqref{OmegaD0}, and the hairpin $\Omega^H$ \eqref{def:hairpin0}. The strip and disk also require vertical shifts, which in the expanded variable $y=\frac{x}{\varepsilon}$ are $O(\lvert\log\varepsilon|)$.
In Section~\ref{sect:operators}, we show how solving a nonlinear Robin problem leads to a solution of the overdetermined problem. In Section~\ref{sec:LinH} we solve the linearized equation about the hairpin. We also introduce a perturbation of the hairpin and study the Robin problem there. In Sections~\ref{sec:TildeOD} and \ref{sect:S} we introduce perturbations of the disk and strip and study the associated Robin problems. The purpose of these perturbed domains is to help reduce the Robin problem in $\Omega_0$ to a system of equations. This is done in Section~\ref{sec:linear}, where the linearized problem about $(\Omega_0,\psi_0)$ is analyzed. How well $(\Omega_0,\psi_0)$ approximates a solution of the scaled version of \eqref{wv} is studied in Section~\ref{sect:error}. Finally in Section~\ref{sec:proof} we give the proof of Theorem~\ref{thm}. Much finer information on the solution $(\Omega,\psi)$ is provided in the proof. Finally, appendix~\ref{sect-p2} contains some computations that were used in previous sections.

\section{Formal linearization}
\label{sectFormal}

In this section we formally obtain the linearization of the overdetermined problem 
\begin{align}
    \label{overd1ax}
    \left\{
    \begin{aligned}
    -\Delta \Psi &= \omega && \text{in }\Omega
    \\
    \Psi &= 0 && \text{on } \partial\Omega
    \\
    |\nabla \Psi|^2 &= 1 && \text{on } \partial\Omega  ,
    \end{aligned}
    \right.
\end{align}
assuming for simplicity that $\Omega$ is bounded and $\partial\Omega$ is connected.
The resulting Robin problem~\eqref{robin3a} is posed in the unperturbed domain $\Omega_0$ and involves a single unknown function, features which will be shared by our reformulation of the full nonlinear problem in Section~\ref{sect:operators}.
Formal linearzations of the irrotational water wave problem about strips and half-spaces are  classical \cite{stoker-1957}, and linearizations about more general solutions often appear in numerical studies of stability \cite{mclean-1982,francius-kharif-2017}; also see \cite[Section~5]{kozlov-lokharu-2023}.

Suppose that $\Omega_0$, $\Psi_0$ is a solution of \eqref{overd1ax} with $\omega = \omega_0$ a constant, that is,
\begin{align}
    \label{overd1a}
    \left\{
    \begin{aligned}
    -\Delta \Psi_0 &= \omega_0 && \text{in }\Omega_0
    \\
    \Psi_0 &= 0 && \text{on } \partial\Omega_0
    \\
    |\nabla \Psi_0|^2 &= 1 && \text{on } \partial\Omega_0  .
    \end{aligned}
    \right.
\end{align}
We introduce an artificial small parameter $\delta$, and consider a perturbation of \eqref{overd1a} of the form
\begin{align}
    \label{overd1}
    \left\{
    \begin{aligned}
    -\Delta \Psi &= \omega(x,\delta) && \text{in }\Omega
    \\
    \Psi &= 0 && \text{on } \partial\Omega
    \\
    |\nabla \Psi|^2 &=  f(x,\delta) && \text{on } \partial\Omega ,
    \end{aligned}
    \right.
\end{align}
where $f(x,\delta)$ and $\omega(x,\delta)$ are smooth functions with
\begin{align*}
    \omega(x,0) &= \omega_0,
    \quad
    f(x,0)=1.
\end{align*}
Moreover, we assume that \eqref{overd1} has a smooth family of solutions $(\Omega_\delta, \Psi_\delta)$ passing through $(\Omega_0,\Psi_0)$ at $\delta = 0$.

Let $\gamma\colon [0,L]\to \R^2$ be an arc-length parametrization of $\partial \Omega_0$, periodic of period $L>0$. 
We choose an orientation for $\gamma$ so that 
\[
    \nu(s) = \gamma'(s)^\perp 
\]
points out of $\Omega_0$ where $(a,b)^\perp = (b,-a)$ is rotation by $\frac{\pi}{2}$ in the clockwise direction. We write $T(s) = \gamma'(s)$. The second derivative of $\gamma$ gives the curvature $\kappa$ of $\partial \Omega_0$, and we take as definition
\begin{align}
    \label{curvature}
    \gamma'' = T'(s) = - \kappa \nu ,
\end{align}
so that if $\Omega_0$ is a disk of radius $R$, then $\kappa=\frac{1}{R}>0$. 

Next we prove a pair of identities for $\Psi_0$. From the two boundary conditions in \eqref{overd1a}, we see that 
\begin{align}
    \label{psi0pm}
    \nabla\Psi_0 = \pm \nu 
    \quad \text{on } \partial\Omega_0 .
\end{align}
Moreover, differentiating the Dirichlet condition twice and using \eqref{curvature} yields 
\begin{align*}
    0 = \frac{d^2}{ds^2} \Psi_0(\gamma(s))
    = D^2 \Psi_0 T \cdot T - \kappa \nabla\Psi_0 \cdot \nu
    = ( - \omega_0 - D^2 \Psi_0 \nu \cdot \nu) - \kappa\nabla\Psi_0 \cdot \nu
    \quad \text{on } \partial\Omega_0,
\end{align*}
where $D^2 \Psi_0$ is the Hessian matrix of $\Psi_0$ and in the last step we have used the fact that 
$-\Delta\Psi_0 = \omega_0$.
Solving for $D^2 \Psi_0 \nu \cdot \nu$ yields
\begin{align}
    \label{psi0nunu}
    D^2 \Psi_0 \nu \cdot \nu 
    =  -  \omega_0  - \kappa\nabla\Psi_0 \cdot \nu
    = - \omega_0 \mp \kappa
\end{align}
where the $\mp$ refers to the sign in \eqref{psi0pm}.

To describe the domains $\Omega_\delta$,
for $h \colon \partial\Omega_0 \to \R$ sufficiently small, let $\Omega_h$ be the domain close to $\Omega_0$ with boundary given by 
\[
    \partial \Omega_h = \{ \, x - \nu(x) h(x) \, | \, x \in \partial \Omega_0 \, \},
\]
where $\nu$ is the exterior unit normal vector to $\Omega_0$. We assume that $\Omega_\delta = \Omega_{h_\delta}$ for a smooth family of functions $h_\delta$ on $\partial\Omega_0$.
Since $h_\delta$ is a function on $\partial \Omega_0$, we can regard it also as a function of $s$, so that 
\begin{align}
    \label{defm-bdy}
    s \mapsto \gamma(s) - \nu(s) h_\delta(s)
\end{align}
is a parametrization of $\partial \Omega_\delta$.

Writing $\Psi_\delta(x)$, $h_\delta(x)$ as $\Psi(x,\delta)$, $h(x,\delta)$,  our objective is to formally compute 
\begin{align}
    \label{psi1h1}
    \Psi_1 = \frac{\partial \Psi}{\partial \delta} \Big|_{\delta=0}, \quad
    h_1  = \frac{\partial h}{\partial \delta} \Big|_{\delta=0}.
\end{align}
Differentiating $ - \Delta \Psi(x,\delta) = \omega(x,\delta) $ is straightforward, and gives
\begin{align*}
    -\Delta \Psi_1= \frac{\partial \omega}{\partial \delta}(x,0) \quad \text{in } \Omega_0 .
\end{align*}
The boundary conditions are more subtle.
Writing the Dirichlet condition in \eqref{overd1} as
\[
    \Psi\big(\gamma(s) - \nu(s) h(s,\delta),\delta\big)=0
\]
and differentiating with respect to $\delta$ at $\delta=0$, we find
\begin{align}
\nonumber
    - \nabla \Psi_0 \cdot \nu h_1 + \Psi_1 = 0
    \quad \text{on }\partial \Omega_0,
\end{align}
with the abuse of notation $h_1(x)= h_1(s)$ for $x = \gamma(s) \in \partial \Omega_0$.
Using \eqref{psi0pm}, this simplifies to
\begin{align}
    \label{linear-rel-psi-h}
    h_1 = \pm\Psi_1 \quad \text{on }\partial \Omega_0.
\end{align}

Finally we turn to the Neumann condition in \eqref{overd1}, which reads
\[
    \bigl|\nabla_x \Psi(\gamma(s) - \nu(s) h(s,\delta),\delta)\bigr|^2 = f(\gamma(s) - \nu(s) h(s,\delta),\delta).
\]
Differentiating in $\delta$ and setting $\delta=0$ gives
\[
    \nabla \Psi_0 \cdot ( - D^2 \Psi_0 \nu h_1 + \nabla \Psi_1) = - \frac 12 \nabla f \cdot \nu h_1 + \frac 12 \frac{\partial f}{\partial \delta}(x,0)   \quad \text{on }\partial \Omega_0,
\]
where we have divided by a factor of $2$.
Since $f(x,0)=1$, the first term on the right hand side vanishes. Inserting \eqref{linear-rel-psi-h} and then \eqref{psi0nunu} into the left hand side we find
\begin{align*}
    \nabla \Psi_0 \cdot ( - D^2 \Psi_0 \nu h_1 + \nabla \Psi_1)
    =
    (\kappa \pm \omega_0 ) \Psi_1 
    \pm \nu \cdot \nabla \Psi_1 
    =
    \pm \Big( \frac{\partial\Psi_1}{\partial \nu}
     + \big(\kappa + (\nabla \Psi_0 \cdot \nu)\omega_0 \big) \Psi_1 \Big).
\end{align*}
Putting everything together, we  obtain the Robin problem
\begin{align}
    \label{robin3a}
    \left\{
    \begin{aligned}
    - \Delta \Psi_1 &= f_1 \quad \text{in }\Omega_0
    \\
    \frac{\partial \Psi_1}{\partial \nu} + (\kappa +(\nabla \Psi_0 \cdot \nu) \omega_0 ) \Psi_1  &= f_2
    \quad \text{on }\partial \Omega_0 ,
    \end{aligned}
    \right.
\end{align}
where $f_1 = \partial \omega/\partial \delta$ and $f_2 = \pm \partial f/\partial \delta$ are fixed functions. 
The variation in the domain is captured by $h_1$, defined in \eqref{psi1h1} and given in terms of $\Psi_1$ by  \eqref{linear-rel-psi-h}.

Note that replacing $\Psi_0$ by $-\Psi_0$ and $\omega_0$ by $-\omega_0$ leaves both \eqref{overd1a} and the left hand side of \eqref{robin3a} unchanged.

\medskip
Let us consider some examples.

\subsection{The disk}
If $\Omega_0 =B_R(0)$ is a disk of radius $R$, centered at the origin and $\omega_0\not=0$ is a constant, then 
\[
\Psi_0(x) = \frac{\omega_0}{4}(R^2 - |x|^2) ,
\]
satisfies \eqref{overd1a}.
To achieve $|\nabla \Psi_0|=1$ on $\partial B_R(0)$ we need $\frac{1}{2} |\omega_0| R =1$. In this case \eqref{robin3a} takes the form:
\begin{align}
\label{robin3}
\left\{
\begin{aligned}
- \Delta \Psi_1 &= f_1 \quad \text{in }B_R(0)
\\
\frac{\partial \Psi_1}{\partial \nu} - \frac{1}{R} \Psi_1  &= f_2
\quad \text{on }\partial B_R(0) .
\end{aligned}
\right.
\end{align}
This linear problem has a kernel. Indeed 
\begin{align}
\label{Zdisc}
Z_1^D = \frac{\partial \Psi_0}{\partial x_1}
= -\frac{\omega_0}{2}x_1, 
\quad 
Z_2^D = \frac{\partial \Psi_0}{\partial x_2} 
= -\frac{\omega_0}{2}x_2 
\end{align}
satisfy the homogeneous version of \eqref{robin3}.
For sufficiently regular $f_1$ and $f_2$, \eqref{robin3} has a solution if and only if
\begin{align*}
\int_{B_R(0)} f_1 Z_j^D = \int_{\partial B_R(0)} f_2 Z_j^D, \quad j=1,2 .
\end{align*}

\subsection{The strip}
Let us consider the strip $\Omega^S$ defined in \eqref{strip1} (with $A=1$) with boundary $\partial \Omega^S $ given by the lines $x_2=-\A$ and $x_2=0$. The function $\Psi^S$ defined in \eqref{psiS1} (with $c=1$) satisfies
\begin{align}
\label{strip00}
\left\{
\begin{aligned}
-\Delta \Psi^S &= \omega \quad \text{in } \Omega^S
\\
\Psi^S &= 0 \quad \text{on }  x_2=0
\\
|\nabla \Psi^S|& =1\quad \text{on } x_2=0
\\
\Psi^S &= \beta  \quad \text{on }  x_2=-\A,
\end{aligned}
\right.
\end{align}
where $\beta$ is given by \eqref{def:B}.
The linearized version of \eqref{strip00} is
\begin{align}
\label{strip0}
\left\{
\begin{aligned}
-\Delta \phi &= f_1 && \text{in } \Omega^S
\\
\frac{\partial \phi}{\partial \nu} - \omega \phi &= f_2 && \text{on } x_2=0
\\
\phi &= 0  && \text{on }  x_2=-\A.
\end{aligned}
\right.
\end{align}
This linear problem is well behaved under the assumption $\omega < 1$.
Let us fix $\mu>0$ such that
\begin{align}
\label{eq:mu}
\omega < \frac{\mu}{\tan(\mu)}.
\end{align}
We claim that if $f_1$ and $f_2$ have exponential decay $\exp(-\mu |x_1|)$ then there is a solution $\phi$ of \eqref{strip0} with the same exponential decay. This is proved using the barrier
\begin{align}
\label{superstrip}
\bar u(x_1,x_2)= \sin(\lambda (x_2+\A)) \exp(-\mu |x_1|),
\end{align}
where $\lambda>\mu$ is such that 
\begin{align*}
\omega \Ae < \frac{\lambda \Ae }{\tan(\lambda \Ae)}.
\end{align*}

\subsection{The hairpin}
Let $\Omega^H$ be the domain \eqref{def:hairpin0} and $\psi^H$ be given by \eqref{psiH-0}.
The linearized version of \eqref{hairpin00} is
\begin{align}
\label{linearH0}
\left\{
\begin{aligned}
\Delta_y \phi  &= f_1 \quad 
\text{in }  \Omega^H 
\\
\frac{\partial \phi}{\partial  \nu_y} 
+ \kappa \phi
& =f_2 
\quad \text{on } \partial \Omega^H  ,
\end{aligned}
\right.
\end{align}
where $\kappa$ is the curvature of $\partial\Omega^H$. 

The homogeneous version of \eqref{linearH0} has two independent solutions
\begin{align}
\nonumber
Z_1^H = \frac{\partial \psi^H}{\partial y_1} , 
\quad 
Z_2^H = \frac{\partial \psi^H}{\partial y_2} ,
\end{align}
which satisfy
\begin{alignat*}{2}
Z_1^H(y) &\to 0 &\quad& \text{as }|y|\to\infty\\
Z_2^H(y_1,y_2) &\to \pm1 &\quad& \text{as }y_2\to\pm\infty .
\end{alignat*}
In the rest of the paper we work assuming even symmetry with respect to the horizontal direction so that orthogonality of the right hand sides with respect to $Z_1^H$ is automatic.
We give a construction of solutions to \eqref{linearH0} for right hand sides with algebraic decay in Lemma~\ref{lemmaH1}.

\section{Approximate solution}
\label{sec:approx}

In this section we construct an approximate solution $\Omega_0$, $\psi_0$ in the scaled variable $y=\frac{x}{\varepsilon}$, that is, we construct them so that
\begin{align*}
\begin{aligned}
\Delta_y \psi_0 + \varepsilon \omega &\approx 0 && \text{in }\Omega_0
\\
\psi_0 &\approx 0 && \text{on } \mathcal{S}_0
\\
\frac{1}{2} |\nabla_y \psi_0 |^2 + g  (\varepsilon y_2+d^S) - \frac{1}{2} & \approx 0 && \text{on }  \mathcal{S}_0
\\
\psi_0 &= \frac{\beta}{\varepsilon} && \text{on }  \mathcal{B}_0
\end{aligned}
\end{align*}
in a sense to be defined later on.
The domains and stream functions constructed here are all symmetric (even) with respect to $x_1$.
The parameter $d^S$ corresponds to a vertical shift that will be defined later.

The construction is essentially by gluing 3 domains together with their stream functions. These domains are the hairpin $\Omega^H$ defined in \eqref{def:hairpin0}, which can also be written as 
\begin{align}
\label{defHairpin}
\Omega^H = \Bigl\{  y \in \R^2 \ \Big| \ |y_1|< \frac{\pi}{2}+ \cosh y_2 \Bigr\},
\end{align}
the disk 
\begin{align}
\label{OmegaD}
\Omega^D = \{ x\in \R^2 \ | \ |x-(0,R+d^D)|< R \},
\end{align}
where $d^D>0$ is to be chosen later on (see \eqref{dDeps}), and the strip
\begin{align}
\label{def:OmegaS0}
\Omega^S = \{ (x_1,x_2) \ | \ -d^S- \A < x_2 < -d^S\}
\end{align}
with boundary $\partial\Omega^S=\mathcal{S}\cup\mathcal{B}$, 
\begin{align}
\label{pOmegaS}
\mathcal{S}&=\{ \ x_2=-d^S\ \}, \quad
\mathcal{B}=\{ \ x_2=-\A-d^S\ \},
\end{align}
where $d^S$ is chosen later (see \eqref{dSeps}).
We note that $\Omega^S$ has been shifted vertically by the amount $d^S$ compared to the definition in \eqref{strip1}, and also the disk $\Omega^D$ has been shifted vertically by $d^D$ compared to \eqref{OmegaD0}.

The stream function $\psi^H$ associated to $\Omega^H$ is defined by
\begin{align}
\label{psiH}
\psi^H(y) = \Re (\cos (w)), \quad y = F(w) .
\end{align}
The stream function $\psi^D$ corresponding to $\Omega^D$ is
\begin{align}
\label{PsiD}
\Psi^D(x) = \frac{\omega}{4}(R^2 - | x- (0,R+d^D)|^2) .
\end{align}
Assuming that $\omega R = 2$, we find
\[
|\nabla \Psi^D(x)|=1,\quad x\in \partial \Omega^D.
\]
The stream function associated to the strip is
\[
\Psi^S(x)  = -\tfrac{1}{2}\omega (x_2+d^S)^2 - (x_2+d^S) , \quad x=(x_1,x_2).
\]

The expanded version of the disk is 
\begin{align}
\nonumber
\frac{1}{\varepsilon} \Omega^D = \{ y \in \R^2 \ | \ |y-(0,(R+d^D) \varepsilon^{-1})|< R \varepsilon^{-1} \} 
\end{align}
with the rescaled stream function
\begin{align}
\label{psiD}
\psi^D(y) := 
\frac{1}{\varepsilon}\Psi^D(\varepsilon y)
=\frac{1}{4}\omega \varepsilon  (R^2 \varepsilon^{-2} - | y- (0,R+d^D) \varepsilon^{-1} |^2).
\end{align}
The scaled strip and its stream function are given by 
\begin{align}
\nonumber
\frac{1}{\varepsilon}\Omega^S = \{ (y_1,y_2) \ | \ {}- \varepsilon^{-1} d^S- \varepsilon^{-1}\Ae < y_2 < -\varepsilon^{-1}d^S\}
\\
\label{psiS}
\psi^S(y)= \frac{1}{\varepsilon} \Psi^S(\varepsilon y) = - \frac{1}{2}\varepsilon\omega ( y_2 + \varepsilon^{-1} d^S)^2- ( y_2 + \varepsilon^{-1} d^S) .
\end{align}

The construction of the approximate solution has the following steps:
\begin{itemize}
\item[(i)]
We compute the expansion of $\psi^H(y)$ as $|y|\to\infty$. 
\item[(ii)]
We introduce $\tilde\psi^H(y)$, an improved version of $\psi^H(y)$, to achieve $-\Delta_y \tilde \psi^H=\varepsilon \omega$, and $\tilde\psi^D$, $\tilde\psi^S$, which are improved versions of $\psi^D$ and $\psi^S$, to achieve a better matching of  $\tilde\psi^H(y)$ and $\tilde\psi^D(y)$, $\tilde\psi^S(y)$ in the regime $1\ll |y| \ll\infty$.
The modifications of the stream functions go together with modifications of the domains, which are denoted by $\Omega^H_\varepsilon$, $\Omega^D_\varepsilon$ and $\Omega^S_\varepsilon$.
\item[(iii)]
We construct the approximate domain and stream function by interpolating the corrected domains and stream functions with cut-offs in the region $\varepsilon^{-b}\leq|y|\leq 2 \varepsilon^{-b}$ where $b$ is a constant in the range
\begin{align*}
\tfrac{1}{2}<b<1.
\end{align*}
\end{itemize}

It is convenient to regard $\psi^H$ defined by \eqref{psiH} extended analytically to a region larger than $\Omega^H$.
For $0 < \delta<\frac{\pi}{2}$ let us define
\begin{align*}
\Omega^{H,+}_\delta  &= \Bigl\{ z \in \C \  \Big| \  z = w+\sin(w), \ 
\Im(w)>0 , \ \Re(w)  \in \Bigl(-\frac{\pi}{2}-\delta,\frac{\pi}{2}+\delta\Bigr) \Bigr\} ,
\\
\Omega^{H,-}_\delta &= \Bigl\{ z \in \C \  \Big| \  z = w+\sin(w), \ 
\Im(w)<0 , \ \Re(w)  \in \Bigl(-\frac{\pi}{2}-\delta,\frac{\pi}{2}+\delta\Bigr)\Bigr\}.
\end{align*}
\begin{lemma}
Let $0 < \delta<\frac{\pi}{2}$.
Then $\psi^H$ has a unique harmonic extension to $\Omega^{H,+}_\delta $ and another unique harmonic extension to $\Omega^{H,-}_\delta $.
\end{lemma}
\begin{proof}

For $M>0$ let $S_M = \{ w\in \C \, | \, |\Re(w)| < \frac{\pi}{2}+\delta,\ 0 < \Im(w) < M\, \}$. It is sufficient to verify that $F$ is injective on $S_M$.

Note that $F'(w) \not= 0$ for all $w$ in a neighborhood of $\overline{S}_M$. The curve $F(\partial S_M)$ is closed, simple, and it is easy to check that the degree of $F\colon \partial S_M \to F(\partial S_M)$ is one. 
The region $F(S_M)$ is bounded, open, and connected. Hence the degree $d(F, S_M, z) = 1$ for all $z \in F(S_M)$. Since $F$ is analytic, for any $z \in F(S_M)$, there is only one $w \in S_M$ such that $z = F(\omega)$.

One can also argue using the Darboux-Picard theorem, see for example \cite[p.~310]{burckel}.
\end{proof}

In particular, $\psi^H$ has a harmonic extension to the set $\R^2 \setminus \{ \, (y_1,0)\, | \, |y_1| \geq \pi \,\}$, which is the one we use in the rest of the paper.

\medskip

We need the asymptotic expansion of $\psi^H(y)$ as $|y|\to\infty$. This is the content of the next lemma, whose proof is given in Appendix~\ref{sect-p2}.

\begin{lemma}
\label{lemma:estGradPsiH}
Let $\psi^H$ be the harmonic extension to $\Omega^{H,+}_\delta$ of the function  \eqref{psiH}. 
For $0 < \delta<\frac{\pi}{2}$ there is $C$ such that for $z \in \Omega^{H,+}_\delta$ with $|z|\geq 1$, the following estimates hold
\begin{align}
    \label{est0}
    \left| \psi(z) - [z_2-\log(2|z|)] \right|
    &
    \leq C\frac{\log |z|}{|z|} 
    \\
    \label{est1a}
    \Bigl| \partial_{z_1} \psi^H(z) + \frac{z_1}{|z|^2} \Bigr|
    &
    \leq C \frac{\log|z|}{|z|^2}
    \\
    \label{est2a}
    \Bigl| \partial_{z_2} \psi^H(z) - 1  \Bigr| 
    &
    \leq  \frac{z_2}{|z|^2} \Bigl( 1 + C\frac{\log |z|}{|z|^2} \Bigr)
    \\
    \label{est3}
    \Bigl| \partial_{z_2} \psi^H(z) - 1  - \frac{z_2}{|z|^2}\Bigr| 
    &
    \leq  C\frac{\log |z|}{|z|^{2}}
    \\
    \label{D2psiH-1}
    \Bigl| D^2_z( \psi^H(z) - z_2  - \log|z|)\Bigr| 
    &
    \leq  C\frac{\log |z|}{|z|^{3}}
    \\
    \label{D2psi}
    |D^k_z \psi^H(z)|
    &
    \leq \frac{C}{|z|^k}, \quad k=2,3,4.
\end{align}
\end{lemma}

To define the approximate stream function we first add a correction to $\psi^H$ 
\begin{align}
\label{tildePsiH}
\tilde \psi^H(y) = \psi^H(y) - \frac{1}{4} \varepsilon \omega  |y|^2 \eta^+_{1}(y) - \frac{1}{2} \varepsilon \omega y_2^2 \eta^-_{1}(y)
\end{align}
where the cut-off functions $\eta_{1}^\pm$ are defined as 
\begin{align}
\label{eta1pm}
\eta_1^+(y) = 1-\eta_0(y_2) , \quad \eta_1^-(y) = 1-\eta_0(-y_2) ,
\end{align}
for some $\eta_0 \in C^\infty(\R)$, $0 \leq \eta_0\leq 1$ which is even and satisfies
\begin{align}
\label{eta0}
\eta_0(s) = 1 \quad s \leq 1, \quad \eta_0(s) = 0 \quad s \geq 2.
\end{align}
In \eqref{tildePsiH} the corrections $- \frac{1}{4} \varepsilon \omega  |y|^2 $, $ \frac{1}{2} \varepsilon \omega y_2^2 $ are explicit quadratic functions to achieve $-\Delta_y \tilde \psi^H=\varepsilon \omega$ (far from $y_2=0$), and the choices in the upper and lower halves $y_2>0$ and $y_2<0$ are to better match the functions $\psi^D$ and $\psi^S$.

\medskip

We would like to construct $\psi_0$ by interpolating $\tilde \psi^ H$ defined in \eqref{tildePsiH} with the scaled stream functions $\psi^D$ and $\psi^S$ \eqref{psiD}, \eqref{psiS}. But doing this directly creates an error which is too big for the scheme to work. Indeed, note that $\psi^D$ \eqref{psiD} can be written as
\begin{align}
\label{ab1}
\psi^D(y)
&=y_2 - \frac{1}{4} \varepsilon\omega|y|^2 + \frac{1}{2} \omega d^D y_2 -\frac{d^D}{\varepsilon}-\frac{1}{4} \omega \frac{(d^D)^2}{\varepsilon} .
\end{align}
On the other hand, Lemma~\ref{lemma:estGradPsiH} gives
\[
\psi^H(y) = |y_2|-\log(2|y|) + O\Bigl( \frac{\log |y|}{|y|}\Bigr), \quad \text{as }|y|\to\infty, \ y \in \Omega^H,
\]
so that from \eqref{tildePsiH}
\begin{align}
\label{ab2}
\tilde \psi^H(y)  = y_2- \frac{1}{4} \varepsilon \omega  |y|^2 -\log(2|y|) + O\Bigl( \frac{\log |y|}{|y|}\Bigr)  , \quad \text{as }|y|\to\infty , \ y_2>0.
\end{align}
Comparing \eqref{ab1} and \eqref{ab2} we see that the term $-\log(2|y|)$ is not present in \eqref{ab1}. 
To obtain a correction to $\psi^D$, we solve the linearized problem in the disk, introduced in Section~\ref{sectFormal}
\begin{align}
\label{eqPsi1D}
\left\{
\begin{aligned}
\Delta \Psi_1^D &= 0 \quad \text{in } \Omega^D = B_R(0,R+d^D)\\
\frac{\partial\Psi_1^D }{\partial \nu} - \frac{1}{R}\Psi_1^D &=  \pi\delta_{P_D} + g_0 (x_2+d^S)  \quad \text{on } \partial\Omega^D ,
\end{aligned}
\right.
\end{align}
where $P_D=(0,d^D)$ is the lowest point in $\partial\Omega^D$ and $g_0$ is chosen so that 
\begin{align}
\nonumber
\int_{\partial \Omega^D } ( \pi\delta_{P_D} + g_0 (x_2+d^S)  ) Z_2^D =0,
\end{align}
where
\begin{align}
\label{Z2D}
Z_2^D(x) = x_2-(R+d^D)
\end{align}
is the same as the one defined in  \eqref{Zdisc}, except a translation and multiplication by a constant. The orthogonality condition with respect to $Z_1^D(x)=x_1$ is automatic. 
It turns out that
\begin{align}
\label{def-g0}
g_0=\frac{1}{R^2}.
\end{align}
Indeed, by centering the disk at the origin, we have
\begin{align*}
g_0 =-\frac{\pi \int_{\partial B_R(0)}\delta_{(0,-R)}x_2}{\int_{\partial B_R(0)}(x_2+R+d^D+d^S)x_2}
=\frac{\pi R}{\int_{\partial B_R(0)}x_2^2}
=\frac{1}{R^2}.
\end{align*}

\begin{lemma}
\label{lem:sol-disc2}
Let $g_0$ be given by \eqref{def-g0}. Then problem \eqref{eqPsi1D} has a solution,
which is even in $x_1$ and such that 
\begin{align}
\label{expPsi1D}
\Psi_1^D(x) &= -\log(|x-P_D|) + O(|x-P_D| |\log|x-P_D||),
\\
\label{expNPsi1D}
\nabla \Psi_1^D(x) &= - \frac{x-P_D}{|x-P_D|^2} +  O( |\log|x-P_D||)
\\
\nonumber
|D^2 \Psi_1^D(x) |& \leq   \frac{C}{|x-P_D|^2} 
\end{align}
as $x\to P_D$.
The solution $\Psi_1^D$ can be extended to a harmonic function on $\R^2$ minus the half line $\{ (0,t) \, | \, t < d^D\}$, and this extension satisfies that 
$\Psi_1^D(x) +  \log(|x-P_D|)$ has a continuous limit at each side of the half line  $\{ (0,t) \, | \, t \leq d^D\}$.
\end{lemma}
The proof of this lemma is standard, but we give a self-contained version in Appendix~\ref{sect-p2} below, to show that $\Psi_1^D$ is defined in a domain larger than $\Omega^D$, as stated in Lemma~\ref{lem:sol-disc2}. Note that $d^D$ and $d^S$ have not been defined yet, but the solution $\Psi_1^D$ depends on $d^D$ only by a vertical shift, and on $d^S$ by an additive constant.
The value of $g_0$ defined by \eqref{def-g0} does not depend on $d^D$ or $d^S$.

\medskip
The addition of $\varepsilon \Psi_1^D$ to $\Psi^D$ leads us naturally to decompose the parameter $g$ in the dynamic boundary condition on $\mathcal S$ as 
\begin{align}
\label{decomp-g}
g = \varepsilon g_0 + g_1
\end{align}
where $g_0$ is defined by \eqref{def-g0}. We will later solve for $g_1$.

\medskip

The corrected stream function on the disk is given by $\Psi^D(x)+\varepsilon\Psi_1^D(x)$, where $\Psi^D$ is defined in \eqref{PsiD} and $\Psi_1^D$ in \eqref{eqPsi1D}. In the scaled domain we define
\begin{align}
\label{tildepsiD}
\tilde \psi^D(y) = \psi^D(y) + \Psi_1^D(\varepsilon y). 
\end{align}
Then from \eqref{ab1} and \eqref{expPsi1D} we get
\begin{align}
\nonumber
\tilde \psi^D(y)
&=y_2 - \frac{1}{4} \varepsilon\omega|y|^2 + \frac{1}{2} \omega d^D y_2 -\frac{d^D}{\varepsilon}-\frac{1}{4} \omega \frac{(d^D)^2}{\varepsilon} 
\\
\label{ab3}
& \quad + |\log\varepsilon| - \log \Bigl| y - \frac{P_D}{\varepsilon}\Bigr| + O\Bigl( \varepsilon |\log\varepsilon| \Bigl| y - \frac{P_D}{\varepsilon}\Bigr| \Bigr),
\end{align}
as $|y|\to \infty$, $y_2>0$.
From \eqref{ab2} and \eqref{ab3}, there is a good matching of the functions $\tilde \psi^H$ and $\tilde \psi^D$ in the region $1\ll|y|\ll\varepsilon^{-1}$, $y_2>0$ (including the $\log|y|$ and constant) if we choose  $d^D$ such that 
\begin{align*}
-\frac{1}{2} \omega \frac{R d^D}{\varepsilon}
-\frac{1}{4} \omega \frac{ (d^D)^2}{\varepsilon} + |\log\varepsilon| &= - \log(2).
\end{align*}
Since by assumption $\frac{1}{2} \omega R = 1$, this gives
\begin{align}
\label{dDeps}
d^D &= \varepsilon |\log\varepsilon| + \varepsilon \log(2)+ O(\varepsilon^2 |\log \varepsilon |^2) ,
\end{align}
as $\varepsilon\to0$.
From \eqref{ab2} and \eqref{ab3} and the choice of $d^D$ we get
\begin{align}
\label{diffpsiHpsiD}
|\tilde\psi^H(y)-\tilde\psi^D(y)| \leq C \varepsilon |\log\varepsilon| |y| + C \frac{|\log\varepsilon|}{|y|}.
\end{align}

The solution to \eqref{eqPsi1D} introduces a correction to the disk $\Omega^D$,
denoted by $\Omega^D_\varepsilon$, such that $\partial\Omega^D_\varepsilon$ is described by 
\begin{align}
\label{pOmegaDeps}
\partial\Omega^D_\varepsilon = \big\{\,  x -\nu_{\partial\Omega^D}(x) h_1^D(x) \, | \, x\in \partial\Omega^D \setminus\{P_D\}\,\big\},
\end{align}
(as in \eqref{defm-bdy}) where $\nu_{\partial\Omega^D}$ is the outward unit normal to $\partial \Omega^D$ and 
\[
h_1^D(x) = - \varepsilon \Psi_1^D(x) ,
\]
see \eqref{linear-rel-psi-h}.
Note that with this definition, $\Omega^D_\varepsilon$ is unbounded, but we will later cut off the unbounded part.

Let 
\begin{align}
\label{gD1}
g_D^\pm(y_1) = \frac{R+d^D}{\varepsilon} \pm \sqrt{\frac{R^2}{\varepsilon^2}-y_1^2} 
\end{align}
be the functions whose graph gives the upper/lower parts of the boundary of the dilated disk $\frac{1}{\varepsilon}\Omega^D$.
The domain $\frac{1}{\varepsilon}\Omega^D_\varepsilon$ has boundary parametrized by
\begin{align}
\label{bdyOmegaeps}
y + \nu_{\frac{1}{\varepsilon}\partial\Omega^D}(y) \Psi_1^D(\varepsilon y) , \quad y \in \frac{1}{\varepsilon} (\partial \Omega^D\setminus\{P_D\}),
\end{align}
where $\nu_{\frac{1}{\varepsilon}\partial\Omega^D}$ is the outer unit normal to $\frac{1}{\varepsilon}\partial\Omega^D$.

Using the implicit function theorem, we can write the part of the curve defined by \eqref{bdyOmegaeps}, which lies between $1 \leq |y_1| \leq \frac{R}{2\varepsilon}$,  $y_2 \leq C \varepsilon y_1^2$, as a graph 
\[
(y_1 , g_{\varepsilon,D}^-(y_1)) , \quad 1 \leq |y_1| \leq \frac{R}{2\varepsilon} ,
\]
where the function $ g_{\varepsilon,D}^-(y_1) $ is smooth.
More precisely, $g_{\varepsilon,D}^-(y_1)$ is computed from the relation
\begin{align*}
(\tilde y_1,g_{\varepsilon,D}^-(\tilde y_1))  = (y_1,g_D^-(y_1) ) + \nu_{\frac{1}{\varepsilon}\partial\Omega^D}(y_1) \Psi_1^D(\varepsilon y_1,\varepsilon g_D^-(y_1) )
\end{align*}
using that 
\begin{align*}
\nu_{\frac{1}{\varepsilon}\partial\Omega^D}(y_1) = \frac{ ( (g_D^-)'(y_1),-1) }{\sqrt{1+(g_D^-)'(y_1)^2}} .
\end{align*}
Equivalently,
\begin{align}
\label{gepsDm}
g_{\varepsilon,D}^-(\tilde y_1)
&= 
g_D^-(y_1)
- \frac{ 1 }{\sqrt{1+(g_D^-)'(y_1)^2}}  \Psi_1^D(\varepsilon y_1,\varepsilon g_D^-(y_1) ),
\\
\nonumber
\tilde y_1 &= y_1+\frac{  (g_D^-)'(y_1)}{\sqrt{1+(g_D^-)'(y_1)^2}}  \Psi_1^D(\varepsilon y_1,\varepsilon g_D^-(y_1) ) .
\end{align}
From this, $g_{\varepsilon,D}^-$ has the expansion
\begin{align}
\label{gepsDm2}
g_{\varepsilon,D}^-(y_1) = g_D^-(y_1) - \Psi_1^D(\varepsilon y_1,\varepsilon	g_D^-(y_1)) + O(\varepsilon^2 |\log\varepsilon| y_1^2 ).
\end{align}

\medskip

For the strip $\Omega^S$ we perform a similar procedure. 
We add a correction to $\psi^S$ by solving
\begin{align}
\label{eqPsi1S}
\left\{
\begin{aligned}
\Delta \Psi_1^S &= 0 && \text{in } \Omega^S = \{ \, -\A-d^S < x_2 < -d^S \, \}\\
\frac{\partial\Psi_1^S }{\partial \nu} - \omega \Psi_1^S &= \pi\delta_{P_S}  && \text{on } x_2=-d^S
\\
 \Psi_1^S &=0 && \text{on } x_2=-\A-d^S ,
\end{aligned}
\right.
\end{align}
with $P_S = (0,-d^S)$.
Again this problem and its solution depend on $d^S$, which has not been defined yet, but the dependence is only through a vertical shift.

\begin{lemma}
\label{lem:sol-strip2}
Let $0<\omega<1$.
Problem \eqref{eqPsi1S} has a solution,
which is even in $x_1$ and such that 
\begin{align}
\label{expPsi1S}
\Psi_1^S(x) = - \log(|x-P_S|) + C^S +  O(|x-P_S| |\log|x-P_S||),\quad \text{as }x\to P_S ,
\end{align}
where $C^S \in \R$ is a constant.
Moreover, for any $\mu>0$ satisfying \eqref{eq:mu}
and $n\in\N$ there is $C$ such that 
\begin{align*}
|D^k \Psi_1^S(x)| \leq C e^{-\mu|x|},\quad|x|>1 ,\ 0\leq k\leq n.
\end{align*}
The solution $\Psi_1^S$ can be extended to a harmonic function on $ \{ \, (x_1,x_2) \, |  \,  -\A-d^S<x_2<\A-d^S \, \}$ minus the half line $\{ (0,t) \, | \, t \geq -d^S\}$, and this extension satisfies that $\Psi_1^S(x) +  \log(|x-P_S|)$ has a continuous limit at each side of the half line  $\{ (0,t) \, | \, t \geq -d^S\}$.
\end{lemma}

The proof is in Appendix~\ref{sect-p2}.

Define
\begin{align*}
\psi_1^S(y) = \Psi_1^S(\varepsilon y)
\end{align*}
and
\begin{align}
\label{tildePsiS}
\tilde \psi^S = \psi^S + \psi_1^S. 
\end{align}
From \eqref{psiS}  and \eqref{expPsi1S} we get
\begin{align}
\nonumber
\tilde \psi^S(y)
& = - \frac{1}{2}\varepsilon\omega ( y_2 + \varepsilon^{-1} d^S)^2- ( y_2 + \varepsilon^{-1} d^S)  
\\
\label{ab4}
& \quad + |\log\varepsilon| + C^S - \log \Bigl| y - \frac{P_S}{\varepsilon}\Bigr| 
+O\Bigl( \varepsilon |\log\varepsilon| \Bigl| y - \frac{P_S}{\varepsilon}\Bigr| \Bigr),
\end{align}
as $|y|\to \infty$, $y_2<0$. The definition \eqref{tildePsiH} implies that 
\begin{align}
\label{ab2b}
\tilde \psi^H(y)  = -y_2-\log(2|y|) + O\Bigl( \frac{\log |y|}{|y|}\Bigr) - \frac{1}{2} \varepsilon \omega  y_2^2  , \quad \text{as }|y|\to\infty , \ y_2<0.
\end{align}
From \eqref{ab2b} and \eqref{ab4}, there is a good matching of the functions $\tilde \psi^H$ and $\tilde \psi^S$ in the region $1\ll|y|\ll\varepsilon^{-1}$, $y_2<0$ if we choose $d^S$ such that 
\begin{align}
\label{dSeps}
-\frac{1}{2} \omega \frac{(d^S)^2}{\varepsilon}
-  \frac{d^S}{\varepsilon}
+ |\log\varepsilon| + C^S + \log 2 =0
\end{align}
which gives
\begin{align*}
d^S = \varepsilon|\log\varepsilon| + \varepsilon  (C^S+\log(2))  + O(\varepsilon^2 |\log(\varepsilon)|^2) .
\end{align*}

As for the disk, the solution $\Psi_1^S$ to  \eqref{eqPsi1S} introduces a correction $\Omega_\varepsilon^S$ to the strip $\Omega^S$ such that
\begin{align}
\label{pOmegaSeps}
\partial\Omega^S_\varepsilon = \mathcal{S}^S_\varepsilon \cup \mathcal{B}
\end{align}
where
\begin{align}
\label{Seps}
\mathcal{S}^S_\varepsilon = \{\,  x -\nu_{\Omega^S}  h_1^S(x) \, | \, x_2 = -d^S, x_1\not=0\,
\},
\end{align}
(as in \eqref{defm-bdy}) where $\nu_{\Omega^S}$ is the outward unit normal and 
\begin{align}
\label{h1S}
h_1^S(x_1) = - \varepsilon \Psi_1^S(x_1,-d^S) ,
\end{align}
see \eqref{linear-rel-psi-h}. The lower boundary $\mathcal{B}$ is the same one of the strip $\Omega^S$ \eqref{pOmegaS}.
More explicitly, the scaled upper boundary $\frac{1}{\varepsilon}\mathcal{S}^S_\varepsilon$ is the graph of the function
\begin{align}
\label{gepsSpl}
g_{\varepsilon,S}^+(y_1) = \Psi_1^S(\varepsilon y_1,-d^S) , \quad y_1\not=0.
\end{align}
Note that with this definition $\Omega^S_\varepsilon$ is unbounded, but we will later cut off the unbounded part. 

Let $\chi_0$ be a radial cut-off function in $\R^2$ such that 
\begin{align}
\label{chi0}
\chi_0(x) = 1 \quad |x| \leq 1,
\quad
\chi_0(x) = 0 \quad |x| \geq 2,
\end{align}
and let $\eta_1^+$ and $ \eta_1^-$ be defined as in \eqref{eta1pm}, \eqref{eta0}.
We define 
\begin{align}
\nonumber
\psi_0(y) 
&= \tilde\psi^H(y) \chi_0(\varepsilon^b y) 
+ \tilde\psi^D(y) [1-\chi_0(\varepsilon^b y)] \eta_1^+(y_2)
\\
\label{psi0}
& \quad + \tilde\psi^S(y) [1-\chi_0(\varepsilon^b y)] \eta_1^-(y_2).
\end{align}

\medskip

Finally we define $\Omega_0$.
To do this we make an approximation of the level set $\tilde \psi^H(y)=0$.
First we write $\partial \Omega^H$ as the graph of functions:
\[
y_2 = g_H^\pm(y_1), \quad |y_1|\geq \frac{\pi}{2}-1,
\]
where
\begin{align}
\label{bdyHairpin}
g_H^\pm(y_1) 
& = \pm \mathop{\textrm{arcosh}}\Bigl(y_1-\frac{\pi}{2}\Bigr)
= \pm \log\Bigl( y_1-\frac{\pi}{2}
+\sqrt{\bigl(y_1-\frac{\pi}{2}\bigr)^2-1} \Bigr).
\end{align}
We define
\begin{align}
\label{tgplus}
\tilde g_H^+(y_1) = g_H^+(y_1) + \frac{1}{4} \varepsilon\omega y_1^2 (1-\eta_0(|y_1|-10)),
\end{align}
where $\eta_0$ is the cut-off function defined in \eqref{eta0}.

The motivation for this approximation of the level set $\{ \tilde \psi^H=0 \}$, is that 
\[
\partial_{y_2} \psi^H(y_1,y_2) \approx 1
\]
in the upper half of the hairpin (see by Lemma~\ref{lemma:estGradPsiH}). Then, if we try to solve for $t $ in $\tilde \psi^H(y_1,t) = 0$ and expand
\begin{align*}
\tilde \psi^H(y_1,t) 
& \approx
\psi^H(y_1,g_H^+(y_1)) 
+ t-g_H^+(y_1)
- \frac{1}{4} \varepsilon \omega  \bigl( y_1^2+t^2 \bigr) ,
\end{align*}
we find that $t = g_H^+(y_1)+ \frac{1}{4} \varepsilon \omega y_1^2 $ at main order.

With the lower part of $\partial\Omega^H$ there is no need to do a modification. So we can define an $\varepsilon$-modification of $\Omega^H$ as the domain $\Omega^H_\varepsilon$ such that its boundary is given by
\begin{align}
\label{OmegaHeps}
\partial\Omega^H_\varepsilon = 
\Bigl\{ \, \big(y_1,\tilde g_H^+(y_1)\big) \, \Big| \, |y_1|\geq\frac{\pi}{2}-1 \,\Bigr\}\cup \Bigl\{ \, \big(y_1,g_H^-(y_1)\big) \, \Big| \, |y_1|\geq\frac{\pi}{2}-1 \,\Bigr\}.
\end{align}

\medskip

The function \eqref{tgplus} is well defined for all $|y_1|\geq \frac{\pi}{2}-1$.
Let $\frac{1}{2}<b<1$. In the region $\varepsilon^{-b}<|y_1|<2\varepsilon^{-b}$ and $y_2>0$ we transition from $g_H^+$ to $g_{\varepsilon,D}^-$ \eqref{gepsDm} using smooth cut-offs.
Let $\eta_0(s)$ be the cut-off described in \eqref{eta0}.
Define 
\begin{align}
\label{upperbdy}
g^+(y_1)=\tilde g_H^+(y_1) \eta_0(\varepsilon^b y_1) + g_{\varepsilon,D}^-(y_1) (1- \eta_0(\varepsilon^b y_1)),
\quad \frac{\pi}{2}-1\leq |y_1|\leq \frac{R}{2\varepsilon},
\end{align}
with $\tilde g_H^+$ the function defined in \eqref{tgplus}.
Similarly, in the lower part of the hairpin we define a transition from $g_H^-$ to  $g_{\varepsilon,S}^+$ \eqref{gepsSpl}, by 
\begin{align}
\label{defgm}
g^-(y_1) = g_H^-(y_1) \eta_0(\varepsilon^b y_1)  +g_{\varepsilon,S}^+(y_1)  (1- \eta_0(\varepsilon^b y_1)),
\quad \frac{\pi}{2}-1\leq |y_1|\leq \frac{1}{\varepsilon}.
\end{align}

Let $0<\delta<\min(1,\frac{R}{2})$ be a fixed constant.
We define $\Omega_0$ so that inside the ball $B_{\frac{\delta}{2\varepsilon}}(0)$ the boundary of $\Omega_0$ is given by the two curves $y_2 = g^\pm(y_1)$ in
\eqref{upperbdy}, \eqref{defgm}.
In the region of the upper half plane outside $B_{\frac{\delta}{2\varepsilon}}(0)$, we let $\Omega_0$ be the part corresponding to $\frac{1}{\varepsilon} \Omega^D_\varepsilon$, and in the region of the lower half plane outside $B_{\frac{\delta}{2\varepsilon}}(0)$, we let $\Omega_0$ be the part corresponding to $\frac{1}{\varepsilon}\Omega^S_\varepsilon$. Figure~\ref{fig:Omega0} shows a representation of $\Omega_0$.

\begin{figure}
\centering
\includegraphics[scale=1,page=1]{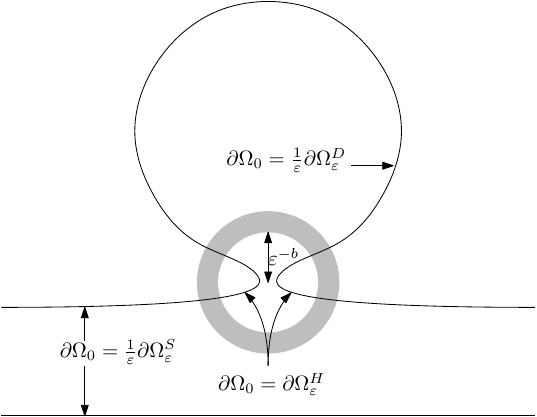}
\caption{Approximate domain $\Omega_0$.
The grey annulus represents the region of transition from the modified hairpin $\Omega^H_\varepsilon$ \eqref{OmegaHeps} to either the modified scaled disk $\frac{1}{\varepsilon}\Omega^D_\varepsilon$ \eqref{pOmegaDeps} in the upper half or the modified scaled strip $\frac{1}{\varepsilon}\Omega^S_\varepsilon$ \eqref{pOmegaSeps} in the lower half.}
\label{fig:Omega0}
\end{figure}

\begin{remark}
\label{rem:pacard}
In \cite{mazzeo-pacard-pollack} two nondegenerate, orientable, immersed,
compact constant mean curvature surfaces $\Sigma_1$, $\Sigma_2$, with nonempty boundary, such that they touch a certain point $P$, are desingularized into a new CMC by replacing the touching point with a catenoid. The approximation involves solving the Jacobi operator on each surface equal to $\varepsilon \delta_P$. This solution, which has a $\log$ singularity that matches the shape of the catenoid, is used to deform each surface in the normal direction. This procedure is comparable to solving \eqref{eqPsi1D} and \eqref{eqPsi1S}, whose solutions enter as a corrections of the main term of the outer stream functions and define deformations of the disk and strip that involve a $\log$ singularity. In our case, the solvability condition of the linearized problem on the disk requires a first adjustment of the parameter $g$, while in \cite{mazzeo-pacard-pollack} this phenomenon is not present by nondegeneracy.
\end{remark}

\section{Setting up the problem}
\label{sect:operators}

In this section we describe a scheme that converts the overdetermined problem
\begin{align}
    \label{overd0}
    \left\{
        \begin{alignedat}{2}
        -\Delta_y \psi &= \varepsilon \omega &\quad& \text{in }\Omega
        \\
        \psi &= 0 &\quad& \text{on } \mathcal{S}
        \\
        \frac{1}{2} |\nabla_y \psi|^2 + g  ( \varepsilon y_2 + d^S) &= \frac{1}{2} &\quad& \text{on } \mathcal{S} 
        \\
        \psi &= \frac{\beta}{\varepsilon} &\quad& \text{on } \mathcal{B}, 
        \end{alignedat}
    \right.
\end{align}
for $(\Omega, \psi)$ close to the approximate solution $(\Omega_0, \psi_0)$ constructed in Section~\ref{sec:approx}, into the nonlinear elliptic equation
\begin{align}
\label{main-problem0}
\left\{
\begin{aligned}
\Delta_y u + Q_1[u] &=0 \quad \text{in }\Omega_0
\\
\frac{\partial u}{\partial \nu} + (\kappa - \varepsilon\omega) u + E(g) + Q_2[u,g] &= 0 
\quad \text{on } \mathcal S_0
\\
u &= 0 \quad \text{on } \mathcal B_0
\end{aligned}
\right.
\end{align}
in the unknowns $u$, $g$.
The main result of this section is that solving \eqref{main-problem0} yields a solution to the overdetermined problem \eqref{overd0}. In equation \eqref{main-problem0}, $E(g)$ is a function independent of $u$ (defined in \eqref{def:E1}), and $Q_1$ and $Q_2$ are nonlinear operators defined below in \eqref{newQ1} and \eqref{newQ2}, respectively. In the boundary condition, $\kappa$ represents the curvature of the boundary with an appropriate sign convention. Additionally, there exists a link between the solution $u$ and the solution $(\Omega, \psi)$ of \eqref{overd0}, as described in \eqref{Omega-psi}.

\begin{proposition}
\label{prop:change}
If \eqref{main-problem0} has a solution $u$ such that the condition \eqref{diffeo} below is satisfied, then $\Omega$, $\psi$ defined by \eqref{Omega-psi} below is a solution to \eqref{overd0}.
\end{proposition}

Regarding the solvability of \eqref{main-problem0}, the smallness of the error $E$ depends on how close $(\Omega_0, \psi_0)$ is to solving the problem. The terms $Q_1$ and $Q_2$ contain either quadratic or small linear operators of $u$. Thus, if $E$ is small, \eqref{main-problem0} can be treated as a perturbation of the linearized problem \eqref{robin3a} around $(\Omega_0, \psi_0)$ (where, for simplicity, we did not include a boundary component with Dirichlet boundary condition). The condition \eqref{diffeo} says that a change of variables defined in terms of $u$ is a diffeomorphism and it holds if $u$ is sufficiently small. Thus, it can be checked after  \eqref{main-problem0} has been solved.

\medskip

The rest of the section is devoted to introduce some notation and give a proof of Proposition~\ref{prop:change}.

Let $\gamma\colon \R \to \R^2$ be an arc-length parametrization
of  $\mathcal{S}_0$ and let us write $\nu(s)$ the exterior unit normal vector to $\partial\Omega_0$ at $\gamma(s)$ and $T(s)= \gamma'(s)$. 
As in Section~\ref{sectFormal},
we choose orientations so that 
\[
\nu = T^\perp  ,
\]
where $(a,b)^\perp = (b,-a)$ and define the curvature of $\kappa$ of $\mathcal{S}_0$ as 
\begin{align*}
\frac{d T}{ds}= - \kappa \nu.
\end{align*}

As in the discussion of the linearization of \eqref{overd0} in Section~\ref{sectFormal}, given a small differentiable function $h\colon \mathcal{S}_0\to\R$, we would like to construct a domain $\Omega_h$ close to $\Omega_0$ such that its upper boundary $\mathcal{S}_h$ is parametrized by  $ \gamma(s) -  h(\gamma(s)) \nu(s)$. This leads us to consider coordinates $(s,t)$ given by 
\[
\gamma(s) - ( t + h(s) )  \nu(s).
\]
We would like to cut-off the effect of $h$ in these coordinates when we are sufficiently far from $\mathcal{S}_h$, as described by a distance function $\rho(s)$, where $s$ is the arc-length parameter of $\mathcal{S}_0$. In constructing this function it is convenient to notice that 
\begin{align}
\label{est-kappa}
|\kappa(y)|\leq C\Bigl( \frac{1}{1+|y|^2}+ \varepsilon\Bigr),\quad y\in\partial\Omega_0.
\end{align}
Considering this estimate for the curvature, we let $\rho \in C^\infty(\R)$ be such that
\begin{align}
\nonumber
\left\{
\begin{aligned}
&\text{if $|\gamma(s)|\geq \varepsilon^{-\frac{1}{2}}$, $\rho(s)=\varepsilon^{-1} d_0$}
\\
&\text{if $|\gamma(s)|\leq \varepsilon^{-\frac{1}{2}}$, then}
\\
& \qquad c_1  (|\gamma(s)|+1)^{2} \leq \rho(s) \leq c_2  (|\gamma(s)|+1)^{2}
\\
&|\rho^{(k)}(s)| \leq C  (|\gamma(s)|+1)^{2-k},
\quad k=1,2,
\\
&|\rho^{(3)}(s)| \leq C \sqrt{\varepsilon},
\end{aligned}
\right.
\end{align}
where $d_0>0$ is a small constant so that $d_0<\frac{1}{10 R}$, $d_0<\frac{\A}{10}$ and $0<c_1<c_2$ are again small constants.
With this definition, $\rho(s)^{-1}$ follows the bound \eqref{est-kappa} for $\kappa$.
More precisely, given any $\bar\delta>0$, by adjusting the constants $c_1$, $c_2$, $d_0$, we achieve
\begin{align}
\label{kappa-d}
|\kappa(s)|\, \rho(s)\leq \bar\delta.
\end{align}

Let $\eta$ be defined by 
\begin{align}
\label{est-eta}
\eta(s,t) = \eta_0\Big( \frac{t}{\rho(s)} \Bigr) ,
\end{align}
where $\eta_0$ is as in \eqref{eta0}.
We define $h$-dependent $(s,t)$ coordinates by
\begin{align}
\label{Xh}
X_h(s,t ) = \gamma(s) - ( t + h(s) \eta(s,t) )  \nu(s).
\end{align}
Next, we define $\tilde \rho\in C^\infty(\R)$ such that 
\begin{align}
\label{properties-tilded}
\left\{
\begin{aligned}
&\text{if $|s|\geq \varepsilon^{-1}$, $\tilde \rho(s)=\varepsilon^{-1} d_0$}
\\
&\text{if $|s|\leq \varepsilon^{-1}$, then}
\\
& \qquad c_1  (\log(|\gamma(s)|+2) + \varepsilon|\gamma(s)|^2)
 \leq \tilde \rho(s) \leq c_2  (\log(|\gamma(s)|+2) + \varepsilon|\gamma(s)|^2)
\end{aligned}
\right.
\end{align}
where $d_0>0$ and $0<c_1<c_2$ are constants.
This function follows the separation of the upper and lower parts of the boundary of $\Omega_0$ (Section~\ref{sec:approx}).

By direct computation we get the following result.
\begin{lemma}
\label{lem:injective}
By taking the constants $c_1$, $c_2$ and $d_0$ small, the function $X_0$ is injective on the set 
$\{(s,t)\,|\, -5 \tilde \rho(s)\leq t\leq 5 \rho(s)\}$.
\end{lemma}

Consider the open neighborhood 
\begin{align}
\nonumber
U = \{ \, X_0(s,t) \, | \,-2 \tilde \rho(s)< t< 2 \rho(s) \,\}
\end{align}
of $\mathcal{S}_0$.
Then $X_0^{-1}$ is well defined on $U$ and this allows us to define
\begin{align}
\label{def:Fh}
F_h(y) = X_h \circ X_0^{-1}(y),\quad y\in U.
\end{align}
Note that $X_h \circ X_0^{-1}(y)=y$ for 
\[
y\in\tilde U:=\{ \, X_0(s,t) \, | \,0< t< 4 \rho(s) \,\}
\]
and therefore
\[
F_h(y)
=\begin{cases}
X_h \circ X_0^{-1}(y) & y \in \overline\Omega_0\cap U
\\
y& y \in \Omega_0\setminus U
\end{cases}
\]
defines a $C^1$ function $F_h \colon\overline \Omega_0 \to \R^2$ (assuming $h$ is $C^1$).

Let
\begin{align}
\label{def:Omegah}
\Omega_h = F_h(\Omega_0) .
\end{align}
Note that for $h$ small (to be made precise later), the boundary of $\Omega_h$ decomposes into its upper part, which is denoted by $\mathcal{S}_h$ and its bottom part, which coincides with $\mathcal{B}_0$.

We consider \eqref{overd0} with $\Omega=\Omega_h$ and the problem becomes to find $h$ such that the overdetermined problem
\begin{align}
\label{overd}
\left\{
\begin{aligned}
-\Delta_y \psi &= \varepsilon \omega && \text{in }\Omega_h
\\
\psi &= 0 && \text{on } \mathcal{S}_h
\\
\frac{1}{2} |\nabla_y \psi|^2 + g  (\varepsilon y_2+d^S) &= \frac{1}{2} && \text{on }\mathcal{S}_h
\\
\psi &= \frac{\beta}{\varepsilon} && \text{on } \mathcal{B}_0,
\end{aligned}
\right.
\end{align}
has a solution.

\medskip

Define $\Psi[h]$ as the solution to the problem
\begin{align}
\label{eqPsi}
\left\{
\begin{aligned}
-\Delta \Psi[h] &= \varepsilon \omega + \Delta \psi_0 && \text{in }\Omega_h
\\
\Psi[h] &= -\psi_0 && \text{on } \mathcal{S}_h
\\
\Psi[h] &= 0 && \text{on } \mathcal{B}_0,
\end{aligned}
\right.
\end{align}
where we have used that the approximate solution $\psi_0$ satisfies $\frac{\beta}{\varepsilon}$ on the lower boundary on $\mathcal{B}_0$.
The idea is to write the solution $\psi$ of \eqref{overd} as $\psi = \psi_0 + \Psi[h]$.
Then we reformulate the problem \eqref{overd} as: find $h$ such that 
\begin{align}
\label{prob3}
\frac{1}{2} |\nabla_y( \Psi[h] + \psi_0)|^2 + g (\varepsilon y_2+d^S) = \frac{1}{2} \quad \text{on }\mathcal{S}_h.
\end{align}

We use the following notation. For a function $\psi$ defined in $\Omega_h$ we let
\[
\tilde \psi (y)= \psi \circ F_h (y)=  \psi \circ X_h \circ X_0^{-1} (y), \quad y \in \Omega_0 ,
\]
\begin{align}
\label{hat-psi}
\hat\psi(s,t)  = \psi \circ X_h(s,t) = \tilde \psi \circ X_0(s,t)
\end{align}
so that 
\[
\tilde \psi  = \hat\psi \circ X_0^{-1} .
\]

Let us write \eqref{eqPsi} in the domain $\Omega_0$ by changing variables. Let
\begin{align}
\nonumber
\tilde \psi_0[h] = \psi_0 \circ F_h, 
\end{align}
\[
\tilde \Psi[h](y) = \Psi[h] ( F_h (y)).
\]
Then $\tilde \Psi[h]$ satisfies
\begin{align}
\label{eqTildePsi}
\left\{
\begin{aligned}
- L_h [\tilde \Psi] &= \varepsilon \omega + L_h [\tilde \psi_0[h]] && \text{in }\Omega_0
\\
\tilde  \Psi &= - \tilde \psi_0[h] && \text{on } \mathcal{S}_0 ,
\\
\tilde  \Psi &= 0 && \text{on } \mathcal{B}_0 ,
\end{aligned}
\right.
\end{align}
where $L_h$ is defined as 
\[
L_h[\tilde \psi] = \Delta_y( \tilde \psi\circ F_h^{-1} ) \circ F_h .
\]
The operator $L_h$ can be computed as follows. In $ \Omega_0 \setminus U $, $L_h = \Delta_y$. 
Let 
\[
\hat L_h [\hat\psi]  = \Delta ( \hat\psi \circ X_h^{-1}) \circ X_h.
\]
Then, in $\Omega_0 \cap U$ we have
\begin{align}
\label{change}
L_h[\tilde \psi] &= \Delta_y( \tilde \psi\circ X_0 \circ X_h^{-1} ) \circ X_h \circ X_0^{-1}
= \hat L_h[ \tilde \psi \circ X_0] \circ X_0^{-1},
\end{align}
where $\hat L_h$ is given by
\begin{align*}
\hat L_h  = \hat L_{1,h} + \hat L_{2,h} + \hat L_{3,h}  
\end{align*}
and
\begin{align*}
\hat L_{1,h} &= \frac{1}{(1-\kappa( t + h \eta))^2 } \partial_{ss} 
- 2 \frac{(h \eta)_s}{(1-\kappa( t + h \eta))^2(1+h \eta_t)}\partial_{st} 
\\
& \quad 
+ \Bigl[
\frac{1}{(1+h \eta_t)^2 } + \frac{ ((h \eta)_s)^2 }{(1-\kappa( t + h \eta))^2(1+h \eta_t)^2}
\Bigr] \partial_{tt} 
\\
\hat L_{2,h}
&= \Bigl[ \partial_s \Bigl( \frac{1}{(1-\kappa( t + h \eta))^2 } \Bigr)
- \partial_t  \Bigl( \frac{(h \eta)_s}{(1-\kappa( t + h \eta))^2(1+h \eta_t) } \Bigr)\Bigr]
\partial_s 
\\
& \quad 
+ \Bigl[
- \partial_s 
\Bigl(
\frac{(h \eta)_s}{(1-\kappa( t + h \eta))^2(1+h \eta_t)}\Bigr)
+ \partial_t\Bigl( \frac{1}{(1+h \eta_t)^2}\Bigr)
\\
& \qquad 
+ \partial_t \Bigl( \frac{ ( (h \eta)_s )^2 }{(1-\kappa( t + h \eta))^2(1+h \eta_t)^2} \Bigr)
\Bigl] \partial_t
\\
\hat L_{3,h}&=
\Bigl[
\frac{\partial_s ( (1-\kappa(t+h\eta)) (1+h\eta_t)) }{ (1-\kappa(t+h\eta)) (1+h\eta_t)) }
\frac{1}{ (1-\kappa(t+h\eta))^2}
\\
& \qquad 
-\frac{\partial_t( (1-\kappa(t+h\eta)) (1+h\eta_t)) }{ (1-\kappa(t+h\eta)) (1+h\eta_t)) }
\frac{(h\eta)_s}{(1-\kappa(t+h\eta))^2 (1+h\eta_t)}
\Bigr] \partial_s
\\
& \quad
\Bigl[
-\frac{\partial_s ( (1-\kappa(t+h\eta)) (1+h\eta_t)) }{ (1-\kappa(t+h\eta)) (1+h\eta_t) }
\frac{(h\eta)_s}{(1-\kappa(t+h\eta))^2 (1+h\eta_t)}
\\
& \qquad
+\frac{\partial_t( (1-\kappa(t+h\eta)) (1+h\eta_t)) }{ (1-\kappa(t+h\eta)) (1+h\eta_t) }
\Bigl( \frac{1}{(1+h\eta_t)^2}
\\
& \qquad 
+ \frac{((h\eta)_s)^2}{(1-\kappa(t+h\eta))^2 (1+h\eta_t)^2}
\Bigr)
\Bigr] \partial_t.
\end{align*}
The terms $\hat L_{j,h}$, $j=1,2,3$ denote a splitting of the operator $\hat L_h$ without any particular meaning.

We expand these operators with respect to $h$:
\begin{align*}
\hat L_{1,h} &= \hat  L_{1}^{(0)} +  \hat L_{1,h}^{(1)} +  \hat L_{1,h}^{(2)} 
\\
\hat L_{1}^{(0)} &= \frac{1}{(1-\kappa t)^2} \partial_{ss}  + \partial_{tt}
\\
\hat L_{1,h}^{(1)} &= \frac{2 \kappa h \eta}{(1-\kappa t)^3} \partial_{ss}  
- 2 \frac{(h\eta)_s}{(1-\kappa t)^2}  \partial_{st}
-2 h \eta_t \partial_{tt}
\end{align*}
and $\hat L_{1,h}^{(2)}$ is defined as 
\[
\hat L_{1,h}^{(2)} = \hat L_{1,h} - \hat  L_{1}^{(0)} - \hat L_{1,h}^{(1)} .
\]
Similarly we split
\begin{align*}
\hat L_{2,h} &= \hat L_{2}^{(0)} +  \hat L_{2,h}^{(1)} +  \hat L_{2,h}^{(2)} 
\\
\hat L_{2}^{(0)} &= 2\frac{\kappa' t}{(1-\kappa t)^3} \partial_s
\\
\hat L_{2,h}^{(1)} &=  \Bigl[ 2\frac{\kappa' h \eta }{(1-\kappa t)^3}
+ 6 \frac{t \kappa \kappa' h \eta}{(1-\kappa t)^3}
- \frac{(h\eta)_{st}}{(1-\kappa t)^2}
\Bigr]\partial_s
+ 
\Bigl[
-\frac{(h\eta)_{ss}}{(1-\kappa t)^2}
- 2 h \eta_{tt}
\Bigr] \partial_t
\end{align*}
and
\begin{align*}
\hat L_{h,3} &= \hat L_{3,h}^{(0)} + \hat L_{3,h}^{(1)} + \hat L_{3,h}^{(2)} 
\\
\hat L_{3}^{(0)} &=  -\frac{\kappa' t}{(1-\kappa t)^3} \partial_s   - \frac{\kappa}{1-\kappa t} \partial_t 
\\
\hat L_{3,h}^{(1)} &= 
\Bigl[
-\frac{\kappa' h \eta}{(1-\kappa t)^3}
-\frac{\kappa (h\eta)_s}{(1-\kappa t)^3}
-3 \frac{t \kappa \kappa'  h}{(1-\kappa t)^4}
+ \frac{\kappa' (h\eta)_s}{(1-\kappa t)^3}
+\frac{\kappa' (h\eta)_s}{(1-\kappa t)^3}
\Bigr] \partial_s u
\\
& \quad 
+ \Bigl[
\frac{\kappa'(h\eta)_s}{(1-\kappa t)^3} 
-\frac{\kappa^2 h \eta}{(1-\kappa t)^2}
+ \frac{\kappa h \eta_t}{1-\kappa t}
\Bigr] \partial_t.
\end{align*}

Let us define
\begin{align*}
\hat L^{(0)} &=  \hat L_{1}^{(0)}+\hat L_{2}^{(0)}+\hat L_{3}^{(0)} 
\\
\hat  L_h^{(1)} &=  \hat L_{1,h}^{(1)}  +   \hat L_{2,h}^{(1)}  +   \hat L_{3,h}^{(1)}  
\\
\hat L_h^{(2)} &=  \hat L_{1,h}^{(2)}  +  \hat  L_{2,h}^{(2)}  +  \hat  L_{3,h}^{(2)}  ,
\end{align*}
and using \eqref{change}, we define
\begin{align*}
L^{(0)}[\tilde \psi] &= \hat L^{(0)}[ \tilde \psi \circ X_0] \circ X_0^{-1} 
\\
L_h^{(1)}[\tilde \psi] &= \hat L_h^{(1)}[ \tilde \psi \circ X_0] \circ X_0^{-1}
\\
L_h^{(2)}[\tilde \psi] &= \hat L_h^{(2)}[ \tilde \psi \circ X_0] \circ X_0^{-1} ,
\end{align*}
so that 
\[
L_h  = L^{(0)}+ L_h^{(1)} + L_h^{(2)} .
\]
We observe that $L^{(0)} = \Delta_y$, $L_h^{(1)}$ depends linearly on $h$ and $L_h^{(2)}$ is quadratic in $h$. Although this decomposition has been defined using a change of variables in $\Omega_0\cap U$, it is natural to extend it to $\Omega_0 \setminus U$, by setting
\[
L^{(0)}[\tilde \psi]  = \Delta_y \tilde \psi, \quad
L_h^{(1)}[\tilde \psi]  = 0, \quad
L_h^{(2)}[\tilde \psi]  = 0 \quad \text{in } \Omega_0 \setminus U.
\]

To rewrite  \eqref{prob3}, we let
\[
B_h(\tilde \psi_1,\tilde \psi_2) =  ( \nabla_y (\tilde \psi_1\circ F_h^{-1} )  \circ F_h ) \cdot  ( \nabla_y (\tilde \psi_2\circ F_h^{-1} )  \circ F_h ) 
\]
where the resulting expression is evaluated at $t=0$,
and note that 
\[
B_h(\tilde \psi_1,\tilde \psi_2) = \hat B_h(\tilde \psi_1 \circ X_0 ,\tilde \psi_2 \circ X_0 ) \circ X_0^{-1} 
\]
where
\begin{align*}
\hat B_h(\hat\psi_1,\hat\psi_2) 
& =  ( \nabla_y (\hat\psi_1 \circ X_h^{-1} )  \circ X_h ) \cdot ( \nabla_y (\hat\psi_2 \circ X_h^{-1} )  \circ X_h ) 
\\
&= \frac{1}{(1-\kappa h)^2 } \partial_s \hat\psi_1 \partial_s \hat\psi_2 
-\frac{h'}{(1-\kappa h)^2}  \partial_s \hat\psi_1 \partial_t \hat  \psi_2
-\frac{h'}{(1-\kappa h)^2}  \partial_s \hat\psi_2 \partial_t \hat  \psi_1
\\
& \quad 
+ \Bigl[ 1+\frac{(h')^2}{(1-\kappa h)^2}  \Bigr] \partial_t \hat\psi_1\partial_t \hat\psi_2 .
\end{align*}

We use the notation
\[
\hat \Psi[h] = \Psi[h]\circ X_h= \tilde \Psi[h]\circ X_0
\]
\begin{align}
\nonumber
\hat\psi_0[h](s,t) = \psi_0(X_h(s,t)),
\end{align}
which is consistent with the notation \eqref{hat-psi} but is explicit regarding the dependence on $h$.
Then problem \eqref{prob3} becomes to find $h$ such that the solution to \eqref{eqTildePsi} satisfies
\begin{align}
\label{bc1}
\hat B_h( \hat \Psi[h] + \hat\psi_0[h], \hat \Psi[h] + \hat\psi_0[h]) + 2 g (\varepsilon\tilde y_2[h] + d^S) = 1
 , \quad \text{on } t=0 .
\end{align}
where 
\[
\tilde y_2[h](s) = (\gamma(s) - \nu(s) h(s))\cdot {\bf e}_2, \quad  {\bf e}_2 = (0,1),
\]
is at main order the $y_2$ coordinate on $\mathcal{S}_0$ with a linear perturbation term in $h$.

We consider $\tilde \Psi_1\colon  \Omega_0\to \R$ defined by 
\begin{align}
\label{tildePsi1}
\tilde \Psi_1[h] = \tilde \Psi[h]  -  \tilde \Psi[0]  + \nabla_y \Psi[0] \cdot ( (\nu \eta h)\circ X_0^{-1}) ,
\end{align}
that is, $\nu$, $\eta$, $h$ are expressed in terms of the $(s,t)$ coordinates defined by $X_0$ and $\Psi[0]$ is the function defined in \eqref{eqPsi} for $h=0$, which is defined in $\Omega_0$.
Note that $ (\nu \eta h) \circ X_0^{-1}$ is well defined in $\Omega_0$ thanks to the cut-off $\eta$.

Next we derive a problem equivalent to \eqref{eqTildePsi}, \eqref{bc1}  but written in terms of $\tilde \Psi_1$.
First we note that for $y\in \mathcal{S}_0$, writing $y = \gamma(s)$, we have
\begin{align}
\nonumber
\tilde \Psi_1[h](y) &= -\psi_0(F_h(y)) + \psi_0(y) + \nabla_y \Psi[0]\cdot \nu(s) h(s)
\\
\nonumber
&= -\psi_0(y-\nu(s) h(s)) + \psi_0(y) + \nabla_y \Psi[0]\cdot \nu(s) h(s)
\\
\nonumber
&= \nabla_y \psi_0(y) \cdot \nu(s) h(s)
-  [\psi_0(y-\nu(s) h(s)) - \psi_0(y) +\nabla_y \psi_0(y) \cdot \nu(s) h(s)]
\\
\nonumber
& \quad 
+ \nabla_y \Psi[0]\cdot \nu(s) h(s) .
\end{align}
This is a nonlinear relation between $\tilde \Psi_1[h](y) $ and $h(s)$, which at main order is $\tilde \Psi_1[h](y) \approx  \nabla_y \psi_0(y) \cdot \nu(s) h(s)\approx - h(s)$.

Then we compute
\begin{align*}
L_h [\tilde \Psi[h] ] + \varepsilon \omega + L_h [\tilde \psi_0[h]] 
&= L_h\left[ \tilde \Psi_1[h] +  \tilde \Psi[0] -  \nabla \Psi[0] \cdot \nu \eta h\right]
 + \varepsilon \omega + L_h [\tilde \psi_0[h]] 
\\
&=
L^{(0)}\left[  \tilde \Psi_1[h] +  \tilde \Psi[0] -  \nabla \Psi[0] \cdot \nu \eta h \right]
\\
& \quad 
+ L_h^{(1)} \left[  \tilde \Psi_1[h] +  \tilde \Psi[0] -  \nabla \Psi[0] \cdot \nu \eta h \right]
\\
& \quad 
+ L_h^{(2)} \left[ \tilde \Psi_1[h] +  \tilde \Psi[0] -  \nabla \Psi[0] \cdot \nu \eta h \right]
\\
& \quad 
+ L^{(0)} [\tilde \psi_0[h]] 
+ L_h^{(1)} [\tilde \psi_0[h]] 
+ L_h^{(2)} [\tilde \psi_0[h]] 
\\ 
& \quad 
+ \varepsilon \omega .
\end{align*}
Let us expand $\tilde \psi_0[h]$ in terms of $h$
\begin{align*}
\tilde \psi_0[h] = \tilde \psi_0[0] - \nabla_y \tilde \psi_0[0] \cdot \nu \eta h 
+  \tilde \psi_0[h]  - \tilde \psi_0[0] + \nabla_y \tilde \psi_0[0] \cdot \nu \eta h  .
\end{align*}
Using that $L^{(0)}=\Delta_y$ and the equation \eqref{eqPsi} for $h=0$ we have
\[
L^{(0)}[\tilde \Psi[0] ] + L^{(0)}[ \tilde \psi_0[0] ] + \varepsilon\omega =0.
\]
Therefore
\begin{align}
\nonumber
L_h [\tilde \Psi[h] ] + \varepsilon \omega + L_h [\tilde \psi_0[h]] 
&=
L^{(0)}[ \tilde \Psi_1[h] ] - L^{(0)}[ \nabla_y \Psi[0] \cdot \nu \eta h ]
\\
\nonumber
& \quad 
+ L_h^{(1)}[ \tilde \Psi_1[h] ] +  L_h^{(1)}[\tilde \Psi[0] ] - L_h^{(1)}[ \nabla_y \Psi[0] \cdot \nu \eta h ]
\\
\nonumber
& \quad 
+ L_h^{(2)}[ \tilde \Psi_1[h] ] +  L_h^{(2)}[\tilde \Psi[0] ] - L_h^{(2)}[ \nabla_y \Psi[0] \cdot \nu \eta h ]
\\
\nonumber
& \quad 
- L^{(0)} [  \nabla_y   \tilde \psi_0[0] \cdot \nu \eta h ]
\\
\nonumber
& \quad 
+ L^{(0)}[ \tilde \psi_0[h]  - \tilde \psi_0[0] + \nabla_y   \tilde \psi_0[0] \cdot \nu \eta h]
\\
\label{eq1}
& \quad 
+ L_h^{(1)} [\tilde \psi_0[h]] 
+ L_h^{(2)} [\tilde \psi_0[h]] .
\end{align}
Next we claim that 
\begin{align}
\label{formula1}
- L^{(0)}[ \nabla_y \Psi[0] \cdot \nu \eta h ]
- L^{(0)} [  \nabla_y   \tilde \psi_0[0] \cdot \nu \eta h ]
+  L_h^{(1)}[\tilde \Psi[0] ]
+ L_h^{(1)} [\tilde \psi_0[0]] =0.
\end{align}
To see this we consider equation \eqref{eqTildePsi}
\[
L_{\tau h}[\tilde \Psi[\tau h] ] + L_{\tau h}[\tilde \psi_0[\tau h] ] + \varepsilon\omega = 0 \quad \text{in }\Omega_0
\]
and differentiate it with respect to $\tau$ at $\tau=0$. 
We get
\begin{align*}
L^{(0)}\left[ {\textstyle \frac{d}{d\tau}\big|_{\tau=0} \tilde \Psi[\tau h] } \right]
+ L_h^{(1)} [ \tilde \Psi[0] ]
+L^{(0)}\left[ {\textstyle \frac{d}{d\tau}\big|_{\tau=0} \tilde \psi_0[\tau h] } \right]
+ L_h^{(1)} [ \tilde \psi_0[0] ] = 0.
\end{align*}
But
\[
\tilde \Psi[\tau h](y) = \Psi[\tau h]( X_{\tau h} ( X_0^{-1}(y) )
\]
and writing $(s,t) = X_0^{-1}(y)$ we see that 
\begin{align*}
\frac{d}{d\tau}\Big|_{\tau=0} \tilde \Psi[\tau h]( y) = 
\frac{d}{d\tau}\Big|_{\tau=0}  \Psi[\tau h](y)  - \nabla_y \Psi[0]\cdot \nu(s) h(s) \eta(s,t) .
\end{align*}
Therefore
\begin{align*}
L^{(0)}\left[  {\textstyle \frac{d}{d\tau}\big|_{\tau=0} \tilde \Psi[\tau h] } \right]
=L^{(0)}\left[  {\textstyle \frac{d}{d\tau}\big|_{\tau=0}  \Psi[\tau h] } \right]
- L^{(0)}[\nabla_y \Psi[0]\cdot \nu h \eta] .
\end{align*}
But $L^{(0)}= \Delta$ and by \eqref{eqPsi} $\Delta \Psi[h] = -\varepsilon\omega - \Delta \psi_0$ is independent of $h$. Hence $L^{(0)}\left[  {\textstyle \frac{d}{d\tau}\big|_{\tau=0}  \Psi[\tau h] } \right]=0$ and so
\begin{align*}
L^{(0)}\left[  {\textstyle \frac{d}{d\tau}\big|_{\tau=0} \tilde \Psi[\tau h] } \right]
=- L^{(0)}[\nabla_y \Psi[0]\cdot \nu h \eta].
\end{align*}
Also 
\begin{align*}
\frac{d}{d\tau}\Big|_{\tau=0} \tilde \psi_0[\tau h](y) = 
\frac{d}{d\tau}\Big|_{\tau=0} \psi_0( X_h(X_0^{-1}(y))
=
 - \nabla_y \Psi[0]\cdot \nu  h \eta.
\end{align*}

Using \eqref{formula1} in \eqref{eq1} we get
\begin{align}
\nonumber
L_h [\tilde \Psi[h] ] + \varepsilon \omega + L_h [\tilde \psi_0[h]] 
&=
L^{(0)}[ \tilde \Psi_1[h] ]  + \tilde Q_1(\tilde \Psi_1[h] ,h) 
\end{align}
where 
\begin{align}
\nonumber
\tilde Q_1(\tilde \Psi_1[h] ,h) &=  L_h^{(1)}[ \tilde \Psi_1[h] ] - L_h^{(1)}[ \nabla_y \Psi[0] \cdot \nu \eta h ]
\\
\nonumber
& \quad 
+ L_h^{(2)}[ \tilde \Psi_1[h] ] +  L_h^{(2)}[\tilde \Psi[0] ] - L_h^{(2)}[ \nabla_y \Psi[0] \cdot \nu \eta h ]
\\
\nonumber
& \quad 
+ L^{(0)}[ \tilde \psi_0[h]  - \tilde \psi_0[0] + \nabla_y   \tilde \psi_0[0] \cdot \nu \eta h]
\\
\label{Q1}
& \quad 
+ L_h^{(1)} [\tilde \psi_0[h] - \tilde \psi_0[0]] 
+ L_h^{(2)} [\tilde \psi_0[h]]  .
\end{align}

Let
\[
\hat \Psi_1[h] =\tilde \Psi_1 [h]\circ X_0
\]
so that
\[
\hat  \Psi[h]  = \hat \Psi_1[h] +  \hat \Psi[0]  -  \nabla_y \Psi[0]\circ X_0 \cdot \nu \eta h .
\]
We want to write the boundary condition \eqref{bc1} in terms of $\hat \Psi_1[h]$:
\begin{align*}
\hat B_h( & \hat \Psi_1[h] +  \hat \Psi[0]  -  \nabla_y \Psi[0]\circ X_0 \cdot \nu  h + \hat\psi_0[h], 
\\
& \hat \Psi_1[h] +  \hat \Psi[0]  -  \nabla_y \Psi[0]\circ X_0 \cdot \nu  h + \hat\psi_0[h]) + 2 g (\varepsilon \tilde y_2[h]+d^S) = 1,
\end{align*}
since $\eta=1$ in a neighborhood of $t=0$.
Using the bilinearity and symmetry of $\hat B_h$ this is the same as
\begin{align}
\nonumber
& \hat B_h ( \hat \Psi_1[h] ,\hat \Psi_1[h] )
+ \hat B_h ( \hat \Psi[0]  ,  \hat \Psi[0] )
+  \hat B_h( \nabla_y \Psi[0]\circ X_0 \cdot \nu  h , \nabla_y \Psi[0]\circ X_0 \cdot \nu  h) 
\\
\nonumber
& \quad 
+  \hat B_h( \hat\psi_0[h], \hat\psi_0[h] )
\\
\nonumber
& \quad 
+ 2 \hat B_h( \hat \Psi_1[h]  ,  \hat \Psi[0] )
- 2 \hat B_h( \hat \Psi_1[h] ,  \nabla_y \Psi[0]\circ X_0 \cdot \nu  h)
+  2  \hat B_h( \hat \Psi_1[h]  , \hat\psi_0[h] )
\\
\nonumber
& \quad 
- 2 \hat B_h( \hat \Psi[0] ,\nabla_y \Psi[0]\circ X_0 \cdot \nu  h)
+ 2 \hat B_h( \hat \Psi[0] , \hat\psi_0[h] )
\\
\nonumber
& \quad 
-  2 \hat B_h( \nabla_y \Psi[0]\circ X_0 \cdot \nu  h ,  \hat\psi_0[h] )
\\
\label{bc2}
& \quad + 2 g (\varepsilon \tilde y_2[h]+d^S) = 1.
\end{align}

We want to expand this expression with respect to $h$. 
First we expand the operator $\hat B_h$:
\begin{align}	
\label{expansion-B}
\hat B_h = \hat B^{(0)} +  \hat B_h^{(1)} +  \hat B_h^{(2)}
\end{align}
where
\begin{align*}
\hat B^{(0)} (\hat\psi_1,\hat\psi_2) 
&= \partial_s \hat\psi_1 \partial_s \hat\psi_2 + \partial_t \hat\psi_1\partial_t \hat\psi_2
\\
\hat B_h^{(1)} (\hat\psi_1,\hat\psi_2) 
&=2 \kappa h  \partial_s \hat\psi_1 \partial_s \hat\psi_2 
- h'  \partial_s \hat\psi_1 \partial_t \hat  \psi_2
- h' \partial_s \hat\psi_2 \partial_t \hat  \psi_1,
\end{align*}
and ${}'=\frac{d}{ds}$.
In \eqref{bc2} we also expand $\hat\psi_0[h]$ with respect to $h$:
\begin{align}
\label{expansion-hatpsi0h}
\hat\psi_0[h]=\hat\psi_0[0]-(\nabla_y\psi_0)\circ X_0 \cdot \nu h  +q_0(h)
\end{align}
where
\[
q_0(h) = \hat\psi_0[h] -  \hat\psi_0[0] + (\nabla_y\psi_0)\circ X_0 \cdot \nu h .
\]
When expanding all terms with respect to $h$ in \eqref{bc2}, we identify the following terms as independent of $h$:
\begin{align}
\label{E}
E_1(g) = \hat B^{(0)}( \hat\psi_0[0] + \hat \Psi[0]  ,  \hat\psi_0[0]+ \hat \Psi[0] )  
+ 2 g (\varepsilon  y_2 +d^S) - 1.
\end{align}

Next, we look at the terms in \eqref{bc2} that are largest as functions of $h$, based on the form of the linearized operator that we expect from the computations in Section~\ref{sectFormal}.
We expand, using \eqref{expansion-B} and \eqref{expansion-hatpsi0h}
\begin{align}
\label{exp1}
\hat B_h( \hat\psi_0[h], \hat\psi_0[h] )
&=\hat B^{(0)}( \hat\psi_0[h], \hat\psi_0[h] )
+\hat B_h^{(1)}( \hat\psi_0[h], \hat\psi_0[h] )
+\hat B_h^{(2)}( \hat\psi_0[h], \hat\psi_0[h] )
\end{align}
and 
\begin{align*}
& \hat B^{(0)}( \hat\psi_0[h], \hat\psi_0[h] )
\\
&\quad =
\hat B^{(0)}( \hat\psi_0[0]-(\nabla_y\psi_0)\circ X_0 \cdot \nu h  +q_0(h), \hat\psi_0[0]-(\nabla_y\psi_0)\circ X_0 \cdot \nu h  +q_0(h))
\\
&\quad =
\hat B^{(0)}( \hat\psi_0[0], \hat\psi_0[0])
+
\hat B^{(0)}( (\nabla_y\psi_0)\circ X_0 \cdot \nu h,(\nabla_y\psi_0)\circ X_0 \cdot \nu h)
\\
& \qquad 
+ \hat B^{(0)}( q_0(h),q_0(h))
-2\hat B^{(0)}( \hat\psi_0[0],(\nabla_y\psi_0)\circ X_0 \cdot \nu h)
\\
&\qquad
+2\hat B^{(0)}( \hat\psi_0[0],q_0(h))
-2\hat B^{(0)}((\nabla_y\psi_0)\circ X_0 \cdot \nu h,q_0(h)).
\end{align*}
The first term in this expression is part of \eqref{E}.
We have that 
\begin{align*}
\hat B^{(0)} ( \hat\psi_0[0] ,\nabla_y \psi_0\circ X_0\cdot \nu  h ) 
&= \partial_s  \hat\psi_0[0]  \partial_s (\nabla_y \psi_0\circ X_0) \cdot \nu  h
+  \partial_s  \hat\psi_0[0]  \nabla_y \psi_0\circ X_0\cdot \nu' h
\\
& \quad 
+  \partial_s  \hat\psi_0[0]  \nabla_y \psi_0\circ X_0\cdot \nu h'
+ \partial_t  \hat\psi_0[0] \partial_t ( \nabla_y \psi_0\circ X_0 \cdot \nu )  h,
\end{align*}
and note that on $t=0$
\begin{align*}
\partial_t ( \nabla_y \psi_0\circ X_0 \cdot \nu ) 
&= -\partial_{tt} \hat\psi_ 0[0]
= -(\Delta_y \psi_0)\circ X_0 + \partial_{ss} \hat\psi_0[0] - \kappa \partial_t \hat\psi_0[0],
\end{align*}
so that
\begin{align*}
&-2\hat B^{(0)}( \hat\psi_0[0],(\nabla_y\psi_0)\circ X_0 \cdot \nu h
\\
&\quad=
-2\partial_s  \hat\psi_0[0]  \partial_s (\nabla_y \psi_0\circ X_0) \cdot \nu  h
-2  \partial_s  \hat\psi_0[0]  \nabla_y \psi_0\circ X_0\cdot \nu' h
\\
& \qquad 
-2 \partial_s  \hat\psi_0[0]  \nabla_y \psi_0\circ X_0\cdot \nu h'
\\
&\qquad
+2 \partial_t  \hat\psi_0[0]
((\Delta_y \psi_0)\circ X_0 - \partial_{ss} \hat\psi_0[0] + \kappa \partial_t \hat\psi_0[0]) h,
\end{align*}
which provides the linear term of order 0 in \eqref{robin3a}. Other linear terms are present in $\hat B_h(\hat\psi_0[h], \hat\psi_0[h])$, such as $\hat B_h^{(1)}(\hat\psi_0[0], \hat\psi_0[0])$, but we will verify that they are small.

The other important linear term in \eqref{bc2} comes from $2\hat B_h( \hat\Psi_1[h],\hat\psi_0[h])$. Treating $\hat \Psi_1[h]$ as unknown, using \eqref{expansion-B} we get
\begin{align}
\label{exp2}
\hat B_h(\hat\Psi_1[h],\hat\psi_0[h])
&=
\hat B^{(0)}(\hat\Psi_1[h],\hat\psi_0[h])
+\hat B_h^{(1)}(\hat\Psi_1[h],\hat\psi_0[h])
+\hat B_h^{(2)}(\hat\Psi_1[h],\hat\psi_0[h])
\end{align}
and we expand, using \eqref{expansion-hatpsi0h}
\begin{align*}
\hat B^{(0)}(\hat\Psi_1[h],\hat\psi_0[h])
&=\hat B^{(0)}(\hat\Psi_1[h],\hat\psi_0[0])
-\hat B^{(0)}(\hat\Psi_1[h],(\nabla_y\psi_0)\circ X_0 \cdot \nu h )
\\
&\quad +\hat B^{(0)}(\hat\Psi_1[h],q_0(h)).
\end{align*}
Then we note that
\begin{align*}
\hat B^{(0)} ( \hat \Psi_1[h],\hat\psi_0[0] ) 
&= \partial_s   \hat \Psi_1[h]  \partial_s  \hat\psi_0[0] +
\partial_t   \hat \Psi_1[h]  \partial_t  \hat\psi_0[0].
\end{align*}
Thus \eqref{bc2} contains the term
\[
2 \partial_t\hat \Psi_1[h]\partial_t\hat\psi_0[0],
\]
which essentially corresponds to the normal derivative term in \eqref{robin3a}.

Let us rename
\begin{align}
\label{rename-u}
u = \tilde \Psi_1[h],\quad\text{and}\quad
\hat u=u\circ X_0.
\end{align}
Then \eqref{bc2} takes the form
\begin{align*}
0&=
2 \partial_t\hat\psi_0[0]\left[\partial_t\hat u-(\Delta_y \psi_0\circ X_0 + \kappa) h \right] + E_1 + \hat Q_{2}(\hat u,h,g) ,
&&\text{on } t=0
\end{align*}
where $\hat Q_{2}$ collects all the terms that have been left out. This includes quadratic terms in $u$, $h$ as well as small linear terms. Explicitly, $\hat Q_{2}$ is given by
\begin{align}
\label{hatQ2}
\hat Q_{2}
&=\hat Q_{2,a}+\hat Q_{2,b}+\hat Q_{2,c}+\hat Q_{2,d}+\hat Q_{2,e}
+\hat Q_{2,f}
\end{align}
where
\begin{align}
\nonumber
\hat Q_{2,a}(\hat u,h,g)
&=\hat B_h(\hat u,\hat u)
+\hat B_h( \nabla_y \Psi[0]\circ X_0 \cdot \nu h, \nabla_y \Psi[0]\circ X_0 \cdot \nu h)+ 2 \hat B_h( \hat u ,  \hat \Psi[0] ) 
\\
\nonumber
&\quad - 2 \hat B_h(\hat u,  \nabla_y \Psi[0]\circ X_0 \cdot \nu  h)
- 2 \hat B_h( \hat \Psi[0] ,\nabla_y \Psi[0]\circ X_0 \cdot \nu  h)
\\
\label{Q2a}
&\quad-  2 \hat B_h( \nabla_y \Psi[0]\circ X_0 \cdot \nu  h ,  \hat\psi_0[h]),
\end{align}
contain terms directly from \eqref{bc2},
\begin{align}
\nonumber
\hat Q_{2,b}(\hat u,h,g)
&=\hat B_h^{(1)}( \hat\psi_0[h], \hat\psi_0[h] )
+\hat B_h^{(2)}( \hat\psi_0[h], \hat\psi_0[h] )
\\
\nonumber
&\quad 
+
\hat B^{(0)}( (\nabla_y\psi_0)\circ X_0 \cdot \nu h,(\nabla_y\psi_0)\circ X_0 \cdot \nu h)
\\
\nonumber
& \quad +
\hat B^{(0)}( q_0(h),q_0(h))
\\
\label{Q2b}
&\quad
+2\hat B^{(0)}( \hat\psi_0[0],q_0(h))
-2\hat B^{(0)}((\nabla_y\psi_0)\circ X_0 \cdot \nu h,q_0(h))
\end{align}
contains the terms originating in $\hat B_h( \hat\psi_0[h], \hat\psi_0[h] )$, see \eqref{exp1},
\begin{align}
\nonumber
\hat Q_{2,c}(\hat u,h,g)
&=
2\hat B_h^{(1)}(\hat u,\hat\psi_0[h])
+2\hat B_h^{(2)}(\hat u,\hat\psi_0[h])
\\
\nonumber
&\quad
-2\hat B^{(0)}(\hat u,(\nabla_y\psi_0)\circ X_0 \cdot \nu h )
\\
\label{Q2c}
&\quad +2\hat B^{(0)}(\hat u,q_0(h))
+2\partial_s\hat u\partial_s\hat\psi_0[0]
\end{align}
contains the terms originating in $2\hat B_h( \hat\Psi_1[h],\hat\psi_0[h])$, see \eqref{exp2},
\begin{align}
\label{Q2d}
\hat Q_{2,d}(\hat u,h,g)
&=
\hat B_h^{(1)} ( \hat \Psi[0]  ,  \hat \Psi[0] )
+\hat B_h^{(2)} ( \hat \Psi[0]  ,  \hat \Psi[0] )
\end{align}
contains terms originating in $\hat B_h^{(1)} ( \hat \Psi[0]  ,  \hat \Psi[0] )$, with $\hat B^{(0)} ( \hat \Psi[0]  ,  \hat \Psi[0] )$ removed (it is part of \eqref{E}),
\begin{align}
\nonumber
\hat Q_{2,e}(\hat u,h,g)
&=
2 \hat B_h^{(1)}( \hat \Psi[0] , \hat\psi_0[h] )
+ 2 \hat B_h^{(2)}( \hat \Psi[0] , \hat\psi_0[h] )
\\
\label{Q2e}
&\quad
+ 2 \hat B_h^{(0)}( \hat \Psi[0] , -(\nabla_y\psi_0)\circ X_0 \cdot \nu h   )
+ 2 \hat B_h^{(0)}( \hat \Psi[0] , q_0(h) )
\end{align}
contains terms that originate in $2 \hat B_h( \hat\Psi[0]+\hat\psi_0[h] , \hat\Psi[0]+\hat\psi_0[h] )$, but where $2 \hat B^{(0)}( \hat\Psi[0] , \hat\psi_0[0])$ was removed (it is part of \eqref{E}), and
\begin{align}
\label{Q2f}
\hat Q_{2,f}(\hat u,h,g)
=-2 \varepsilon g \nu\cdot e_2 h,
\end{align}
where $e_2=(0,1)$.

At this point, problem \eqref{bc1} can be written as follows:
Find $h$ such that $u=\tilde \Psi_1[h]$ satisfies 
\begin{align}
\label{eq2}
0&= \Delta_y u + \tilde Q_1(u,h) , && \text{in }\Omega_0 ,
\\
\label{bcD}
u(y) &= \nabla_y \psi_0(x) \cdot \nu(s) h(s) + Q_3(h(s),s) ,&& y = \gamma(s) \in \mathcal{S}_0
\\
\nonumber
0&=
2 \partial_t\hat\psi_0[0]\left[\partial_t\hat u-( (\Delta_y \psi_0)\circ X_0 - \partial_{ss} \hat\psi_0[0] + \partial_t \hat\psi_0[0] \kappa) h \right] 
\\
\label{bcN}
&\qquad + E_1(g)  + \hat Q_{2}(\hat u, h) ,
&&\text{on } t=0
\\
\label{bcB0}
0&=u&&\text{on }\mathcal{B}_0
\end{align}
where $\tilde Q_1$ is defined in \eqref{Q1}, $\hat Q_2$ in \eqref{hatQ2}, $E_1$ in \eqref{E}, and 
\begin{align}
\nonumber
Q_3(h,s) &=  -[\psi_0(y-\nu(s) h(s)) - \psi_0(y) +\nabla_y \psi_0(y) \cdot \nu(s) h(s)]
\\
\label{Q3}
&\quad
+ \nabla_y \Psi[0]\cdot \nu(s) h(s).
\end{align}
The condition \eqref{bcB0} follows directly from the definition \eqref{tildePsi1} since $\tilde \Psi[h]=\tilde \Psi[0]=0$ and
$\nabla_y \Psi[0] \cdot ( (\nu \eta h)\circ X_0^{-1})=0$ on $\mathcal{B}_0$.

Assuming that $u$ is sufficiently small on $\mathcal{S}_0$, there is a unique small solution $h(s)$ of \eqref{bcD},
which we denote by the function
\begin{align}
\label{def:a}
h(y) = \mathbf{h}( u(y),y ).
\end{align}
Imposing \eqref{bcD} through \eqref{def:a}, we formulate the problem \eqref{eq2}, \eqref{bcD}, \eqref{bcN}, \eqref{bcB0} as the following nonlinear Robin problem in the unknown $u$:
\begin{align}
\label{main-problem}
\left\{
\begin{aligned}
\Delta_y u + Q_1[u]&=0 \quad \text{in }\Omega_0
\\
\frac{\partial u}{\partial \nu}
+ (\kappa - \varepsilon\omega) u + E + Q_2[u,g] &= 0 
\quad \text{on } \mathcal{S}_0
\\
u&=0\quad \text{on } \mathcal{B}_0
\end{aligned}
\right.
\end{align}
where
\begin{align}
\label{newQ1}
Q_1[u](y)&=\tilde Q_1(u,\mathbf{h}(u,y)) 
\end{align}
\begin{align}
\label{def:E1}
E(g) = \frac{1}{2\partial_t \hat\psi_0[0]} E_1(g)
\end{align}
with $E_1$ defined in \eqref{E}, and
\begin{align}
\nonumber
Q_2[u,g]
&= - \frac{\hat Q_{2}(\hat u, \mathbf{h}(u,y))}{2 \partial_t\hat\psi_0[0]} 
\\
\nonumber
&\quad-(\kappa-\varepsilon\omega)\left(u\Bigl(\frac{1}{\nabla_y \psi_0 \cdot \nu}+1\Bigr)-\frac{Q_3(\mathbf{h}(u,y),s)}{\nabla_y \psi_0 \cdot \nu}\right)
\\
\label{newQ2}
&\quad
-( \Delta_y \psi_0 +\varepsilon\omega- \partial_{ss} \hat\psi_0[0] + (\partial_t \hat\psi_0[0]-1) \kappa) \mathbf{h}(u,y)
\end{align}
with $\hat Q_2$ defined in \eqref{hatQ2} and $Q_3$ in \eqref{Q3}.

\begin{proof}[Proof of Proposition~\ref{prop:change}]
If we find a solution $u$ to \eqref{main-problem}, then we define $h$ by \eqref{def:a} and $F_h$ by \eqref{def:Fh}. Then we let  $\Omega_h$ be defined by \eqref{def:Omegah} and $\psi=\psi_0+\Psi[h]$. 
We summarize these definitions as
\begin{align}
\label{Omega-psi}
h(y) = \mathbf{h}(u(y),y ) , \quad
\Omega_h = F_h(\Omega_0) , \quad
\psi=\psi_0+\Psi[h] .
\end{align}
By the preceding computations, if
\begin{align}
\label{diffeo}
\text{$F_h$ is a diffeomorphism from $\overline\Omega_0$ onto $\overline\Omega_h$}
\end{align}
then $(\Omega_h, \psi)$ satisfies \eqref{overd}.
\end{proof}

We conclude this section by clarifying the sense in which $h$ and $u$ must be small.

We begin with $u$, the unknown in the nonlinear problem \eqref{main-problem}. To ensure the problem's well-formulation, we must be able to uniquely solve for $h$ in equation \eqref{bcD}. This equation is nearly quadratic, and to smoothly solve for $h(y)$ in terms of $u(y)$ and $y$, we require
\begin{align}
\label{small-nl}
|D_y^2\psi_0(y)| |u(y)|\leq \bar\delta,\quad y\in\partial\Omega_0,
\end{align}
where $\bar\delta>0$ is a small constant. But from the definition of $\psi_0$ \eqref{psi0} we get
\begin{align*}
|D_y^2 \psi_0(y)|\leq C \Bigl(\frac{1}{|y|^2}+\varepsilon \Bigr).
\end{align*}
We will work with $u$ in a class where 
\begin{align}
\label{hyp-u}
|u(y)| \leq \varepsilon^\theta 
\begin{cases}
(1+|y|)^{1-\sigma}&\quad |y| \leq \frac{\delta_1}{\varepsilon}
\\
(\frac{\delta_1}{\varepsilon})^{1-\sigma }
&\quad |y| \geq  \frac{\delta_1}{\varepsilon},
\end{cases}
\quad y \in \Omega_0,
\end{align}
for some $\theta>0$ and $0<\sigma<1$. 
Under assumption \eqref{hyp-u} on $u$ we get \eqref{small-nl} for $\varepsilon>0$ small. Solving for $h$ in \eqref{bcD} we get that 
\begin{align}
\label{est-h}
|h(y)| \leq C \varepsilon^\theta 
\begin{cases}
(1+|y|)^{1-\sigma}&\quad |y| \leq \frac{\delta_1}{\varepsilon}
\\
(\frac{\delta_1}{\varepsilon})^{1-\sigma }
&\quad |y| \geq  \frac{\delta_1}{\varepsilon},
\end{cases}
\quad y \in\mathcal{S}_0.
\end{align}

After solving for $h$, we need to check that the operators that define $\tilde Q_1$ are well defined, since they involve coefficients of the form 
\begin{align}
\nonumber
 \frac{1}{(1-\kappa(t+h\eta))^a (1+h\eta_t)^b}
\end{align}
for some integers $a,b\geq 0$.
These coefficients have to be estimated in a neighborhood of $\partial\Omega_0$ of the form $t \leq 2 \rho(s)$, where $s$, $t$ are the $(s,t)$ coordinates defined by $X_0$ in \eqref{Xh}.

We use that $\kappa$ satisfies \eqref{est-kappa}
and the following estimates derived from \eqref{est-eta}, for $y=X_0(s,t)$ with $|y|\leq\frac{\delta_1}{\varepsilon}$:
\begin{align*}
|\partial_t \eta(s,t)|&\leq \frac{C}{1+|y|},
&
|\partial_{tt} \eta(s,t)|&\leq \frac{C}{1+|y|^2}
\\
|\partial_s \eta(s,t)|&\leq \frac{C}{1+|y|},
&
|\partial_{ss} \eta(s,t)|&\leq \frac{C}{1+|y|^2} .
\end{align*}
These derivatives are supported in the region $\rho(s)\leq t\leq 2\rho(s)$.

By \eqref{kappa-d} and \eqref{est-h} we see that if $u$ satisfies \eqref{hyp-u} and $h$ is given by \eqref{def:a}, for $\varepsilon>0$ small 
\begin{align*}
| \kappa(s)(t+h(s)\eta(s,t))|\leq \tfrac{1}{2},
\end{align*}
for $s\in \R$, $t\leq 2\rho(s)$. Similarly, by taking $\varepsilon$ small, if $u$ satisfies \eqref{hyp-u} and $h$ given by \eqref{def:a}, 
\begin{align*}
|h\eta_t| \leq \tfrac{1}{2}
\quad s\in \R, \ t\leq 2\rho(s).
\end{align*}

\section{Estimate of the error}
\label{sect:error}

In this section we estimate the size of the error in problem \eqref{overd0} produced by the approximation $\Omega_0$, $\psi_0$ constructed in Section \ref{sec:approx}.

\medskip

The following norms are suitable for the right hand sides in problem \eqref{main-problem}.
Let $\Omega \subset \R^2$ be an open set with $C^2$ boundary. Let $0<\sigma,\alpha <1$, $\mu>0$ as in \eqref{eq:mu} and $\delta_1>0$. We assume that $\mu$ satisfies \eqref{eq:mu}.
For a function $f$ defined on $\emptyset\not=D\subset \R^2$ we use the notation
\[
[f]_{\alpha,D}
= \mathop{\sup_{y_1,y_2 \in D}}_{y_1\not=y_2}
\frac{|f(y_1)-f(y_2)|}{|y_1-y_2|^\alpha} .
\]
For $f_1 \in C^\alpha(\overline\Omega)$ we let
\begin{align}
\label{defNf1Omega0}
\| f_1 \|_{**,\Omega} & = \text{the least $M$ such that for $y \in \Omega$}
\\
\nonumber
& |f_1(y)|+(1+|y|)^\alpha [f_1]_{\alpha,B(y,\frac{|y|}{10})\cap\Omega}
\leq (1+|y|)^{-1-\sigma} M && \textstyle\text{for } |y| \leq \frac{\delta_1}{\varepsilon}
\\
\nonumber
& 
\textstyle
|f_1(y)| +(\frac{\delta_1}{\varepsilon})^\alpha [f_1]_{\alpha,B(y,\frac{\delta_1}{10 \varepsilon})\cap\Omega}
\leq (\frac{\varepsilon}{\delta_1} )^{1+\sigma} e^{-\mu \varepsilon |y|} M
&&  \textstyle \text{for } |y| \geq \frac{\delta_1}{\varepsilon} .
\end{align}
\index{$\norm \  \norm_{**} $, norm for the RHS of the linearized problem in $\Omega_0$}
For $f_2 \in C^{1,\alpha}(\partial \Omega)$ we let
\begin{align}
\label{defNf2Omega0}
\| f_2 \|_{**,\partial\Omega} &= \text{the least $M$ such that for $y \in \partial\Omega$}
\\
\nonumber
& \textstyle |f_2(y)|
+(1+|y|) |f_2'(y)|
+(1+|y|)^{1+\alpha} [f_2']_{\alpha,B(y,\frac{|y|}{10})\cap\partial\Omega}
\leq (1+|y|)^{-\sigma} M 
\\
\nonumber
& \qquad \qquad \qquad \qquad \qquad \qquad \qquad \qquad \qquad \qquad \qquad \qquad  \textstyle \text{for } |y| \leq \frac{\delta_1}{\varepsilon}
\\
\nonumber
& 
\textstyle
|f_2(y)| 
+\frac{\delta_1}{\varepsilon} |f_2'(y)|
+(\frac{\delta_1}{\varepsilon})^{1+\alpha} [f_2']_{\alpha,B(y,\frac{\delta_1}{10 \varepsilon})\cap\partial\Omega}
\leq (\frac{\varepsilon}{\delta_1})^{\sigma} e^{-\mu \varepsilon |y|} M
\\
\nonumber
& \qquad \qquad \qquad \qquad \qquad \qquad \qquad \qquad \qquad \qquad \qquad \qquad  \textstyle \text{for }  |y| \geq \frac{\delta_1}{\varepsilon} ,
\end{align}
where $f_2'$ denotes the derivative of $f_2$ with respect to arc length.
\index{$\norm \  \norm_{*} $, norm for the RHS on the boundary of the linearized problem in $\Omega_0$}

\medskip
Recall the definition of $E$ \eqref{def:E1} and also that $g=\varepsilon g_0+g_1$ \eqref{decomp-g} where $g_0$ is the constant defined in \eqref{def-g0}. Then it is natural to define
\begin{align}
\label{def:E0}
E_0(y;g_1) = E(y;\varepsilon g_0+g_1) - g_1 ( \varepsilon y_2+d^S).
\end{align}
The idea is to isolate the main order contribution of the term $g_1 ( \varepsilon y_2+d^S)$ from the error, which is taken care of by the linear theory. The cancellation is not exact and $E_0(y;g_1) $ has still a small dependence on $g_1$.
We assume that
\begin{align}
\label{hg}
|g_1| \leq C |\varepsilon|^{1+m},
\end{align}
for some constant $0<m<1$.

\begin{proposition}
\label{prop:est-error-bdy}
Assume $\frac12<b<\frac23$, \eqref{hg}, and that $\mu$ satisfies \eqref{eq:mu}.
Then
\begin{align*}
\|E_0(g_1)\|_{**, \partial \Omega_0 }\leq C \varepsilon^{\ell_0} |\log\varepsilon|
\end{align*}
where 
\begin{align*}
\ell_0 = \min(1-\sigma b, 4(1-b) - \sigma b, 1+m+b(1-\sigma) ).
\end{align*}
\end{proposition}

The proof is divided into several computations that we state as separate lemmas 
which are proved later in the section.

Recall that $E$ is defined in \eqref{def:E1} where $E_1$ is given by \eqref{E}.
Let us write
\begin{align*}
P(y) = \partial_t \hat\psi_0[0]|_{t=0}(y) = - \nabla_y \psi_0\cdot \nu (y)
\end{align*}
and decompose
\begin{align*}
E_0 &= \frac{E_{0,1}}{2P}+\frac{E_{0,2}}{2P}+\frac{E_{0,3}}{2P} ,
\end{align*}
where
\begin{align*}
E_{0,1} &=
|\nabla_y \psi_0|^2-1+ 2 \varepsilon g_0  (\varepsilon y_2 +d^S) 
\\
E_{0,2}&=
2(1-P) g_1 (\varepsilon y_2+d^S)
\\
E_{0,3}&=
2 \nabla_y \psi_0 \cdot \nabla_y \Psi[0] + |\nabla_y \Psi[0]|^2 .
\end{align*}

We start with some properties of $P$.
\begin{lemma}
\label{lem:normal-psi0}
For $\varepsilon>0$ sufficiently small
\begin{align*}
\frac{1}{2}\leq  P(y) \leq 2, \quad y\in \partial\Omega_0,
\end{align*}
\begin{align}
\label{estP}
P(y)=1+O(\varepsilon|\log\varepsilon|)+O\Bigl(\frac{1}{|y|}\Bigr), \quad 
|y|\geq 2\varepsilon^{-b} ,
\end{align}
and for $k=1,2$
\begin{align*}
(1+|y|)^k |D^k_y P(y)| \leq C \Bigl( \frac{1}{1+|y|} + \varepsilon (1+|y|)\Bigr) , \quad |y| \leq \frac{\delta_1}{\varepsilon} .
\end{align*}
\end{lemma}

\begin{lemma}
\label{lem:error1}
Assume that $\mu$ satisfies \eqref{eq:mu}.
Then for $k=0,1,2$:
\begin{align}
\label{estE1k1}
(1+|y|)^k |D^k_y E_{0,1}(y)|
\leq 
C \varepsilon^{\ell_1}
(1+|y|)^{-\sigma} , \quad  |y| \leq \frac{\delta_1}{\varepsilon}
\end{align}
\begin{align}
\label{estE1k2}
|D^k_y E_{0,1}(y)| \leq C \varepsilon^{\ell_1}
(\frac{\varepsilon}{\delta_1})^{k+\sigma} e^{-\mu \varepsilon|y|}, \quad  |y| \geq \frac{\delta_1}{\varepsilon}
\end{align}
where $\ell_1$ is given by 
\begin{align}
\label{ell1}
\ell_1 = \min(1-\sigma b, 4(1-b) - \sigma b) .
\end{align}
\end{lemma}

\medskip

\begin{lemma}
\label{lem:error1-2}
Let $\frac12<b<\frac23$.
Assume that $g_1$ satisfies \eqref{hg} and that $\mu$ satisfies \eqref{eq:mu}.
Then for $k=0,1,2$:
\begin{align}
\label{estE1k1-2}
(1+|y|)^k |D^k_y E_{0,2}(y)|
\leq 
C \varepsilon^{1+m+b(1-\sigma)}|\log\varepsilon|
(1+|y|)^{-\sigma} , \quad  |y| \leq \frac{\delta_1}{\varepsilon}
\end{align}
\begin{align}
\nonumber
|D^k_y E_{0,2}(y)| \leq C \varepsilon^{1+m+b(1-\sigma)}|\log\varepsilon|
(\frac{\varepsilon}{\delta_1})^{k+\sigma} e^{-\mu \varepsilon|y|}, \quad  |y| \geq \frac{\delta_1}{\varepsilon}.
\end{align}
\end{lemma}

\medskip

To find an estimate for $E_{0,3}$ let us write
\begin{align*}
	\Psi_0[0] = u_0 + u_1
\end{align*}
where $u_0$ is the solution to 
\begin{align}
	\label{equ0}
	\left\{
	\begin{aligned}	
		\Delta u_0 &= 0  && \text{in } \Omega_0, 
		\\
		u_0 &= \psi_0  && \text{on } \mathcal{S}_0,
		\\
		u_0 &= 0  && \text{on } \mathcal{B}_0,
	\end{aligned}
	\right.
\end{align}
and $u_1$ is the solution to 
\begin{align}
	\label{equ1}
	\left\{
	\begin{aligned}
		\Delta u_1 &= \Delta \psi_0 + \varepsilon\omega && \text{in } \Omega_0,
		\\
		u_1 &= 0 && \text{on } \partial\Omega_0 = \mathcal{B}_0\cup\mathcal{S}_0.
	\end{aligned}
	\right.
\end{align}
When we say the solution to either \eqref{equ0} or \eqref{equ1} we mean the unique bounded solution to the problem.

\begin{lemma}
	\label{lem:error-u1}
	Let $u_1$ be the bounded solution to \eqref{equ1}.
	Then 
	\begin{align*}
		\Bigl\|  \frac{\partial u_1}{\partial \nu}
		\Bigr\|_{**,\partial\Omega_0} 
		\leq C\varepsilon^{1-\sigma b}|\log\varepsilon|.
	\end{align*} 
\end{lemma}

\begin{lemma}
	\label{lem:error-u0}
	Assume
	\begin{align}
		\label{restricb1}
		\frac{1}{2}<b<\frac{2}{3}.
	\end{align}
	Let $u_0$ be the bounded solution to \eqref{equ0}.
	Then
	\begin{align*}
		\Bigl\|  \frac{\partial u_0}{\partial \nu}
		\Bigr\|_{**,\partial\Omega_0} \leq C \varepsilon^{1-b \sigma}|\log\varepsilon|.
	\end{align*}
\end{lemma}

\begin{proof}[Proof of Lemma~\ref{lem:normal-psi0}]
	We consider the regions:
	\begin{align}
		\label{regions}
		\begin{aligned}
			R_1 &= \{ y\in\partial\Omega_0, \ |y|\leq \varepsilon^{-b} \}
			\\
			R_2 &= \{ y\in\partial\Omega_0, \ \varepsilon^{-b}\leq |y|\leq 2\varepsilon^{-b} \}
			\\
			R_3 &= \{ y\in\partial\Omega_0, \  |y|\geq 2\varepsilon^{-b} \}
		\end{aligned}
	\end{align}
	and also distinguish $y_2\geq 0$ or $y_2\leq 0$, which we write as $R_j^{\pm}$ respectively, where $y=(y_1,y_2)$.
	
	If $L>0$ is fixed, then
	\begin{align*}
		P(y) = 1 + O(\varepsilon)\quad \text{for }y\in\partial\Omega_0, \ |y|\leq L.
	\end{align*}
	
	Next we estimate in the rest of $R_1^+$.
	For $y\in R_1^+$
	\begin{align*}
		\nu(y_1) = \frac{( (g^+)'(y_1),-1)}{\sqrt{1+( (g^+)'(y_1))^2}}
	\end{align*}
	where $g^+=\tilde g^+$ \eqref{tgplus}, \eqref{upperbdy}.
	Therefore for $y = (y_1,g^+(y_1)) \in \partial\Omega_0$, 
	\begin{align*}
		P(y) = - \nabla_y \psi_0(y)\cdot \nu (y) 
		&=  - \frac{ (\tilde g^+)'(y_1)}{\sqrt{1+( (\tilde g^+)'(y_1))^2}} \Bigl( \partial_{y_1} \psi^H - \frac{1}{2} \varepsilon \omega y_1 \Bigr)
		\\
		& \quad 
		+\frac{1}{\sqrt{1+( (\tilde g^+)'(y_1))^2}} \Bigl( \partial_{y_2} \psi^H - \frac{1}{2} \varepsilon \omega y_2 \Bigr)
		\\
		&= 1+O\Big( \frac{1}{|y|}+\varepsilon |y|\Bigr),
	\end{align*}
	so that the statement $\frac{1}{2}\leq P(y) \leq 2$ holds for $y\in R_1^+$.
	
	Let us consider $y\in R_3^+$. It is more convenient to work in the variable $x$, by considering
	\begin{align*}
		\tilde P(x) = P\Bigl(\frac{x}{\varepsilon}\Bigr) .
	\end{align*}
	Then
	\begin{align*}
		\tilde P(x) = -\nabla_x \Psi(x) \cdot \nu_{\partial \Omega^D_\varepsilon}(x),\quad x\in\partial\Omega^D_\varepsilon,
	\end{align*}
	where 
	\begin{align}
		\label{def:Psi}
		\Psi =  \Psi^D + \varepsilon \Psi_1^D
	\end{align}
	with $\Psi^D$ and $\Psi_1^D$ defined in \eqref{PsiD}, \eqref{eqPsi1D} respectively, and $\nu_{\partial \Omega^D_\varepsilon}$ denotes the unit outer normal to $\partial \Omega^D_\varepsilon$ \eqref{pOmegaDeps}.
	If $x\in\partial\Omega^D_\varepsilon$ we write
	\[
	x = z+\varepsilon\Psi_1^D(z)\nu_{\partial\Omega^D}(z)
	\]
	with $z\in \partial\Omega^D$ and $\nu_{\partial\Omega^D}$ the unit outer normal to $\partial\Omega^D$. Then we compute
	\begin{align*}
		\nu_{\partial\Omega^D_\varepsilon}(x) 
		& = \frac{1+\varepsilon R \Psi_1^D(z)}{\sqrt{(1+\varepsilon R \Psi_1^D(z))^2+\varepsilon^2 (\nabla \Psi_1^D(z) \cdot \nu_{\partial\Omega^D}^\perp(z))^2}}\nu_{\partial\Omega^D}(z) 
		\\
		& \quad - \frac{\varepsilon \nabla \Psi_1^D(z) \cdot \nu_{\partial\Omega^D}^\perp(z)}{\sqrt{(1+\varepsilon R \Psi_1^D(z))^2+\varepsilon^2 (\nabla \Psi_1^D(z) \cdot \nu_{\partial\Omega^D}^\perp(z))^2}}\nu_{\partial\Omega^D}^\perp(z) .
	\end{align*}
	We also compute
	\begin{align*}
		\nabla\Psi(x)
		&= \nabla \Psi^D(z+\varepsilon\Psi_1^D(z)\nu_{\partial\Omega^D}(z)) 
		+ \varepsilon \Psi_1^D(z+\varepsilon\Psi_1^D(z)\nu_{\partial\Omega^D}(z)) 
		\\
		&= -\nu_{\partial\Omega^D}(z)+O\Bigl(\frac{\varepsilon}{|z|}\Bigr) ,
	\end{align*}
	by Lemma~\ref{lem:sol-disc2}. 
	This gives
	\begin{align*}
		P(y)=1+O(\varepsilon|\log\varepsilon|)+O\Bigl(\frac{1}{|y|}\Bigr)
	\end{align*}
	for $y\in R_3^+$ and in particular $\frac{1}{2}\leq P(y)\leq 2$ in this region for $\varepsilon$ small. 
	The estimate $\frac{1}{2}\leq P(y)\leq 2$ in the other regions and the estimate for the derivatives of $P$ are obtained similarly.
\end{proof}

\begin{proof}[Proof of Lemma~\ref{lem:error1}]
	
	We use the regions $R_j^\pm$ defined in \eqref{regions}.
	
	\medskip
	\noindent
	{\bf Estimate in $R_1^+$.}
	We claim that 
	\begin{align}
		\label{estR1p}
		\sup_{y\in R_1^+}
		(1+|y|)^{\sigma} \,
		\left| |\nabla_y \psi_0|^2 -1  + 2 \varepsilon g_0(\varepsilon y_2 + d^S)  \right| 
		\leq C \varepsilon^{\ell_1} |\log\varepsilon|,
	\end{align}
	where $\ell_1$ is given by \eqref{ell1}.
	Indeed, in the region $R_1^+$ we estimate
	\begin{align*}
		\left| |\nabla_y \psi_0|^2 -1  + 2 \varepsilon g_0(\varepsilon y_2 + d^S)  \right| 
		\leq 
		\left| |\nabla_y \psi_0|^2 -1   \right| 
		+
		2 \left|   \varepsilon g_0(\varepsilon y_2 + d^S)  \right|  .
	\end{align*}
	Using \eqref{dSeps} and that for $y\in R_1^+$, $|y_2| \leq C( |\log\varepsilon|+ \varepsilon y_1^2)$ we get
	\begin{align*}
		\sup_{y\in R_1^+}  (1+|y|)^\sigma |\varepsilon g_0(\varepsilon y_2 + d^S) |
		& \leq C 
		\sup_{y\in R_1^+} (1+|y|)^\sigma ( \varepsilon^2|\log\varepsilon| + \varepsilon^3 |y|^2)
		\\
		& \leq C  \varepsilon^{3-2b-\sigma b}.
	\end{align*}
	
	Next we estimate $(1+|y|)^{\sigma} \,\left| |\nabla_y \psi_0|^2 -1   \right| $ in the region $R_1^+$. We note that for any $K>0$
	\begin{align*}
		(1+|y|)^{\sigma} \,\left| |\nabla_y \psi_0|^2 -1   \right| \leq C \varepsilon,
		\quad |y|\leq K, \ y\in\partial\Omega_0.
	\end{align*}
	Next consider the region
	\begin{align}
		\label{reg1b}
		y \in \partial \Omega_0 , \quad 
		10 \leq  |y|\leq  \varepsilon^{-b},
		\quad 
		y_2>0 .
	\end{align}
	Here we have $ \psi_0(y) = \tilde \psi^H(y)=\psi^H(y) - \frac{1}{4}\varepsilon \omega |y|^2 $ and so
	\begin{align}
		\label{grad2Psi}
		|\nabla_y \tilde \psi^H|^2 
		= |\nabla_y \psi^H|^2
		-  \varepsilon \omega \partial_{y_1} \psi^H y_1
		-  \varepsilon \omega \partial_{y_2} \psi^H y_2
		+ \frac{1}{4} \varepsilon^2 \omega |y|^2 .
	\end{align}
	
	We have 
	\begin{align*}
		\partial_{y_j} \psi^H\Bigl(y_1,g_H^+(y_1) + \frac{1}{4}\varepsilon \omega y_1^2 \Bigr)
		& = 
		\partial_{y_j} \psi^H(y_1,g_H^+(y_1) )
		+ \frac{1}{4}\partial_{y_j,y_2} \psi^H(y_1,g_H^+(y_1) + \xi )
		\varepsilon \omega y_1^2 
	\end{align*}
	for some $0<\xi<\frac{1}{4}\varepsilon \omega y_1^2$.
	Then by \eqref{D2psi}
	\begin{align*}
		\partial_{y_j} \psi^H\Bigl(y_1,g_H^+(y_1) + \frac{1}{4}\varepsilon \omega y_1^2 \Bigr)
		& = 
		\partial_{y_j} \psi^H(y_1,g_H^+(y_1) )
		+O(\varepsilon)
	\end{align*}
	and so
	\begin{align}
		\label{b1}
		\Bigl|\nabla  \psi^H\Bigl(y_1,g_H^+(y_1) + \frac{1}{4}\varepsilon \omega y_1^2 \Bigr)
		\Bigr|^2 - 1 = O(\varepsilon) ,
	\end{align}
	in the region \eqref{reg1b}.
	
	Suppose now that in addition to \eqref{reg1b} we have $|y_1| < (\varepsilon^{-1} |\log \varepsilon|)^{1/2} $.
	We are evaluating \eqref{grad2Psi} at $ y_2 = \tilde g_H^+(y_1)$ (c.f. \eqref{tgplus}).
	Note that \eqref{est1a} implies $ |\partial_{y_1} \psi^H(y)| \leq \frac{C}{|y|}$ and so
	\begin{align}
		\label{b2a}
		|\varepsilon \omega \partial_{y_1} \psi^H y_1|
		& \leq C \varepsilon .
	\end{align}
	Similarly, from \eqref{est2a}
	\begin{align}
		\label{b2b}
		| \varepsilon \omega \partial_{y_2} \psi^H y_2|
		& \leq C \varepsilon \log(y_1) \leq C \varepsilon |\log\varepsilon| .
	\end{align}
	Also, since  $ y_2 = \tilde g_H^+(y_1)$, we have
	\begin{align}
		\label{b2c}
		\varepsilon^2 \omega |y|^2 & \leq C \varepsilon^2 y_1^2.
	\end{align}
	Combining \eqref{b1}--\eqref{b2c}  we get
	\begin{align*}
		\mathop{\sup_{y \in \partial \Omega_0, \, y_2>0}}_{|y|^2\leq \varepsilon^{-1}|\log \varepsilon |}
		(1+|y|)^{\sigma} \,
		\left| |\nabla \psi_0|^2 -1  \right|
		\leq C \varepsilon^{1-\sigma/2} |\log\varepsilon|^{1+\sigma/2}.
	\end{align*}
	
	Consider the term $|\nabla \psi_0|^2-1$ in the region \eqref{reg1b} but now also assuming that  $|y_1| \geq  (\varepsilon^{-1} |\log \varepsilon|)^{1/2} $.
	Let us write \eqref{grad2Psi} as 
	\begin{align}
		\nonumber
		|\nabla \tilde \psi^H|^2 
		= |\nabla \psi^H|^2
		-  \varepsilon \omega \partial_{y_1} \psi^H y_1
		+ (A) 
	\end{align}
	where 
	\begin{align*}
		(A) = -  \varepsilon \omega \partial_{y_2} \psi^H y_2
		+ \frac{1}{4} \varepsilon^2 \omega^2 |y|^2 .
	\end{align*}
	Using that $y_2 = g_H^+(y_1) + \frac{1}{4} \varepsilon\omega y_1^2$ and \eqref{est2a} we get
	\begin{align*}
		(A) 
		&=  -  \varepsilon \omega \Bigl( 1 + O(\frac{1}{|y_1|}) \Bigr)\Bigl(g_H^+(y_1) + \frac{1}{4} \varepsilon\omega y_1^2\Bigr)
		+ \frac{1}{4} \varepsilon^2 \omega^2 y_1^2
		\\
		& \quad 
		+ \frac{1}{4} \varepsilon^2 \omega^2 
		\Bigl(g_H^+(y_1) + \frac{1}{4} \varepsilon\omega y_1^2\Bigr)^2 
		\\
		&=  -  \varepsilon \omega g_H^+(y_1)  
		-  \varepsilon \omega \Bigl(g_H^+(y_1) + \frac{1}{4} \varepsilon\omega y_1^2\Bigr)O\Bigl(\frac{1}{|y_1|}\Bigr)
		\\
		& \quad 
		+ \frac{1}{4} \varepsilon^2 \omega^2 
		\Bigl( g_H^+(y_1) + \frac{1}{4} \varepsilon\omega y_1^2 \Bigr)^2 .
	\end{align*}
	From \eqref{bdyHairpin}, for $y$ satisfying \eqref{reg1b} and  $|y_1| \geq  (\varepsilon^{-1} |\log \varepsilon|)^{1/2} $ we find that 
	\begin{align}
		\label{A}
		|(A)| \leq C ( \varepsilon  \log(|y_1|)
		+ \varepsilon^2 |y_1| 
		+ \varepsilon^4 y_1^4).
	\end{align}
	Using  \eqref{A} and \eqref{b1} we obtain
	\begin{align*}
		\left| |\nabla \psi_0|^2 -1  \right| \leq  C ( \varepsilon  \log(|y_1|)
		+ \varepsilon^2 |y_1| 
		+ \varepsilon^4 y_1^4).
	\end{align*}
	for $y$ satisfying \eqref{reg1b} and  $|y_1| \geq  (\varepsilon^{-1} |\log \varepsilon|)^{1/2} $.
	From here we deduce the validity of \eqref{estR1p}.
	
	\medskip
	\noindent
	{\bf Estimate in $R_2^+$.}
	We claim that 
	\begin{align}
		\label{estR2p}
		\sup_{y\in R_2^+}
		(1+|y|)^{\sigma} \,
		\left| |\nabla \psi_0|^2 -1  + 2 \varepsilon g_0(\varepsilon y_2 + d^S)  \right| 
		\leq C \varepsilon^{\ell_1} |\log\varepsilon|,
	\end{align}
	where $\ell_1$ is defined in \eqref{ell1}.
	Similarly as in $R_1^+$, we have
	\begin{align}
		\label{ab0}
		\sup_{y\in R_2^+}  (1+|y|)^\sigma |\varepsilon g_0(\varepsilon y_2 + d^S) |
		& \leq C  \varepsilon^{3-2b-\sigma b}.
	\end{align}
	
	Next we compute an estimate for $|\nabla \psi_0|^2-1$ in the region $\varepsilon^{-b}<|y|<2\varepsilon^{-b}$ $y_2>0$.
	In this region
	\begin{align}
		\nonumber
		\psi_0
		&= \tilde\psi^H \chi + \tilde\psi^D (1-\chi)
	\end{align}
	where $\tilde\psi^H$ and $\tilde\psi^D$ are defined in \eqref{tildePsiH} and \eqref{tildepsiD} respectively, and 
	\begin{align*}
		\chi(y) = \chi_0(\varepsilon^b y).
	\end{align*}
	Then
	\begin{align*}
		\nabla \psi_0 
		&= e_2 -\frac{1}{2}\varepsilon\omega y  - \frac{y}{|y|^2}
		+ \Bigl(\nabla \tilde \psi^H-e_2+\frac{1}{2}\varepsilon\omega y + \frac{y}{|y|^2}\Bigr) \chi 
		\\
		& \quad  
		+ \Bigl(\nabla\tilde\psi^D-e_2+\frac{1}{2}\varepsilon\omega y + \frac{y}{|y|^2} \Bigr) (1-\chi)
		+ (\tilde\psi^H-\tilde\psi^D) \nabla \chi ,
	\end{align*}
	and we use this to compute
	\begin{align*}
		|\nabla\psi_0|^2 
		&= \Bigl| e_2 -\frac{1}{2}\varepsilon\omega y - \frac{y}{|y|^2}\Bigr|^2 + R_1(y),
	\end{align*}
	where 
	\begin{align*}
		|R_1(y)| 
		&\leq 
		C \Bigl|  \nabla \tilde \psi^H-e_2+\frac{1}{2}\varepsilon\omega y + \frac{y}{|y|^2} \Bigr| 
		+C\Bigl| \nabla\tilde\psi^D-e_2+\frac{1}{2}\varepsilon\omega y + \frac{y}{|y|^2}  \Bigr|
		\\
		&\quad 
		+C\Bigl| (\tilde\psi^H-\tilde\psi^D) \nabla \chi\Bigr|
		\\
		& \quad 
		+C\Bigl| \nabla \tilde \psi^H-e_2+\frac{1}{2}\varepsilon\omega y + \frac{y}{|y|^2} \Bigr|^2 +C\Bigl| \nabla\tilde\psi^D-e_2+\frac{1}{2}\varepsilon\omega y + \frac{y}{|y|^2}   \Bigr|^2
		\\
		&\quad 
		+C\Bigl| (\tilde\psi^H-\tilde\psi^D) \nabla \chi\Bigr|^2.
	\end{align*}
	Writing $y =(y_1,y_2)$ with $y_2=g^+(y_1)$ (defined in \eqref{upperbdy}) we start with
	\begin{align}
		\label{a10}
		\Bigl| e_2 -\frac{1}{2}\varepsilon\omega y - \frac{y}{|y|^2}\Bigr|^2
		&= 1 -  \varepsilon\omega y_2+ \frac{1}{4}\varepsilon^2\omega^2 |y|^2
		-2\frac{y_2}{|y|^2} + \varepsilon\omega
		+ \frac{1}{|y|^2} .
	\end{align}
	From \eqref{upperbdy} we get
	\begin{align*}
		g^+(y_1) = \frac{1}{4}\varepsilon \omega y_1^2 + O(|\log\varepsilon|)
		+O(\varepsilon^3 y_1^4).
	\end{align*} 
	This combined with \eqref{a10} gives
	\begin{align*}
		\Bigl| e_2 -\frac{1}{2}\varepsilon\omega y - \frac{y}{|y|^2}\Bigr|^2
		&= 1 - \varepsilon\omega \Bigl( \frac{1}{4}\varepsilon\omega y_1^2 
		+ O(|\log\varepsilon|) +O(\varepsilon^3 y_1^4)\Bigr) 
		+ \frac{1}{4}\varepsilon^2\omega^2 y_1^2 
		\\
		& \quad 
		+ \frac{1}{4}\varepsilon^2\omega^2 \Bigl(\frac{1}{4}\varepsilon\omega y_1^2 
		+ O(|\log\varepsilon|) +O(\varepsilon^3 y_1^4)\Bigr)^2
		+ O( \varepsilon ) .
	\end{align*}
	Therefore
	\begin{align}
		\label{gr1}
		\left| \Bigl| e_2 -\frac{1}{2}\varepsilon\omega y - \frac{y}{|y|^2}\Bigr|^2 -1 \right|
		& \leq C \left( \varepsilon |\log\varepsilon|+\varepsilon^4 |y|^4
		\right).
	\end{align}
	
	We estimate the terms in $R_1$.
	From \eqref{est1a} we get
	\begin{align*}
		\left| \partial_{y_1} \tilde \psi^H + \frac{1}{2} \varepsilon \omega y_1 + \frac{y_1}{|y|^2}\right| \leq C \frac{\log |y|}{|y|^2}
	\end{align*}
	and from \eqref{est2a} 
	\begin{align*}
		\left| \partial_{y_2} \tilde \psi^H -1+ \frac{1}{2} \varepsilon \omega y_2 + \frac{y_2}{|y|^2}\right| 
		&\leq
		\left| \partial_{y_2} \tilde \psi^H -1+ \frac{1}{2} \varepsilon \omega y_2 \right| +  \frac{|y_2|}{|y|^2}
		\\
		&\leq  C \frac{y_2}{|y|^2}
		\\ 
		& \leq C \varepsilon ,
	\end{align*}
	because $y_2 \leq C \varepsilon |y|^2 + \log |y|$ and $|y|\geq \varepsilon^{-b}\geq \varepsilon^{-\frac{1}{2}}$.
	Therefore
	\begin{align}
		\label{gr2}
		\Bigl|  \nabla \tilde \psi^H-e_2+\frac{1}{2}\varepsilon\omega y + \frac{y}{|y|^2} \Bigr| 
		& \leq C \varepsilon + C \frac{\log |y|}{|y|^2} .
	\end{align}
	From the definition \eqref{tildepsiD} and \eqref{expNPsi1D}, together with $|d^D|\leq C\varepsilon|\log\varepsilon|$ (see \eqref{dDeps}) we obtain
	\begin{align}
		\label{gr3}
		\Bigl|\nabla  \tilde \psi^D-e_2+\frac{1}{2}\varepsilon\omega y + \frac{y}{|y|^2} \Bigr| 
		\leq C \varepsilon |\log\varepsilon|
	\end{align}
	and thanks to \eqref{diffpsiHpsiD} we get
	\begin{align}
		\label{gr4}
		\Bigl| (\tilde\psi^H-\tilde\psi^D) \nabla \chi\Bigr|
		\leq C \varepsilon^{b} \Bigl[ \varepsilon |\log \varepsilon| |y| + \frac{|\log\varepsilon|}{|y|} \Bigr].
	\end{align}
	Combining \eqref{gr1}--\eqref{gr4} we get
	\begin{align}
		\label{estR1}
		|R_1(y)|\leq C \varepsilon |\log\varepsilon| + C \frac{\log |y|}{|y|^2} +  C \varepsilon^{b} \Bigl[ \varepsilon |\log \varepsilon| |y| + \frac{|\log\varepsilon|}{|y|} \Bigr] ,
	\end{align}
	for $y \in R_2^+$.
	Therefore
	\begin{align}
		\label{abc1}
		\sup_{y \in R_2^+}
		(1+|y|)^{\sigma} \,
		\left| |\nabla \psi_0|^2 -1  \right|
		\leq C \varepsilon^{\ell_1} |\log\varepsilon|,
	\end{align}
	where $\ell_1$ is defined in \eqref{ell1}.
	
	From \eqref{ab0} and \eqref{abc1} we deduce \eqref{estR2p}.
	
	\medskip
	\noindent
	{\bf Estimate in $R_3^+$.}
	We claim that
	\begin{align}
		\label{estR3p}
		\sup_{y \in R_3^+}
		(1+|y|)^{\sigma} \,
		\left| |\nabla_y \psi_0|^2 -1  + 2 \varepsilon g_0(\varepsilon y_2 + d^S)  \right| 
		\leq C \varepsilon^{b(2-\sigma)}.
	\end{align}
	
	In this region $\psi_0$ is given by $\tilde\psi^D$ \eqref{tildepsiD} and the boundary of $\Omega_0$ is given by $\frac{1}{\varepsilon} \partial \Omega_\varepsilon^D $ \eqref{bdyOmegaeps}.
	So it is more natural to do the computations in the variable $x=\varepsilon y$. 
	Then we have
	\begin{align*}
		& \mathop{\sup_{y \in \partial \Omega_0, \, y_2>0}}_{ |y|\geq 2\varepsilon^{-b}}
		(1+|y|)^{\sigma} \,
		\left| |\nabla_y \psi_0|^2 -1  + 2  \varepsilon g_0(\varepsilon y_2 + d^S)  \right| 
		\\
		& \quad 
		\leq 
		C \varepsilon^{-\sigma} \sup_{x\in \tilde R_3^+} |x|^\sigma \left| |\nabla_x \Psi|^2 -1  + 2  \varepsilon g_0( x_2 + d^S)  \right| .
	\end{align*}
	where $\Psi$ is defined in \eqref{def:Psi} and where the region $\tilde R_3^+$ is given by  $x \in \partial\Omega^D_\varepsilon$ \eqref{pOmegaDeps}, $ |x|\geq 2\varepsilon^{1-b}$.
	We have
	\begin{align*}
		|\nabla_x \Psi|^2
		&=
		|\nabla_x \Psi^D|^2 + 2 \varepsilon \nabla_x \Psi^D \cdot  \nabla_x\Psi_1^D + \varepsilon^2 | \nabla_x \Psi_1^D|^2  .
	\end{align*}
	We recall that $ \partial\Omega^D_\varepsilon$ is described by  $x + \varepsilon \nu_{\partial\Omega^D} \Psi_1^D(x)$ with $x \in \partial \Omega^D$. Let us write $\nu = \nu_{\partial\Omega^D}$. Using Taylor's theorem, the equation satisfied by $\Psi^D$ and the smoothness of $\Psi^D$ and its derivatives, for $x\in \partial\Omega^D$ we have
	\begin{align*}
		|\nabla_x \Psi^D|^2 (x + \varepsilon \nu \Psi_1^D)
		&=
		|\nabla_x \Psi^D|^2 (x)
		+ 2 \varepsilon D^2 \Psi^D(x) \nabla \Psi^D(x) \cdot \nu  \Psi^D_1(x)
		+ O ( \varepsilon^2 \Psi_1^D(x)^2 )
		\\ 
		&=
		|\nabla_x \Psi^D|^2 (x)
		+ 2 \varepsilon D^2 \Psi^D(x) \nabla \Psi^D(x) \cdot \nu  \Psi^D_1(x)
		+ O ( \varepsilon^2 \Psi_1^D(x)^2 )
		\\
		&=
		1 + 2 \frac{\varepsilon}{R}\Psi_1^D(x)+ O ( \varepsilon^2 \Psi_1^D(x)^2 )
	\end{align*} 
	\begin{align*}
		\nabla_x \Psi^D \cdot  \nabla_x\Psi_1^D (x + \varepsilon \nu \Psi_1^D)
		&= 
		\nabla_x \Psi^D \cdot  \nabla_x\Psi_1^D (x)
		+ O(\varepsilon  |\nabla_x \Psi_1^D(x)|)
		+ O(\varepsilon  |D^2_x \Psi_1^D(x)|)
		\\
		&= -  \nabla_x\Psi_1^D (x) \cdot \nu 
		+ O(\varepsilon  |\nabla_x \Psi_1^D(x)|)
		+ O(\varepsilon  |D^2_x \Psi_1^D(x)|) .
	\end{align*}
	Therefore, using \eqref{eqPsi1D}, for $x\in \partial\Omega^D$
	\begin{align*}
		|\nabla_x \Psi|^2
		&= 1 - 2 \varepsilon  g_0(x_2+ d^S)
		+ O ( \varepsilon^2 \Psi_1^D(x)^2 )
		+ O(\varepsilon^2  |\nabla_x \Psi_1^D(x)|^2)
		\\
		& \quad 
		+ O(\varepsilon^2  |D^2_x \Psi_1^D(x)|) .
	\end{align*}
	We thus obtain
	\begin{align*}
		|\nabla_x \Psi|^2 - 1+ 2  \varepsilon g_0( x_2 + d^S) 
		&= 
		O ( \varepsilon^2 \Psi_1^D(x)^2 )
		+ O(\varepsilon^2  |\nabla_x \Psi_1^D(x)|^2)
		\\
		& \quad 
		+ O(\varepsilon^2  |D^2_x \Psi_1^D(x)|) .
	\end{align*}
	Then by Lemma~\ref{lem:sol-disc2}
	\begin{align}
		\label{abc2}
		\left| |\nabla_x \Psi|^2 - 1+ 2  \varepsilon g_0( x_2 + d^S) \right|
		& \leq  C \frac{\varepsilon^2}{|x|^2} ,
	\end{align}
	and therefore from \eqref{abc2} we get
	\begin{align*}
		\varepsilon^{-\sigma} \sup_{x\in \tilde R_3^+} |x|^\sigma \left| |\nabla_x \Psi|^2 -1  + 2  \varepsilon g_0( x_2 + d^S)  \right| 
		\leq  C \varepsilon^{b(2-\sigma)},
	\end{align*}
        which is \eqref{estR3p}.
	
	\medskip
	
	Putting together \eqref{estR1p}, \eqref{estR2p} and \eqref{estR3p} we find that 
	\begin{align}
		\label{error-upper}
		\sup_{y \in \partial \Omega_0, \ y_2>0} (1+|y|)^{\sigma} \,
		\left| |\nabla \psi_0|^2 -1  + 2  \varepsilon g_0(\varepsilon y_2 + d^S)  \right| 
		\leq C \varepsilon^{\ell_1} |\log\varepsilon|
	\end{align}
	because $b>\frac{1}{2}$.
	
	\medskip
	Next we estimate $(1+|y|)^{\sigma} \,
	\left| |\nabla \psi_0|^2 -1  + 2 \varepsilon g_0(\varepsilon y_2 + d^S)  \right| $ in the regions $R_j^-$, $j=1,2,3$.
	
	\medskip
	\noindent
	{\bf Estimate in $R_1^-$.}
	We have
	\begin{align}
		\label{estR1m}
		\sup_{y\in R_1^-}
		(1+|y|)^{\sigma} \,
		\left| |\nabla \psi_0|^2 -1  + 2 \varepsilon g_0(\varepsilon y_2 + d^S)  \right| 
		\leq C \varepsilon^{1-\sigma b} |\log\varepsilon|,
	\end{align}
	and the proof is very similar to that of \eqref{estR1p}.
	
	\medskip
	\noindent
	{\bf Estimate in $R_2^-$.}
	We have
	\begin{align}
		\label{estR2m}
		\sup_{y\in R_2^-}
		(1+|y|)^{\sigma} \,
		\left| |\nabla \psi_0|^2 -1  + 2 \varepsilon g_0(\varepsilon y_2 + d^S)  \right| 
		\leq C \varepsilon^{1-\sigma b}|\log\varepsilon|.
	\end{align}
	Indeed, in the region $R_2^-$, 
	\begin{align}
		\nonumber
		\psi_0
		&= \tilde\psi^H \chi + \tilde\psi^S (1-\chi)
	\end{align}
	where $\tilde\psi^H$ and $\tilde\psi^S$ are defined in \eqref{tildePsiH} and \eqref{tildePsiS} respectively, and 
	\begin{align*}
		\chi(y) = \chi_0(\varepsilon^b y) .
	\end{align*}
	
	Then
	\begin{align*}
		\nabla \psi_0 
		&= e_2 - \varepsilon\omega y_2 e_2  - \frac{y}{|y|^2}
		+ \Bigl(\nabla \tilde \psi^H-e_2+\varepsilon\omega y_2 e_2 + \frac{y}{|y|^2}\Bigr) \chi 
		\\
		& \quad  
		+ \Bigl(\nabla\tilde\psi^S-e_2+\varepsilon\omega y_2 e_2 + \frac{y}{|y|^2} \Bigr) (1-\chi)
		+ (\tilde\psi^H-\tilde\psi^S) \nabla \chi ,
	\end{align*}
	and we use this to compute
	\begin{align*}
		|\nabla\psi_0|^2 
		&= \Bigl| e_2 -\varepsilon\omega y_2 e_2 - \frac{y}{|y|^2}\Bigr|^2 + R_2(y),
	\end{align*}
	where
	\begin{align*}
		|R_2(y)| 
		&\leq 
		C \Bigl|  \nabla \tilde \psi^H-e_2+\varepsilon\omega y_2 e_2 + \frac{y}{|y|^2} \Bigr| 
		+C\Bigl| \nabla\tilde\psi^S-e_2+\varepsilon\omega y_2 e_2 + \frac{y}{|y|^2}  \Bigr|
		\\
		&\quad 
		+C\Bigl| (\tilde\psi^H-\tilde\psi^S) \nabla \chi\Bigr|
		\\
		& \quad 
		+C\Bigl| \nabla \tilde \psi^H-e_2+\varepsilon\omega y_2 e_2 + \frac{y}{|y|^2} \Bigr|^2 +C\Bigl| \nabla\tilde\psi^S-e_2+\varepsilon\omega y_2 e_2 + \frac{y}{|y|^2}  \Bigr|^2
		\\
		&\quad 
		+C\Bigl| (\tilde\psi^H-\tilde\psi^S) \nabla \chi\Bigr|^2.
	\end{align*}
	
	Writing $y =(y_1,y_2)$ with $y_2=g^-(y_1)$ (defined in \eqref{defgm}) we have
	\begin{align}
		\nonumber
		\Bigl| e_2 - \varepsilon\omega y_2 e_2 - \frac{y}{|y|^2}\Bigr|^2
		&= 1 -  2\varepsilon\omega y_2+ \varepsilon^2\omega^2 y_2^2
		-2\frac{y_2}{|y|^2} + 2\frac{\varepsilon\omega y_2^2}{|y|^2} 
		+ \frac{1}{|y|^2} .
	\end{align}
	But $|y_2|\leq C |\log y_1|$ on $R_2^-$, so 
	\begin{align}
		\label{abc3}
		\left| \Bigl| e_2 - \varepsilon\omega y_2 e_2 - \frac{y}{|y|^2}\Bigr|^2 -1\right|
		&\leq C \varepsilon |\log \varepsilon| + C\frac{|\log\varepsilon|}{|y|^2} .
	\end{align}
	Similar to \eqref{estR1} we have, for $y\in R_2^-$
	\begin{align*}
		|R_2(y)|\leq  C \frac{\log |y|}{|y|^2} +  C \varepsilon^{b} \Bigl[ \varepsilon |\log \varepsilon| |y| + \frac{|\log\varepsilon|}{|y|} \Bigr]  .
	\end{align*}
	This estimate and \eqref{abc3} imply \eqref{estR2m}.
	
	\medskip
	\noindent
	{\bf Estimate in $R_3^-$.}
	We claim that
	\begin{align}
		\label{estR3m1}
		\mathop{\sup_{ y \in \Omega_0}}_{y_2<0, \, 2\varepsilon^{-b}|y| \leq \frac{\delta_1}{\varepsilon}}
		(1+|y|)^{\sigma} \,
		\left| |\nabla_y \psi_0|^2 -1  + 2 \varepsilon g_0(\varepsilon y_2 + d^S)  \right| 
		\leq C \varepsilon^{b(2-\sigma)} .
	\end{align}
	
	A computation similar to \eqref{estR3p} gives
	\begin{align}
		\label{abc4}
		\left| |\nabla_x \Psi|^2 - 1 \right|
		& \leq  C \frac{\varepsilon^2}{|x|^2} ,
		\quad 2\varepsilon^{1-b}\leq |x|\leq \delta_1,
	\end{align}
	while
	\begin{align}
		\label{abc5}
		|\varepsilon g_0(x_2+d^S)| \leq C \varepsilon^2 |\log \varepsilon|,\quad 2\varepsilon^{1-b}\leq |x|\leq \delta_1.
	\end{align}
	This is because for the strip $\Omega^S$ \eqref{def:OmegaS0} we have $x_2+d^S\equiv 0$ on its upper boundary and the perturbed strip $\Omega^S_\varepsilon$ \eqref{pOmegaSeps}, \eqref{Seps} is obtained by adding a function of the form $O(\varepsilon \log|x|)$.
	From \eqref{abc4} and \eqref{abc5} we deduce \eqref{estR3m1}.
	
	We claim that
	\begin{align}
		\label{estR3m2}
		\mathop{\sup_{ y \in \Omega_0}}_{y_2<0, \, |y| \geq \frac{\delta_1}{\varepsilon}}
		e^{\mu \varepsilon|y|} \,
		\left| |\nabla_y \psi_0|^2 -1  + 2 \varepsilon g_0(\varepsilon y_2 + d^S)  \right| 
		\leq C \varepsilon^{2-\sigma}
	\end{align}
	The proof is similar to that of \eqref{estR3p}. 
	We change variables and use that $ |\nabla_x \Psi|^2 -1 $  is quadratic in $\varepsilon \Psi_1^D$ and the estimate
	\begin{align*}
		|g(x_2+d^S)| \leq C \varepsilon^2 e^{-\mu |x|},\quad |x|\geq \delta_1 .
	\end{align*}
	
	\medskip
	Combining \eqref{error-upper}, \eqref{estR1m}, \eqref{estR2m}, \eqref{estR3m1} and \eqref{estR3m2} we deduce the validity of the estimates \eqref{estE1k1} and \eqref{estE1k2} for $k=0$. For the derivatives the proof is similar, differentiating the expansions that we have already found in each region.
\end{proof}

\begin{proof}[Proof of Lemma~\ref{lem:error1-2}]
	The proof is similar to the one of Lemma~\ref{lem:error1}.
	Using the regions regions $R_j^\pm$ defined in \eqref{regions}, the estimates in the regions $R_1^\pm$ and $R_2^\pm$ are similar and yield
	\begin{align}
		\nonumber
		\sup_{y \in R_1^+\cup R_2^+}
		(1+|y|)^{\sigma} \,
		\left| (1-P) g_1(\varepsilon y_2 + d^S)  \right| 
		\leq C  \varepsilon^{3+m-2b-\sigma b}.
	\end{align}
	
	We claim that
	\begin{align}
		\nonumber
		\sup_{y \in R_3^+}
		(1+|y|)^{\sigma} \,
		\left| (1-P) g_1(\varepsilon y_2 + d^S)  \right| 
		\leq C\varepsilon^{1+m+b(1-\sigma)} |\log\varepsilon|.
	\end{align}
	Indeed, for $y \in R_3^+$, by the estimate on $P$ \eqref{estP} and the assumption on $g_1$ \eqref{hg},
	\begin{align*}
		(1+|y|)^{\sigma} \left| (1-P) g_1(\varepsilon y_2 + d^S)  \right|
		&\leq
		C (1+|y|)^{\sigma} \left(\varepsilon|\log\varepsilon|+\frac{1}{|y|}\right) |g_1|
		\\
		&\leq
		C \left( \varepsilon^{1-\sigma b}|\log\varepsilon|+\varepsilon^{b(1-\sigma)}\right) |g_1|
		\\
		&\leq
		C \varepsilon^{1+m+b(1-\sigma)}.
	\end{align*}
	The condition $\frac12<b<\frac23$ implies that $3+m-2b-\sigma b>b(1-\sigma)+1+m$.
	Hence \eqref{estE1k1-2} for $k=0$ follows for $y$ in the upper part of $\partial\Omega_0$. In the lower part of $\partial\Omega_0$ the estimate is similar. For $k=1,2$ the result follows by differentiating the terms and using similar estimates as above.
\end{proof}

For the proof of Lemma~\ref{lem:error-u1} it will be convenient to have the following intermediate computation.

\begin{lemma}
	We have that
	\begin{align}
		\label{decomp-err-la}
		\Delta\psi_0+\varepsilon\omega = f_1+f_2\quad \text{in } \overline{\Omega}_0
	\end{align}
	where 
	\begin{align}
		\label{f1x}
		|f_1(y)|\leq C\varepsilon 
		\text{ and }
		\supp(f_1)\subset B_L(0)
	\end{align}
	with $L$ independent of $\varepsilon$,  and
	\begin{align}
		\label{f2x}
		|f_2(y)|\leq \varepsilon^{1+b}|\log\varepsilon|,
		\text{ and }
		\supp(f_2)\subset \{\varepsilon^{-b}\leq|y|\leq2\varepsilon^{-b}\}.
	\end{align}
\end{lemma}
\begin{proof}
	For $|y|\leq \varepsilon^{-b}$ we have
	\begin{align}
		\nonumber
		|\Delta \psi_0 + \varepsilon\omega |\leq C \varepsilon  \chi_{B(0,L)} 
	\end{align}
	for some fixed $L>0$, and for $|y|\geq 2 \varepsilon^{-b}$ we have $\Delta \psi_0 + \varepsilon\omega =0$.
	
	So let $y =(y_1,y_2)$ be such that 
	\begin{align}
		\label{reg1}
		y \in \Omega_0 , \quad 
		\varepsilon^{-b} \leq  |y|\leq 2 \varepsilon^{-b} .
	\end{align}
	Here we have
	\begin{align*}
		\Delta \psi_0 
		&=-\varepsilon \omega 
		+  2 \varepsilon^b \nabla ( \tilde \psi^H - \tilde \psi^D) \cdot \nabla \chi_0(\varepsilon^b y)
		+ \varepsilon^{2b}  ( \tilde \psi^H - \tilde \psi^D)  \Delta \chi_0(\varepsilon^b y),
	\end{align*}
	where $\tilde\psi^H$ is defined in \eqref{tildePsiH}, $\tilde\psi^D$ is defined in \eqref{tildepsiD}, and $\chi_0$ is the cut-off function defined in \eqref{chi0}.
	
	We claim that for $y$ in the region \eqref{reg1} we have
	\begin{align}
		\label{estErrLapl}
		|\Delta \psi_0 +\varepsilon \omega |
		&\leq C \varepsilon^{1+b} |\log\varepsilon| .
	\end{align}
	We start with an estimate for $\varepsilon^b \nabla ( \tilde \psi^H - \tilde \psi^D) \cdot \nabla \chi_0(\varepsilon^b y) $.
	Let $y=(y_1,y_2)$ be in the region \eqref{reg1} and assume that $y_2>0$. 
	By \eqref{tildePsiH} we have
	\begin{align*}
		\nabla_y \tilde \psi^H(y) = \nabla \psi^H(y) - \tfrac{1}{2} \varepsilon \omega y 
	\end{align*}
	and from \eqref{tildepsiD}, using $\omega R = 2$, 
	\begin{align*}
		\nabla_y \tilde \psi^D(y) = e_2 - \tfrac{1}{2} \varepsilon \omega y + \tfrac{1}{2} \omega d^D e_2 + \varepsilon \nabla_x \Psi_1^D(\varepsilon y)  .
	\end{align*}
	Then by \eqref{expNPsi1D} 
	\begin{align}
		\nonumber
		\nabla_y (\tilde\psi^H(y)-\tilde \psi^D(y))
		&= \nabla \psi^H(y) -e_2  - \varepsilon \nabla_x \Psi_1^D(\varepsilon y)  
		\\
		\label{m-1}
		&= \nabla \psi^H(y) -e_2  + \frac{y}{|y|^2} + O\Bigl( \frac{|\log\varepsilon|}{|y|^2}\Bigr)
		+ O(\varepsilon|\log\varepsilon|)
	\end{align}
        and so \eqref{est1a} and \eqref{est3} imply
	\begin{align*}
		\left|
		\Bigl(\nabla \psi^H(y) -e_2  - \frac{y}{|y|^2}\Bigl) \cdot \nabla_z \chi_0(\varepsilon^b y)\right|
		&\leq C \frac{\log |y|}{|y|^2}.
	\end{align*}
	It follows from this and \eqref{m-1} that
	\begin{align*}
		\Bigl| \nabla_y (\tilde\psi^H(y)-\tilde \psi^D(y)) \cdot \nabla \chi_0(\varepsilon^b y) \Bigr|
		& \leq C  \frac{\log|y|}{|y|^{2}} + C \varepsilon|\log\varepsilon| 
		\\
		& \leq C \varepsilon^{2b} |\log\varepsilon| + C \varepsilon|\log\varepsilon|  .
	\end{align*}
	Thus
	\begin{align}
		\label{lapl-n}
		\left| \varepsilon^b \nabla ( \tilde \psi^H - \tilde \psi^D) \cdot \nabla \chi_0(\varepsilon^b y) \right| \leq C \varepsilon^{1+b} |\log\varepsilon| 
	\end{align}
	for $y=(y_1,y_2)$ in the region \eqref{reg1} with $y_2>0$.
	A similar argument shows that \eqref{lapl-n} holds also for $y=(y_1,y_2)$  in the region \eqref{reg1} and $y_2<0$. 
	
	Consider $ \varepsilon^{2b}  ( \tilde \psi^H - \tilde \psi^D)  \Delta \chi_0(\varepsilon^b y)$.  Let $y$ in the region \eqref{reg1} and assume that $y_2>0$.  By \eqref{diffpsiHpsiD},
	\[
	|\tilde\psi^H(y)-\tilde\psi^D(y)| \leq C \varepsilon |\log\varepsilon| |y| + C \frac{|\log\varepsilon|}{|y|}.
	\]
	It follows that in this region
	\begin{align}
		\label{EstDiffPsi2}
		| \varepsilon^{2b}  ( \tilde \psi^H - \tilde \psi^D)  \Delta \chi_0(\varepsilon^b y) |
		& \leq C ( \varepsilon^{1+b} + \varepsilon^{3b}) |\log \varepsilon| .
	\end{align}
	Combining \eqref{lapl-n} and \eqref{EstDiffPsi2} we get \eqref{estErrLapl}.
\end{proof}

\begin{proof}[Proof of Lemma~\ref{lem:error-u1}]
	Let us write $\Delta\psi_0+\varepsilon\omega=f_1+f_2$ as in \eqref{decomp-err-la} with $f_1$ and $f_2$ satisfying \eqref{f1x}, \eqref{f2x}.
	Estimates \eqref{f1x} and \eqref{f2x} imply that
	\begin{align*}
		|\Delta\psi_0+\varepsilon \omega|\leq C \varepsilon^{1-\sigma b}|\log\varepsilon| 
		\frac{1}{(1+|y|)^{1+\sigma}} .
	\end{align*}
	Moreover, the support of $\Delta\psi_0+\varepsilon \omega$ is contained in $|y|\leq 2 \varepsilon^{-b}$.
Using Lemma~\ref{lem:lh2}, there is a function $\bar u_1$ satisfying
\begin{align}
\label{pHairpinD3}
\left\{
\begin{aligned}
-\Delta_y \bar u_1 &=  \varepsilon^{1-\sigma b}|\log\varepsilon| 
\frac{1}{(1+|y|)^{1+\sigma}}
&&
\text{in }  \Omega_0 \cap B_{2\varepsilon^{-b}} (0)
\\
\bar u_1 & =0 && \text{on } \partial \Omega_0 \cap B_{2\varepsilon^{-b}} (0)
\end{aligned}
\right.
\end{align}
and the inequalities 
\begin{alignat}{2}
\label{eq:u1-upper}
\bar u_1(y) &\leq C_2\varepsilon^{1-\sigma b}|\log\varepsilon| (1+|y|)^{1-\sigma} 
,
&\quad& y\in \Omega_0 \cap B_{2\varepsilon^{-b}} (0) \\
\label{eq:u1-lower}
\bar u_1(y) &\geq
C_1\varepsilon^{1-\sigma b}|\log\varepsilon| (1+|y|)^{1-\sigma} 
,
&\quad& y\in  \Omega_0 \cap B_{2\varepsilon^{-b}} (0).
\end{alignat}
See Remark~\ref{rem:u1}.
	We let $K>0$ be such that $(\frac{K}{C_1})^{\frac{1}{1-\sigma}}=2$ and define
	\begin{align*}
		\bar u_2 = \min(\bar u_1, K \varepsilon^{1-b}|\log\varepsilon|).
	\end{align*}
	Then $\bar u_2$ is well defined in $\Omega_0$ and it is a weak supersolution of 
	\begin{align}
		\nonumber
		\left\{
		\begin{aligned}
-\Delta_y u &=  \varepsilon^{1-\sigma b}|\log\varepsilon| 
			\frac{1}{(1+|y|)^{1+\sigma}}
			&&
			\text{in }\Omega_0
			\\
			u & =0 && \text{on } \partial\Omega_0.
		\end{aligned}
		\right.
	\end{align}
	By the maximum principle,
	\begin{align}
		\label{cota01}
		|u_1|\leq \bar u_2 \quad\text{in }\Omega_0.
	\end{align}
	This in particular implies 
	\begin{align*}
		|u_1(y)|\leq K \varepsilon^{1-b}|\log\varepsilon|, \quad y\in\Omega_0, \quad |y|=\frac{\delta_1}{\varepsilon},
	\end{align*}
	Using suitable barriers we obtain
	\begin{align}
		\label{cota02}
		|u_1(y)|\leq C \varepsilon^{1-b}|\log\varepsilon| e^{-\mu \varepsilon |y|},  \quad y\in\Omega_0, \quad |y|\geq\frac{\delta_1}{\varepsilon}.
	\end{align}
	Combining \eqref{cota01} and \eqref{cota02} we get
	\begin{align}
		\nonumber
		|u_1(y)| 
		\leq C 
		\begin{cases}
			\varepsilon^{1-\sigma b}|\log\varepsilon| (1+|y|)^{1-\sigma} , &  |y|\leq\frac{\delta_1}{\varepsilon},
			\\
			\varepsilon^{1-\sigma b}
			\varepsilon^{\sigma-1}
			|\log\varepsilon|  e^{-\mu \varepsilon |y|},&|y|\geq\frac{\delta_1}{\varepsilon}.
		\end{cases}
	\end{align}
	This and standard elliptic estimates imply
	\begin{align}
		\nonumber
		|\nabla_y u_1(y)|
		\leq C 
		\begin{cases}
			\varepsilon^{1-\sigma b} |\log\varepsilon| (1+|y|)^{-\sigma} , &|y|\leq\frac{\delta_1}{\varepsilon},
			\\
			\varepsilon^{1-\sigma b}
			\varepsilon^{\sigma}
			|\log\varepsilon|  e^{-\mu \varepsilon |y|},&|y|\geq\frac{\delta_1}{\varepsilon}
		\end{cases}
	\end{align}
	and hence
	\begin{align}
		\nonumber
		\sup_{y\in\mathcal{S}_0,\, |y| \leq \frac{\delta_1}{\varepsilon}}(1+|y|)^{\sigma}  |\nabla_y u_1(y)|&\leq C \varepsilon^{1-\sigma b}|\log\varepsilon|
            \\
		\nonumber
		\sup_{y\in\mathcal{S}_0,\,|y| \geq \frac{\delta_1}{\varepsilon}}\varepsilon^\sigma e^{\mu \varepsilon |y|}|\nabla_y u_1(y)|&\leq C \varepsilon^{1-\sigma b}|\log\varepsilon|,
	\end{align}
which are part of the norm $\| \  \|_{**,\partial\Omega} $ in \eqref{defNf2Omega0}.
By similar estimates for the higher derivatives of $u_1$, using \eqref{D2psiH-1}, we get
\begin{align*}
\Bigl\|  \frac{\partial u_1}{\partial \nu}
\Bigr\|_{**,\partial\Omega_0} \leq C \varepsilon^{1-\sigma b}|\log\varepsilon|.
\end{align*}

\end{proof}

\begin{proof}[Proof of Lemma~\ref{lem:error-u0}]\ 
	
	\medskip
	\noindent \textbf{Step 1.}
	Estimate $\psi_0$ on $\mathcal{S}_0$. 
	We claim that for $y\in\mathcal{S}_0$%
	\begin{align}
		\label{psi0R1}
		|\psi_0(y)| &\leq  C \varepsilon \log^2(2+|y|)+ C \varepsilon^2 (1+|y|)^2+C\varepsilon^3 (1+|y|)^4 , \quad |y| \leq \varepsilon^{-b}
		\\
		\label{psi0R2}
		|\psi_0(y)|&\leq C  \varepsilon^3 |y|^4
		+ C \varepsilon |\log\varepsilon| \,|y|  , \quad \varepsilon^{-b} \leq |y| \leq2\varepsilon^{-b}
		\\
		\label{psi0R3}
		|\psi_0(y)| &\leq C 
		\begin{cases}
			\frac{|\log\varepsilon|}{|y|} , &  2\varepsilon^{-b}\leq|y|\leq\frac{\delta_1}{\varepsilon},
			\\
			\varepsilon e^{-\mu \varepsilon |y|},&|y|\geq\frac{\delta_1}{\varepsilon},
		\end{cases}
	\end{align}
	for some $\mu>0$ (the same $\mu$ as in the statement of Lemma~\ref{lem:sol-strip2}).
	We also claim the estimates%
	\begin{align}
		\label{dpsi0R1}
		|\psi_0^{(j)}(y)| &\leq 
		C \varepsilon \frac{\log (2+|y|)}{(1+|y|)^j}+ C \varepsilon^2 (1+|y|)^{2-j}+C\varepsilon^3 (1+|y|)^{4-j},\quad|y|\leq\varepsilon^{-b}
		\\
		\label{dpsi0R2}
		|\psi_0^{(j)}(y)|&\leq C  \varepsilon^3 |y|^{4-j}
		+ C \varepsilon |\log\varepsilon| \,|y|^{1-j}  , \quad \varepsilon^{-b} \leq |y| \leq2\varepsilon^{-b}
		\\
		\label{dpsi0R3}
		|\psi_0^{(j)}(y)| &\leq C 
		\begin{cases}
			\frac{|\log\varepsilon|}{|y|^{1+j}} , &  2\varepsilon^{-b}\leq|y|\leq\frac{\delta_1}{\varepsilon},
			\\
			\varepsilon^{1+j} e^{-\mu \varepsilon |y|},&|y|\geq\frac{\delta_1}{\varepsilon},
		\end{cases}
	\end{align}
	for $j=1,2,3$, where $\psi_0^{(j)}$ is the $j$-th derivative with respect to arc-length on $\mathcal{S}_0$. 
	
	\medskip
	
	Let us prove \eqref{psi0R1}. For $y=(y_1,y_2)\in\partial\mathcal{S}_0$ with $|y|\leq\varepsilon^{-b}$, $\psi_0(y)$ is given by $\tilde\psi^H(y)$ \eqref{tildePsiH}.
	Consider the case $y_2>0$. Write $y=(y_1, \tilde g_H^+(y_1) )$.
	We have for some $0<\xi < \frac{1}{4}\varepsilon\omega y_1^2$, 
	\begin{align}
		\nonumber
		\tilde \psi^H (y_1, \tilde g_H^+(y_1) )
		&= 
		\psi^H \Bigl( y_1, g_H^+(y_1) + \frac{1}{4} \varepsilon\omega y_1^2 \Bigr)
		- \frac{1}{4} \varepsilon \omega y_1^2  \eta^+_{1}(y) 
		\\
		\nonumber
		& \quad 
		- \frac{1}{4} \varepsilon \omega \Bigl( g_H^+(y_1) 
		+ \frac{1}{4} \varepsilon\omega y_1^2 \Bigr)^2   \eta^+_{1}(y) 
		\\
		&= \partial_{y_2} \psi^H ( y_1, g_H^+(y_1) + \xi ) \frac{1}{4} \varepsilon\omega y_1^2 
		\nonumber
		- \frac{1}{4} \varepsilon \omega y_1^2  \eta^+_{1}(y) 
		\\
		\nonumber
		& \quad 
		- \frac{1}{4} \varepsilon \omega \Bigl( g_H^+(y_1) 
		+ \frac{1}{4} \varepsilon\omega y_1^2 \Bigr)^2   \eta^+_{1}(y) 
		\\
		\nonumber
		&= [ \partial_{y_2} \psi^H ( y_1, g_H^+(y_1) + \xi )-1] \frac{1}{4} \varepsilon\omega y_1^2 
		- \frac{1}{4} \varepsilon \omega y_1^2  ( \eta^+_{1}(y) -1) 
		\\
		& \quad 
		\label{cota3}
		- \frac{1}{4} \varepsilon \omega \Bigl( g_H^+(y_1) 
		+ \frac{1}{4} \varepsilon\omega y_1^2 \Bigr)^2   \eta^+_{1}(y) .
	\end{align}
	Using \eqref{est2a} we get
	\begin{align*}
		|\tilde \psi^H (y_1, \tilde g_H^+(y_1) )|
		&\leq C \varepsilon \log^2 (|y_1|)+ C \varepsilon^2 y_1^2+C\varepsilon^3 y_1^4.
	\end{align*}
	This proves \eqref{psi0R1} in the case $y_2>0$. If $y_2<0$ the computation is analogous.
	From \eqref{cota3} we also get estimates for the derivatives:
	\begin{align*}
		\Bigl|\frac{d^j}{d y_1^j}\tilde \psi^H (y_1, \tilde g_H^+(y_1) )\Bigr|
		&\leq C \varepsilon \frac{\log (|y_1|)}{|y_1|^j}+ C \varepsilon^2 y_1^{2-j}+C\varepsilon^3 y_1^{4-j},
	\end{align*}
	$j=1,2,3$, which gives \eqref{dpsi0R1}.
	\medskip
	
	We prove \eqref{psi0R3} next, because it will be used later in the proof of \eqref{psi0R2}. We consider $y=(y_1,y_2)\in\partial\mathcal{S}_0$ with $ 2\varepsilon^{-b}\leq|y|\leq\frac{\delta_1}{\varepsilon}$ and $y_2>0$. In this region $\psi_0(y)$ is given by $\tilde\psi^D(y)$ \eqref{tildepsiD}, and it is more convenient to do the computation in the variable $x$.
	We write $x\in \partial\Omega^D_\varepsilon$ as $x = z+\varepsilon\Psi_1^D(z)\nu_{\partial\Omega^D}(z)$, $z\in\partial\Omega^D$.
	Using \eqref{tildepsiD} and $x=\varepsilon y$,%
	\begin{align*}
		\psi_0(y)&=\tilde\psi^D(y)\\
		&=\frac{1}{\varepsilon} \left[ \Psi^D(x) + \varepsilon \Psi_1^D(x)\right]
		\\
		&=
		\frac{1}{\varepsilon} \left[ \Psi^D(z+\varepsilon\Psi_1^D(z)\nu_{\partial\Omega^D}(z)) + \varepsilon \Psi_1^D(z+\varepsilon\Psi_1^D(z)\nu_{\partial\Omega^D}(z))\right]
		\\
		&=
		\frac{1}{\varepsilon} \Bigl[ \Psi^D(z) + \varepsilon\Psi_1^D(z) \nabla\Psi^D(z)\cdot \nu_{\partial\Omega^D}(z)
		+ O(\varepsilon^2 |\Psi_1^D(z)|^2)
		\\
		&\quad
		+ \varepsilon \Psi_1^D(z+\varepsilon\Psi_1^D(z)\nu_{\partial\Omega^D}(z))\Bigr] .
	\end{align*}
	But $\Psi^D(z)=0$ and $\nabla\Psi^D(z)\cdot \nu_{\partial\Omega^D}(z)=-1$. Then, by Lemma~\ref{lem:sol-disc2},
	\begin{align*}
		|\psi_0(y)|
		&=
		\frac{1}{\varepsilon} \left[ O(\varepsilon^2 |\Psi_1^D(z)|^2)
		+ O(\varepsilon^2 |\Psi_1^D(z)| |\nabla \Psi_1^D(z)| ) \right] 
		\\
		&\leq C\varepsilon |\log\varepsilon|^2 + C \frac{\varepsilon |\log|z||}{|z|}
		\\
		&\leq C\varepsilon |\log\varepsilon|^2 + C \frac{|\log|\varepsilon y||}{|y|}
		\\ 
		& \leq  C \frac{|\log|\varepsilon y||}{|y|},
	\end{align*}
	in the region under consideration. This proves \eqref{psi0R3} when $|y|\geq2\varepsilon^{-b}$ and $y_2>0$. 
	Note that when $y_2>0$ we naturally have $|y|\leq \frac{C}{\varepsilon}$ so there is no need to prove exponential decay.
	In the lower half the computation is the same, but we need to use the exponential decay of $\Psi_1^S$ in Lemma~\ref{lem:sol-strip2}.
	
	\medskip
	The proof of \eqref{dpsi0R3} is similar. 
	\medskip
	
	We prove \eqref{psi0R2} for $y=(y_1,y_2)$ with $y_2>0$.
	In this region
	\begin{align}
		\label{psi0b}
		\psi_0
		= \tilde\psi^H \chi
		+ \tilde\psi^D (1-\chi),
	\end{align}
	(c.f. \eqref{psi0}) where
	\begin{align*}
		\chi(y) = \chi_0(\varepsilon^b y) ,
	\end{align*}
	$\chi_0$ is as in \eqref{chi0}, and $y_2= g^+(y_1)$ \eqref{upperbdy}.
	Then
	\begin{align*}
		|\psi_0(y)| & \leq |\tilde\psi^D(y)| + |(\tilde\psi^H(y)-\tilde\psi^D(y))|.
	\end{align*}
	By \eqref{diffpsiHpsiD},
	\begin{align}
		\label{bb1}
		| ( \tilde\psi^H(y)-\tilde\psi^D(y) )| 
		&\leq C \varepsilon |\log\varepsilon| |y_1| + C \frac{|\log\varepsilon|}{|y_1|},
		\quad y=(y_1,g^+(y_1))
	\end{align}
	for $\varepsilon^{-b}\leq |y|\leq2\varepsilon^{-b}$, using that $\frac{1}{2}<b<1$.
	
	We next estimate $\tilde\psi^D(y)$ for $y=(y_1,y_2)$ and $y_2=g^+(y_1)$.
	Recall that 
	\begin{align}
		\label{gplus}
		g^+(y_1) = g_{\varepsilon,D}^-(y_1) + (\tilde g_H^+(y_1)-g_{\varepsilon,D}^-(y_1))\eta(y_1),
	\end{align}
	where
	\[
	\eta(y_1)=\eta_0(\varepsilon^b y_1)
	\]
	and $\eta_0$ is defined in \eqref{eta0}.
	Then
	\begin{align*}
		\tilde \psi^D(y_1,g^+(y_1))
		&= \tilde \psi^D(y_1, g_{\varepsilon,D}^-(y_1) + (\tilde g_H^+(y_1)-g_{\varepsilon,D}^-(y_1))\eta(y_1))
		\\
		&= 
		\tilde \psi^D(y_1, g_{\varepsilon,D}^-(y_1) )
		+\nabla_y\tilde\psi^D(y_1,g_{\varepsilon,D}^-(y_1)+\xi) (\tilde g_H^+(y_1)-g_{\varepsilon,D}^-(y_1))\eta(y_1)
	\end{align*}
	for some $|\xi|\leq |(\tilde g_H^+(y_1)-g_{\varepsilon,D}^-(y_1))\eta(y_1)|$.
	The estimate \eqref{psi0R3} holds also for $\varepsilon^{-b}\leq|y|\leq2\varepsilon^{-b}$, and therefore
	\begin{align*}
		|\tilde \psi^D(y_1, g_{\varepsilon,D}^-(y_1) )|\leq C \frac{|\log\varepsilon|}{|y|}.
	\end{align*}
	On the other hand, from Lemma~\ref{lem:sol-disc2}
	\begin{align*}
		|\nabla \tilde \psi^D(y)|\leq C
	\end{align*}
	for $\varepsilon^{-b}\leq|y|\leq2\varepsilon^{-b}$.
	This implies
	\begin{align}
		\label{bb2}
		|\tilde \psi^D(y)|\leq C \frac{|\log\varepsilon|}{|y|} + C |\tilde g_H^+(y_1)-g_{\varepsilon,D}^-(y_1)|.
	\end{align}
	To estimate $\tilde g_H^+ - g_{\varepsilon,D}^-$, we recall \eqref{gepsDm} 
	\begin{align}
		\nonumber
		g_{\varepsilon,D}^-(y_1) = g_D^-(y_1) - \Psi_1^D(\varepsilon y_1,\varepsilon	g_D^-(y_1)) + O(\varepsilon^2 |\log\varepsilon| y_1^2 ).
	\end{align}
	From \eqref{gD1},
	\begin{align*}
		g_D^-(y_1) 
		&=  \frac{d^D}{\varepsilon} 
		+\frac{\varepsilon y_1^2}{2 R}+O\Bigl(\frac{\varepsilon^3 y_1^4}{R^3}\Bigr)
	\end{align*}
	as $|y_1|\to\infty$. We expand $\Psi_1^D(\varepsilon y)$ using Lemma~\ref{lem:sol-disc2}
	\begin{align*}
		\Psi_1^D(\varepsilon y)
		&= - \log(| \varepsilon y - P^D | ) + O(| \varepsilon y - P^D | \log(| \varepsilon y - P^D |) )
		\\
		&=\log\varepsilon+\log(|y|)+O(\frac{|\log\varepsilon|}{|y|})
		+ O( \varepsilon |\log\varepsilon| \,| y|)
	\end{align*}
	in the region under consideration.
	Therefore
	\begin{align*}
		g_{\varepsilon,D}^-(y_1) 
		&= \frac{d^D}{\varepsilon} 
		+\frac{\varepsilon y_1^2}{2 R}+O\Bigl(\frac{\varepsilon^3 y_1^4}{R^3}\Bigr) + \log\varepsilon+\log(y_1)+O\Bigl(\frac{|\log\varepsilon|}{y_1}\Bigr) 
		\\
		& \quad 
		+ O(\varepsilon^2 |\log\varepsilon| y_1^2 )
		+ O( \varepsilon |\log\varepsilon| \,| y|).
	\end{align*}
	On the other hand, from \eqref{bdyHairpin}
	\begin{align*}
		\tilde g_H^+(y_1) = \log( 2 y_1) +  \frac{1}{4}\varepsilon\omega y_1^2  + O\Bigl( \frac{1}{y_1}\Bigr) ,
	\end{align*}
	as $|y_1|\to\infty$.
	Since $\omega R=2$ we get
	\begin{align}
		\nonumber
		|g_{\varepsilon,D}^-(y_1) -\tilde g_H^+(y_1) |
		& \leq C\frac{\varepsilon^3 y_1^4}{R^3} +C\frac{|\log\varepsilon|}{|y_1|}
		+ C \varepsilon |\log\varepsilon| \,| y_1|
		\\
		\label{diff-g}
		& \leq C \varepsilon^3 y_1^4
		+ C \varepsilon |\log\varepsilon| \,| y_1| ,
	\end{align}
	for $\varepsilon^{-b}\leq|y_1|\leq2\varepsilon^{-b}$.
	This combined with \eqref{bb2} yields
	\begin{align}
		\label{bb3}
		|\tilde \psi^D(y)|\leq C  \varepsilon^3 |y|^4
		+ C \varepsilon |\log\varepsilon| \,|y| .
	\end{align}
	Then from \eqref{bb3} and \eqref{bb1} we obtain
	\begin{align*}
		|\psi_0(y)|\leq 
		C  \varepsilon^3 |y|^4
		+ C \varepsilon |\log\varepsilon| \,|y| ,
	\end{align*}
	which is the inequality \eqref{psi0R2}.
	
	\medskip
	
	For the proof of \eqref{dpsi0R2} we proceed similarly. We consider $y=(y_1,y_2)\in\mathcal{S}_0$, $y_2>0$ and $\varepsilon^{-b} \leq |y| \leq2\varepsilon^{-b}$.
	We compute from \eqref{psi0b}
	\begin{align*}
		\frac{d}{dy_1} \psi_0(y_1,g^+(y_1))
		&=\frac{d}{dy_1}  [\tilde\psi^H(y_1,g^+(y_1)) - \tilde\psi^D(y_1,g^+(y_1)]\chi(y_1,g^+(y_1)) 
		\\
		&\quad+  [\tilde\psi^H(y_1,g^+(y_1)) - \tilde\psi^D(y_1,g^+(y_1)]\frac{d}{dy_1} \chi(y_1,g^+(y_1))
		\\
		&\quad+ \frac{d}{dy_1} \tilde\psi^D (y_1,g^+(y_1)).
	\end{align*}
	Similarly to \eqref{diffpsiHpsiD} we have
	\begin{align}
		\nonumber
		|\nabla\tilde\psi^H(y)-\nabla\tilde\psi^D(y)| \leq C \varepsilon |\log\varepsilon|  + C \frac{|\log\varepsilon|}{|y|^2},
	\end{align}
	which implies that 
	\begin{align*}
		\left| \frac{d}{dy_1}  [\tilde\psi^H(y_1,g^+(y_1)) - \tilde\psi^D(y_1,g^+(y_1)]\chi(y_1,g^+(y_1)) \right|\leq C \varepsilon |\log\varepsilon|  + C \frac{|\log\varepsilon|}{|y|^2}.
	\end{align*} 
	From \eqref{diffpsiHpsiD} we also get
	\begin{align*} 
		\left|  [\tilde\psi^H(y_1,g^+(y_1)) - \tilde\psi^D(y_1,g^+(y_1)]\frac{d}{dy_1} \chi(y_1,g^+(y_1)) \right|\leq C \varepsilon |\log\varepsilon|  + C \frac{|\log\varepsilon|}{|y|^2}.
	\end{align*}
	To estimate $\frac{d}{dy_1} \tilde\psi^D (y_1,g^+(y_1))$ we use \eqref{gplus}, the estimate \eqref{dpsi0R3}, which reads
	\begin{align*}
		\left|\frac{d}{dy_1} \tilde\psi^D (y_1, g_{\varepsilon,D}^-(y_1))\right|\leq C\frac{|\log\varepsilon|}{|y|^2},
	\end{align*}
	and 
	\begin{align*}
		| (g_{\varepsilon,D}^-)'(y_1) -(\tilde g_H^+)'(y_1) |
		& \leq C \varepsilon^3 |y_1|^3
		+ C \varepsilon |\log\varepsilon|  
	\end{align*}
	for $\varepsilon^{-b}\leq|y_1|\leq2\varepsilon^{-b}$, which is proved similarly to \eqref{diff-g}. Combining the previous estimates we get
	\begin{align*}
		\left|\frac{d}{dy_1} \psi_0(y_1,g^+(y_1))\right|
		\leq C \varepsilon |\log\varepsilon| + C \varepsilon^3 |y_1|^3+C\frac{|\log\varepsilon|}{|y|^2}.
	\end{align*}
	This is the proof of \eqref{dpsi0R2} for $j=1$. For $j=2,3$ the computations are similar.
	
\medskip
\noindent \textbf{Step 2.}
Next we construct a function $\bar u$ which is a supersolution of \eqref{equ0}. We do the details of the construction assuming \eqref{restricb1}, which simplifies some steps.
Note that from 
\eqref{psi0R1}--\eqref{psi0R3} 
we get for $y\in\mathcal{S}_0$:
\begin{align}
\nonumber
|\psi_0(y)| \leq C 
\begin{cases}
C \varepsilon |\log\varepsilon| (1+|y|) , & |y| \leq 2\varepsilon^{-b}
\\
\frac{|\log\varepsilon|}{|y|} , &  2\varepsilon^{-b}\leq|y|\leq\frac{\delta_1}{\varepsilon}
\\
\varepsilon e^{-\mu \varepsilon |y|},&|y|\geq\frac{\delta_1}{\varepsilon}.
\end{cases}
\end{align}
This implies the following (less sharp but simpler) bound
\begin{align}
\nonumber
|\psi_0(y)| \leq C 
\begin{cases}
\varepsilon^{1-\sigma b} |\log\varepsilon| (1+|y|)^{1-\sigma} , & |y| \leq 2\varepsilon^{-b}
\\
\varepsilon^{1-b}|\log\varepsilon|, &  2\varepsilon^{-b}\leq|y|\leq\frac{\delta_1}{\varepsilon}
\\
\varepsilon^{1-b}|\log\varepsilon| e^{-\mu \varepsilon |y|},&|y|\geq\frac{\delta_1}{\varepsilon}.
\end{cases}
\end{align}
Using Lemma~\ref{lem:lh2}, there exists a function $\bar u_0$ satisfying
\begin{align}
\label{eq:bar-u0}
\left\{
\begin{aligned}
\Delta \bar u_0 & = 0 && \text{in } \Omega_0 \cap B_{2\varepsilon^{-b}} (0)
\\
\bar u_0 &= \varepsilon^{1-b\sigma} |\log\varepsilon| (1+|y|^2)^{\frac{1-\sigma}{2}} && \text{on }\partial\Omega_0 \cap B_{2\varepsilon^{-b}} (0)
\end{aligned}
\right.
\end{align}
and the inequalities
\begin{align}
\label{ineq-baru1}
C_1 \varepsilon^{1-b\sigma} |\log\varepsilon|  (1+|y|)^{1-\sigma}\leq 
|\bar u_0(y)|\leq C_2 \varepsilon^{1-b\sigma} |\log\varepsilon|  (1+|y|)^{1-\sigma}\quad\text{in }  \Omega_0 \cap B_{2\varepsilon^{-b}} (0) ,
\end{align}
for some constants $0<C_1<C_2$. 
See Remark~\ref{rem:u1}.
Note that 
\begin{align}
\label{psi0i}
|\psi_0(y)|\leq C \bar u_0(y), \quad y\in\partial\Omega_0, \quad |y|\leq 2 \varepsilon^{-b}.
\end{align}
We let $K>0$ be such that $(\frac{K}{C_1})^{\frac{1}{1-\sigma}}=2$ and define
\begin{align*}
\bar u_2 = \min(\bar u_0, K \varepsilon^{1-b}|\log\varepsilon|).
\end{align*}
This function is initially defined in $\Omega_0\cap B_{2\varepsilon^{-b}}(0)$, but because of the choice of $K$ and \eqref{ineq-baru1}, $\bar u_0(y)= K \varepsilon^{1-b}|\log\varepsilon|$ for $y\in\tilde\Omega^H$ and $|y|\geq2\varepsilon^{-b}$. Therefore $\bar u_2$ can be extended as $  K \varepsilon^{1-b}$ to $\Omega_0$. 
By \eqref{psi0R3} and \eqref{psi0i} we have 
\begin{align*}
|\psi_0(y)|\leq C \bar u_2(y), \quad y\in\partial\Omega_0.
\end{align*} 
Since $\bar u_2$ is weakly superharmonic, we obtain that
\begin{align}
\label{cota1}
|u_0(y)|\leq C \bar u_2(y), \quad y\in\Omega_0.
\end{align}
Using suitable barriers and the bound
\begin{align*}
|u_0(y)|\leq K \varepsilon^{1-b}|\log\varepsilon| \quad y\in\Omega_0, \quad |y|=\frac{\delta_1}{\varepsilon},
\end{align*}
we obtain
\begin{align}
\label{cota2}
|u_0(y)|\leq C \varepsilon^{1-b}|\log\varepsilon| e^{-\mu \varepsilon |y|},  \quad y\in\Omega_0, \quad |y|\geq\frac{\delta_1}{\varepsilon}.
\end{align}
Combining \eqref{cota1}, \eqref{cota2} we get
\begin{align}
\nonumber
|u_0(y)| 
\leq C 
\begin{cases}
\varepsilon^{1-\sigma b} |\log\varepsilon| (1+|y|)^{1-\sigma} , & |y| \leq 2\varepsilon^{-b}
\\
\varepsilon^{1-b}|\log\varepsilon|, &  2\varepsilon^{-b}\leq|y|\leq\frac{\delta_1}{\varepsilon}
\\
\varepsilon^{1-b}|\log\varepsilon| e^{-\mu \varepsilon |y|},&|y|\geq\frac{\delta_1}{\varepsilon}.
\end{cases}
\end{align}
This combined with standard elliptic estimates and \eqref{dpsi0R1}--\eqref{dpsi0R3} give
\begin{align*}
|\nabla_y u_0(y)|
\leq C 
\begin{cases}
\varepsilon^{1-\sigma b} |\log\varepsilon| (1+|y|)^{-\sigma} , & |y| \leq 2\varepsilon^{-b}
\\
\varepsilon^{1-b}|\log\varepsilon| |y|^{-1}, &  2\varepsilon^{-b}\leq|y|\leq\frac{\delta_1}{\varepsilon}
\\
\varepsilon^{2-b}|\log\varepsilon| e^{-\mu \varepsilon |y|},&|y|\geq\frac{\delta_1}{\varepsilon}.
\end{cases}
\end{align*}
This directly gives that 
\begin{align*}
\sup_{y\in\mathcal{S}_0,\, |y| \leq \frac{\delta_1}{\varepsilon}}(1+|y|)^{\sigma}  |\nabla_y u_0(y)|\leq C \varepsilon^{1-\sigma b}|\log\varepsilon|.
\end{align*}
and
\begin{align}
\nonumber
\sup_{y\in\mathcal{S}_0,\,|y| \geq \frac{\delta_1}{\varepsilon}}  \varepsilon^\sigma
e^{\mu \varepsilon |y|}|\nabla_y u_0(y)|\leq C \varepsilon^{1-\sigma b}|\log\varepsilon|.
\end{align}
which are part of the norm $\| \  \|_{**,\partial\Omega} $ \eqref{defNf2Omega0}.

By similar estimates for the higher derivatives of $u_0$ we get
\begin{equation*}
\Bigl\|  \frac{\partial u_0}{\partial \nu}
\Bigr\|_{**,\partial\Omega_0} \leq C \varepsilon^{1-b \sigma}|\log\varepsilon|.
\qedhere
\end{equation*}
\end{proof}

\begin{proof}[Proof of Proposition~\ref{prop:est-error-bdy}]
	The result follows by combining the estimates of Lemmas~\ref{lem:normal-psi0}, \ref{lem:error1}, \ref{lem:error1-2}, \ref{lem:error-u1}, and \ref{lem:error-u0}.
\end{proof}

\section{Modification of the hairpin and linearization}
\label{sec:LinH}

A byproduct of the result of Section~\ref{sect:operators} is that the linear problem
\begin{align}
	\label{linear-0}
	\left\{
	\begin{aligned}
		\Delta_y \phi  &= f_1  && \text{in } \Omega_0
		\\
		\frac{\partial \phi}{\partial\nu}
		+ (\kappa-\varepsilon\omega) \phi &= f_2
		&& \text{on } \mathcal{S}_0
		\\
		\phi &= 0 && \text{on } \mathcal{B}_0
	\end{aligned}
	\right.
\end{align}
plays an important role in the solution of \eqref{main-problem}. To solve \eqref{linear-0} we try to find $\phi$ in the form
\begin{align}
	\nonumber
	\phi(y) &= \phi^H(y) \chi^H(x)
	+ \frac{1}{\varepsilon}\phi^{S}(x)\chi^{S}(x)
	+ \frac{1}{\varepsilon}\phi^{D}(x)\chi^D(x),
\end{align}
where $y=\frac{x}{\varepsilon}$ and $\phi^H$, $\phi^{S}$ and $\phi^{D}$ are the new unknowns. The cut-off functions are defined in \eqref{cutoffs} and depend on the parameters $0<\delta_1<\delta_2\ll1$. The idea is that $\phi^H$, $\phi^{S}$, $\phi^{D}$ solve linear problems in domains $\tilde \Omega^H$, $\tilde\Omega^S$ and $\tilde\Omega^D$, which are close to the hairpin $\Omega^H$, the strip $\Omega^S$, and the disk $\Omega^D$. The purpose of Sections~\ref{sec:LinH}--\ref{sect:S} is to construct these domains, and study some linear problems on them.

\medskip
In this section we define a domain $\tilde\Omega^H$, which is a perturbation of the hairpin $\Omega^H$ \eqref{defHairpin}, and study the linear problem \eqref{robin3a} on $\tilde\Omega^H$. 
Let $\delta_2>0$ be fixed. We construct $\tilde\Omega^H$ so that coincides with $\Omega_0$ up to $|y|\leq 2 \delta_2\varepsilon^{-1}$ and otherwise coincides with the hairpin $\Omega^H$.
\index{OH@$\tilde \Omega^H$, coincides with $\Omega_0$ up to $\abs{y} \leq 2 \delta_2\varepsilon^{-1}$ and otherwise coincides with the hairpin $\Omega^H$}

We deal first with the upper part of $\tilde\Omega^H$. Let us recall that part of the boundary of $\Omega_0$ was defined as a graph
\[
y_2 = g^+(y_1) ,\quad \frac{\pi}{2}-1\leq |y_1|\leq \frac{R}{2\varepsilon},
\]
where
\[
g^+ =\tilde g_H^+  \eta_{\varepsilon^{-b}} + g_{\varepsilon,D}^- (1- \eta_{\varepsilon^{-b}}),
\]
(see \eqref{upperbdy}), where 
\[
\eta_\lambda(y_1,y_2) = \eta_0\Bigl(\frac{y_1}{\lambda}\Bigr)
\]
and $\eta_0$ is defined in \eqref{eta0}.
We define
\begin{align}
	\label{tildegp}
	\tilde g^+= g^+ \eta_{2\delta_2/\varepsilon}  +  g_H^+ \left(  1-  \eta_{2\delta_2/\varepsilon} \right).
\end{align}

For the lower part we do something similar. 
We have defined 
\[
g^- = g_H^-  \eta_{\varepsilon^{-b}}  +g_{\varepsilon,S}^+  (1- \eta_{\varepsilon^{-b}}),
\quad \frac{\pi}{2}-1\leq |y_1|\leq \frac{\delta}{\varepsilon}
\]
in \eqref{defgm}, and now let
\begin{align}
	\label{tildegm}
	\tilde g^- = g^-  \eta_{2\delta_2/\varepsilon}  +  g_H^- \left(  1-  \eta_{2\delta_2/\varepsilon} \right) .
\end{align}
We extend the functions $\tilde g^\pm$ as $g_H^\pm$ for $|y_1|\geq \frac{2\delta_2}{\varepsilon}$.
Then we define
\begin{align}
	\label{tOmegaH}
	\tilde\Omega^H
	= 
	\Bigl\{\, |y_1| < \frac{\pi}{2}-1, \  y_2 \in \R\Bigr\}
	\cup
	\Bigl\{\,  |y_1| >  \frac{\pi}{2}-1, \ 
	\tilde g^-(y_1)< y_2 < \tilde g^+(y_1)
	\Bigr\} .
\end{align}

Next we consider
\begin{align}
	\label{pHairpinv0}
	\left\{
	\begin{aligned}
		\Delta_y \phi &= f_1 && \text{in } \tilde \Omega^H
		\\
		\frac{\partial \phi}{\partial  \nu_y} 
		+ (\kappa - \varepsilon \omega \chi^H_2) \phi
		&=f_2  && \text{on } \partial \tilde \Omega^H.
	\end{aligned}
	\right.
\end{align}
Here $\chi_2^H(y)$ is defined as 
\begin{align*}
	\chi_2^H(y) = \chi_0\Bigl( \frac{\varepsilon y}{2\delta_2} \Bigl),
\end{align*}
where $\chi_0$ is the radial cut-off function in $\R^2$ defined in \eqref{chi0}.

We introduce some norms that are well adapted to the problem \eqref{pHairpinv0}, with the main feature that $\nabla \phi$ decays at the rate $|y|^{-\sigma}$, where $0<\sigma<1$.
We let $\Omega$ be either $\Omega^H$ or $\tilde\Omega^H$ and let $0<\sigma,\alpha <1$. 
For a function $\phi$ defined in $\Omega$, we let
\begin{align}
	\nonumber
	\| \phi \|_{\sharp,\Omega}
	&= \sup_{y \in \Omega}\,
	(1+|y|)^{\sigma-1}
	\Bigl[
	|\phi(y)| + (1+|y|) |\nabla \phi(y)|
	+ (1+|y|)^2 |D^2\phi(y)|
	\\
	\label{nH3}
	& \qquad \qquad \qquad \qquad 
	+ (1+|y|)^{2+\alpha} [D^2 \phi ]_{\alpha,B(y,\frac{|y|}{10})\cap \Omega}
	\Bigr] .
\end{align}
\index{h3@$\norm \ \norm_{\sharp,\Omega}$, \eqref{nH3} weighted H\"older norm for the solution  in the hairpin}
For $f_1$ defined in $\Omega$ we define
\begin{align}
	\label{nH1}
	\| f_1 \|_{\sharp\sharp,\Omega}
	&= \sup_{y \in \Omega}\,
	(1+|y|)^{1+\sigma}
	\left[
	|f_1(y)| + (1+|y|)^{\alpha} [f_1]_{\alpha,B(y,\frac{|y|}{10})\cap \Omega}
	\right] ,
\end{align}
\index{h2@$\norm \ \norm_{\sharp\sharp,\Omega}$, \eqref{nH1}, weighted H\"older norm for the RHS in the equation in the hairpin}
and for $f_2$ defined on $\partial \Omega$ we let
\begin{align}
\nonumber
\| f_2 \|_{\sharp\sharp,\partial\Omega}
&= 
\sup_{y \in \partial \Omega}
(1+|y|)^{\sigma}
\Bigl[ |f_2(y)|   + (1+|y|)|f_2'(y)|
\\
\label{nH2}
& \qquad \qquad \qquad \qquad 
+ (1+|y|)^{1+\alpha} [f_2']_{\alpha,B(y,\frac{|y|}{10})\cap \partial\Omega} 
\Bigr] ,
\end{align}
\index{h3@$\norm \ \norm_{\sharp\sharp,\partial\Omega}$, \eqref{nH2} weighted H\"older norm for the RHS on the boundary in the hairpin}
where ${}'$ is the derivative with respect to arc length.

\begin{proposition}
\label{prop:lh1}
There is $\delta_2^*>0$ and $C>0$ so that for  $0<\delta_2<\delta_2^*$ there is $\varepsilon^*>0$ such that for $0<\varepsilon<\varepsilon^*$ the following holds.
Let $f_1$ and $f_2$ satisfy $\|f_1\|_{\sharp\sharp,\tilde\Omega^H}<\infty$,
$\|f_2\|_{\sharp\sharp,\partial\tilde\Omega^H}<\infty$ and be even with respect to $y_1$.
Then there is a function $\phi$ even in $y_1$ and solving \eqref{pHairpinv0}, which defines a linear operator of $f_1$, $f_2$ such that
\begin{align}
\nonumber
\|\phi\|_{\sharp,\tilde\Omega^H} \leq C ( \|f_1\|_{\sharp\sharp,\tilde\Omega^H} +\|f_2\|_{\sharp\sharp,\partial\tilde\Omega^H}). 
\end{align}
\end{proposition}

A first step is to analyze the problem in the unperturbed hairpin $\Omega^H$:
\begin{align}
	\label{pHairpin3}
	\left\{
	\begin{aligned}
		\Delta_y \phi  &= f_1 \quad 
		\text{in }  \Omega^H 
		\\
		\frac{\partial \phi}{\partial  \nu_y} 
		+ \kappa \phi
		& =f_2 
		\quad \text{on } \partial \Omega^H  ,
	\end{aligned}
	\right.
\end{align}
where $\kappa $ is the curvature of $\partial \Omega^H$. 

We recall that $\Omega^H$ can be written as 
\[
\Omega^H=F(S), \quad S = \Bigl\{ \, w\in \C \  \Big| \  | \Re(w)|  < \frac{\pi}{2} \, \Bigr\} ,
\]
where
\begin{align}
	\label{Fhairpin}
	F(w) = w + \sin(w) 
\end{align}
and that 
\[
\psi^H(F(w))=\Re(\cos(w)), \quad w \in S,
\]
is harmonic in $\Omega^H$ and satisfies
\[
\psi^H = 0 , \quad |\nabla \psi^H|=1, \quad 
\text{on } \partial S. 
\]
A calculation shows that 
\[
Z_1^H = \partial_{y_1} \psi^H, \quad
Z_2^H = \partial_{y_2} \psi^H ,
\]
satisfy \eqref{pHairpin3} with $f_1=0$ and $f_2=0$.

We use the same sign conventions and orientations described in Section~\ref{sectFormal}.
Then the curvature of the boundary of $\Omega^H$ is given by
\[
\kappa(F(w)) = -\frac{1}{|\sin(w)|^2} = - \frac{1}{\cosh^2(t)}, 
\quad w = \pm \frac{\pi}{2} + i t, \quad t\in \R.
\]

For the unperturbed case we have a result without assuming even symmetry in $y_1$ of $f_1$, $f_2$, but instead assuming 
\begin{align}
	\label{orthog}
	\int_{\Omega^H} f_1 Z_1^H dy = \int_{\partial\Omega^H} f_2 Z_1^H d\ell(y),
\end{align}
where $d\ell$ is arclength.

We note that the integrals in \eqref{orthog} are well defined if  $\|f_1\|_{\sharp\sharp,\Omega^H}<\infty$, $\|f_2\|_{\sharp\sharp,\partial\Omega^H}<\infty$, by the decay of $Z_1^H$ from Lemma~\ref{lemma:estGradPsiH}.
We also remark that if $f_1$ and $f_2$ are even with respect to $y_1$, then \eqref{orthog} holds.

\begin{lemma}
	\label{lemmaH1}
	Let $f_1$ and $f_2$ satisfy $\|f_1\|_{\sharp\sharp,\Omega^H}<\infty$, $\|f_2\|_{\sharp\sharp,\partial\Omega^H}<\infty$, and assume that \eqref{orthog} holds.
	Then there is a function $\phi$ with $\|\phi\|_{\sharp,\Omega^H}<\infty$ solving \eqref{pHairpin3}, which defines a linear operator of $f_1$, $f_2$ and such that 
	\begin{align}
		\label{estPhiH3}
		\|\phi\|_{\sharp,\Omega^H} \leq C ( \|f_1\|_{\sharp\sharp,\Omega^H} +\|f_2\|_{\sharp\sharp,\partial\Omega^H}). 
	\end{align}
\end{lemma}
\begin{proof}
	We use the conformal map  $F$ defined in \eqref{Fhairpin}.
	Let
	\[
	\tilde \phi = \phi \circ F . 
	\]
	Then \eqref{pHairpin3} is equivalent to
	\begin{align}
		\label{pHairpin4}
		\left\{
		\begin{aligned}
			\Delta_w \tilde \phi  &= \tilde f_1 \quad 
			\text{in }  S
			\\
			\frac{\partial \tilde\phi}{\partial  \nu} 
			+ \tilde \kappa \tilde\phi
			& = \tilde f_2 
			\quad \text{on } \partial S  ,
		\end{aligned}
		\right.
	\end{align}
	where
	\begin{align*}
		\tilde f_1(w) &= | F'(w) |^2 \, f_1(F(w)) ,
		\\
		\tilde f_2(w) &= | F'(w) | \, f_2(F(w)) ,
		\\
		\tilde \kappa(w) &= | F'(w) | \, \kappa(F(w)) .
	\end{align*}
	In \eqref{pHairpin4}, $\nu$ denotes the outward unit normal vector to $\partial S$, that is, for $t$ real, $\nu(\pm\frac{\pi}{2}+it)=(\pm 1,0)$ in the $w$-plane.
	A computation gives
	\[
	\tilde \kappa (w) = -\frac{1}{|\sin(w)|} = - \frac{1}{\cosh(w_2)} , \quad w = \pm \frac{\pi}{2} + i w_2, \quad w_2\in \R.
	\]
	The condition \eqref{orthog} takes the form
	\begin{align*}
		\int_S \tilde f_1(w) Z_1^H(F(w)) dw = \int_{\partial S} \tilde f_2(w) Z_1^H(F(w)) d\ell(w), 
	\end{align*}
	where again $d\ell(w)$ is arc-length.
	
	We solve \eqref{pHairpin4} in two stages. First we solve the Dirichlet problem
	\begin{align}
		\label{pHairpinD}
		\left\{
		\begin{aligned}
			\Delta_w \tilde \phi_1  &= \tilde f_1 \quad 
			\text{in }  S
			\\
			\tilde\phi_1 & = 0
			\quad \text{on } \partial S  ,
		\end{aligned}
		\right.
	\end{align}
	and then we solve the Robin problem
	\begin{align}
		\label{pHairpin6}
		\left\{
		\begin{aligned}
			\Delta_w \tilde \phi_2  &= 0\quad 
			\text{in }  S
			\\
			\frac{\partial \tilde\phi_2}{\partial  \nu} 
			+ \tilde \kappa \tilde\phi_2
			& = h
			\quad \text{on } \partial S,
		\end{aligned}
		\right.
	\end{align}
	where
	\begin{align*}
		h = \tilde f_2  - \frac{\partial\tilde\phi_1}{\partial\nu}. 
	\end{align*}
	The solution to \eqref{pHairpin4} is then given by $\tilde \phi = \tilde \phi_1 + \tilde \phi_2$.
	
	\medskip
	
	Let us solve \eqref{pHairpinD}. Using that $ |F'(w)|^2 = |1+\cos(w)|^2$, the relation \eqref{expansionw2b}, and the definition of $\|f_1\|_{\sharp\sharp,\Omega^H}$ \eqref{nH1}, we get
	\[
	|\tilde f_1(w) | \leq C \| f_1\|_{\sharp\sharp,\Omega^H}
	e^{(1-\sigma)|w_2|}, \quad w = w_1+ i w_2.
	\]
	Using the barrier 
	\[
	\bar\phi_1(w) =  \cosh( (1-\sigma) w_2 ) \cos( a w_1)
	\]
	with $(1-\sigma)<a<1$ we find existence of a unique solution to \eqref{pHairpinD} satisfying
	\begin{align}
		\nonumber
		|\tilde \phi_1(w) |\leq C  \| f_1\|_{\sharp\sharp,\Omega^H}e^{(1-\sigma)|w_2|}.
	\end{align}
	By standard elliptic estimates we also get
	\begin{align}
		\label{x2}
		\Bigl| \frac{\partial \tilde \phi_1}{\partial \nu}(w) \Bigr|\leq C  \| f_1\|_{\sharp\sharp,\Omega^H}e^{(1-\sigma)|w_2|}.
	\end{align}
	By \eqref{x2} and the definition of $\|f_2\|_{\sharp\sharp,\partial\Omega^H}$ in \eqref{nH2}, we have
	\begin{align*}
		|h(w)|\leq C ( \|f_1\|_{\sharp\sharp,\Omega^H} + \|f_2\|_{\sharp\sharp,\partial\Omega^H} ) e^{(1-\sigma)|w_2|}, \quad w \in \partial S.
	\end{align*}
	We note that
	\begin{align}
		\nonumber
		&\int_{\partial S} h(w) Z_1^H(F(w)) d\ell(w)
		\\
		\nonumber
		&=\int_{\partial S}\tilde f_2(w)Z_1^H(F(w)) d\ell(w)
		-\int_{\partial S}\frac{\partial\tilde\phi_1}{\partial\nu}	Z_1^H(F(w)) d\ell(w)
		\\
		\nonumber
		&=\int_{\partial S}\tilde f_2(w)Z_1^H(F(w)) d\ell(w)
		- \int_S\tilde f_1(w)Z_1^H(F(w))dw
		\\
		\label{int0}
		&=0.
	\end{align}
	
	\medskip
	
	Let us consider problem \eqref{pHairpin6}. Changing the notation, \eqref{pHairpin6} becomes
	\begin{align}
		\label{pHairpin7}
		\left\{
		\begin{aligned}
			\Delta  u  &= 0 \quad 
			\text{in }  S
			\\
			\frac{\partial u}{\partial  \nu} 
			-\frac{1}{|\sin(w)|} u
			& = h
			\quad \text{on } \partial S.
		\end{aligned}
		\right.
	\end{align}
	We further rewrite the boundary condition in \eqref{pHairpin7} as 
	\begin{align}
		\nonumber
		\partial_{w_1} u - \frac{1}{\sin(w)} u = \tilde h
		\quad \text{on } \partial S,
	\end{align}
	where 
	\[
	\tilde h\Big(\pm\frac{\pi}{2}+i w_2\Bigr) = \pm h\Bigl(\pm\frac{\pi}{2}+i w_2\Bigr) , \quad w_2\in \R.
	\]
	We note that $\sin(\pm \frac{\pi}{2}+iw_2) = \pm \cosh(w_2) $ for $w_2\in \R$, and so we may write \eqref{pHairpin7} as 
	\begin{align}
		\label{u2}
		\left\{
		\begin{aligned}
			\Delta u &=0 \quad \text{in }S
			\\
			\partial_{w_1} u - \frac{1}{\sin(w)} u &= \sin(w) h_2
			\quad \text{on } \partial S,
		\end{aligned}
		\right.
	\end{align}
	where 
	\begin{align*}
		h_2(w)=\frac{\tilde h(w)}{\sin(w)}, \quad w=\pm\frac{\pi}{2}+iw_2, \ w_2\in\R.
	\end{align*}
	Note that $h_2$ is real valued and satisfies
	\begin{align}
		\label{g2f}
		|h_2(w) | \leq C ( \|f_1\|_{\sharp\sharp,\Omega^H} + \|f_2\|_{\sharp\sharp,\partial\Omega^H} ) e^{-\sigma|w_2|}, \quad w \in \partial S.
	\end{align}
	Actually $h_2$ is H\"older continuous on $\partial S$, because $\tilde f_1 $ in \eqref{pHairpinD} is H\"older continuous, and by standard elliptic estimates $ \frac{\partial\tilde \phi_1}{\partial \nu}$ is H\"older continuous on $\partial S$. Moreover, the H\"older continuity of $h_2$ has the following form:
	\begin{align}
		\label{g2b}
		\frac{|h_2(w_1,w_2')-h_2(w_1,w_2'')|}{|w_2'-w_2''|^\alpha} 
		\leq C  ( \|f_1\|_{\sharp\sharp,\Omega^H} + \|f_2\|_{\sharp\sharp,\partial\Omega^H} ) e^{-\sigma w_2}, \quad w_1=\pm\frac{\pi}{2},
	\end{align}
	where $w_2 = \min(|w_2'|,|w_2''|)$.
	
	We remark that for $w=\pm\frac{\pi}{2}+iw_2$, $w_2\in\R$,
	\begin{align*}
		Z_1^H(F(w)) = \Re\Bigl( \frac{\sin(w)}{1+\cos(w)}\Bigr) 
		= \pm\frac{1}{\cosh(w_2)}.
	\end{align*}
	Hence, for $w=\pm\frac{\pi}{2}+iw_2$, $w_2\in\R$,
	\begin{align*}
		h(w) Z_1^H(F(w))
		&=\pm h_2(w) \sin(w) \Re\Bigl( \frac{\sin(w)}{1+\cos(w)}\Bigr)=\pm h_2(w) ,
	\end{align*}
	and we get from \eqref{int0} that
	\begin{align}
		\label{int0b}
		\int_{-\infty}^\infty \Bigl(h_2\Bigl(\frac{\pi}{2}+iw_2\Bigr)-h_2\Bigl(-\frac{\pi}{2}+iw_2\Bigr) \Bigr) dw_2=0.
	\end{align}
	
	To motivate the formula that we use later on, let us proceed  assuming that we have a solution to \eqref{u2}.
	Since $u$ is harmonic in $S$, we can write it as the real part of an analytic function $f$ on $S$, $u = \Re(f)$.
	Let $H_2$ be the bounded analytic function on $S$ such that $\Re(H_2)=h_2$ on $\partial S$.
	Then $\Im(H_2)$ is unique up to the addition of  a constant. 
	Then by \eqref{u2}
	\[
	\Re\Bigl[ f'(w) -  \frac{1}{\sin(w)} f\Bigr] = \Re\left[\sin(w) H_2 \right] , \quad \text{on }\partial S ,
	\]
	where $f'(w)$ is the complex derivative of $f$.
	By the maximum principle it is reasonable to expect that this is also true on $S$. Then we can solve for $f$ using the integrating factor $\frac{\cos(\frac{w}{2})}{\sin(\frac{w}{2})}$
	\[
	f(w) = \frac{\sin(\frac{w}{2})}{\cos(\frac{w}{2})} \Bigl( \int_0^w \frac{\cos(\frac{v}{2})}{\sin(\frac{v}{2})} \sin(v) H_2(v) \,dv + c\Bigr), 
	\quad w \in S,
	\]
	for some constant $c$.
	
	We use now this formula to construct a solution to \eqref{u2}. Given $h_2$ satisfying \eqref{g2f}, \eqref{g2b} and \eqref{int0b} we let $H_2$ be the bounded analytic function in $S$ such that $\Re(H_2)=h_2$ on $\partial S$. The imaginary part of $H_2$ is defined up to a constant that we fix by setting $\Im(H_2)(0)=0$.
	Let $j_2=\Im(H_2)$.
	Then $j_2$ is harmonic in $S$ and $H_2 = h_2 + i j_2$. Using that $h_2$ satisfies \eqref{g2f} and \eqref{g2b}, we get that
	\begin{align}
		\label{h2}
		U = \lim_{w_2\to\infty} j_2(w_1,w_2),\quad
		L = \lim_{w_2\to-\infty} j_2(w_1,w_2)
	\end{align}
	exist uniformly for $w_1\in[-\frac{\pi}{2},\frac{\pi}{2}]$. Moreover,
	\begin{align*}
		|j_2(w_1,w_2)-U|&\leq C(\|f_1\|_{\sharp\sharp,\Omega^H} + \|f_2\|_{\sharp\sharp,\partial\Omega^H} ) e^{-\sigma w_2},
		\\
		|j_2(w_1,w_2)-L|&\leq C(\|f_1\|_{\sharp\sharp,\Omega^H} + \|f_2\|_{\sharp\sharp,\partial\Omega^H} ) e^{-\sigma w_2}.
	\end{align*}
	The proof of this is presented in Section~\ref{sec:holder-strip}.
	
	Integrating $H_2$ on the closed curve $\partial\mathcal R$ where $\mathcal R = [-\frac{\pi}{2},\frac{\pi}{2}]\times [-R,R]$ and letting $R\to \infty$ we find that
	\begin{align*}
		\int_{-\infty}^\infty \Bigl(h_2\Bigl(\frac{\pi}{2}+iw_2\Bigr)-h_2\Bigl(-\frac{\pi}{2}+iw_2\Bigr)\Bigr)dw_2 = \pi ( L - U).
	\end{align*}
	From \eqref{int0b} we deduce that $L=U$.
	We then define
	\begin{align*}
		\tilde H_2 = h_2 + i j_2 -i L.
	\end{align*}
	It follows from \eqref{g2f} and \eqref{h2} that 
	\begin{align}
		\label{tG2}
		|\tilde H_2(w) | \leq C ( \|f_1\|_{\sharp\sharp,\Omega^H} + \|f_2\|_{\sharp\sharp,\partial\Omega^H} ) e^{-\sigma |w_2|}, \quad w \in S.
	\end{align}
	
	Then we define 
	\[
	f(w) = \frac{\sin(\frac{w}{2})}{\cos(\frac{w}{2})} \int_0^w \frac{\cos(\frac{v}{2})}{\sin(\frac{v}{2})} \sin(v) \tilde H_2(v) \,dv, 
	\quad w \in S.
	\]
	There is no singularity in the integrand, so $f(w)$ is single valued.
	Then $u = \Re(f)$ satisfies \eqref{u2} and from \eqref{tG2} we have that
	\[
	|f(w) | \leq C ( \|f_1\|_{\sharp\sharp,\Omega^H} + \|f_2\|_{\sharp\sharp,\partial\Omega^H} ) e^{(1-\sigma)|w_2|}, \quad w \in S.
	\]
	
	\medskip
	
	Going back to \eqref{pHairpin6}, we have constructed a solution $\tilde \phi_2$, which satisfies
	\begin{align}
		\nonumber
		|\tilde \phi_2(w) |\leq C  ( \|f_1\|_{\sharp\sharp,\Omega^H} + \|f_2\|_{\sharp\sharp,\partial\Omega^H} )
		e^{(1-\sigma)|w_2|}.
	\end{align}
	This gives the desired solution $\phi$ of \eqref{pHairpin3}.
	By standard elliptic estimates, it satisfies \eqref{estPhiH3}.
\end{proof}

\begin{proof}[Proof of Proposition~\ref{prop:lh1}]
	Let us rewrite the problem \eqref{pHairpinv0} to be solved as
	\begin{align}
		\label{pHairpin2}
		\left\{
		\begin{aligned}
			\Delta_y \tilde\phi &= \tilde f_1 &&
			\text{in } \tilde \Omega^H 
			\\
			\frac{\partial \tilde\phi}{\partial  \nu_{\partial\tilde\Omega^H}} 
			+ (\kappa_{\partial\tilde\Omega^H} - \varepsilon \omega \chi^H_2) \tilde\phi
			&= \tilde f_2 && \text{on } \partial \tilde \Omega^ H  ,
		\end{aligned}
		\right.
	\end{align}
	where $\tilde\Omega^H$ is given in \eqref{tOmegaH} and we have changed the unknown to $\tilde\phi$.

	Note that the functions $\tilde g^+$ and $\tilde g^-$ used to define $\tilde \Omega^H$ (in \eqref{tildegp} and \eqref{tildegm}) can be decomposed as 
	\begin{align*}
		\tilde g^+ = g_H^+ + \hat g^+, \qquad
		\tilde g^- = g_H^- + \hat g^-,
	\end{align*}
	where
	\begin{align}
		\nonumber
		\hat g^+
		& = g_H^+ ( \eta_{\varepsilon^{-b}} -  \eta_{2\delta_2/\varepsilon}  ) 
		+ \tfrac{1}{4}\varepsilon\omega y_1^2
		(1-\eta_0(|y_1|-10)) \eta_{\varepsilon^{-b} }
		\\
		\label{defHatgplus}
		& \quad 
		+ g_D^- (1-\eta_{\varepsilon^{-b}} ) \eta_{2\delta_2/\varepsilon} ,
	\end{align}
	and
	\begin{align*}
		\hat g^- = g_H^- ( \eta_{\varepsilon^{-b}} -  \eta_{2\delta_2/\varepsilon}  ) - \varepsilon^{-1}  d^S(1-\eta_{\varepsilon^{-b}}) \eta_{2\delta_2/\varepsilon}.
	\end{align*}
	
	We make the following change of variables
	\[
	y \in \tilde\Omega^H \mapsto z = K(y)\in\Omega^H
	\]
	where
	\begin{align*}
		K(y) &= y + (0,\varphi(y)), 
		\\    
		\varphi(y_1,y_2) &= \begin{cases}
			-\hat g^+(y_1)\eta_0(\frac{\varepsilon y_2}{\delta_2}), & y_2>0
			\\
			-\hat g^-(y_1)\eta_0(\frac{\varepsilon y_2}{\delta_2}), & y_2<0,
		\end{cases}
	\end{align*}
	and $\eta_0$ is the cut-off function defined in \eqref{eta0}.
	Let 
	\[
	\phi(K(y)) = \tilde\phi(y),\quad y\in\tilde\Omega^H.
	\]
	Then the problem \eqref{pHairpin2} is transformed into
	\begin{align}
		\label{pHairpin2b}
		\left\{
		\begin{aligned}
			\Delta_z \phi + B_1[\phi]&= f_1&&
			\text{in } \Omega^H 
			\\
			\frac{\partial \phi}{\partial  \nu_{\partial\Omega^H}} 
			+ \kappa_{\partial\Omega^H}\phi 
			+ B_2[\phi]
			&=f_2 && \text{on } \partial \Omega^ H  ,
		\end{aligned}
		\right.
	\end{align}
	where
	\begin{align*}
		B_1[\phi]&=
		2 \partial_{z_1,z_2}^2 \phi \partial_{y_1} \varphi
		+2 \partial_{z_2,z_2}^2 \phi \partial_{y_2} \varphi
		+\partial_{z_2,z_2}^2 \phi |\nabla_y\varphi|^2
		\\
		&\quad+\partial_{z_2}\phi\Delta_{y}\varphi
        \\
		B_2[\phi] &= - \varepsilon \omega \chi^H_2\phi 
		+ (\hat\kappa-\kappa_{\partial\Omega^H})\phi
		\\
		& \quad
		+\partial_{z_1}\phi
		\left(\frac{\frac{d\tilde g^+}{dy_1}}{(1+(\frac{d\tilde g^+}{dy_1})^2)^{\frac{1}{2}}}-\frac{\frac{dg^+_H}{dy_1}}{(1+(\frac{dg^+_H}{dy_1})^2)^{\frac{1}{2}}}\right)
		\\
		& \quad
		+\partial_{z_2}\phi
		\left(-\frac{1}{(1+(\frac{d\tilde g^+}{dy_1})^2)^{\frac{1}{2}}}+\frac{1}{(1+(\frac{dg^+_H}{dy_1})^2)^{\frac{1}{2}}}-\frac{\frac{d\tilde g^+}{dy_1}\frac{d\hat g^+}{dy_1}}{(1+(\frac{d\tilde g^+}{dy_1})^2)^{\frac{1}{2}}}
		\right)
	\end{align*}
	and
	\begin{align*}
		\hat\kappa(K(y)) = \kappa_{\partial\tilde \Omega^H}(y),
		\qquad f_j(K(y)) = \tilde f_j(y).
	\end{align*}
	
	For given right hand sides $f_1$, $f_2$ we construct a solution $\phi$ of \eqref{pHairpin2b}, which is a linear operator of $f_1$ and $f_2$.
	To do this we treat \eqref{pHairpin2b} as a perturbation of \eqref{pHairpin3}. 
	Let $\phi=T[f_1,f_2]$ be the operator constructed in Lemma~\ref{lemmaH1}, that given $f_1$ and $f_2$ produces a solution $\phi$ of \eqref{pHairpin3} with the estimate
	\begin{align}
		\nonumber
		\|T[f_1,f_2]\|_{\sharp,\Omega^H} \leq C ( \|f_1\|_{\sharp\sharp,\Omega^H} +\|f_2\|_{\sharp\sharp,\partial\Omega^H}). 
	\end{align}
	We construct a solution to \eqref{pHairpin2b} as a fixed point of the operator
	\[
	\phi \mapsto F[\phi]=T[f_1-B_1[\phi],f_2-B_2[\phi]].
	\]
	The existence of a unique fixed point follows from the following claim: there exist $\delta_2^*>0$ and $C>0$ so that for  $0<\delta_2<\delta_2^*$ there is $\varepsilon^*>0$ such that for $0<\varepsilon<\varepsilon^*$
	\begin{align}
		\label{abc}
		\|T[f_1-B_1[\phi],f_2-B_2[\phi]]\|_{\sharp,\Omega^H} \leq 
		\tfrac{1}{2}\|\phi\|_{\sharp,\Omega^H}
		+ C \|f_1\|_{\sharp\sharp,\Omega^H}+ C \|f_2\|_{\sharp\sharp,\partial\Omega^H}.
	\end{align}
	
	Indeed, let us estimate $\| B_1[\phi] \|_{\sharp\sharp,\Omega^H}$.
	For this we need to estimate first
	\[
	\sup_{z\in \Omega^H} (1+|z|)^{1+\sigma} |B_1[\phi]|.
	\]
	Let $z = (z_1,z_2) \in \Omega^H$ and assume first that $z_2>0$. 
	From \eqref{defHatgplus} we get 
	\[
	|(\hat g^+)'(z_1)| \leq C m_1
	\]
	where
	\[
	m_1 = |\log\varepsilon| \varepsilon^b + \varepsilon^{1-b} + \frac{\varepsilon}{\delta_2}\log\Bigl( \frac{\delta_2}{\varepsilon}\Bigr) + \delta_2 .
	\]
	Note that $m_1$ is small  if $\delta_2$ is small and then $\varepsilon$ is small depending on $\delta_2$.
	Then
	\begin{align}
		\label{B1}
		| \partial_{z_1,z_2}^2 \phi \partial_{y_1} \varphi |
		& \leq C   \frac{m_1}{(1+|z|)^{1+\sigma}} \| \phi \|_{\sharp,\Omega^H}.
	\end{align}
	The computation is similar for the other terms in $B_1$ involving second order derivatives.
	To estimate the term with first order derivative in $B_1$ we note that 
	\begin{align*}
		|\partial_{y_1,y_1}^2\varphi|\leq|(\hat g^+)''(z_1)|\leq C \frac{\varepsilon m_2}{\delta_2}
	\end{align*}
	where
	\begin{align*}
		m_2 = \delta_2  + \frac{\varepsilon}{\delta_2} \log\Bigl( \frac{\delta_2}{\varepsilon}\Bigr) .
	\end{align*}
	We have used here that $b>\frac{1}{2}$ so that $ \varepsilon^{2b} |\log\varepsilon| \leq \varepsilon$ for $\varepsilon$ small.
	Similarly,
	\begin{align*}
		|\partial_{y_2,y_2}^2\varphi|\leq C \frac{\varepsilon^2}{\delta_2^2}
		\leq C \frac{\varepsilon m_2}{\delta_2}.
	\end{align*}
	We note that $\Delta_y \varphi$ is supported on $|z|\leq \frac{4\delta_2}{\varepsilon}$.
	Then we estimate, for $z_2>0$ and $|z_1|\leq \frac{4\delta_2}{\varepsilon}$:
	\begin{align}
		\nonumber
		|\partial_{z_2} \phi \Delta_y \varphi|
		& \leq C \frac{\varepsilon m_2}{\delta_2(1+|z|)^\sigma} \|\phi\|_{\sharp,\Omega^H}
		\\
		\label{B2}
		& \leq C \frac{m_2}{ (1+|z|)^{1+\sigma}} \|\phi\|_{\sharp,\Omega^H} .
	\end{align}
	Then note that  $m_2$ is small  if $\delta_2$ is small and then $\varepsilon$ is small depending on $\delta_2$.
	
	Combining \eqref{B1} and \eqref{B2} and a similar calculation for $z_2<0$, we obtain
	\begin{align*}
		\sup_{z\in \Omega^H } (1+|z|)^{1+\sigma} |B_1[\phi]| \leq C (m_1+m_2) \|\phi\|_{\sharp,\Omega^H}.
	\end{align*}
	The H\"older part of the norm $\| \ \|_{\Omega^H}$ is handled similarly and we obtain
	\begin{align}
		\label{Bx1}
		\| B_1[\phi] \|_{\sharp\sharp,\Omega^H} \leq C m_3 \|\phi\|_{\sharp,\Omega^H}
	\end{align}
	where $m_3$ is small if $\delta_2$ and $\varepsilon$ are chosen appropriately small and $C$ is fixed.
	
	The estimate of the term $B_2[\phi]$ is similar. We analyze some of the terms in this expression. 
	We have
	\begin{align*}
		|\varepsilon \omega \chi^H_2(z)\phi(z) | &\leq C \varepsilon (1+|z|)^{1-\sigma}  \chi^H_2(z) \|\phi\|_{\sharp,\Omega^H}
		\leq C \frac{\delta_2 }{(1+|z|)^\sigma}\|\phi\|_{\sharp,\Omega^H}.
	\end{align*}
	An analogous computation gives
	\begin{align*}
		\| \varepsilon \omega \chi^H_2 \phi\|_{\sharp\sharp,\partial\Omega}  \leq C \delta_2 \|\phi\|_{\sharp,\Omega^H}.
	\end{align*}
	All the other terms in $B_2$ are estimated similarly and we get
	\begin{align}
		\label{Bx2}
		\| B_2[\phi] \|_{\sharp\sharp,\partial\Omega^H} \leq C m_4 \|\phi\|_{\sharp,\Omega^H},
	\end{align}
	where $m_4$ is small if $\delta_2$ and $\varepsilon$ are chosen appropriately small and $C$ is fixed.
	
	From \eqref{Bx1} and \eqref{Bx2} we obtain \eqref{abc}.
\end{proof}

We also consider the Dirichlet problem
\begin{align}
	\label{pHairpinD2}
	\left\{
	\begin{aligned}
		\Delta_y \phi  &= f_1 && \text{in }  \tilde\Omega^H 
		\\
		\phi & =f_2 &&\text{on } \partial \tilde\Omega^H.
	\end{aligned}
	\right.
\end{align}
We use the following norm for $f_2$ defined on $\partial\Omega$, where $\Omega$ is either $\Omega^H$ or $\tilde\Omega^H$. 
For $0<\sigma,\alpha <1$ we define 
\begin{align}
	\nonumber
	\| f_2 \|_{\sharp,\partial\Omega}
	&= \sup_{y \in \partial\Omega}
	(1+|y|)^{\sigma-1}
	\Bigl[
	|f_2(y)| + (1+|y|) |f_2'(y)|
	+ (1+|y|)^2 |f_2''(y)|
	\\
	\nonumber
	& \qquad \qquad \qquad \qquad 
	+ (1+|y|)^{2+\alpha} [ f_2'' ]_{\alpha,B(y,\frac{|y|}{10})\cap \partial\Omega} 
	\Bigr] ,
\end{align}
where $f_2'$ denotes the derivative of $f_2$ with respect to arc length.
This is similar to the norm $\| \ \|_{\sharp,\Omega}$ \eqref{nH3} but for functions restricted to the boundary of $\Omega$. 

\begin{lemma}
\label{lem:lh2}
There is $\delta_2^*>0$ and $C>0$ so that for  $0<\delta_2<\delta_2^*$ there is $\varepsilon^*>0$ such that for $0<\varepsilon<\varepsilon^*$ the following holds.
Let $f_1$ and $f_2$ satisfy $\|f_1\|_{\sharp\sharp,\tilde\Omega^H}<\infty$,
$\|f_2\|_{\sharp,\partial\tilde\Omega^H}<\infty$. 
Then there is a function $\phi$ with $\|\phi\|_{h,\tilde\Omega^H}<\infty$ solving \eqref{pHairpinD2}, which defines a linear operator of $f_1$, $f_2$ and such that 
\begin{align}
\nonumber
\|\phi\|_{\sharp,\tilde\Omega^H} \leq C ( \|f_1\|_{\sharp\sharp,\tilde\Omega^H} +\|f_2\|_{\sharp,\partial\tilde\Omega^H}). 
\end{align}
\end{lemma}

\begin{proof}
    First solve the problem in the hairpin $\Omega^H$ and then use a perturbation argument as in Proposition~\ref{prop:lh1}. To solve the problem in $\Omega^H$ we use the conformal map $F$ to map the problem onto a strip, as in Lemma~\ref{lemmaH1}, and use a barrier and then elliptic estimates.
\end{proof}

\begin{remark}
\label{rem:u1}
To construct a function $\bar u_1$ as used in Section~\ref{sect:error}, \eqref{pHairpinD3} one can solve, using Lemma~\ref{lem:lh2},
\begin{align}
\nonumber
\left\{
\begin{aligned}
-\Delta_y \bar u_1 &=  \varepsilon^{1-\sigma b}|\log\varepsilon| 
\frac{1}{(1+|y|)^{1+\sigma}}
&&
\text{in } \tilde\Omega^H
\\
\bar u_1 & =0 && \text{on } \partial\tilde\Omega^H.
\end{aligned}
\right.
\end{align}
The restriction of $\bar u_1$ to $\Omega_0$ satisfies \eqref{pHairpinD3} and the inequalities \eqref{eq:u1-upper}, \eqref{eq:u1-lower}. The lower bound \eqref{eq:u1-lower} can be obtained by passing to the strip and using an appropriate subsolution.

Likewise, a function $\bar u_0$ as needed in Section~\ref{sect:error}, satisfying \eqref{eq:bar-u0} and \eqref{ineq-baru1} can be obtained by solving
\begin{align}
\nonumber
\left\{
\begin{aligned}
\Delta \bar u_0 & = 0 && \text{in } \tilde\Omega^H
\\
\bar u_0 &= \varepsilon^{1-b\sigma} |\log\varepsilon| (1+|y|^2)^{\frac{1-\sigma}{2}} && \text{on }\partial\tilde\Omega^H
\end{aligned}
\right.
\end{align}
using Lemma~\ref{lem:lh2}.
The upper bound comes from Lemma~\ref{lem:lh2} while the lower bound is again proved by changing variables to the strip and using a suitable subsolution.
\end{remark}

\section{Modification of the disk and linearization}
\label{sec:TildeOD}

Similarly to what is done in Section~\ref{sec:LinH}, we define here a domain $\tilde\Omega^D$ close to $\Omega^D_\varepsilon$ in \eqref{pOmegaDeps} and study the linear equation \eqref{robin3a} on $\tilde\Omega^D$. 

The main reason to modify $\Omega^D_\varepsilon$ is that $\partial\Omega^D_\varepsilon$ has a singularity. 
But even if we worked with the disk $\Omega^D$ \eqref{OmegaD}, at a certain point (for example Lemma~\ref{lemma:RZHbdyD}) we need to estimate derivatives with respect to $x$ of $Z^H(\frac{x}{\varepsilon})$ on the boundary of the disk, where $Z^H = \partial_{y_2} \psi^H$.
We have by \eqref{D2psi}
\begin{align*}
	\Bigl|\nabla_x \Bigl[ Z^H\Bigl(\frac{x}{\varepsilon}\Bigr)\Bigr]\Bigr| &= \frac{1}{\varepsilon} \Bigl|\nabla_y Z^H\Bigl(\frac{x}{\varepsilon}\Bigr)\Bigr| \leq  C \frac{\varepsilon}{|x|^2+\varepsilon^2}.
\end{align*}
Control of this term depends on a lower bound for $|x|$ for $x\in \partial\Omega^D$.
If we work with the disk $\Omega^D$ \eqref{OmegaD}, since $d^D = O(\varepsilon|\log\varepsilon|)$ we would get a control of the form 
\begin{align*}
	|\nabla_x Z^H| \leq  C \frac{\varepsilon}{|x|^2+\varepsilon^2} \leq C \frac{1}{\varepsilon |\log\varepsilon |^2},
\end{align*}
which is not good enough for our purposes.

In \eqref{pOmegaDeps} we defined $\partial\Omega^D_\varepsilon$ and wrote part of the lower half of the scaled curve $\frac{1}{\varepsilon} \partial\Omega^D_\varepsilon$ as the graph of the function $g_{\varepsilon,D}^-(y_1)$, $1\leq|y_1|\leq\frac{R}{2\varepsilon}$, \eqref{gepsDm}.
Thus, part of the lower boundary $\partial\Omega^D_\varepsilon$ is the graph of the function
\begin{align}
	\label{scaled}
	\varepsilon g_{\varepsilon,D}^-\Bigl(\frac{x_1}{\varepsilon}\Bigr),
	\quad  \varepsilon\leq |x|\leq \frac{R}{2}.
\end{align}
We will define
\begin{align*}
	\tilde g_{\varepsilon,D}^-\colon \Bigl[-\frac{R}{2},\frac{R}{2}\Bigr]\to \R
\end{align*}
close to the scaled function \eqref{scaled} and use it to define $\tilde\Omega^D$.

The construction depends on a small parameter $\delta_1>0$, which is independent of $\varepsilon$. More precisely, we will find $\delta_1^*>0$ so that for all $0<\delta_1<\delta_1^*$ there is $\varepsilon^*(\delta_1^*)>0$ such that for $0<\varepsilon<\varepsilon^*(\delta_1^*)$ the following construction is made, and the results in this section hold.

In the following definitions it is worth to keep in mind the expansions
\begin{align}
	\nonumber
	g_D^-(y_1) &=  \frac{d^D}{\varepsilon} 
	+\frac{\varepsilon \omega y_1^2}{4}+O(\varepsilon^3 y_1^4)
	\\
	\nonumber
	g_{\varepsilon,D}^-(y_1) 
	&= \frac{d^D}{\varepsilon} 
	+\frac{\varepsilon \omega y_1^2}{4}+O(\varepsilon^3 y_1^4) + \log\varepsilon+\log(|y_1|)+O\Bigl(\frac{|\log\varepsilon|}{|y_1|}\Bigr)
	+ O( \varepsilon |\log\varepsilon| \, |y_1| ) ,
\end{align}
valid for $1\ll |y_1|< \varepsilon^{-1} \delta_1$, which follow from \eqref{gD1} and \eqref{gepsDm2}.
Then the effect of the scaling \eqref{scaled} is
\begin{align}
	\label{ex-g-disc}
	\varepsilon g_{D}^-\Bigl(\frac{x_1}{\varepsilon}\Bigr)
	&=  d^D+\frac{1}{4}\omega x_1^2+O(x_1^4)
	\\
	\label{ex-g-disc-eps-mod}
	\varepsilon g_{\varepsilon,D}^-\Bigl(\frac{x_1}{\varepsilon}\Bigr)
	&= d^D+\frac{1}{4}\omega x_1^2+O(x_1^4)
	+ \varepsilon\log(|x_1|)+O\Bigl(\frac{\varepsilon^2|\log\varepsilon|}{|x_1|}\Bigr)
	+ O( \varepsilon |\log\varepsilon| \, |x_1| ) ,
\end{align}
valid for $\varepsilon \ll |x_1|< \delta_1$.
We scale these functions by the parameter $\delta_1$ in such a way as to preserve the quadratic part in $x_1$.
Thus, for $\xi_1\in[-\frac{R}{2}\delta_1,\frac{R}{2}\delta_1]\setminus ( -\frac{\varepsilon}{\delta_1},\frac{\varepsilon}{\delta_1})$ define
\[
g_1(\xi_1) = \frac{\varepsilon}{\delta_1^2} g_{\varepsilon,D}^- \Bigl(\frac{\delta_1 \xi_1}{\varepsilon}\Bigr).
\]
The idea is to modify the function $g_1$ so that for $|\xi_1|\leq 1$ the modification coincides with 
\begin{align*}
	\frac{\varepsilon}{\delta_1^2} g_D^- \Bigl(\frac{\delta_1 \xi_1}{\varepsilon}\Bigr) + \delta_1^2.
\end{align*}
This corresponds to the graph of the lower part of the boundary of $\Omega^D$ shifted vertically, to avoid a bad estimate on $\nabla_x Z^H(\frac{x}{\varepsilon})$. 
We note that we are not shifting the disk $\Omega^D$ (or its perturbed version $\Omega_\varepsilon^D$), but only a piece of its lower boundary.
To do this, we let $\eta_0 \in C^\infty(\R)$, $\eta_0\geq0$ be even and such that $\eta_0(s)=1$ for $|s|\leq 1$, $\eta_0(s)=0$ for $|s|\geq 2$ (as in \eqref{eta0}).
Next we define
\begin{align*}
	\tilde g_1(\xi_1) = g_1(\xi_1) ( 1 - \eta_0(\xi_1)) + 
	\Bigl(  \frac{\varepsilon}{\delta_1^2} g_D^- \Bigl(\frac{\delta_1 \xi_1}{\varepsilon}\Bigr) + \delta_1^2 \Bigr) \eta_0(\xi_1).
\end{align*}
Then $\tilde g_1\colon [-\frac{R}{2}\delta_1,\frac{R}{2}\delta_1]\to \R$ is a smooth even function such that
\begin{align}
	\label{tildeg1}
	\left\{
	\begin{aligned}
		& \tilde g_1(\xi_1) = g_1(\xi_1) ,\quad|\xi_1|\geq[-2,2],
		\\
		& \tilde g_1(\xi_1) \geq \delta_1^2  ,\quad \xi_1 \in [-1,1].
	\end{aligned}
	\right.
\end{align}

\begin{lemma}
	We have 
	\begin{align}
		\label{tildeg2}
		& \| \tilde g_1 - g_1 \|_{C^4([-2,2]) } \leq C \delta_1^2 + C \frac{\varepsilon|\log\varepsilon|}{\delta_1^2} .
	\end{align}
\end{lemma}
\begin{proof}
	The proof is by direct computation, using the expansions \eqref{ex-g-disc} and \eqref{ex-g-disc-eps-mod} and similar ones for the derivatives.  The biggest term comes from the difference
	\begin{align*}
		g_1(\xi_1)-\Bigl[\frac{\varepsilon}{\delta_1^2} g_D^- \Bigl(\frac{\delta_1 \xi_1}{\varepsilon}\Bigr)+\delta_1^2\Bigr]
		&=
		\frac{\varepsilon}{\delta_1^2} g_{\varepsilon,D}^- \Bigl(\frac{\delta_1 \xi_1}{\varepsilon}\Bigr)
		-\Bigl[\frac{\varepsilon}{\delta_1^2} g_D^- \Bigl(\frac{\delta_1 \xi_1}{\varepsilon}\Bigr)+\delta_1^2\Bigr],
	\end{align*}
	which, by \eqref{ex-g-disc} and \eqref{ex-g-disc-eps-mod}, for $|\xi_1|\in[1,2]$ can be estimated by 
	\begin{align*}
		\left|g_1(\xi_1)-\Bigl[\frac{\varepsilon}{\delta_1^2} g_D^- \Bigl(\frac{\delta_1 \xi_1}{\varepsilon}\Bigr)+\delta_1^2\Bigr]\right|
		&\leq C \delta_1^2 + C\frac{\varepsilon |\log(\delta_1\xi_1)|}{\delta_1^2}
		+C\frac{\varepsilon^2|\log\varepsilon|}{\delta_1|\xi_1|}
		\\
		&\quad
		+C\frac{\varepsilon}{\delta_1} |\log\varepsilon| \, |\xi_1| 
		\\
		&\leq 
		C \delta_1^2 + C\frac{\varepsilon |\log \varepsilon|}{\delta_1^2}.
        \qedhere
	\end{align*}
\end{proof}
Define
\begin{align}
	\label{tgDm}
	\tilde g_D^-(x_1) = \delta_1^2 \tilde g_1\Bigl(\frac{x_1}{\delta_1}\Bigr) 
\end{align}
and the domain $\tilde\Omega^D$ by
\begin{align}
	\label{tildeOD}
	\left\{
	\begin{aligned}
		&\text{in $|x_1|>\frac{R}{2}$, $\tilde\Omega^D$ agrees with $\Omega_\varepsilon^D$}
		\\
		&\text{in $|x_1|\leq\frac{R}{2}$, $\tilde\Omega^D$ is given by}
		\\
		&\quad  \Bigl\{ \, (x_1,x_2) \, \Big|  \, x_1 \in \Bigl[-\frac{R}{2},\frac{R}{2}\Bigr], \ \tilde g_D^-(x_1)<x_2<g_D^+(x_1) \, \Bigr\} 
        .
	\end{aligned}
	\right.
\end{align}
\index{OD@$\tilde \Omega^D$, a domain close to a disk}
We also define the point $P_D$ as the lowest point in $\tilde\Omega^D$
\begin{align*}
	P_D = (0,\tilde g_D^-(0)).
\end{align*}

\begin{figure}
	\centering
	\includegraphics[scale=1,page=1]{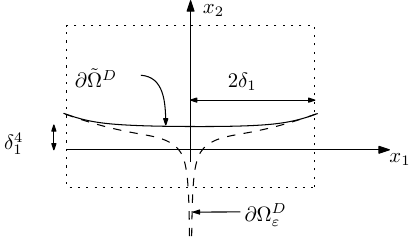}
	\caption{Sketch of $\partial\Omega^D$. }
	\label{fig:OmegaD}
\end{figure}

Figure~\ref{fig:OmegaD} shows a sketch of $\partial\tilde\Omega^D$ near the origin. Outside the dotted box with sidelength $4\delta_1$, $\partial\tilde\Omega^D$ coincides with $\partial\Omega^D_\varepsilon$. In particular, outside the dotted box, $\partial\tilde\Omega^D$ is an $\varepsilon$-perturbation of the circle $\partial\Omega^D$ of radius $R$ centered at $(0,R+d^D)$. Inside the box, $\partial\tilde\Omega^D$ is a smooth graph, close to the circle $\partial\Omega^D$, and a distance $\delta_1^4$ above the $x_1$ axis. This comes from the inequality $\tilde g_1(\xi_1) \geq \delta_1^2$ in \eqref{tildeg1} and the scaling \eqref{tgDm}.

Here we present some properties needed in other sections regarding the domain $\tilde\Omega^D$ defined in \eqref{tildeOD}.
Let $\kappa_{\partial\Omega^D}$ and $\kappa_{\partial\tilde \Omega^D}$  denote the curvatures of $\partial\Omega^D$ and $\partial\tilde\Omega^D$ respectively. 
Then we have the following properties.

\begin{lemma}
	Let us write $\tilde \kappa$ for the function  $\kappa_{\partial\tilde \Omega^D}$ but measured on $\partial\Omega^D$ by the transformation
	\begin{align*}
		\tilde\kappa(x_1,x_2) = \kappa_{\partial\tilde \Omega^D}(x_1,x_2- g_D^-(x_1) + \tilde g_D^-(x_1) ), \quad (x_1,x_2) \in \partial\Omega^D.
	\end{align*}
	Then
	\begin{align*}
		\| \kappa_{\partial\Omega^D} - \tilde \kappa\|_{C^{1,\alpha},\partial\Omega^D} \leq C \delta_1^{1-\alpha}+C\frac{\varepsilon|\log\varepsilon|}{\delta_1^{3+\alpha}}.
	\end{align*}
	Moreover, there is $c>0$ such that 
	\begin{align}
		\label{tOD0}
		|x|\geq c \delta_1^4 , \quad \text{for all } x\in \partial\tilde\Omega^D.
	\end{align}
\end{lemma}
\begin{proof}
	This follows directly from \eqref{tildeg1} and \eqref{tildeg2}.
\end{proof}

Now we consider the linearized problem in $\tilde\Omega^D$:
\begin{align}
	\label{pDisc0}
	\left\{
	\begin{aligned}
		\Delta \phi&= f_1 
		&&\text{in }  \tilde \Omega^D
		\\
		\frac{\partial \phi}{\partial \nu}
		+ (\kappa-\omega) \phi
		&= f_2
		&&\text{on }  \partial \tilde \Omega^D .
	\end{aligned}
	\right.
\end{align}
For this problem we work with standard H\"older norms. 
Let $0<\alpha<1$ and define
\begin{align}
    \label{nD1}
	\|f_1\|_{C^\alpha,\tilde\Omega^D}
	&= 
	\|f_1\|_{L^\infty(\tilde \Omega^D)}
	+
	[f_1]_{\alpha,\tilde\Omega^D}
\index{$\norm \ \norm_{C^\alpha,\tilde\Omega^D}$}
\intertext{and}
    \label{nD2}
	\| f_2 \|_{C^{1,\alpha},\partial\tilde\Omega^D}
	&= 
	\|f_2\|_{L^\infty(\partial\tilde \Omega^D)}
	+\|f_2'\|_{L^\infty(\partial\tilde \Omega^D)}
	+
	[f_2']_{\alpha,\partial\tilde\Omega^D}
    \\
    \index{$\norm \ \norm_{C^{1,\alpha},\partial\tilde\Omega^D}$}
    \label{nD3}
	\| \phi \|_{C^{2,\alpha},\tilde\Omega^D}
	&= 
	\|\phi\|_{L^\infty(\tilde \Omega^D)}
	+\|\nabla\phi\|_{L^\infty(\tilde \Omega^D)}
	+\|D^2\phi\|_{L^\infty(\tilde \Omega^D)}
	+  [D^2 \phi]_{\alpha,\tilde\Omega^D}.
\end{align}
\index{$\norm \ \norm_{C^{2,\alpha},\tilde\Omega^D}$}

Since $\tilde\Omega^D$ is close to $\Omega^D$ and problem \eqref{pDisc0} in $\Omega^D$ has a non-trivial kernel, some orthogonality conditions are needed to get existence and good estimates.
We work with functions that are even in $x_1$ so that the orthogonality condition with respect to horizontal translations is automatically satisfied. Then to get existence and a good estimate for the solution of \eqref{pDisc0} we project the right hand side.
This leads us to consider the modified equation
\begin{align}
	\label{pDiscProj}
	\left\{
	\begin{aligned}
		\Delta \phi & = f_1(x)
		&& \text{in }  \tilde \Omega^D
		\\
		\frac{\partial \phi}{\partial \nu}
		+ (\kappa-\omega)   \phi
		&= f_2(x) + g (x_2+d^S)
		&&\text{on }  \partial \tilde \Omega^D .
	\end{aligned}
	\right.
\end{align}
We include the shift $d^S$ in the boundary condition, because this appears in the formulation \eqref{main-problem0}, see \eqref{def:E1} and \eqref{E}. 

\begin{proposition}
	\label{propDisc1}
	There are $\delta_1^*>0$, $C>0$ so that for $0<\delta_1<\delta_1^*$ there is $\varepsilon^*>0$ so that for $0<\varepsilon<\varepsilon^*$ the following holds.
	Let $f_1\in C^\alpha(\overline{\tilde\Omega^D})$ and $f_2\in C^{1,\alpha}(\partial\tilde\Omega^D)$ be even in $x_1$.
	Then there is  $\phi \in C^{2,\alpha}(\overline{\tilde \Omega^D})$, even in $x_1$, and $g \in \R$ that solve \eqref{pDiscProj}. Moreover they are linear functions of $f_1$, $f_2$ and satisfy
	\begin{align*}
		\| \phi \|_{C^{2,\alpha},\tilde\Omega^D} + |g| \leq C ( \|f_1\|_{C^\alpha,\tilde\Omega^D}+ \|f_2\|_{C^{1,\alpha},\partial\tilde\Omega^D} ).
	\end{align*}
\end{proposition}
\begin{proof}
	We shift vertically by the amount $R+d^D$ and so problem \eqref{pDiscProj} gets transformed into
	\begin{align}
		\label{pDiscProj2}
		\left\{
		\begin{aligned}
			\Delta \phi & = f_1(x)
			&& \text{in }  (\tilde \Omega^D-(R+d^S)e_2)
			\\
			\frac{\partial \phi}{\partial \nu}
			+ (\kappa-\omega)   \phi
			&= f_2(x) + g (x_2+R+d^D+d^S)
			&&\text{on }  \partial (\tilde \Omega^D-(R+d^S)e_2) ,
		\end{aligned}
		\right.
	\end{align}
	where $e_2=(0,1)$.
	
	For $\delta_1>0$, $\varepsilon>0$ small, we can write 
	\[
	\partial\tilde\Omega^D = \{x + h_{\delta_1,\varepsilon}(x) \nu_{\Omega^D}(x) \ | \ x\in\partial\Omega^D\},
	\]
	for some smooth function $h_{\delta_1,\varepsilon}\colon \partial\Omega^D\to\R$, such that $h_{\delta_1,\varepsilon}\to0$ as $\delta_1\to0$, $\varepsilon\to 0$ in $C^3$. Here and in the rest of this proof, when we say ``as $\delta_1\to0$, $\varepsilon\to 0$'', this means that the statement is valid for all $\delta_1>0$ small and for all $0<\varepsilon<\varepsilon^*(\delta_1)$ where $\varepsilon^*(\delta_1)>0$.
	Then the function
	\begin{align*}
		F_{\delta_1,\varepsilon}(x)=x-h\Bigl(\frac{x}{|x|}\Bigr) \eta_0(4(1-|x|)) \frac{x}{|x|}
	\end{align*}
	maps $B_R(0)$ bijectively into $(\tilde \Omega^D-(R+d^S)e_2)$. 
	The cut-off $\eta_0$ is the same as in \eqref{eta0}.
	Moreover $F_{\delta_1,\varepsilon}$ is smooth and tends to the identity in $C^3$ as $\delta_1\to0$, $\varepsilon\to 0$.
	
	Let $\tilde\phi(x)=\phi(F(x))$, $x\in B_R(0)$. Then \eqref{pDiscProj2} becomes
	\begin{align}
		\label{pBallProj2}
		\left\{
		\begin{aligned}
			L_{\delta_1,\varepsilon} \tilde\phi & = \tilde f_1(x)
			&& \text{in }  B_R(0)
			\\
			B_{\delta_1,\varepsilon}\tilde\phi
			+ (\tilde\kappa_{\delta_1,\varepsilon}-\omega)\tilde\phi
			&= \tilde f_2(x) + g \tilde Z_{\delta_1,\varepsilon}(x)
			&&\text{on }  \partial B_R(0),
		\end{aligned}
		\right.
	\end{align}
	where $L_{\delta_1,\varepsilon}$ is a second order operator and $B$ is a first order operator such that $L_{\delta_1,\varepsilon}\to\Delta$ and $B_{\delta_1,\varepsilon}\to\nu\cdot\nabla$ as $\delta_1\to0$, $\varepsilon\to0$. Here $\tilde Z_{\delta_1,\varepsilon}(x)$ tends to $x_2+R$ as $\delta_1\to0$, $\varepsilon\to0$.
	
	Consider
	\begin{align}
		\label{pBall0}
		\left\{
		\begin{aligned}
			\Delta \phi & = f_1(x)
			&& \text{in } B_R(0)\\
			\frac{\partial \phi}{\partial \nu}
			+ (\kappa-\omega)   \phi
			&= f_2(x) + g (x_2+R)
			&&\text{on }  \partial B_R(0),
		\end{aligned}
		\right.
	\end{align}
	Using Fourier series and standard elliptic estimates one can prove:
	Let $f_1\in C^\alpha(\overline{B_R(0)})$ and $f_2\in C^{1,\alpha}(\partial B_R(0))$ be even in $x_1$.
	Then there is  $\phi \in C^{2,\alpha}(\overline{B_R(0)})$, even in $x_1$, and $g \in \R$ that solve \eqref{pBall0}. Moreover they are linear functions of $f_1$, $f_2$ and satisfy
\begin{align*}
\| \phi \|_{C^{2,\alpha},B_R(0)} + |g| \leq C ( \|f_1\|_{C^\alpha,B_R(0)}+ \|f_2\|_{C^{1,\alpha},\partial B_R(0)} ).
\end{align*}
	By treating \eqref{pBallProj2} as a perturbation of \eqref{pBall0} we see that given $\tilde f_1\in C^\alpha(\overline{B_R(0)})$ and $\tilde f_2\in C^{1,\alpha}(\partial B_R(0))$ even in $x_1$, there is  $\tilde\phi \in C^{2,\alpha}(\overline{B_R(0)})$, even in $x_1$, and $g \in \R$ that solve \eqref{pBallProj2}. From here we obtain the result in the statement of the proposition.
\end{proof}

We also need the following estimate.

\begin{lemma}
	\label{lemmaNormZD}
	Let $Z_2^D$ be the function defined in \eqref{Z2D}. Then
	\begin{align}
		\label{errorZD}
		\Bigl\|  \frac{\partial Z_2^D}{\partial  \nu_{\partial\tilde\Omega^D}} +  (\kappa_{\partial\tilde\Omega^D}-\omega) Z_2^D \Bigr\|_{C^{1,\alpha},\partial\tilde\Omega^D} & \leq C \delta_1^{1-\alpha}+C\frac{\varepsilon|\log\varepsilon|}{\delta_1^{3+\alpha}}
		\\
		\label{X2}
		\| x_2 \|_{C^{1,\alpha}(\partial\tilde\Omega^D) } & \leq C +C\delta_1^{3-\alpha}+C\frac{\varepsilon|\log\varepsilon|}{\delta_1^{1+\alpha}}.
	\end{align}
\end{lemma}
\begin{proof}
	The function $Z_2^D$ \eqref{Z2D} is defined and harmonic in $\R^2$.
	The idea is that $Z_2^D$ satisfies
	\begin{align}
		\label{c1}
		\frac{\partial Z_2^D}{\partial \nu_{\partial\Omega^D}} +  (\kappa_{\partial\Omega^D}-\omega) Z_2^D=0
		\quad \text{on } \partial\Omega^D, 
	\end{align}
	and since $\partial\Omega^D$ is close to $\partial\tilde\Omega^D$ then $ \frac{\partial Z_2^D}{\partial \nu_{\partial\tilde\Omega^D}} +  (\kappa_{\partial\tilde\Omega^D}-\omega) Z_2^D$ should be small on $\partial\tilde\Omega^D$. To make this precise we parametrize part of the lower half of $\partial\tilde\Omega^D$ with
	\begin{align*}
		x_1\mapsto (x_1,\tilde g_D^-(x_1))
	\end{align*}
	and compute
	\begin{align*}
		\frac{\partial Z_2^D}{\partial \nu_{\partial\tilde\Omega^D}} (x_1,\tilde g_D^-(x_1))&=
		-\frac{1}{\sqrt{1+((\tilde g_D^-)'(x_1))^2}}
        \\
		\frac{\partial Z_2^D}{\partial \nu_{\partial\Omega^D}} (x_1,g_D^-(x_1))&=
		-\frac{1}{\sqrt{1+((g_D^-)'(x_1))^2}} .
	\end{align*}
	From this and \eqref{tildeg2} we get
    \begin{align}
        \label{c2}
        \left\|\frac{\partial Z_2^D}{\partial \nu_{\partial\tilde\Omega^D}} (\cdot,\tilde g_D^-(\cdot))
        -\frac{\partial Z_2^D}{\partial \nu_{\partial\Omega^D}} (\cdot,g_D^-(\cdot))\right\|_{C^{1,\alpha}([-1/2,1/2])} \leq C \delta_1^{2-\alpha}+C \frac{\varepsilon|\log\varepsilon|}{\delta_1^{2+\alpha}}.
    \end{align}
    
    Similarly, 
    \begin{align*}
        [(\kappa_{\partial\Omega^D}-\omega)Z_2^D ](x_1,g_D^{-}(x_1))=\Bigl( \frac{{g_D^{-}}''(x_1)}{(1+({g_D^{-}}')^2(x_1) )^{3/2}} - \omega\Bigr)(g_D^{-}(x_1)-R-d^D).
    \end{align*}
    Using this formula, a similar one for the corresponding expression involving $\partial\tilde\Omega^D$, and \eqref{tildeg2} we find that
    \begin{align}
        \nonumber
        \bigl\| &(\kappa_{\partial\tilde\Omega^D}-\omega)Z_2^D ](x_1,\tilde g_D^{-}(x_1)) -(\kappa_{\partial\Omega^D}-\omega)Z_2^D ](x_1,g_D^{-}(x_1)) \bigr\|_{C^{1,\alpha}([-1/2,1/2])} 
        \\
        \label{c3}
        & \quad \leq C \delta_1^{1-\alpha}+C\frac{\varepsilon|\log\varepsilon|}{\delta_1^{3+\alpha}}.
    \end{align}
    Combining \eqref{c1}, \eqref{c2} and \eqref{c3} we obtain the stated estimate \eqref{errorZD}.
    The proof of \eqref{X2} is similar.
\end{proof}
			
We will also need another estimate of the same type.
Let
\begin{align}
    \label{chiH1}
    \chi^H(x) = \chi_0\Bigl( \frac{x}{\delta_2} \Bigl),\quad x\in\R^2,
\end{align}
where $\delta>0$ and $\chi_0$ is a radial cut-off function in $\R^2$ such that 
\[
\chi_0(x) = 1 \quad |x| \leq 1,
\quad
\chi_0(x) = 0 \quad |x| \geq 2.
\]
We assume that $0<\delta_1<\delta_2$.
			
			\begin{lemma}
				\label{lemma:RZHbdyD}
				Given $0<\delta_1<\delta_2$, for $\varepsilon>0$ sufficiently small we have
				\begin{align}
					\label{BopZchiH}
					\Bigl\| \frac{\partial (Z^H (\frac{\cdot}{\varepsilon})\chi^H)}{\partial \nu} +  (\kappa-\omega) Z^H (\frac{\cdot}{\varepsilon})\chi^H\Bigr\|_{C^{1, \alpha},\tilde\Omega^D}  \leq C \delta_2^{-(2+\alpha)}
				\end{align}
				and
				\begin{align}
					\label{LZchiH}
					\Bigl\| \Delta\Bigl (Z^H\Bigl(\frac{\cdot}{\varepsilon}\Bigr) \chi^H \Bigr) \Bigr\|_{C^\alpha,\tilde\Omega^D} 
					\leq C \delta_2^{-(2+\alpha)}.
				\end{align}
			\end{lemma}
			\begin{proof}
				Write 
				\begin{align*}
					\frac{\partial (Z^H \chi^H)}{\partial \nu} +  (\kappa-\omega) Z^H \chi^H
					&=  Z^H  \frac{\partial  \chi^H}{\partial \nu}
					+    \frac{\partial Z^H }{\partial \nu} \chi^H
					+    (\kappa-\omega) Z^H  \chi^H .
				\end{align*}
				
				Using that $Z^H = \partial_{y_2} \psi^H$ and  \eqref{D2psi}
				we find
				\begin{align*}
					|\nabla_x Z^H| &= \frac{1}{\varepsilon} |\nabla_y Z^H|  \leq  \frac{C}{\varepsilon |y|^2} 
					\leq  C \frac{\varepsilon}{|x|^2} 
				\end{align*}
				since $y = \frac{x}{\varepsilon}$.
				In a similar way we get for $j=2,3$:
				\begin{align*}
					|D^j_x Z^H(x)| \leq C \frac{\varepsilon}{|x|^{1+j}}, \quad x \in \partial\tilde\Omega^D.
				\end{align*}
				But for $x\in \partial\tilde\Omega^D$ we have $|x|\geq c \delta_1^4$ for some $c>0$, by \eqref{tOD0}, so
				\begin{align*}
					|D_x^j Z^H(x)| \leq C \frac{\varepsilon}{\delta_1^{4(1+j)}}, \quad x \in \partial\tilde\Omega^D.
				\end{align*}
				From this and estimates for the derivative of $\chi^H$ we obtain
				\begin{align}
					\label{bx1}
					\Bigl\| \frac{ \partial \chi^H}{\partial \nu} Z^H\Bigr\|_{C^{1,\alpha},\tilde\Omega^D} 
					&\leq  C \frac{\varepsilon}{\delta_1^{4(2+\alpha)} \delta_2} 
					+\frac{C}{\delta_2^{2+\alpha}} .
				\end{align}
				A similar argument gives 
				\begin{align}
					\label{bx2}
					\Bigl\| \frac{\partial Z^H }{\partial \nu}  \chi^H \Bigr\|_{C^{1,\alpha},\tilde\Omega^D} 
					\leq C \frac{\varepsilon}{\delta_1^{4(3+\alpha)}} + C\frac{\varepsilon}{\delta_1^8\delta_2^{1+\alpha}} .
				\end{align}
				
				To estimate $(\kappa-\omega) Z^H \chi^H$ let us recall that $\kappa=\kappa_{\partial\tilde\Omega^D}$ is the curvature of $\partial\tilde\Omega^D$, which is parametrized by $\tilde g_D^-$ \eqref{tgDm}. Using \eqref{tildeg1}, by a similar computation to \eqref{c3}, we see that
				\begin{align*}
					\|\kappa\|_{C^{1,\alpha},\tilde\Omega^D} \leq C+C \delta_1^{1-\alpha}+C\frac{\varepsilon|\log\varepsilon|}{\delta_1^{3+\alpha}}.
				\end{align*}
				Using this and estimates for the derivatives of $ Z^H \chi^H$ as before we obtain
				\begin{align}
					\label{bx3}
					\| (\kappa-\omega) Z^H \chi^H \|_{C^{1,\alpha},\tilde\Omega^D} 
					& \leq 
					C+C \delta_1^{1-\alpha}+C\frac{\varepsilon|\log\varepsilon|}{\delta_1^{3+\alpha}}
					+C\frac{\varepsilon}{\delta_1^{4(1+\alpha)}}
					+C\frac{1}{\delta_2^{1+\alpha}}.
				\end{align}
				Combining \eqref{bx1}, \eqref{bx2} and \eqref{bx3}, we see that given $\delta_2>0$ and $\delta_1>0$ then for $\varepsilon>0$ sufficiently small we get \eqref{BopZchiH}.
				
				\medskip
				Let us consider $ \Delta(Z^H(\frac{\cdot}{\varepsilon}) \chi^H)$. Since $Z^H$ is harmonic we have $\Delta(Z^H  \chi^H ) = 2 \nabla Z^H \cdot \nabla \chi^H + Z^H \Delta \chi^H$.
				But the following estimates hold
				\begin{align*}
					\Bigl| Z^H \Bigl(\frac{x}{\varepsilon}\Bigr)\Delta \chi^H(x) \Bigr| \leq 
					C \delta_2^{-2}
				\end{align*}
				\begin{align*}
					\left[ Z^H \Bigl(\frac{\cdot}{\varepsilon}\Bigr)\Delta \chi^H\right]_{\alpha,\tilde\Omega^D} \leq C  \delta_2^{-(2+\alpha) }.
				\end{align*}
				This and similar estimates for the term $ \nabla Z^H \cdot \nabla \chi^H$ give \eqref{LZchiH}.
			\end{proof}

   	\section{Modification and linear theory on the strip}
			\label{sect:S}
			
			Similarly to the modifications of the hairpin and disk, we need to modify the perturbed strip $\Omega^S_\varepsilon$ defined in \eqref{pOmegaSeps} to avoid the singularity.
			By \eqref{Seps} and \eqref{h1S}, $\mathcal{S}^S_\varepsilon$ is the graph of the function
			\[
			g_S(x_1)= -d^S+\varepsilon \Psi_1^S(x_1,-d^S) 
			\]
			where $\Psi_1^S$ is the function constructed in Lemma~\ref{lem:sol-strip2}.
			Let $\eta_0$ be the cut-off in \eqref{eta0}.
			We define 
			\begin{align}
                \notag
    			\tilde g_S(x_1)&= \Bigl(1-\eta_0\Bigl(\frac{x_1}{\delta_1}\Bigr)\Bigr) g_S(x_1) + \eta_0\Bigl(\frac{x_1}{\delta_1}\Bigr) g_S(\delta_1)
                \\
				\label{tOmegaS}
				\tilde\Omega^S &= \{ \ (x_1,x_2) \ | \ -\A-d^S<x_2<\tilde g_S(x_1) \ \}.
			\end{align}
			The boundary of $\tilde\Omega^S$ consists of 
			\begin{align*}
				\tilde{\mathcal{S}}=\{\ x_2=\tilde g_S(x_1) \ \},
				\quad
				\mathcal{B} =\{\ x_2=-\A-d^S\ \}.
			\end{align*}
We define $P_S$ as 
\begin{align}
\label{def:PS}
P_S=(0,\tilde g_S(0)).
\end{align}

			Next we consider problem 
			\begin{align}
				\label{pStrip0}
				\left\{
				\begin{aligned}
					\Delta \phi &= f_1 
					&&\text{in } \tilde\Omega^S
					\\
					\frac{\partial\phi}{\partial \nu} -\omega\phi
					&= f_2 
					&&\text{on }\tilde{\mathcal{S}}
					\\
					\phi&= 0 &&\text{on }  \mathcal{B}
				\end{aligned}
				\right.
			\end{align}
			
Let $\mu>0$, $0<\alpha<1$.
We define
\begin{align}
\label{nS1}
\|f_1\|_{C^\alpha,\tilde\Omega^S}
&= 
\sup_{x\in\tilde\Omega^S}  \cosh(\mu x_1) 
\left( |f_1(x)| + [f_1]_{\alpha,B(x,1)\cap\tilde\Omega^S} \right)
\\
\label{nS2}
\|f_2\|_{C^{1,\alpha},\partial\tilde\Omega^S}
&= 
\sup_{x\in \partial\tilde\Omega^S}  \cosh(\mu x_1) 
\left( |f_2(x)| +|\nabla f_2(x)| + [\nabla f_2]_{\alpha,B(x,1) \cap \partial\tilde\Omega^S} \right)
\\
\label{nS3}
\| \phi \|_{C^{2,\alpha},\tilde\Omega^S}
&= \sup_{x\in\tilde\Omega^S}  \cosh(\mu x_1) 
\left( |\phi(x)| +|\nabla \phi(x)|+|D^2\phi(x)|  +  [D^2 \phi]_{\alpha,B(x,1)\cap\tilde\Omega^S} \right).
\end{align}
			
\begin{proposition}
\label{propS}
Assume that $\omega \Ae<1$.
Then for $\mu>0$ small as in \eqref{eq:mu}  (depending on $1-\omega \Ae$),
there are $\delta_1^*>0$, $C>0$ so that for $0<\delta_1<\delta_1^*$ there is $\varepsilon^*>0$ so that for $0<\varepsilon<\varepsilon^*$ the following holds. For $\|f_1\|_{C^\alpha,\tilde\Omega^S}<\infty$ and  $\|f_2\|_{C^{1,\alpha},\tilde\Omega^S}<\infty$ there is a unique solution $\phi$ of \eqref{pStrip0} with $\|\phi \|_{C^{2,\alpha},\tilde\Omega^S}<\infty$,
and it satisfies
\begin{align*}
\| \phi \|_{C^{2,\alpha},\tilde\Omega^S} \leq C ( \|f_1\|_{C^\alpha,\tilde\Omega^S}+ \|f_2\|_{C^{1,\alpha},\partial\tilde\Omega^S} ).
\end{align*}
\end{proposition}
The proof is by perturbation of the same result in the strip $\Omega^S$ \eqref{strip1}.
For the strip $\Omega^S$ the result follows by using standard H\"older estimates and the supersolution \eqref{superstrip}.

\section{The linearized equation in \texorpdfstring{$\Omega_0$}{Omega0}}
\label{sec:linear}
			
			We consider here the linear problem
			\begin{align}
				\label{main-linear}
				\left\{
				\begin{aligned}
					\Delta_y \phi  &= f_1  && \text{in } \Omega_0
					\\
					\frac{\partial \phi}{\partial\nu}
					+ (\kappa-\varepsilon\omega) \phi &= f_{2,0} + g ( \varepsilon y_2+d^S)
					&& \text{on } \mathcal{S}_0
					\\
					\phi &= 0 && \text{on } \mathcal{B}_0,
				\end{aligned}
				\right.
			\end{align}
			for $\varepsilon>0$ sufficiently small, where $\Omega_0$ is the domain constructed in Section~\ref{sec:approx} and $\mathcal{S}_0$, $\mathcal{B}_0$ are the upper and lower components of $\partial\Omega_0$.
			We recall that $\mathcal{S}_0$ is asymptotic to $y_2=-\varepsilon^{-1}d^S$ as $|y|\to \infty$, where $d^S$ is a parameter of size $O(\varepsilon|\log\varepsilon|)$, see \eqref{dSeps}.
			In \eqref{main-linear}, $\kappa$ is the curvature of $\partial\Omega_0$ with the sign convention of Section~\ref{sectFormal} and $\omega>0$ is a constant.
			At some points later on we will use the notation
			\begin{align}
				\label{decompf2}
				f_2(y) = f_{2,0}(y)+g ( \varepsilon y_2+d^S).
			\end{align}
			
			The domain $\Omega_0$ defined in Section~\ref{sec:approx} is dilated, in the sense that the part that corresponds to the disk has radius $O(\frac{1}{\varepsilon})$ and the neck of the hairpin has size $O(1)$. Problem \eqref{main-linear} can be transformed, by the change 
			\begin{align}
				\label{tf}
				\tilde \phi(x) &= \varepsilon \phi\Bigl(\frac{x}{\varepsilon}\Bigr), \quad 
				\tilde f_1(x) =\frac{1}{\varepsilon} f_1\Bigl(\frac{x}{\varepsilon}\Bigr), \quad 
				\tilde f_2(x) = f_2\Bigl(\frac{x}{\varepsilon}\Bigr), 
				\\
				\label{tkappa}
				\tilde \kappa(x) &= \frac{1}{\varepsilon}\kappa\Bigl(\frac{x}{\varepsilon}\Bigr)\end{align}
			into the problem in the original scale
			\begin{align}
				\nonumber
				\left\{
				\begin{aligned}
					\Delta_x \tilde\phi  &= \tilde f_1  && \text{in } \varepsilon\Omega_0
					\\
					\frac{\partial \tilde\phi}{\partial\nu}
					+ (\tilde \kappa-\omega) \tilde\phi &= \tilde f_{2,0} + g (x_2+	d^S)
					&& \text{on } \varepsilon\mathcal{S}_0
					\\
					\tilde\phi &= 0 && \text{on } \varepsilon\mathcal{B}_0.
				\end{aligned}
				\right.
			\end{align}
			
For $f_1$ and $f_2$ in \eqref{main-linear} we consider the norms defined in \eqref{defNf1Omega0} and \eqref{defNf2Omega0}.
Let $0<\sigma,\alpha<1$, $\delta_1>0$ be small, and $\mu>0$ satisfy \eqref{eq:mu}.
For the solution $\phi$ of \eqref{main-linear} we define the norm
			\begin{align}
				\label{norm-phi}
				\| \phi \|_{*,\Omega_0} & = \text{the least $M$ such that for $y \in \Omega_0$}
				\\
				\nonumber
				& |\phi(y)|
				+(1+|y|) |\nabla \phi(y)|
				+(1+|y|)^2 |D^2 \phi(y)|
				+(1+|y|)^{2+\alpha} [D^2\phi]_{\alpha,B(y,\frac{|y|}{10})}
				\\
				\nonumber
				& \qquad \qquad  \qquad 
				\leq (1+|y|)^{1-\sigma} M 
				\qquad \qquad\text{for } |y| \leq \frac{\delta_1}{\varepsilon}
				\\
				\nonumber
				& 
				|\phi(y)| 
				+\Bigl(\frac{\delta_1}{\varepsilon}\Bigr) |\nabla\phi(y)|
				+\Bigl(\frac{\delta_1}{\varepsilon}\Bigr)^2 |D^2\phi(y)|
				+\Bigl(\frac{\delta_1}{\varepsilon}\Bigr)^{2+\alpha} [D^2\phi]_{\alpha,B(y,\frac{\delta_1}{10 \varepsilon})}
				\\
				\nonumber
				& \quad \qquad \qquad 
				\leq \Bigl(\frac{\delta_1}{\varepsilon}\Bigr)^{1-\sigma} \exp(-\mu \varepsilon |y|) M \qquad\qquad   \text{for } |y| \geq \frac{\delta_1}{\varepsilon} .
			\end{align}
			\index{$\norm \  \norm_{*,\Omega_0}$, norm for the solution of the linearized problem in $\Omega_0$}
			
			The main result of this section is the following.
			
			\begin{proposition}
				\label{prop:main-linear}
				There is $C>0$ and $\varepsilon_0>0$ so that for $0<\varepsilon<\varepsilon_0$  for any $f_1$ and $f_{2,0}$ with $\|f_1\|_{**, \Omega_0}<\infty$ and $\|f_{2,0}\|_{**, \partial \Omega_0} <\infty$ there exist $g$ and $\phi$ that satisfy \eqref{main-linear}, are linear functions of $f_1$, $f_{2,0}$ and satisfy the estimate
				\begin{align}
					\nonumber
					\|\phi\|_{*,\Omega_0} &\leq C \varepsilon^{\sigma-1}( \|f_1\|_{**, \Omega_0}+\|f_{2,0}\|_{**, \partial \Omega_0} )
					\\
					\nonumber
					|g|&\leq C \varepsilon^{\sigma}( \|f_1\|_{**, \Omega_0}+\|f_{2,0}\|_{**, \partial \Omega_0} ).
				\end{align}
			The norms $\| \cdot \|_{**, \Omega_0}$ and $\| \cdot \|_{**, \partial \Omega_0}$ are defined in \eqref{defNf1Omega0} and \eqref{defNf2Omega0} respectively.
			\end{proposition}
			
			The strategy of the proof, postponed to Section~\ref{sectProofSys}, is to decompose the solution $\phi$ into three functions that roughly will solve a linear coupled system in domains close to the hairpin, the disk and the strip.
			Let $0<\delta_1<\delta_2$ be small numbers to be chosen later on. They will be fixed independently of $\varepsilon$ and we will need $\varepsilon\ll\delta_1\ll\delta_2$.
			Let $\chi_0$ be a radial cut-off function in $\R^2$ such that 
			\[
			\chi_0(x) = 1 \quad |x| \leq 1,
			\quad
			\chi_0(x) = 0 \quad |x| \geq 2.
			\]
			Define
			\begin{align}
				\label{cutoffs}
				\left\{
				\begin{aligned}
					\chi^H(x) &= \chi_0\Bigl( \frac{x}{\delta_2} \Bigl)    
					\\
					\chi^{S}(x) &= \Bigl[ 1-\chi_0\Bigl( \frac{x}{\delta_1} \Bigl)\Bigr]
					\eta_1^-\Bigl( \frac{x}{\varepsilon} \Bigl)
					\\
					\chi^{D}(x) &=\Bigl[ 1-\chi_0\Bigl( \frac{x}{\delta_1} \Bigl)\Bigr]
					\eta_1^+\Bigl( \frac{x}{\varepsilon} \Bigl),
				\end{aligned}
				\right.
			\end{align}
			where
			\[
			\eta_1^+(y) = 1-\eta_0(y_2) , \quad \eta_1^-(y) = 1-\eta_0(-y_2) ,
			\]
			\[
			\eta_0 \in C^\infty(\R) ,\quad
			0 \leq \eta_0\leq 1, \quad 
			\eta_0(s) = 1 \quad s \leq 1, \quad \eta_0(s) = 0 \quad s \geq 2 .
			\]
			$\eta_1^+$, $\eta_1^-$, $\eta_0$  are the same functions defined in \eqref{eta1pm}, \eqref{eta0}. The cut-off $\chi^H$ defined here is the same as in \eqref{chiH1}.
			
			\medskip
			We also define the cut off functions with larger supports
				\begin{align}
                \nonumber
				\left\{
				\begin{aligned}
					\chi^H_2 (x) &= \chi_0\Bigl( \frac{x}{2\delta_2} \Bigl)    
					\\
					\chi^{S}_2 (x) &= \Bigl[ 1-\chi_0\Bigl( \frac{2x}{\delta_1} \Bigl)\Bigr]
					\eta_1^-\Bigl( \frac{2x}{\varepsilon} \Bigl) 
					\\
					\chi^{D}_2(x) &=\Bigl[ 1-\chi_0\Bigl( \frac{2x}{\delta_1} \Bigl)\Bigr]
					\eta_1^+\Bigl( \frac{2x}{\varepsilon} \Bigl), 
				\end{aligned}
				\right.
			\end{align}
		Observe that
		$$
		\chi^H (x) \, \chi^H_2 (x) = \chi^H (x) , \quad 	\chi^S (x) \, \chi^S_2 (x) = \chi^S (x) , 	\quad \chi^D (x) \, \chi^D_2 (x) = \chi^D (x) 
		$$
		for all $x \in \R^2$.
			\medskip
			We look for a solution of \eqref{main-linear} of the form
			\begin{align}
				\label{phi}
				\phi(y) &=  ( \phi^H(y) + \lambda Z^H(y)) \chi^H(x)
				+ \frac{1}{\varepsilon}\phi^{S}(x)\chi^{S}(x)
				+ \frac{1}{\varepsilon}(\phi^{D}(x)+ \beta Z^D(x))\chi^D(x),
			\end{align}
			where
			\begin{align*}
				y &= \frac{x}{\varepsilon}.
			\end{align*}
			
			The functions $Z^H(y)$, $Z^D(x)$ are kernel elements due to vertical translations in the hairpin $\Omega^H$ and disk $\Omega^D$:
			\begin{align*}
				Z^H(y) = \partial_{y_2} \psi^H(y),\quad
				Z^D(x) = \partial_{x_2} \psi^D(x),
			\end{align*}
where $\phi^H $ is the hairpin stream function \eqref{psiH} and  $\phi^D $ is the disk stream function \eqref{psiD}. In other sections we have written $Z^D$ as $Z_2^D$ to distinguish from the other element in the kernel $Z_1^D$ due to horizontal translation. The symmetry in $x_1$ yields that the  only relevant translation is vertical. Hence in this section we keep the notation $Z^D$.
			
			In \eqref{phi}, $\lambda$, $\beta$ are part of the solution.
			
			The function $\phi^H$ is defined on the domain $\tilde\Omega^H$ constructed in \eqref{tOmegaH}, $\phi^D$ is defined on $\tilde\Omega^D$ constructed in \eqref{tildeOD}, and $\phi^S$ is defined on the strip $\tilde\Omega^S$ defined in \eqref{tOmegaS}.
			The domains $\tilde\Omega^H$, $\tilde\Omega^D$, $\tilde\Omega^S$ and the cut-off functions have been constructed so that the terms in the equations \eqref{gHairpin}--\eqref{gDisc} below are correctly defined. 
			Indeed, consider for example $ \varepsilon (\phi^D + \beta Z^D) \Delta \chi^D$, which appears in \eqref{gHairpin}. We claim that this function is well-defined in $\tilde \Omega^H$.
			To see this, let $y \in \tilde\Omega^H$ be such that $x = \varepsilon y \in \supp( \Delta \chi^D)$. Then $y=(y_1,y_2)$ satisfies $y_2\geq 1$ and $\frac{\delta_1}{\varepsilon} \leq |y| \leq 2 \frac{\delta_1}{\varepsilon}$. 
			But the transition from the hairpin to the disk occurs in the region $\varepsilon^{-b}\leq |y|\leq 2 \varepsilon^{-b}$.
			Then $\varepsilon y \in \tilde \Omega^D$. 
			
			\subsection{The system}
			Consider the following system of equations for the functions $\phi^H$, $\phi^D$, $\phi^S$ and the parameters $\lambda$, $\beta$:
			\begin{align}
				\label{gHairpin}
				\left\{
				\begin{aligned}
					& \Delta_y \phi^H  + \varepsilon \phi^S \Delta \chi^S +2 \varepsilon \nabla \phi^S \cdot \nabla \chi^S
					+ \varepsilon (\phi^D + \beta Z^D) \Delta \chi^D
					\\
					& \quad 
					+ 2\varepsilon \nabla (\phi^D + \beta Z^D)\cdot \nabla \chi^D =  f_1 \chi_2^H \qquad \qquad 
					\text{in } \tilde \Omega^H \cap B_{2\delta_2/\varepsilon}(0)
					\\
					& \frac{\partial \phi^H}{\partial  \nu_y} 
					+ (\kappa - \varepsilon \omega \chi_2^H) \phi^H
					+\phi^S \frac{\partial \chi^S}{\partial  \nu_x}
					+  ( \phi^D +\beta Z^D) \frac{\partial \chi^D}{\partial  \nu_x}
					\\
					& \qquad \qquad 
					+ \varepsilon\omega \lambda Z^D ( 1-\chi^D-\chi^S)
					= f_2 \chi_2^H 
					\quad \text{on } \partial \tilde \Omega^ H  \cap B_{2\delta_2/\varepsilon}(0)
				\end{aligned}
				\right.
			\end{align}
\begin{align}
\label{gStrip}
\left\{
\begin{aligned}
& \Delta ( \phi^S + \varepsilon\lambda Z^H(y) \chi^H )
+2 \nabla_y \phi^H \cdot \nabla \chi^H
+ \varepsilon \phi^H \Delta \chi^H
= \tilde f_1 \chi_2^S
\quad \text{in }  \tilde\Omega^S
\\
& \frac{\partial}{\partial \nu}(  \phi^S + \varepsilon\lambda Z^H(y) \chi^H )
+ (\tilde\kappa-\omega)  ( \phi^S + \varepsilon\lambda Z^H(y) \chi^H ) 
+ \varepsilon \phi^H \frac{\partial \chi^H}{\partial \nu}
\\ & \quad 
= \tilde f_2 \chi_2^S
\quad \text{on }  \partial \tilde\Omega^S
\end{aligned}
\right.
\end{align}
			\begin{align}
				\label{gDisc}
				\left\{
				\begin{aligned}
					& \Delta ( \phi^D + \varepsilon\lambda Z^H(y) \chi^H + \beta Z^D)
					+2 \nabla_y \phi^H \cdot \nabla \chi^H
					+ \varepsilon \phi^H \Delta \chi^H
					= \tilde f_1 \chi_2^D
					\quad \text{in }  \tilde \Omega^D 
					\\
					& \frac{\partial}{\partial \nu}(  \phi^D + \varepsilon\lambda Z^H(y) \chi^H + \beta Z^D)
					+ (\tilde\kappa-\omega) (  \phi^D + \varepsilon\lambda Z^H(y) \chi^H + \beta Z^D) 
					\\ & \quad 
					+ \varepsilon \phi^H \frac{\partial \chi^H}{\partial \nu}
					= \tilde f_2 \chi_2^D 
					\quad \text{on }  \partial \tilde \Omega^D
				\end{aligned}
				\right.
			\end{align}
where   $f_2(y)$ is given in \eqref{decompf2}, and  $\tilde f_j$ and $\tilde\kappa$ are defined in \eqref{tf}, \eqref{tkappa}.
			
			In \eqref{gHairpin}--\eqref{gDisc} 
            the differential operators without subscript act on the variable $x$. 
			So, for example, in \eqref{gStrip},
			\begin{align*}
				\Delta \chi^H = ( \partial_{x_1x_1} + \partial_{x_2x_2}) \chi^H.
			\end{align*}
			
			We claim that if $\phi^H$, $\phi^D$, $\phi^S$, $\lambda$, $\beta$ satisfy \eqref{gHairpin}--\eqref{gDisc}, then $\phi$ given by \eqref{phi} satisfies \eqref{main-linear}.
			To see this, we multiply \eqref{gHairpin} by $\varepsilon^{-1}\chi^H$,  \eqref{gStrip} by $\chi^S$ and  \eqref{gDisc} by $\chi^D$ and add. 
			Let us keep track of the term containing $Z^H$ in the resulting expression:
			\begin{align*}
				\varepsilon \lambda \Delta( Z^H \chi^ H) ( \chi^S+\chi^D)
				&= 
				\varepsilon \lambda \Delta Z^H \chi^ H ( \chi^S+\chi^D)
				+ 
				2\varepsilon \lambda \nabla Z^H \cdot \nabla \chi^ H
				( \chi^S+\chi^D)
				\\
				&\quad 
				+ 
				\varepsilon \lambda  Z^H \Delta \chi^ H
				( \chi^S+\chi^D)
				\\ 
				&=
				2\varepsilon \lambda \nabla Z^H \cdot \nabla \chi^ H
				+ 
				\varepsilon \lambda  Z^H \Delta \chi^ H
				\\
				&=
				\varepsilon\lambda \Delta (Z^H \chi^H)  ,
			\end{align*}
			because $\Delta Z^H =0$ and the form of the cut-offs.
			
			Problem \eqref{gHairpin} leads us to consider linear problems of the form
			\begin{align}
				\label{pHairpin}
				\left\{
				\begin{aligned}
					\Delta_y \phi^H  &= j_1 &&\text{in } \tilde \Omega^H 
					\\
					\frac{\partial \phi^H}{\partial  \nu_y} 
					+ (\kappa - \varepsilon \omega \chi^H_2) \phi^H
					&=j_2 && \text{on } \partial \tilde \Omega^ H.
				\end{aligned}
				\right.
			\end{align}
			In \eqref{gHairpin} the equation and boundary condition are needed on part of $\tilde\Omega^H$ or its boundary.
			In \eqref{pHairpin} we consider the problem in the whole $\tilde \Omega^H$ but then restrict its solution to $\tilde \Omega^H \cap B_{2\delta_2/\varepsilon}(0)$. Existence and estimates of solutions for \eqref{pHairpin} have been obtained in Proposition~\ref{prop:lh1}.
			
			By taking the unknown in \eqref{gDisc} to be $\tilde \phi^D=\phi^D + \varepsilon\lambda Z^H(y) \chi^H + \beta Z^D$, we arrive at the problem
			\begin{align}
				\label{pDisc}
				\left\{
				\begin{aligned}
					\Delta \tilde \phi^D&= j_1 
					&&\text{in }  \tilde \Omega^D
					\\
					\frac{\partial \tilde \phi^D}{\partial \nu}
					+ (\kappa-\omega) \tilde \phi^D 
					&= j_2
					&&\text{on }  \partial \tilde \Omega^D .
				\end{aligned}
				\right.
			\end{align}
			Existence and estimates of solutions for \eqref{pDisc} have been obtained in Proposition~\ref{propDisc1}.
			
			Similarly, by taking the unknown in \eqref{gStrip} to be $ \tilde \phi^S(x) = \phi^S(x) + \lambda Z^H(y) \chi^H(x)$, we are led to
			\begin{align}
				\label{pStrip}
				\left\{
				\begin{aligned}
					\Delta \tilde \phi^S &= j_1 
					&&\text{in } \tilde\Omega^S
					\\
					\frac{\partial \tilde \phi^S}{\partial \nu} -\omega \tilde \phi^S 
					&= j_2 
					&&\text{on }\tilde{\mathcal{S}}
					\\
					\tilde \phi^S &= 0 &&\text{on }  \mathcal{B}
				\end{aligned}
				\right.
			\end{align}
			where $\tilde\Omega^S$ is defined in \eqref{tOmegaS}.
			Existence and estimates of solutions for \eqref{pStrip} are given in Proposition~\ref{propS}.
			
			\subsection{Solving the system}
			To solve  the system \eqref{gHairpin}--\eqref{gDisc} we proceed in three steps. First we solve \eqref{gStrip} as an operator of $\phi^H$. That is, given $\phi^H$ we find
			\begin{align}
				\nonumber
				\phi^S = \mathcal{T}_S[\phi^H,f_1,f_2], \quad 
				\lambda = \Lambda_S[\phi^H,f_1,f_2],
			\end{align}
			that solve \eqref{gStrip}, and such that 
			\[
			\phi^S(P_S) = 0
			\]
where $P_S$ is defined in \eqref{def:PS}.

Next, given $\phi^D$ and $\beta$ we consider \eqref{gHairpin} with $\phi^S = \mathcal{T}_S[\phi^H,f_1,f_2] $ and $\lambda = \Lambda_S[\phi^H,f_1,f_2]$. Assuming that $(\phi^D + \beta Z^D)(P_D)=0$ we find 
\[
\phi^H = \mathcal{T}_H[\phi^D,\beta,f_1,f_2]
\]
that solves \eqref{gHairpin}.	

In the last step we solve for $\phi^D$, $\beta$ and $g$ in equation \eqref{gDisc} with $\phi^H = \mathcal{T}_H[\phi^D,\beta,f_1,f_2]$ and $\lambda = \Lambda_S[ \mathcal{T}_H[\phi^D,\beta,f_1,f_2],f_1,f_2]$.
			
			\medskip
			
			The next proposition allows us to implement the first step, namely  given $\phi^H$ we find
			$\phi^S$ and $\lambda$ solving \eqref{gStrip} as operators on $\phi^H$.
			
			\begin{proposition}
				\label{propS2}
				Let $0<\delta_1<\delta_2$. Then there is $C$ and $\varepsilon_0>0$ such that for $0<\varepsilon<\varepsilon_0$ there are linear operators 
				$\mathcal{T}_S$, $\Lambda_S$
				so that if $ \| \phi^H\|_{\sharp, \tilde\Omega^H}<\infty$, 
				then 
				\begin{align}
					\nonumber
					\phi^S = \mathcal{T}_S[\phi^H,f_1,f_2], \quad 
					\lambda = \Lambda_S[\phi^H,f_1,f_2],
				\end{align}
				satisfy \eqref{gStrip} and 
				\[
				\phi^S(P_S) = 0.
				\]
				Moreover
\begin{align}
\nonumber
\|  \mathcal{T}_S[\phi^H,f_1,f_2] \|_{C^{2,\alpha},\tilde\Omega^S} 
&\leq 
C \varepsilon^{\sigma}  \delta_2^{-1-\alpha-\sigma} \|\phi^H\|_{\sharp , \tilde\Omega^H} + C \| \tilde f_1 \chi_2^S \|_{C^\alpha, \tilde \Omega^S}
\\
\label{TS}
& \quad + C \| \tilde f_2 \chi_2^S \|_{C^{1,\alpha},\partial \tilde \Omega^S}
\\
\nonumber
\varepsilon | \Lambda_S[\phi^H,f_1,f_2] | & 
\leq 
C \varepsilon^{\sigma}  \delta_2^{-1-\alpha-\sigma} \| \phi^H\|_{\sharp, \tilde\Omega^H} + C \| \tilde f_1 \chi_2^S \|_{C^\alpha,\tilde \Omega^S}
\\
\label{Lambda}
& \quad 
+ C \| \tilde f_2 \chi_2^S \|_{C^{1,\alpha},\partial \tilde \Omega^S}.
\end{align}
			\end{proposition}
We refer to \eqref{nH1}--\eqref{nH3} for the norms in $\tilde \Omega^H$ and to \eqref{nS1}--\eqref{nS3} for the norms in $\tilde \Omega^S$.
			\begin{proof}
				Letting
				\[
				\tilde \phi^S = \phi^S + \varepsilon\lambda Z^H(y) \chi^H ,
				\]
				we need to solve \eqref{pStrip} with
				\begin{align*}
					j_1 &= - 2 \nabla_y \phi^H \cdot \nabla \chi^H
					-\varepsilon \phi^H \Delta \chi^H
					+ \tilde f_1 \chi_2^S
					\\
					j_2 &= - \varepsilon \phi^H \frac{\partial \chi^H}{\partial \nu}
					+ \tilde f_2 \chi_2^S .
				\end{align*}
Using Proposition~\ref{propS}, we need to estimate
\begin{align*}
\| \nabla_y \phi^H \cdot \nabla \chi^H \|_{C^\alpha,\tilde\Omega^S}
+ \| \varepsilon \phi^H \Delta \chi^H \|_{C^\alpha,\tilde\Omega^S} \quad {\mbox {and}} \quad \| \varepsilon \phi^H \frac{\partial \chi^H}{\partial \nu}\|_{C^{1,\alpha},\partial\tilde \Omega^S}.
\end{align*}
				We have
				\begin{align}
					\nonumber
					|  \varepsilon \phi^H \Delta \chi^H (x)|
					&\leq \varepsilon \Bigl(\frac{|x|}{\varepsilon}\Bigr)^{1-\sigma}
					\delta_2^{-2} \Bigl|\Delta\chi_0\Bigl(\frac{x}{\delta_2}\Bigr)\Bigr| \|\phi^H\|_{\sharp, \tilde\Omega^H}
					\\
					\label{calcPhiH}
					& \leq C \varepsilon^{\sigma} \delta_2^{-1-\sigma} \|\phi^H\|_{\sharp , \tilde\Omega^H}.
				\end{align}
				
For the H\"older part of $ \| \varepsilon \phi^H \Delta \chi^H \|_{C^\alpha,\tilde \Omega^S}$, we need only to estimate
				\[
				[  \varepsilon \phi^H \Delta \chi^H  ]_{\alpha,B_1(0)\cap \tilde \Omega^S},
				\]
				because the support of $ \Delta \chi^H $ is contained in the ball $B_1(0)$.
				Let $x_1,x_2\in B_1(0)\cap \tilde \Omega^S$ and let $y_j= \frac{x_j}{\varepsilon}$.
				Let us first assume that $|x_1-x_2|\leq \frac 1{10} \max(|x_1|,|x_2|)$.
				We may assume that $\max(|x_1|,|x_2|)=|x_1|$.
				Then write
				\begin{align*}
					\phi^H\Bigl( \frac{x_1}{\varepsilon}\Bigr) \Delta \chi^H(x_1)
					-\phi^H\Bigl( \frac{x_2}{\varepsilon}\Bigr) \Delta \chi^H(x_2)
					&=
					\Bigl[\phi^H\Bigl( \frac{x_1}{\varepsilon}\Bigr) 
					-\phi^H\Bigl( \frac{x_2}{\varepsilon}\Bigr) \Bigr]
					\Delta \chi^H(x_1)
					\\
					& \quad 
					+\phi^H\Bigl( \frac{x_2}{\varepsilon}\Bigr) 
					[\Delta \chi^H(x_1)
					- \Delta \chi^H(x_2)] .
				\end{align*}
				We have
				\begin{align}
					\nonumber
					\Bigl|  \varepsilon \Bigl[ \phi^H\Bigl(\frac{x_1}{\varepsilon}\Bigl) -\phi^H\Bigl(\frac{x_2}{\varepsilon}\Bigl) \Bigr]\Delta \chi^H (x)\Bigr|
					&\leq \varepsilon 
					\frac{|x_1-x_2|}{\varepsilon}
					\Bigl(\frac{|x_1|}{\varepsilon}\Bigr)^{-\sigma}
					\delta_2^{-2} \Bigl|\Delta\chi_0\Bigl(\frac{x}{\delta_2}\Bigr)\Bigr| \|\phi^H\|_{\sharp, \tilde\Omega^H}
					\\
					\nonumber
					&\leq \varepsilon^{\sigma} |x_1-x_2|^\alpha |x_1|^{1-\alpha}|x_1|^{-\sigma} \delta_2^{-2}\Bigl|\Delta\chi_0\Bigl(\frac{x}{\delta_2}\Bigr)\Bigr| \|\phi^H\|_{\sharp, \tilde\Omega^H}
					\\
					\nonumber
					&\leq C\varepsilon^{\sigma} |x_1-x_2|^\alpha \delta_2^{-1-\alpha-\sigma} \|\phi^H\|_{\sharp, \tilde\Omega^H},
				\end{align}
				and similarly,
				\begin{align}
					\label{a5}
					\Bigl|\phi^H\Bigl( \frac{x_2}{\varepsilon}\Bigr) 
					[\Delta \chi^H(x_1)
					- \Delta \chi^H(x_2)]\Bigr|
					\leq 
					C\varepsilon^{\sigma} |x_1-x_2|^\alpha \delta_2^{-1-\alpha-\sigma} \|\phi^H\|_{\sharp, \tilde\Omega^H},
				\end{align}
				if we assume that $|x_1-x_2|\leq \frac 1{10} \max(|x_1|,|x_2|)$.

				Assume now that $|x_1-x_2| > \frac{\max(|x_1|,|x_2|)}{10}$.
				We estimate
				\begin{align*}
					\Bigl|
					\phi^H\Bigl( \frac{x_1}{\varepsilon}\Bigr) \Delta \chi^H(x_1)
					-\phi^H\Bigl( \frac{x_2}{\varepsilon}\Bigr) \Delta \chi^H(x_2)
					\Bigr| 
					&\leq 
					\Bigl|
					\phi^H\Bigl( \frac{x_1}{\varepsilon}\Bigr) \Delta \chi^H(x_1)
					\Bigr|
					\\
					& \quad 
					+
					\Bigl|
					\phi^H\Bigl( \frac{x_2}{\varepsilon}\Bigr) \Delta \chi^H(x_2)
					\Bigr| .
				\end{align*}
				If $ \Delta \chi^H(x_1)\not=0$ then $1 \leq C \delta_2^{-\alpha} |x_1-x_2|^\alpha$ and so
				\begin{align*}
					\Bigl|
					\phi^H\Bigl( \frac{x_1}{\varepsilon}\Bigr) \Delta \chi^H(x_1)
					\Bigr|
					&\leq C \varepsilon^\sigma \delta_2^{-1-\sigma}
					\|\phi^H\|_{\sharp, \tilde\Omega^H}
					\\
					& \leq C \varepsilon^\sigma \delta_2^{-1-\sigma-\alpha} |x_1-x_2|^\alpha
					\|\phi^H\|_{\sharp, \tilde\Omega^H} .
				\end{align*}
				The other term is handled similarly and we get that \eqref{a5} is valid also in the case  $|x_1-x_2| > \frac 1{10} \max(|x_1|,|x_2|)$.
				We thus obtain
\begin{align*}
\| \varepsilon \phi^H \Delta \chi^H \|_{C^\alpha,\tilde\Omega^S} 
\leq C \varepsilon^{\sigma}  \delta_2^{-1-\alpha-\sigma} \|\phi^H\|_{\sharp,\tilde\Omega^H}.
\end{align*}
We also have
\begin{align*}
\| \nabla_y \phi^H \cdot \nabla \chi^H \|_{C^\alpha,\tilde \Omega^S} 
\leq C \varepsilon^{\sigma}  \delta_2^{-1-\alpha-\sigma} \|\phi^H\|_{\sharp,\tilde\Omega^H}.
\end{align*}
				
				Finally, let us estimate
				\begin{align*}
					\Bigl\| \varepsilon \phi^H \frac{\partial \chi^H}{\partial \nu} \Bigr\|_{C^{1,\alpha},\partial \tilde \Omega^S} .
				\end{align*}
				We compute for $x\in\partial\tilde\Omega^S$, 
				\begin{align*}
					\Bigl|  \varepsilon \phi^H\Bigl(\frac{x}{\varepsilon}\Bigr) \frac{\partial \chi^H(x)}{\partial \nu}  \Bigr|
					&\leq \varepsilon \Bigl(\frac{|x|}{\varepsilon}\Bigr)^{1-\sigma}
					\delta_2^{-1} \Bigl|\nabla\chi_0\Bigl(\frac{x}{\delta_2}\Bigr)\Bigr| \|\phi^H\|_{\sharp, \tilde\Omega^H}
					\\
					& \leq C \varepsilon^{\sigma }\delta_2^{-\sigma}
					\|\phi^H\|_{\sharp, \tilde\Omega^H}.
				\end{align*}
				Similarly we obtain for $x\in\partial\tilde\Omega^S$
				\begin{align*}
					\Bigl|  \nabla_x \Bigl[
					\varepsilon \phi^H\Bigl(\frac{x}{\varepsilon}\Bigr) \frac{\partial \chi^H(x)}{\partial \nu} \Bigr] \Bigr|
					&\leq C \varepsilon^{\sigma} \delta_2^{-1-\sigma}
					\|\phi^H\|_{\sharp, \tilde\Omega^H}
				\end{align*}
				and
				\begin{align*}
					\Bigl[ \nabla_x \Bigl( \varepsilon \phi^H\Bigl(\frac{x}{\varepsilon}\Bigr)  \frac{\partial \chi^H}{\partial \nu} \Bigr) \Bigr]_{\alpha,B_1(0)\cap \partial\tilde\Omega^S} 
					\leq C \varepsilon^{\sigma} \delta_2^{-1-\sigma-\alpha}
					\|\phi^H\|_{\sharp, \tilde\Omega^H}.
				\end{align*}
				Therefore
\begin{align*}
\Bigl\| \varepsilon \phi^H \frac{\partial \chi^H}{\partial \nu} \Bigr\|_{C^{1,\alpha},\partial\tilde\Omega^S} 
\leq C \varepsilon^{\sigma}  \delta_2^{-1-\alpha-\sigma} \|\phi^H\|_{\sharp,\tilde\Omega^H}.
\end{align*}

				By Proposition~\ref{propS} we find that \eqref{pStrip} has a unique solution $\tilde\phi^S$, which satisfies
\begin{align}
\label{tPhiS}
\|  \tilde\phi^S \|_{C^{2,\alpha},\tilde\Omega^S}  
& \leq 
C \varepsilon^{\sigma}  \delta_2^{-1-\alpha-\sigma} \|\phi^H\|_{\sharp, \tilde\Omega^H} + C \| \tilde f_1 \chi_2^S \|_{C^\alpha, \tilde \Omega^S}
+ C \| \tilde f_2 \chi_2^S \|_{C^{1,\alpha},\partial \tilde \Omega^S}.
\end{align}

Finally we set 
\begin{align*}
\lambda=\Lambda_S[\phi^H,f_1,f_2] = \frac{1}{\varepsilon} 
					\frac{\tilde\phi^S(P_S)}{ Z^H(\frac{P_S}{\varepsilon} ) }
				\end{align*}
				and
				\begin{align*}
					\mathcal{T}_S[\phi^H,f_1,f_2](x) = \tilde \phi^S(x) - \varepsilon\lambda Z^H\Bigl(\frac{x}{\varepsilon} \Bigr) .
				\end{align*}
				The estimates \eqref{TS} and \eqref{Lambda} follow from \eqref{tPhiS}.
			\end{proof}

			\medskip
			We now turn to problem \eqref{gHairpin}.

			\begin{proposition}
				\label{propH2}
				Consider \eqref{gHairpin} with $\phi^S = \mathcal{T}_S[\phi^H,f_1,f_2] $ and $\lambda = \Lambda_S[\phi^H,f_1,f_2]$, where $\mathcal{T}_S$, $\Lambda_S$ are the operators defined in Proposition~\ref{propS2}.
				
				There is $\delta_0>0 $, $\varepsilon_0>0$ and $C>0$ such that if 
				\begin{align}
					\label{delta1delta2}
					\delta_1^\sigma \delta_2^{-1-\alpha-\sigma}<\delta_0
				\end{align}
				and $0<\varepsilon<\varepsilon_0$, then there is a linear operator
				$\mathcal{T}_H$ such that if $\|\phi^D\|_{C^{2,\alpha}, \tilde \Omega^D}<\infty$ and $\beta\in \R$ satisfy $\phi^D(P_D) + \beta Z^D(P_D)=0$, then
				\begin{align}
					\nonumber
					\phi^H = \mathcal{T}_H[\phi^D,\beta,f_1,f_2] 
				\end{align}
				satisfies \eqref{gHairpin}.
				Moreover
				\begin{align}
					\nonumber
					\|  \mathcal{T}_H[\phi^D,\beta,f_1,f_2] \|_{\sharp, \tilde\Omega^H} 
					& \leq  
					C \delta_1^\sigma \varepsilon^{-\sigma} ( \| \phi^D \|_{C^{2,\alpha}, \tilde \Omega^D} +|\beta| )
					\\
					\nonumber
					& \quad 
					+ C \delta_1^\sigma \varepsilon^{-\sigma} ( \| \tilde f_1 \chi_2^S \|_{C^\alpha,\tilde \Omega^S} +  \| \tilde f_2 \chi_2^S \|_{C^{1,\alpha},\partial \tilde \Omega^S} )
					\\
					\label{estH}
					& \quad 
					+ C \| f_1 \chi_2^H \|_{\sharp\sharp, \tilde\Omega^H}
					+C \| f_2 \chi_2^H \|_{\sharp\sharp,\partial\tilde\Omega^H} .
				\end{align}
			\end{proposition}
We refer to \eqref{nH1}--\eqref{nH3} for the norms in $\tilde \Omega^H$, to \eqref{nS1}--\eqref{nS3} for the norms in $\tilde \Omega^S$ and to \eqref{nD1}--\eqref{nD3} for the norms in $\tilde \Omega^D$.
			\begin{proof}
				Proposition~\ref{prop:lh1} defines a linear operator $L_H$ such that if $\|j_1\|_{\sharp\sharp,\tilde\Omega^H}<\infty$,
				$\|j_2\|_{\sharp\sharp,\partial\tilde\Omega^H}<\infty$, then $\phi^H=L_H[j_1,j_2]$ is a solution of \eqref{pHairpin} and we have the estimate
				\begin{align}
					\label{estLH}
					\| L_H[j_1,j_2] \|_{\sharp,\tilde\Omega^H} \leq C ( \|j_1\|_{\sharp\sharp,\tilde\Omega^H} +\|j_2\|_{\sharp\sharp,\partial\tilde\Omega^H}). 
				\end{align}
				Let
				\begin{align}
                    \notag
					G_1[\phi^H,\phi^D,\beta,f_1,f_2] &=
					-\varepsilon \phi^S \Delta \chi^S 
					-2 \varepsilon \nabla \phi^S \cdot \nabla \chi^S
					+ \varepsilon (\phi^D + \beta Z^D) \Delta \chi^D
					\\
                    \label{G1def}
					& \quad 
					- 2\varepsilon \nabla (\phi^D + \beta Z^D)\cdot \nabla \chi^D 
					+ f_1 \chi_2^H
					\\
                    \notag
					G_2[\phi^H,\phi^D,\beta,f_1,f_2] &=
					-  \phi^S \frac{\partial \chi^S}{\partial  \nu_x}
					-  ( \phi^D +\beta Z^D) \frac{\partial \chi^S}{\partial  \nu_x}
					- \varepsilon\omega \lambda Z^D ( 1-\chi^S-\chi^D)
					\\
     \notag
					& \quad 
					+ f_2 \chi_2^H ,
				\end{align}
				where in the above expressions we have written
				\begin{align}
					\label{phiS2}
					\phi^S = \mathcal{T}_S[\phi^H,f_1,f_2] 
					\\
					\nonumber
					\lambda = \Lambda_S[\phi^H,f_1,f_2].
				\end{align}

				The proof of Proposition~\ref{propH2} consists in showing that the fixed point problem
				\begin{align}
					\label{fixed1}
					\phi^H = L_H[G_1[\phi^H,\phi^D,\beta,f_1],G_2[\phi^H,\phi^D,\beta,f_2]]
				\end{align}
				has a unique solution with the properties described in the statement.

				We claim that
				\begin{align}
					\nonumber
					\| G_1[\phi^H,\phi^D,\beta,f_1,f_2] \|_{\sharp\sharp,\tilde\Omega^H}
					& \leq
					C \delta_1^{\sigma}  \delta_2^{-1-\alpha-\sigma} \|\phi^H\|_{\sharp, \tilde\Omega^H} 
					+ C \delta_1^\sigma \varepsilon^{-\sigma} (  \| \phi^D \|_{C^{2, \alpha}, \tilde \Omega^D}  +|\beta| )
					\\
					\nonumber
					& \quad
					+ C \delta_1^\sigma \varepsilon^{-\sigma} ( \| \tilde f_1 \chi_2^S \|_{C^\alpha,\tilde \Omega^S} +  \| \tilde f_2 \chi_2^S \|_{C^{1,\alpha},\partial \tilde\Omega^S} )
					\\
					\label{estG1}
					& \quad 
					+ C \| f_1 \chi_2^H \|_{\sharp\sharp,\tilde\Omega^H},
				\end{align}
				and
				\begin{align}
					\nonumber
					\|  G_2[\phi^H,\phi^D,\beta,f_1,f_2] \|_{\sharp\sharp,\partial\tilde\Omega^H} 
					&\leq
					C \delta_1^{\sigma}  \delta_2^{-1-\alpha-\sigma} \|\phi^H\|_{\sharp, \tilde\Omega^H} 
					+ C \delta_1^\sigma \varepsilon^{-\sigma} (  \| \phi^D \|_{C^{2, \alpha},\tilde \Omega^D}  +|\beta| )
					\\
					\nonumber
					& \quad
					+ C \delta_1^\sigma \varepsilon^{-\sigma} ( \| \tilde f_1 \chi_2^S \|_{C^\alpha,\tilde\Omega^S} +  \| \tilde f_2 \chi_2^S \|_{C^{1,\alpha},\partial\tilde\Omega^S} )
					\\
					\label{estG2}
					& \quad 
					+ C \| f_2 \chi_2^H \|_{\sharp\sharp,\partial\tilde\Omega^H}.
				\end{align}
				
				Let us prove \eqref{estG1}, starting with the term $\varepsilon \phi^S \Delta \chi^S $ in \eqref{G1def}. We note that $\phi^S$ given by \eqref{phiS2} satisfies, by Proposition~\ref{propS2}, $\phi^S(P_S)=0$. Hence 
\begin{align}
\nonumber
| \varepsilon \phi^S(\varepsilon y) \Delta \chi^S (y)| 
& \leq \varepsilon (\varepsilon |y|) \delta_1^{-2} \Bigl| \Delta \chi_0\Bigl( \frac{y}{\delta_1} \Bigr)\Bigr| \| \phi^S \|_{C^{2,\alpha},\tilde\Omega^S}  
\\
\nonumber
& = \varepsilon^2 \delta_1^{-2} \frac{1+|y|^{2+\sigma}}{1+|y|^{1+\sigma}}\Bigl| \Delta \chi_0\Bigl( \frac{y}{\delta_1} \Bigr)\Bigr| \| \phi^S \|_{C^{2,\alpha},\tilde\Omega^S}  
\\
\nonumber
& \leq C \varepsilon^2 \delta_1^{-2} \frac{(\delta_1/\varepsilon)^{2+\sigma}}{1+|y|^{1+\sigma}}\Bigl| \Delta \chi_0\Bigl( \frac{y}{\delta_1} \Bigr)\Bigr| \| \phi^S \|_{C^{2,\alpha},\tilde\Omega^S} 
\\ 
\nonumber
& \leq C  \delta_1^\sigma \varepsilon^{-\sigma} \frac{1}{1+|y|^{1+\sigma}} \| \phi^S \|_{C^{2,\alpha},\tilde\Omega^S} .
\end{align}
Now let us compute the H\"older part of the norm $\| \varepsilon \phi^S(\varepsilon y) \Delta \chi^S \|_{\sharp\sharp,\tilde\Omega^H} $. 
				Let $y \in \tilde \Omega^H$
				and $y_1,y_2 \in B(y,\frac{|y|}{10})\cap \tilde\Omega^H$, $y_1\not=y_2$. 
				Writing
				\begin{align*}
					\phi^S(\varepsilon y_1)\Delta \chi^S(\varepsilon y_1) - 
					\phi^S(\varepsilon y_2)\Delta \chi^S(\varepsilon y_2) 
					&= 
					[\phi^S(\varepsilon y_1) - \phi^S(\varepsilon y_2)]
					\Delta \chi^S(\varepsilon y_1) 
					\\
					& \quad 
					+ \phi^S(\varepsilon y_2) [\Delta \chi^S(\varepsilon y_1)-\Delta \chi^S(\varepsilon y_2)] ,
				\end{align*}
				we estimate
				\begin{align*}
					&
					\varepsilon |\phi^S(\varepsilon y_1) - \phi^S(\varepsilon y_2) | 
					|\Delta \chi^S(\varepsilon y_1) |
					\\
					& \quad \leq \varepsilon^2 |y_1-y_2| \delta_1^{-2}
					\Bigl| \Delta \chi_0\Bigl( \frac{y_1}{\delta_1} \Bigr)\Bigr| \| \phi^S \|_{C^{2, \alpha}, \tilde \Omega^S} 
					\\
					& \quad  \leq 
					C \varepsilon^2 \delta_1^{-2} |y_1-y_2|^\alpha |y|^{1-\alpha}\frac{|y|^{1+\sigma+\alpha}}{1+|y|^{1+\sigma+\alpha}}
					\Bigl| \Delta \chi_0\Bigl( \frac{y_1}{\delta_1} \Bigr)\Bigr| \| \phi^S \|_{C^{2, \alpha}, \tilde \Omega^S} 
					\\
					& \quad \leq 
					C \varepsilon^2 \delta_1^{-2} |y_1-y_2|^\alpha \frac{(\delta_1/\varepsilon)^{2+\sigma}}{1+|y|^{1+\sigma+\alpha}}
					\Bigl| \Delta \chi_0\Bigl( \frac{y_1}{\delta_1} \Bigr)\Bigr| \| \phi^S \|_{C^{2, \alpha}, \tilde \Omega^S}  
					\\
					\nonumber
					& \quad \leq C  \delta_1^\sigma \varepsilon^{-\sigma} |y_1-y_2|^\alpha\frac{1}{1+|y|^{1+\sigma+\alpha}}\| \phi^S \|_{C^{2, \alpha}, \tilde \Omega^S}  .
				\end{align*}
				We then get
				\begin{align*}
					\|  \varepsilon \phi^S(\varepsilon y) \Delta \chi^S (y)\|_{\sharp\sharp,\tilde \Omega^H} \leq 
					C  \delta_1^\sigma \varepsilon^{-\sigma} \| \phi^S \|_{C^{2, \alpha}, \tilde \Omega^S}  .
				\end{align*}
				Combining this inequality with \eqref{TS} yields
				\begin{align*}
					\|  \varepsilon \phi^S(\varepsilon y) \Delta \chi^S (y)\|_{\sharp\sharp,\tilde \Omega^H} 
					&\leq 
					C \delta_1^{\sigma}  \delta_2^{-1-\alpha-\sigma} \|\phi^H\|_{\sharp,\tilde\Omega^H} +  C  \delta_1^\sigma \varepsilon^{-\sigma}\| \tilde f_1 \chi_2^S \|_{C^\alpha,\tilde \Omega^S}
					\\
					& \quad + C   \delta_1^\sigma \varepsilon^{-\sigma}\| \tilde f_2 \chi_2^S \|_{C^{1,\alpha},\partial \tilde\Omega^S}.
				\end{align*}
				
				Let us analyze the term $  \varepsilon \nabla \phi^S \cdot \nabla \chi^S$ that appears in \eqref{G1def}. 
                We have
				\begin{align*}
					\left| \varepsilon \nabla \phi^S \cdot \nabla \chi^S \right|
					& \leq \varepsilon  \delta_1^{-1} \Bigl| \nabla \chi_0\Bigl( \frac{y}{\delta_1} \Bigr)\Bigr| \| \phi^S \|_{C^{2, \alpha}, \tilde \Omega^S}  
					\\
					& \leq C \varepsilon \delta_1^{-1} \frac{(\delta_1/\varepsilon)^{1+\sigma}}{1+|y|^{1+\sigma}}\Bigl| \nabla \chi_0\Bigl( \frac{y}{\delta_1} \Bigr)\Bigr| \| \phi^S \|_{C^{2, \alpha}, \tilde \Omega^S} 
					\\
					& \leq C  \delta_1^\sigma \varepsilon^{-\sigma} \frac{1}{1+|y|^{1+\sigma}} \| \phi^S \|_{C^{2, \alpha}, \tilde \Omega^S}  .
				\end{align*}
				The H\"older part is estimated as before, and we get, using \eqref{TS},
				\begin{align*}
					\|  \varepsilon \nabla \phi^S \cdot \nabla \chi^S \|_{\sharp\sharp,\tilde \Omega^H} &\leq 
					C  \delta_1^\sigma \varepsilon^{-\sigma} \| \phi^S \|_{C^{2, \alpha}, \tilde \Omega^S} 
					\\
					&\leq 
					C \delta_1^{\sigma}  \delta_2^{-1-\alpha-\sigma} \|\phi^H\|_{\sharp,\tilde\Omega^H} +  C  \delta_1^\sigma \varepsilon^{-\sigma}\| \tilde f_1 \chi_2^S \|_{C^\alpha,\tilde \Omega^S}
					\\
					& \quad + C   \delta_1^\sigma \varepsilon^{-\sigma}\| \tilde f_2 \chi_2^S \|_{C^{1,\alpha},\partial \tilde \Omega^S}.
				\end{align*}
				The remaining terms in $G_1$ are estimated similarly and we get \eqref{estG1}.

				\bigskip

				Next, let us prove \eqref{estG2}.
				Note that $ ( 1-\chi^S-\chi^D)(\varepsilon y) =0$ if $|y|\geq 2 \frac{\delta_1}{\varepsilon}$. Since 
				$Z^D$ is bounded this implies that 
				\begin{align*}
					\sup_{y \in \tilde \Omega^H}
					(1+|y|)^\sigma
					| \varepsilon  \lambda Z^D ( 1-\chi^S-\chi^D)(\varepsilon y) |
					\leq C \varepsilon |\lambda| \delta_1^\sigma \varepsilon^{-\sigma}.
				\end{align*}
				The derivative and its H\"older part are estimated analogously and we get
				\begin{align*}
					\| \varepsilon \lambda Z^D ( 1-\chi^S-\chi^D) \|_{\sharp\sharp,\partial\tilde\Omega^H} \leq C \varepsilon |\lambda| \delta_1^\sigma \varepsilon^{-\sigma}.
				\end{align*}
				This and \eqref{TS} yield
				\begin{align*}
					\| \varepsilon \lambda Z^D ( 1-\chi^S-\chi^D) \|_{\sharp\sharp,\partial\tilde\Omega^H} 
					&\leq 
					C \delta_1^\sigma  \delta_2^{-1-\alpha-\sigma} \|\phi^H\|_{\sharp, \tilde\Omega^H} + C \delta_1^\sigma \varepsilon^{-\sigma} \| \tilde f_1 \chi_2^S \|_{C^\alpha,\tilde \Omega^S}
					\\
					& \quad + C\delta_1^\sigma \varepsilon^{-\sigma} \| \tilde f_2 \chi_2^S \|_{C^{1,\alpha},\partial \tilde \Omega^S}
				\end{align*}
				Similar calculations as before yield the validity of \eqref{estG2}.
				
				\medskip
				
				Using \eqref{estG1} and \eqref{estG2} combined with the  estimates for $\mathcal{T}_S$ and $\Lambda_S$ \eqref{TS} and \eqref{Lambda}, and the estimate for $L_H$ in \eqref{estLH}, we obtain
				\begin{align}
					\nonumber
					\| L_H[G_1[\phi^H,\phi^D,\beta,f_1],&G_2[\phi^H,\phi^D,\beta,f_2]] \|_{\sharp ,\tilde\Omega^H} 
					\\
					\nonumber
					& \leq  
					C \delta_1^\sigma \delta_2^{-1-\alpha-\sigma}    \| \phi^H\|_{\sharp,\tilde\Omega^H}
					+C\delta_1^\sigma\varepsilon^{-\sigma}  ( \| \phi^D \|_{C^{2, \alpha}, \tilde \Omega^D}  +|\beta| )
					\\
					\nonumber
					& \qquad 
					+ C \delta_1^\sigma\varepsilon^{-\sigma}  ( \| \tilde f_1 \chi_2^S \|_{C^\alpha,\tilde\Omega^S}  +  \| \tilde f_2 \chi_2^S \|_{C^{1,\alpha}, \partial\tilde \Omega^S} )
					\\
					\label{cont}
					& \qquad 
					+ C \| f_1 \chi_2^H \|_{\sharp\sharp ,\tilde\Omega^H}
					+C \| f_2 \chi_2^H \|_{\sharp\sharp ,\partial\tilde\Omega^H} .
				\end{align}
				Inequality \eqref{cont} and the assumption \eqref{delta1delta2} allow us to apply the contraction mapping principle and we find a unique $\phi^H$ solving the fixed point problem \eqref{fixed1}.
				Moreover, also from \eqref{cont} we find that this fixed point satisfies \eqref{estH}.
			\end{proof}

			Finally, to fully solve the system \eqref{gHairpin}--\eqref{gDisc}, we consider  problem \eqref{gDisc} where we regard $\phi^H = \mathcal{T}_H[\phi^D,f_1,f_2]$ as defined in Proposition~\ref{propH2}, and $\lambda$ is given by
			\begin{align}
				\label{alpha}
				\lambda = \Lambda_S[\mathcal{T}_H[\phi^D,\beta,f_1,f_2],f_1,f_2].
			\end{align}
			We note that $\Lambda_S$ is a linear function of  $\phi^D,\beta,f_1,f_2$ and has the estimate 
			\begin{align}
				\nonumber
				\varepsilon | \Lambda_S[\mathcal{T}_H[\phi^D,&\beta,f_1,f_2],f_1,f_2]| 
				\\
				\nonumber
				&\leq 
				C \delta_1^\sigma \delta_2^{-1-\alpha-\sigma}
				\left[
				\| \phi^D \|_{C^{2, \alpha}, \tilde \Omega^D}  +|\beta| \right]
				\\
				\nonumber
				& \qquad 
				+  C   \| \tilde f_1 \chi_2^S \|_{C^\alpha, \tilde\Omega^S}
				+ C \| \tilde f_2 \chi_2^S \|_{C^{1,\alpha},\partial\tilde\Omega^S} 
				\\
				\label{estLambda}
				& \qquad 
				+ C  \varepsilon^{\sigma} \delta_2^{-1-\alpha-\sigma} \left[ \| f_1 \chi_2^H \|_{\sharp\sharp,\tilde\Omega^H}
				+  \| f_2 \chi_2^H \|_{\sharp\sharp ,\partial\tilde\Omega^H}
				\right],
			\end{align}
			thanks to \eqref{Lambda} and \eqref{estH}.
			
			\begin{proposition}
				\label{propPhiD2}
				Consider \eqref{gDisc} with $\phi^H = \mathcal{T}_H[\phi^D,\beta,f_1,f_2]$ and $\lambda$ given by \eqref{alpha}.
				For any $\delta_2>0$ small there is $\bar\delta_1>0$ so that if $0<\delta_1<\bar\delta_1$ there is $\bar\varepsilon>0$ and $C>0$ so that if
				$0<\varepsilon<\bar\varepsilon$ then there are linear operators
				$\mathcal{T}_D[f_1,f_2]$, $\mathcal{B}_D[f_1,f_2] $, $\mathcal{G}_D[f_1,f_2] $  such that 
				\begin{align}
					\nonumber
					\phi^D &= \mathcal{T}_D[f_1,f_2] 
					\\
					\nonumber
					\beta &= \mathcal{B}_D[f_1,f_2] 
					\\
					\nonumber
					g &= \mathcal{G}_D[f_1,f_2] 
				\end{align}
				satisfy \eqref{gDisc}.
				Moreover,
				\begin{align}
					\nonumber
					\|\phi^D\|_{C^{2,\alpha},\tilde\Omega^D}
					+ |\beta|
					&  \leq C  \| \tilde f_1 \chi_2^D\|_{C^\alpha,\tilde\Omega^D}
					+   \| \tilde f_{2,0} \chi_2^D\|_{C^{1,\alpha},\partial\tilde\Omega^D}.
					\\
					\nonumber
					& \quad + C \varepsilon^{\sigma}  \delta_2^{-3-2\alpha-\sigma} ( \| f_1 \chi_2^H \|_{\sharp\sharp ,\tilde\Omega^H}
					+ \| f_{2,0} \chi_2^H \|_{\sharp\sharp ,\partial\tilde\Omega^H} 
					)
					\\
					\label{estphiD}
					& \quad + C \delta_2^{-(2+\alpha)} ( \| \tilde f_1 \chi_2^S \|_{C^\alpha,\tilde \Omega^S}+  \| \tilde f_{2,0} \chi_2^S \|_{C^{1,\alpha},\partial \tilde \Omega^S})
				\end{align}
				and
				\begin{align}
					\nonumber
					|g| & \leq C   \left[ 
					\| \tilde f_1 \chi_2^D\|_{C^\alpha,\tilde\Omega^D} 
					+   \| \tilde f_{2,0} \chi_2^D\|_{C^{1,\alpha},\partial\tilde\Omega^D} \right]
					\\
					\nonumber
					& \quad 
					+ C \varepsilon^{\sigma}  \delta_1^{3-\alpha} \delta_2^{-3-2\alpha-\sigma} 
					\left[ \| f_1 \chi_2^H \|_{\sharp\sharp ,\tilde\Omega^H}
					+ \| f_{2,0} \chi_2^H \|_{\sharp\sharp ,\partial\tilde\Omega^H} \right]
					\\
					\label{estg3}
					& \quad + 
					C \delta_1^\sigma  \delta_2^{-1-\sigma-\alpha}\left[ 
					\| \tilde f_1 \chi_2^S \|_{C^\alpha,\tilde \Omega^S}+  \| \tilde f_{2,0} \chi_2^S \|_{C^{1,\alpha},\partial \tilde\Omega^S}
					\right] .
				\end{align}
			\end{proposition}
			\begin{proof}
				We set up a fixed point problem for $\phi^D$.
				Given  $\phi^D \in C^{2,\alpha}(\overline{\tilde\Omega^D})$, we let 
				\begin{align}
					\label{beta}
					\beta(\phi^D) = -\frac{\phi^D(P_D)}{Z^D(P_D)}
				\end{align}
				so that $(\phi^D + \beta(\phi^D) Z^D)(P_D)=0$. 
				Let $\mathcal{T}_H$ be the operator defined in Proposition~\ref{propH2} and let us write
				\begin{align}
					\label{phiH1}
					\phi^H = \mathcal{T}_H[ \phi^D , \beta(\phi^D), f_1,f_2].
				\end{align}

				Let $L_D$ be the linear operator constructed in Proposition~\ref{propDisc1} that to $f_1 \in C^\alpha(\overline{\tilde\Omega^D})$ and $f_2 \in C^{1,\alpha}(\partial\tilde\Omega^D)$, assumed to be even in $x_1$, assigns $\phi = L_D[f_1,f_2]\in C^{2,\alpha}(\overline{\tilde \Omega^D})$, even in $x_1$, and $g \in \R$ that solve \eqref{pDiscProj}. Moreover we have
				\begin{align}
					\label{estLD}
					\| L_D[f_1,f_2] \|_{C^{2, \alpha}, \tilde \Omega^D}  + |g| \leq C \left( \|f_1\|_{C^\alpha, \tilde \Omega^D}+ \|f_2\|_{C^{1, \alpha},\partial \tilde \Omega^D} \right).
				\end{align}
				Let us write
				\begin{align}
					\label{lambda2}
					\lambda = \Lambda_S[\mathcal{T}_H[\phi^D,\beta(\phi^D),f_1,f_2],f_1,f_2], 
				\end{align}
				which is the same as \eqref{alpha} but with $\beta$ given by \eqref{beta}.

				With the notation for $\beta(\phi^D)$, $\phi^H$, $\lambda$ defined in \eqref{beta}, \eqref{phiH1}, and \eqref{lambda2}, we let
				\begin{align}
					\nonumber
					\mathcal{A}_D[\phi^D,f_1,f_2] = L_D[G_1[\phi^D,f_1,f_2],G_2[\phi^D,f_1,f_2]] 
				\end{align}
				where
				\begin{align*}
					G_1[\phi^D,f_1,f_2] 
					&= 
					- \varepsilon \lambda \Delta(Z^H(y) \chi^H)
					-2 \nabla_y \phi^H \cdot \nabla \chi^H
					- \varepsilon \phi^H \Delta \chi^H
					+ \tilde f_1 \chi_2^D
					\\
					G_2[\phi^D,f_1,f_2] 
					&= 
					- \varepsilon \lambda \Bigl( \frac{\partial (Z^H \chi^H)}{\partial \nu} +  (\tilde\kappa-\omega) Z^H \chi^H\Bigr) 
					-\beta(\phi^D) \Bigl( \frac{\partial Z^D}{\partial \nu} +  (\tilde\kappa-\omega) Z^D\Bigr) 
					\\
					& \qquad 
					- \varepsilon \phi^H \frac{\partial \chi^H}{\partial \nu}
					+ \tilde f_{2,0} \chi_2^D + g x_2  ,
                    \\
				\tilde f_{2,0}(x) &= \frac{1}{\varepsilon} f_{2,0} \Bigl(\frac{x}{\varepsilon}\Bigr),
				\end{align*}
				and $f_{2,0}$ is the function in the decomposition \eqref{decompf2}.
				To find a solution to \eqref{gDisc} it is sufficient to find $\phi^D$ such that 
				\begin{align}
					\label{fixedAD}
					\phi^D &= \mathcal{A}_D[\phi^D,f_1,f_2] .
				\end{align}

				\medskip

				Let us prove that $\mathcal{A}_D$ has a fixed point. 
				By \eqref{estLD}
				\begin{align}
					\nonumber
					\| \mathcal{A}_D[\phi^D,f_1,f_2]  \|_{C^{2, \alpha}, \tilde \Omega^D}
					&=
					\|  L_D[G_1[\phi^D,f_1,f_2],G_2[\phi^D,f_1,f_2]] \|_{C^{2, \alpha}, \tilde \Omega^D}
					\\
					\label{estAD}
					& \leq C \| G_1[\phi^D,f_1,f_2] \|_{C^\alpha,\tilde\Omega^D}
					+  C \| G_2[\phi^D,f_1,f_2] \|_{C^{1,\alpha},\partial\tilde\Omega^D} .
				\end{align}
				We claim that
				\begin{align}
					\nonumber
					\| G_1[\phi^D,f_1,f_2] \|_{C^\alpha,\tilde\Omega^D}
					& 
					\leq C   \delta_2^{-3-2\alpha-\sigma}  \delta_1^\sigma
					\| \phi^D \|_{C^{2, \alpha}, \tilde \Omega^D} + C \| \tilde f_1 \chi_2^D\|_{C^\alpha,\tilde\Omega^D}
					\\
					\nonumber
					& \quad 
					+ C \varepsilon^{\sigma}  \delta_2^{-3-2\alpha-\sigma} ( \| f_1 \chi_2^H \|_{\sharp\sharp ,\tilde\Omega^H}
					+ \| f_2 \chi_2^H \|_{\sharp\sharp ,\partial\tilde\Omega^H} 
					)
					\\
					\label{estG1x}
					& \quad + C \delta_2^{-(2+\alpha)} (\| \tilde f_1 \chi_2^S \|_{C^\alpha,\tilde \Omega^S}+  \| \tilde f_2 \chi_2^S \|_{C^{1,\alpha},\partial \tilde \Omega^S}).
				\end{align}

				The estimate of the terms $-2 \nabla_y \phi^H \cdot \nabla \chi^H
				- \varepsilon \phi^H \Delta \chi^H$ in $G_1$ is similar to the one in the proof of Proposition~\ref{propS2} and is given by 
				\begin{align*}
					\| 2 \nabla_y \phi^H \cdot \nabla \chi^H
					+\varepsilon \phi^H \Delta \chi^H\|_{C^\alpha,\tilde\Omega^D} 
					\leq 
					C \varepsilon^{\sigma}  \delta_2^{-1-\alpha-\sigma} \| \phi^H\|_{\sharp, \tilde\Omega^H}  .
				\end{align*}
But $\phi^H$ is given by \eqref{phiH1} so, using \eqref{estH} and the definition of $\beta$ \eqref{beta}, we get
\begin{align}
\nonumber
& \| 2 \nabla_y \phi^H \cdot \nabla \chi^H
+\varepsilon \phi^H \Delta \chi^H\|_{C^\alpha,\tilde\Omega^D} 
\\
\nonumber
& \quad \leq 
C\delta_2^{-1-\alpha-\sigma}\delta_1^{\sigma} ( \| \phi^D \|_{C^{2,\alpha},\tilde \Omega^D} 
+ \| \tilde f_1 \chi_2^S \|_{C^\alpha,\tilde \Omega^S} +  \| \tilde f_2 \chi_2^S \|_{C^{1,\alpha},\partial \tilde \Omega^S} )
\\
\label{G1a}
& \qquad 
+ C  \varepsilon^\sigma \delta_2^{-1-\alpha-\sigma} ( \| f_1 \chi_2^H \|_{\sharp\sharp ,\Omega^H}
+ \| f_2 \chi_2^H \|_{\sharp\sharp ,\tilde\Omega^H} ).
\end{align}

				From \eqref{LZchiH} and \eqref{estLambda} we get
				\begin{align}
					\nonumber
					\Bigl\|  \varepsilon\lambda\Delta\Bigl (Z^H\Bigl(\frac{\cdot}{\varepsilon}\Bigr) \chi^H \Bigr) \Bigr\|_{C^\alpha,\tilde\Omega^D} 
					&\leq 
					C   \delta_2^{-3-2\alpha-\sigma}  \delta_1^\sigma
					( \| \phi^D \|_{C^{2, \alpha}, \tilde \Omega^D} +|\beta| )
					\\
					\nonumber
					& \quad 
					+ C \varepsilon^{\sigma}  \delta_2^{-3-2\alpha-\sigma} ( \| f_1 \chi_2^H \|_{\sharp\sharp ,\tilde\Omega^H}
					+ \| f_2 \chi_2^H \|_{\sharp\sharp ,\partial\tilde\Omega^H} 
					)
					\\
					\label{G1b}
					& \quad + C \delta_2^{-(2+\alpha)} ( \| \tilde f_1 \chi_2^S \|_{C^\alpha,\tilde \Omega^S}+  \| \tilde f_2 \chi_2^S \|_{C^{1,\alpha},\partial \tilde \Omega^S}).
				\end{align}
				Combining \eqref{G1a} and \eqref{G1b} we obtain \eqref{estG1x}.
				
				\medskip
				
				Next we estimate $\| G_2[\phi^D,f_1,f_2]  \|_{C^{1,\alpha},\partial\tilde\Omega^D}$.
				We claim that
				\begin{align}
					\nonumber
					\| G_2[\phi^D,f_1,f_2] \|_{C^{1,\alpha}, \tilde \Omega^D}
					& 
					\leq C   (\delta_2^{-3-2\alpha-\sigma}  \delta_1^\sigma +  \delta_1^{1-\alpha}+\varepsilon|\log\varepsilon| \delta_1^{-3-\alpha})
					\| \phi^D \|_{C^{2, \alpha}, \tilde \Omega^D} + C \| \tilde f_2\|_{C^{1,\alpha},\partial\tilde\Omega^D}
					\\
					\nonumber
					& \quad 
					+ C \varepsilon^{\sigma}  \delta_2^{-3-2\alpha-\sigma} ( \| f_1 \chi_2^H \|_{\sharp\sharp ,\tilde\Omega^H}
					+ \| f_2 \chi_2^H \|_{\sharp\sharp ,\partial\tilde\Omega^H} 
					)
					\\
					\label{estG2x}
					& \quad + C \delta_2^{-(2+\alpha)} (\| \tilde f_1 \chi_2^S \|_{C^\alpha,\tilde \Omega^S}+  \| \tilde f_2 \chi_2^S \|_{C^{1,\alpha},\partial \tilde \Omega^S}) .
				\end{align}
				Indeed, from estimate \eqref{BopZchiH} in combination with \eqref{estLambda} we get
				\begin{align}
					\nonumber
					\varepsilon |\lambda | &
					\Bigl\| 
					\frac{\partial (Z^H \chi^H)}{\partial \nu} +  (\tilde\kappa-\omega) Z^H \chi^H
					\Bigr\|_{C^\alpha,\tilde \Omega^S}
					\\
					\nonumber
					&\leq 
					C \delta_1^\sigma \delta_2^{-3-2\alpha-\sigma}
					\left[
					\| \phi^D \|_{C^{2, \alpha}, \tilde \Omega^D} +|\beta| \right]
					\\
					\nonumber
					& \qquad 
					+  C  \delta_2^{-2-\alpha} \left[ \| \tilde f_1 \chi_2^S \|_{C^\alpha,\tilde \Omega^S}
					+  \| \tilde f_2 \chi_2^S \|_{C^{1,\alpha},\partial \tilde \Omega^S}  \right]
					\\
					\label{G2x1}
					& \qquad 
					+ C  \varepsilon^{\sigma} \delta_2^{-3-2\alpha-\sigma} \left[ \| f_1 \chi_2^H \|_{\sharp\sharp ,\tilde\Omega^H}
					+  \| f_2 \chi_2^H \|_{\sharp\sharp ,\partial\tilde\Omega^H}
					\right] .
				\end{align}
				A direct calculation also gives
				\begin{align*}
					\Bigl\| \varepsilon \phi^H \frac{\partial \chi^H}{\partial \nu}\Bigr\|_{C^{1,\alpha},\tilde\partial\Omega^D}
					\leq 
					C \varepsilon^{\sigma}  \delta_2^{-1-\alpha-\sigma} \| \phi^H\|_{\sharp, \tilde\Omega^H}  .
				\end{align*}
				Therefore, using  \eqref{phiH1}, \eqref{estH}  we get
				\begin{align}
					\nonumber
					\Bigl\| \varepsilon \phi^H \frac{\partial \chi^H}{\partial \nu}\Bigr\|_{C^{1,\alpha},\tilde\partial\Omega^D}
					& \leq C \delta_2^{-1-\alpha-\sigma}\delta_1^{\sigma}  \| \phi^D \|_{C^{2, \alpha}, \tilde \Omega^D} 
					\\
					\nonumber
					& \quad +
					C\delta_2^{-1-\alpha-\sigma}\delta_1^{\sigma} 
					\left[ \| \tilde f_1 \chi_2^S \|_{C^\alpha,\tilde \Omega^S} +  \| \tilde f_2 \chi_2^S \|_{C^{1,\alpha},\partial \tilde \Omega^S} \right]
					\\
					\label{G2x2}
					& \quad 
					+ C  \varepsilon^\sigma \delta_2^{-1-\alpha-\sigma} \left[ \| f_1 \chi_2^H \|_{\sharp\sharp ,\Omega^H}
					+ \| f_2 \chi_2^H \|_{\sharp\sharp ,\tilde\Omega^H} \right].
				\end{align}

				Regarding the term $\beta(\phi^D) ( \frac{\partial Z^D}{\partial \nu} +  (\tilde\kappa-\omega) Z^D)  $ we use that $|\beta(\phi^D)|\leq C \|\phi^D\|_{C^{2,\alpha},\tilde\Omega^D}$ and the following inequality
				\begin{align}
					\label{G2x3}
					\Bigl\|  \frac{\partial Z^D}{\partial \nu} +  (\tilde\kappa-\omega) Z^D \Bigr\|_{C^{1,\alpha},\tilde\Omega^D}
					\leq C \delta_1^{1-\alpha}+C\frac{\varepsilon|\log\varepsilon|}{\delta_1^{3+\alpha}},\end{align}
				from Lemma~\ref{lemmaNormZD}.
				
				Combining \eqref{G2x1}, \eqref{G2x2} and \eqref{G2x3} we obtain \eqref{estG2x}.
				
				\medskip
				
				From  \eqref{estAD}, \eqref{estG1x} and \eqref{estG2x}, 
				we see that if $\delta_2>0$, then $\delta_1>0$ is fixed sufficiently small, and then $\varepsilon>0$ is small, there is a unique $\phi^D$ satisfying the fixed point problem \eqref{fixedAD} and $\phi^D$ has the estimate
				\begin{align}
					\label{estPhiD2}
					\|\phi^D\|_{C^{2,\alpha},\tilde\Omega^D}
					&  \leq C F_D ,
				\end{align}
				where 
				\begin{align*}
					F_D&=  \| \tilde f_1 \chi_2^D\|_{C^\alpha,\tilde\Omega^D}
					+   \| \tilde f_2 \chi_2^D\|_{C^{1,\alpha},\partial\tilde\Omega^D}.
					\\
					& \quad +  \varepsilon^{\sigma}  \delta_2^{-3-2\alpha-\sigma} ( \| f_1 \chi_2^H \|_{\sharp\sharp ,\tilde\Omega^H}
					+ \| f_2 \chi_2^H \|_{\sharp\sharp ,\partial\tilde\Omega^H} 
					)
					\\
					& \quad +  \delta_2^{-(2+\alpha)}\| \tilde f_1 \chi_2^S \|_{C^\alpha,\tilde \Omega^S}+  \| \tilde f_2 \chi_2^S \|_{C^{1,\alpha},\partial\tilde \Omega^S}).
				\end{align*}
				
				We will now obtain estimates for $\phi^H$ and $\phi^S$.
				Using  \eqref{estH} and \eqref{estPhiD2} we obtain
				\begin{align}
					\label{estPhiH4}
					\|  \phi^H \|_{\sharp, \tilde\Omega^H}  \leq C F_H ,
				\end{align}
				where
				\begin{align*}
					F_H &=
					\varepsilon^{-\sigma} \delta_1^\sigma\left[ 
					\| \tilde f_1 \chi_2^D\|_{C^\alpha,\tilde\Omega^D} 
					+   \| \tilde f_2 \chi_2^D\|_{C^{1,\alpha},\partial\tilde\Omega^D} \right]
					\\
					& \quad +  \| f_1 \chi_2^H \|_{\sharp\sharp ,\tilde\Omega^H}
					+ \| f_2 \chi_2^H \|_{\sharp\sharp ,\partial\tilde\Omega^H} 
					\\
					& \quad
					+\varepsilon^{-\sigma} \delta_1^\sigma \delta_2^{-(2+\alpha)} \left[  \| \tilde f_1 \chi_2^S \|_{C^\alpha,\tilde \Omega^S}+  \| \tilde f_2 \chi_2^S \|_{C^{1,\alpha},\partial \tilde \Omega^S} \right]
				\end{align*}
				assuming that 
				\begin{align}
					\label{del12}
					\delta_1^\sigma  \delta_2^{-3-2\alpha-\sigma}  \leq 1 ,
				\end{align}
				which is compatible with the statement of the proposition.
				
				We also get, using \eqref{TS} and \eqref{estPhiH4}
				\begin{align}
					\label{estPhiS}
					\|  \phi^S \|_{C^{2, \alpha}, \tilde \Omega^S}
					& \leq C F_S
				\end{align}
				where
				\begin{align*}
					F_S &= 
					\| \tilde f_1 \chi_2^S \|_{C^\alpha, \tilde \Omega^S}
					+ \| \tilde f_2 \chi_2^S \|_{C^{1,\alpha},\partial \tilde \Omega^S}
					\\
					&\quad + 
					\varepsilon^{\sigma}  \delta_2^{-1-\alpha-\sigma}
					\left[ \| f_1 \chi_2^H \|_{\sharp\sharp ,\tilde\Omega^H}
					+ \| f_2 \chi_2^H \|_{\sharp\sharp ,\partial\tilde\Omega^H} \right]
					\\
					&\quad + 
					\delta_2^{-1-\alpha-\sigma}\delta_1^\sigma
					\left[ 
					\| \tilde f_1 \chi_2^D\|_{C^\alpha,\tilde\Omega^D} 
					+   \| \tilde f_2 \chi_2^D\|_{C^{1,\alpha},\partial\tilde\Omega^D} \right]
				\end{align*}
				assuming  \eqref{del12}.

				Let us prove the estimate for $g$. For this we multiply the equation \eqref{gDisc} by $Z^D$ and integrate. We get
				\begin{align*}
					\int_{\tilde\Omega^D}
					& 
					[ \tilde f_1 \chi_2^D - 2 \nabla_y \phi^H \cdot \nabla \chi^H
					-\varepsilon \phi^H \Delta \chi^H ] Z^D \, dx
					\\
					&=
					- \int_{\partial\tilde\Omega^D}
					\Bigl( \frac{\partial Z^D}{\partial\nu} + (\tilde\kappa-\omega) Z^D \Bigr) \phi^D
					- \varepsilon \lambda 
					\int_{\partial\tilde\Omega^D}
					\Bigl( \frac{\partial Z^D}{\partial\nu} + (\tilde\kappa-\omega) Z^D \Bigr) Z^H \chi^H
					\\
					& \quad 
					-\beta   
					\int_{\partial\tilde\Omega^D}
					\Bigl( \frac{\partial Z^D}{\partial\nu} + (\tilde\kappa-\omega) Z^D \Bigr) Z^D
					- \varepsilon \int_{\partial\tilde\Omega^D}\phi^H \frac{\partial \chi^H}{\partial \nu} Z^D
					\\
					& \quad 
					+ \int_{\partial\tilde\Omega^D}
					\tilde f_{2,0}  Z^D 
					+ g \int_{\partial\tilde\Omega^D}
					(x_2+d^S)  Z^D .
				\end{align*}
				Since 
				\begin{align}
					\label{g0}
					\lim_{\delta_1\to 0} \int_{\partial\tilde\Omega^D} (x_2+d^S)  Z^D \not=0,
				\end{align}
				to get an estimate for $g$ it is sufficient to show that the dependence on $g$ of the remaining terms is of the form $o(1)|g|$, choosing appropriately $\delta_1$, $\delta_2$ and then letting $\varepsilon\to 0$.
				
				Using \eqref{estLambda} and \eqref{estPhiD2} we get, recalling the notation for  $\lambda$ in \eqref{lambda2}
				\begin{align}
					\nonumber
					\varepsilon | \lambda |
					& \leq 
					C \delta_1^\sigma  \delta_2^{-1-\alpha-\sigma}  \left[ 
					\| \tilde f_1 \chi_2^D\|_{C^\alpha,\tilde\Omega^D} 
					+   \| \tilde f_2 \chi_2^D\|_{C^{1,\alpha},\partial\tilde\Omega^D} \right]
					\\
					\nonumber
					& \quad 
					+ C \varepsilon^{\sigma}  \delta_2^{-1-\alpha-\sigma} 
					\left[ \| f_1 \chi_2^H \|_{\sharp\sharp ,\tilde\Omega^H}
					+ \| f_2 \chi_2^H \|_{\sharp\sharp ,\partial\tilde\Omega^H} \right]
					\\
					\label{lambdaeps1}
					& \quad 
					+ C \| \tilde f_1 \chi_2^S \|_{C^\alpha,\tilde \Omega^S} + C \| \tilde f_2 \chi_2^S \|_{C^{1,\alpha},\partial \tilde \Omega^S} ,
				\end{align}
				assuming again \eqref{del12}.
				
				In this bound for $\varepsilon\lambda$ the terms that depend on $g$ are the ones that depend on $f_2$ because of the decomposition \eqref{decompf2}.
				We use the following estimates:
				\begin{align}
					\label{gInD}
					\| (x_2+d^S)\chi_2^D \|_{C^{1,\alpha}, \partial\tilde\Omega^D } & \leq C ,
				\end{align}
				which follows from Lemma~\ref{lemmaNormZD}, and 
				\begin{align}
					\label{gInH}
					\left\| (\varepsilon y_2 + d^S)\chi_2^H\right\|_{\sharp\sharp ,\partial\tilde\Omega^H}
					 &\leq C \varepsilon^{-\sigma}\delta_2^{2+\sigma}
                    \\
					\label{gInS}
					\| (x_2+d^S)\chi_2^S \|_{C^{1,\alpha},\partial\tilde\Omega^D }  &\leq C \frac{\varepsilon}{\delta_1^{1+\alpha}}.
				\end{align}
				The proof of \eqref{gInH} and \eqref{gInS} are direct computations.
				
				Using the decomposition \eqref{decompf2} and  \eqref{gInD}--\eqref{gInS}, we get from \eqref{lambdaeps1} that
\begin{align}
\nonumber
\varepsilon |\lambda|
& \leq 
C \delta_1^\sigma  \delta_2^{-1-\alpha-\sigma}  \left[ 
\| \tilde f_1 \chi_2^D\|_{C^\alpha,\tilde\Omega^D} 
+   \| \tilde f_{2,0} \chi_2^D\|_{C^{1,\alpha},\partial\tilde\Omega^D} \right]
\\
\nonumber
& \quad 
+ C \varepsilon^{\sigma}  \delta_2^{-1-\alpha-\sigma} 
\left[ \| f_1 \chi_2^H \|_{\sharp\sharp ,\tilde\Omega^H}
+ \| f_{2,0} \chi_2^H \|_{\sharp\sharp ,\partial\tilde\Omega^H} \right]
\\
\nonumber
& \quad 
+ C \| \tilde f_1 \chi_2^S \|_{C^\alpha,\tilde\Omega^S} + C \| \tilde f_{2,0} \chi_2^S \|_{C^{1,\alpha},\partial\tilde\Omega^S}
\\
\nonumber
& \quad +C |g| ( \delta_1^\sigma  \delta_2^{-1-\alpha-\sigma} + \varepsilon\delta_1^{-1-\alpha}).
\end{align}	
Since the support of $\chi^H$ is contained in a ball of radius $\delta_2$,
\begin{align}
\nonumber
& \left| \varepsilon \lambda
\int_{\partial\tilde\Omega^D} \Bigl( \frac{\partial Z^D}{\partial\nu} + (\tilde\kappa-\omega) Z^D \Bigr) Z^H \chi^H
\right|
\\
\nonumber
&\qquad  \leq  
C \delta_1^\sigma  \delta_2^{1-\alpha-\sigma}  \left[ 
\| \tilde f_1 \chi_2^D\|_{C^\alpha,\tilde\Omega^D} 
+   \| \tilde f_{2,0} \chi_2^D\|_{C^{1,\alpha},\partial\tilde\Omega^D} \right]
\\
\nonumber
& \qquad \quad 
+ C \varepsilon^{\sigma}  \delta_2^{1-\alpha-\sigma} 
\left[ \| f_1 \chi_2^H \|_{\sharp\sharp ,\tilde\Omega^H}
+ \| f_{2,0} \chi_2^H \|_{\sharp\sharp ,\partial\tilde\Omega^H} \right]
\\
\nonumber
& \qquad \quad 
+ C \delta_2^2 \left[ \| \tilde f_1 \chi_2^S \|_{C^\alpha,\tilde\Omega^S} +  \| \tilde f_{2,0} \chi_2^S \|_{C^{1,\alpha},\partial\tilde\Omega^S} \right]
\\
\label{g1}
& \qquad \quad +C |g|  ( \delta_1^\sigma  \delta_2^{-1-\alpha-\sigma} + \varepsilon\delta_1^{-1-\alpha}).
\end{align}
				
				Let us estimate $ \int_{\tilde\Omega^D} \varepsilon \phi^H \Delta \chi^H  Z^D  dx$. From \eqref{calcPhiH}
				\begin{align}
					\nonumber
					|  \varepsilon \phi^H \Delta \chi^H (x)|
					& \leq C \varepsilon^{\sigma} \delta_2^{-1-\sigma} \|\phi^H\|_{\sharp , \tilde\Omega^H}
				\end{align}
				and this function is supported in a region of area $O(\delta_2^2)$, so, in combination with \eqref{estPhiH4} we get
				\begin{align*}
					\left|\int_{\tilde\Omega^D}
					\varepsilon \phi^H \Delta \chi^H  Z^D  dx\right|
					& \leq C \varepsilon^{\sigma} \delta_2^{1-\sigma} \|\phi^H\|_{\sharp, \tilde\Omega^H}
					\\ 
					& \leq C  \delta_1^\sigma \delta_2^{1-\sigma} \left[ 
					\| \tilde f_1 \chi_2^D\|_{C^\alpha,\tilde\Omega^D} 
					+   \| \tilde f_2 \chi_2^D\|_{C^{1,\alpha},\partial\tilde\Omega^D} \right]
					\\
					& \quad
					+ C \varepsilon^{\sigma} \delta_2^{1-\sigma} \left[
					\| f_1 \chi_2^H \|_{\sharp\sharp ,\tilde\Omega^H}
					+ \| f_2 \chi_2^H \|_{\sharp\sharp ,\partial\tilde\Omega^H} 
					\right]
					\\
& \quad + 
C \delta_1^\sigma  \delta_2^{-1-\sigma-\alpha}\left[ 
\| \tilde f_1 \chi_2^S \|_{C^\alpha,\tilde\Omega^S}+  \| \tilde f_2 \chi_2^S \|_{C^{1,\alpha},\partial\tilde\Omega^S}
\right].
				\end{align*}
				
				Using the decomposition \eqref{decompf2} we find
				\begin{align}
					\label{g2}
					\left|\int_{\tilde\Omega^D}
					\varepsilon \phi^H \Delta \chi^H  Z^D  dx\right|
					& \leq C M + C \left(\delta_1^\sigma \delta_2^{1-\sigma} + \delta_2^3 + \varepsilon \delta_1^{\sigma-1-\alpha}\delta_2^{-1-\sigma-\alpha}\right)  |g|,
				\end{align}
				where
				\begin{align*}
					M&= \delta_1^\sigma \delta_2^{1-\sigma} \left[ 
					\| \tilde f_1 \chi_2^D\|_{C^\alpha,\tilde\Omega^D} 
					+   \| \tilde f_{2,0} \chi_2^D\|_{C^{1,\alpha},\partial\tilde\Omega^D} \right]
					\\
					& \quad
					+ C \varepsilon^{\sigma} \delta_2^{1-\sigma} \left[
					\| f_1 \chi_2^H \|_{\sharp\sharp,\tilde\Omega^H}
					+ \| f_{2,0} \chi_2^H \|_{\sharp\sharp, \partial\tilde\Omega^H} 
					\right]
					\\
					& \quad + 
					C \delta_1^\sigma  \delta_2^{-1-\sigma-\alpha}\left[ 
					\| \tilde f_1 \chi_2^S \|_{C^\alpha,\tilde\Omega^S}+  \| \tilde f_{2,0} \chi_2^S \|_{C^{1,\alpha},\partial\tilde\Omega^S})
					\right] .
				\end{align*}
				
				Similar computations also give
				\begin{align}
					\nonumber
					&\left| \int_{\tilde\Omega^D}
					[2 \nabla_y \phi^H \cdot \nabla \chi^H
					-\varepsilon \phi^H \Delta \chi^H ] Z^D \, dx
					\right|
					\\
					\label{g3}
					& \qquad \leq C M +C \left(\delta_1^\sigma \delta_2^{1-\sigma} + \delta_2^3 + \varepsilon \delta_1^{\sigma-1-\alpha}\delta_2^{-1-\sigma-\alpha}\right)  |g|,
				\end{align}
				and
				\begin{align}
					\label{g4}
					\varepsilon\left|
					\int_{\partial\tilde\Omega^D}
					\phi^H \frac{\partial \chi^H}{\partial \nu} Z^D
					\right|
					& \leq C M + C \left(\delta_1^\sigma \delta_2^{1-\sigma} + \delta_2^3 + \varepsilon \delta_1^{\sigma-1-\alpha}\delta_2^{-1-\sigma-\alpha}\right)  |g|.
				\end{align}
				
				We note that 
				\begin{align}
					\label{g5}
					\left|
					\int_{\partial\tilde\Omega^D}
					\Bigl( \frac{\partial Z^D}{\partial\nu} + (\tilde\kappa-\omega) Z^D \Bigr) Z^D
					\right| \leq C \delta_1^{2-\alpha}+C\frac{\varepsilon|\log\varepsilon|}{\delta_1^{2+\alpha}},
				\end{align}
				by \eqref{errorZD}.
				
				Combining the information in \eqref{g0}, \eqref{g1}, \eqref{g2}, \eqref{g3}, \eqref{g4}, \eqref{g5} we find that 
				\begin{align*}
					|g| & \leq C   \left[ 
					\| \tilde f_1 \chi_2^D\|_{C^\alpha,\tilde\Omega^D} 
					+   \| \tilde f_{2,0} \chi_2^D\|_{C^{1,\alpha},\partial\tilde\Omega^D} \right]
					\\
					\nonumber
					& \quad 
					+ C \varepsilon^{\sigma}  \delta_1^{3-\alpha} \delta_2^{-3-2\alpha-\sigma} 
					\left[ \| f_1 \chi_2^H \|_{\sharp\sharp, \tilde\Omega^H}
					+ \| f_{2,0} \chi_2^H \|_{\sharp\sharp ,\partial\tilde\Omega^H} \right]
					\\
					& \quad + 
					C \delta_1^\sigma  \delta_2^{-1-\sigma-\alpha}\left[ 
					\| \tilde f_1 \chi_2^S \|_{C^\alpha, \tilde\Omega^S}+  \| \tilde f_{2,0} \chi_2^S \|_{C^{1,\alpha},\partial \tilde\Omega^S}
					\right] .
				\end{align*}
				This proves \eqref{estg3}.
				Combining \eqref{estg3} with \eqref{estPhiD2} we obtain \eqref{estphiD}  and this concludes the proof.
			\end{proof}
			
			\subsection{Proof of Proposition~\ref{prop:main-linear}}
			\label{sectProofSys}
			
			To simplify the notation, in what follows we fix $0<\delta_1<\delta_2$ so that the previous results in this section hold for all $\varepsilon>0$ small.
			We will not keep track of the dependence of the constants on $\delta_1$, $\delta_2$.
			
			Proposition~\ref{propPhiD2} gives a solution 
			$\phi^D$, $\phi^S$, $\phi^H$, $\lambda$, $\beta$, $g$ of the system \eqref{gHairpin}, \eqref{gStrip}, \eqref{gDisc}. The formula 
			\begin{align}
				\label{phi2}
				\phi(y) &=  ( \phi^H(y) + \lambda Z^H(y)) \chi^H(x)
				+ \frac{1}{\varepsilon}\phi^{S}(x)\chi^{S}(x)
				+ \frac{1}{\varepsilon}(\phi^{D}(x)+ \beta Z^D(x))\chi^D(x),
			\end{align}
			with $y=\frac{x}{\varepsilon}$
			(this is \eqref{phi}) gives a solution to \eqref{main-linear}, which is a linear operator of the right hand sides $f_1$ and $f_{2,0}$. We collect the estimates for the different terms of $\phi$. From \eqref{estPhiD2}, \eqref{estPhiH4}, and \eqref{estPhiS} we find that 
			\begin{align}
				\label{est-parts-phi}
				\|\phi^H\|_{\sharp,\tilde\Omega^H}
				+ \varepsilon^{-\sigma}\|\phi^D\|_{C^{2,\alpha},\tilde\Omega^D}
				+ \varepsilon^{-\sigma} \|\phi^S\|_{C^{2,\alpha},\tilde\Omega^S}
				&\leq C M_1 + C M_2 
                ,
			\end{align}
			where
			\begin{align*}
				M_1&=
				\| f_1 \chi_2^H \|_{\sharp\sharp ,\tilde\Omega^H}
				+  
				\varepsilon^{-\sigma} 
				\| \tilde f_1 \chi_2^D\|_{C^\alpha,\tilde\Omega^D} 
				+\varepsilon^{-\sigma}  \| \tilde f_1 \chi_2^S \|_{C^\alpha,\tilde \Omega^S},
				\\
				M_2&=
				\| f_{2,0} \chi_2^H \|_{\sharp\sharp ,\partial\tilde\Omega^H}  
				+   \varepsilon^{-\sigma}  \| \tilde f_{2,0}\chi_2^D\|_{C^{1,\alpha},\partial\tilde\Omega^D}
				+\varepsilon^{-\sigma}    \| \tilde f_{2,0} \chi_2^S \|_{C^{1,\alpha},\partial\tilde \Omega^S}  .
			\end{align*}
			We note that for $\varepsilon>0$ small, 
			$M_1$ is equivalent to the norm $ \| f_1 \|_{**, \Omega_0 }$ defined in \eqref{defNf1Omega0} and $M_2$ is equivalent to the norm $ \| f_{2,0} \|_{**, \partial \Omega_0 }$ defined in \eqref{defNf2Omega0}. In particular,
			\begin{align*}
				M_1\leq C  \| f_1 \|_{**, \Omega_0 },\quad
				M_2\leq C  \| f_{2,0} \|_{**, \partial \Omega_0}.
			\end{align*}
			
			Let us rewrite \eqref{phi2} as
			\begin{align*}
				\phi(y) = \phi_0(y) + \lambda Z^H(y) \chi^H(\varepsilon y)
			\end{align*}
			where
			\begin{align*}
				\phi_0(x) = \phi^H(y) \chi^H(x)
				+ \frac{1}{\varepsilon}\phi^{S}(x)\chi^{S}(x)
				+ \frac{1}{\varepsilon}(\phi^{D}(x)+ \beta Z^D(x))\chi^D(x),\quad x = \varepsilon y.
			\end{align*}
			Then from  \eqref{est-parts-phi} we get that, for the norm defined in \eqref{norm-phi},
			\begin{align*}
				\|\phi_0\|_{*, \Omega_0 } &\leq C ( \|f_1\|_{**, \Omega_0 } +  \|f_{2,0}\|_{**, \partial \Omega_0}).
			\end{align*}
			On the other hand, estimate \eqref{estLambda} gives
			\begin{align*}
				\| \lambda Z^H \chi^H(\varepsilon \, \cdot ) \|_{*, \Omega_0}  \leq C \varepsilon^{-1+\sigma}  ( \|f_1\|_{**, \Omega_0} +  \|f_{2,0}\|_{**, \partial \Omega_0 }).
			\end{align*}
			Therefore
			\begin{align*}
				\|\phi\|_{*, \Omega_0}\leq C \varepsilon^{-1+\sigma}  ( \|f_1\|_{**, \Omega_0 } +  \|f_{2,0}\|_{**, \partial\Omega_0 }).
			\end{align*}
			For $g$ we have, by \eqref{estg3}
			\begin{align*}
				|g|\leq C \varepsilon^{\sigma}  ( \|f_1\|_{**, \Omega_0 } +  \|f_{2,0}\|_{**, \partial \Omega_0 }).
			\end{align*}

\section{Proof of Theorem~\ref{thm}}
\label{sec:proof}

The proof consists in solving the nonlinear problem \eqref{main-problem}
and then checking that the map $F[h]$ defined by \eqref{def:Fh}, where $h$ is given by \eqref{def:a}, is a diffeomorphism.

\subsection{Solving problem (\ref{main-problem})}
\label{sect:solve}

We rewrite \eqref{main-problem} as
			\begin{align}
				\label{main-problem3}
				\left\{
				\begin{aligned}
					\Delta_y u + Q_1[u] &=0 && \text{in }\Omega_0
					\\
					\frac{\partial u}{\partial\nu}
					+ (\kappa - \varepsilon\omega)u + E_0(g_1) + Q_2[u,\varepsilon g_0+g_1] &= -g_1 ( \varepsilon y_2+d^S)
					&& \text{on } \mathcal S_0 
					\\
					u&=0 
					&& \text{on } \mathcal B_0 
				\end{aligned}
				\right.
			\end{align}
			where $E_0(g_1)$ is defined in \eqref{def:E0}.

   \subsubsection{Formulation as a fixed point problem}
			Let $L$ denote the linear operator constructed in Proposition~\ref{prop:main-linear}, which to functions $f_1$ and $f_2$ satisfying  $\|f_1\|_{**, \Omega_0}<\infty$ and $\|f_2\|_{**, \partial \Omega_0}<\infty$ associates $(\phi,g_1)=L[f_1,f_2]$ that satisfy 
			\begin{align}
				\nonumber
				\left\{
				\begin{aligned}
					\Delta_y \phi  &= f_1  && \text{in } \Omega_0
					\\
					\frac{\partial \phi}{\partial\nu}
					+ (\kappa-\varepsilon\omega) \phi &= f_2-g_1 ( \varepsilon y_2+d^S) 
					&& \text{on } \mathcal{S}_0
					\\
					\phi &= 0 && \text{on } \mathcal{B}_0,
				\end{aligned}
				\right.
			\end{align}
			and 
			\begin{align}
				\label{L1b}
				\|\phi\|_{*,\Omega_0} &\leq C \varepsilon^{\sigma-1}( \|f_1\|_{**, \Omega_0}+\|f_2\|_{**, \partial \Omega_0} )
				\\
				\label{L2b}
				|g_1| &\leq C \varepsilon^{\sigma}( \|f_1\|_{**, \Omega_0 }+\|f_2\|_{**, \partial \Omega_0} ) .
			\end{align}
			We refer to \eqref{defNf1Omega0} and \eqref{defNf2Omega0} for the definition of the norms $\| \cdot \|_{**, \Omega_0}$ and $\| \cdot \|_{**, \partial \Omega_0}$, and to \eqref{norm-phi} for the definition of $\| \cdot \|_{*, \Omega_0}$.

			Then to find a solution to \eqref{main-problem3} it is sufficient that $u$, $g_1$ satisfy the fixed point problem
			\begin{align}
				\label{fixed}
				(u,g_1)=-L\big[ Q_1[u],E_0(g_1)+Q_2[u,\varepsilon g_0+g_1]\big].
			\end{align}
			Define the space
			\begin{align*}
				X = \{ \ u \in C^{2,\alpha}(\overline{\Omega}_0) \ | \ \|u\|_{*,\Omega_0} < \infty\}.
			\end{align*}
			For simplicity we assume in what remains of this section that
			\begin{align}
				\nonumber
				\frac12<b<\frac23
			\end{align}
			so that the hypotheses in Lemmas~\ref{lem:error1-2} and \ref{lem:error-u0} are satisfied, and 
			\begin{align}
				\label{sigmab}
				\min(1-\sigma b  ,4(1-b)\sigma b) = 1-\sigma b,
			\end{align}
			which appears in the exponent in Proposition~\ref{prop:est-error-bdy}.
			
			\medskip
			We solve \eqref{fixed} for $(u,g_1) \in B \times I$ where
			\begin{align*}
				B&=\{ \ u \in X \ | \ \|u\|_{*,\Omega_0} \leq M \varepsilon^{\sigma(1-b)}|\log\varepsilon| \ \} ,
				\\
				I&=\big[- M \varepsilon^{1+m} |\log \varepsilon| ,M \varepsilon^{1+m} |\log \varepsilon|\big],
			\end{align*}
			and $M>0$ is a constant that will be defined later on, and $0<m<\sigma(1-b)$.
			We note that the hypothesis on $g_1$ \eqref{hg} is then satisfied.
			We consider the map $\mathcal F(u,g_1)$, $(u,g_1)\in B\times I$ defined by 
			\begin{align*}
				\mathcal{F}(u,g_1) = -L[Q_1[u],E_0(g_1)+Q_2[u,g_1]].
			\end{align*}
			
			We claim that $\mathcal{F}(u,g_1)$ is well defined for $(u,g_1)\in B\times I$. Recall that $Q_1$ and $Q_2$ \eqref{newQ1}, \eqref{newQ2} are defined in terms of the function $\mathbf{h}(u,y)$ \eqref{def:a}, which gives the solution of \eqref{bcD} in terms of $u$. But in this equation we can solve for $h(y)$ as a function of $u(y)$ and $y$, if $u(y)$ is sufficiently small. 
			The hypothesis $u\in B$ is sufficient for this purpose (see the end of Section~\ref{sect:operators}).
			
			We introduce the norm
			\begin{align*}
				\|  h \|_{*,\partial\Omega_0} & = \text{the least $M$ such that for $y \in\partial\Omega_0$}
				\\
				\nonumber
				& | h(y)|
				+(1+|y|) | h'(y)|
				+(1+|y|)^2 | h'' \phi(y)|
				+(1+|y|)^{2+\alpha} [ h'']_{\alpha,B(y,\frac{|y|}{10})\cap\partial\Omega_0}
				\\
				\nonumber
				& \qquad \qquad  \qquad 
				\leq (1+|y|)^{1-\sigma} M 
				\qquad \qquad\text{for } |y| \leq \frac{\delta_1}{\varepsilon}
				\\
				\nonumber
				& 
				| h(y)| 
				+\Bigl(\frac{\delta_1}{\varepsilon}\Bigr) | h'(y)|
				+\Bigl(\frac{\delta_1}{\varepsilon}\Bigr)^2 | h''(y)|
				+\Bigl(\frac{\delta_1}{\varepsilon}\Bigr)^{2+\alpha} [ h'']_{\alpha,B(y,\frac{\delta_1}{10 \varepsilon})\cap\partial\Omega_0}
				\\
				\nonumber
				& \quad \qquad \qquad 
				\leq \Bigl(\frac{\varepsilon}{\delta_1}\Bigr)^{1-\sigma} \exp(-\mu \varepsilon |y|) M \qquad\qquad   \text{for } |y| \geq \frac{\delta_1}{\varepsilon} .
			\end{align*}
			This norm is similar to the norm $\| \ \|_{*,\Omega_0}$ in \eqref{norm-phi} but restricted to $\partial\Omega_0$.
			Using estimates for the derivatives of  $\psi_0$ and $\Psi[0]$, which can be obtained similarly as in Section~\ref{sect:error}, we find that
			\begin{align}
				\label{h-from-u}
				\| \mathbf{h}(u,\cdot)\|_{*,\partial\Omega_0} \leq C \| u\|_{*,\Omega_0},\quad u \in B.
			\end{align}
			There is a similar Lipschitz estimate
			\begin{align*}
				\| \mathbf{h}(u,\cdot) - \mathbf{h}(v,\cdot) \|_{*,\partial\Omega_0} \leq C \| u-v\|_{*,\Omega_0},\quad u,v\in B.
			\end{align*}
			
			In order for $\mathcal{F}$ to be well defined we also need to verify that the operators $Q_1$ and $Q_2$ are well defined, since they involve coefficients of the form  $(1-\kappa(t+h\eta))^{-a} (1+h\eta_t)^{-b}$ for some integers $a,b\geq 0$. This holds for $u\in B$ as discussed at the end of Section~\ref{sect:operators}.

   \subsubsection{Estimates}
			
                Next we will estimate the quantities $\|Q_1[u]\|_{**, \Omega_0}$, $\|E_0(g_1)\|_{**, \partial \Omega_0}$, and $\|Q_2[u,\varepsilon g_0+g_1]\|_{**, \partial \Omega_0}$ and also their Lipschitz constants.
                
                \medskip\paragraph{\textbf{Estimates for $\|Q_1[u]\|_{**, \Omega_0}$.}}
			We claim that 
			\begin{align}
				\label{estQ1}
				\|Q_1[u]\|_{**, \Omega_0} &\leq C \|u\|_{*,\Omega_0}^2, \quad u\in B,
			\end{align}
			and
			\begin{align}
				\label{estQ1-lip}
				\|Q_1[u]-Q_1[v]\|_{**, \Omega_0} &\leq C \varepsilon^{\sigma(1-b)} \, |\log\varepsilon| \,  \|u-v\|_{*,\Omega_0}, \quad u,v\in B .
			\end{align}
			
			In \eqref{rename-u} we have renamed $u=\tilde \Psi_1[h]$, therefore the expression $\tilde Q_1$ defined in \eqref{Q1} becomes
			\begin{align}
				\nonumber
				\tilde Q_1(u ,h) &=  L_h^{(1)}[u] - L_h^{(1)}[ \nabla_y \Psi[0] \cdot \nu \eta h ]
				\\
				\nonumber
				& \quad 
				+ L_h^{(2)}[u] +  L_h^{(2)}[\tilde \Psi[0] ] - L_h^{(2)}[ \nabla_y \Psi[0] \cdot \nu \eta h ]
				\\
				\nonumber
				& \quad 
				+ L^{(0)}[ \tilde \psi_0[h]  - \tilde \psi_0[0] + \nabla_y   \tilde \psi_0[0] \cdot \nu \eta h]
				\\
				& \quad 
				\nonumber
				+ L_h^{(1)} [\tilde \psi_0[h] - \tilde \psi_0[0]] 
				+ L_h^{(2)} [\tilde \psi_0[h]]  .
			\end{align}
			
			We estimate the terms in $L_h^{(1)}[u]$, for $|y|\leq \frac{\delta_1}{\varepsilon}$,
			\begin{align*}
				\Bigl|\frac{2 \kappa h \eta}{(1-\kappa t)^3} \partial_{ss} u\Bigr|&\leq C \frac{1}{1+|y|} (1+|y|)^{1-\sigma}  (1+|y|)^{-1-\sigma}\|u\|_{*,\Omega_0}\|h\|_{*,\partial\Omega_0}
				\\
				&\leq \frac{C}{(1+|y|)^{1+2\sigma}}  \|u\|_{*,\Omega_0}^2
				\\
				\Bigl|\frac{(h\eta)_s}{(1-\kappa t)^2}  \partial_{st}u\Bigr|
				&\leq \frac{C}{(1+|y|)^{1+2\sigma}}  \|u\|_{*,\Omega_0}\|h\|_{*,\partial\Omega_0}
				\\
				|2 h \eta_t \partial_{tt}u|
				&\leq \frac{C}{(1+|y|)^{1+2\sigma}}   \|u\|_{*,\Omega_0}^2.
			\end{align*}
			In the region $|y|\geq \frac{\delta_1}{\varepsilon}$ we have
			\begin{align*}
				\Bigl|\frac{2 \kappa h \eta}{(1-\kappa t)^3} \partial_{ss} u\Bigr|&\leq C \varepsilon \Bigl(\frac{\delta_1}{\varepsilon}\Bigr)^{1-\sigma} \exp(-\mu \varepsilon |y|) \|h\|_{*,\partial\Omega_0} \Bigl(\frac{\delta_1}{\varepsilon}\Bigr)^{-1-\sigma} \exp(-\mu \varepsilon |y|) \|u\|_{*,\Omega_0}
				\\
				&\leq 
				C \delta_1^{1-\sigma}\varepsilon^\sigma\Bigl(\frac{\delta_1}{\varepsilon}\Bigr)^{-1-\sigma} \exp(-\mu \varepsilon |y|) \|u\|_{*,\Omega_0}^2.
			\end{align*}
			The other terms are estimated similarly and we get
			\begin{align*}
				\| \hat L_{1,h}^{(1)}[u] \|_{**, \Omega_0 } \leq C \|u\|_{*,\Omega_0}^2.
			\end{align*}
			Similar estimates hold for the other terms of $Q_1$ and we obtain \eqref{estQ1}.
			
			Concerning the proof of \eqref{estQ1-lip}, the computations are similar as before. Let us consider them term $L_{h_u}^{(1)}[u]-L_{h_v}^{(1)}[v]$ where $h_u=\mathbf{h}(u,y)$, $h_v=\mathbf{h}(v,y)$ with $\mathbf{h}$ defined in \eqref{def:a}. We have, for $|y|\leq\frac{\delta_1}{\varepsilon}$:
			\begin{align*}
				&\Bigl|\frac{2 \kappa h_u \eta}{(1-\kappa t)^3} \partial_{ss} u-\frac{2 \kappa h_v \eta}{(1-\kappa t)^3} \partial_{ss} v\Bigr|
				\\
				&\quad=
				\Bigl|\frac{2 \kappa  \eta}{(1-\kappa t)^3} \Bigl[ h_u\partial_{ss} u-h_u\partial_{ss} u\Bigr] \Bigr|
				\\
				&\quad\leq C \frac{1}{1+|y|} 
				(1+|y|)^{1-\sigma}  
				\|h_u-h_v\|_{*,\partial\Omega_0}
				(1+|y|)^{-1-\sigma} \|u\|_{*,\Omega_0}
				\\
				&\qquad
				+ C \frac{1}{1+|y|} 
				(1+|y|)^{1-\sigma}  
				\|h_v\|_{*,\partial\Omega_0}
				(1+|y|)^{-1-\sigma} \|u-v\|_{*,\Omega_0}
				\\
				&\quad \leq C 
				M \varepsilon^{\sigma(1-b)}\, |\log\varepsilon| \,  \|u-v\|_{*,\Omega_0}.
			\end{align*}
			Similar estimates hold for the other terms of $Q_1$ and we obtain \eqref{estQ1-lip}.
			
                \medskip\paragraph{\textbf{Estimates for $\|E_0(g_1)\|_{**, \partial \Omega_0}$.}}
			We have estimated $\|E_0(g_1)\|_{**,\partial\Omega_0}$ in Proposition~\ref{prop:est-error-bdy} and we have
			\begin{align}
				\label{E0g1}
				\|E_0(g_1)\|_{**, \partial \Omega_0 }\leq C \varepsilon^{\ell_0} |\log\varepsilon|,
			\end{align}
			where $\ell_0 = \min(1-\sigma b, 4(1-b) - \sigma b, 1+m+b(1-\sigma) )$.
			We also have the estimate
			\begin{align}
				\label{lipE0}
				\|E_0(g_1)-E_0(\bar g_1)\|_{**, \partial \Omega_0 }\leq C \varepsilon^{b(1-\sigma)} \, |\log\varepsilon| \, |g_1-\bar g_1|.
			\end{align}
			The proof is the same as the one of Lemma~\ref{lem:error1-2}.
			Note that by \eqref{sigmab}
			\[
			\ell_0=1-\sigma b.
			\]
			
                \medskip\paragraph{\textbf{Estimates for $\|Q_2[u,\varepsilon g_0+g_1]\|_{**, \partial \Omega_0}$.}}
			Next we consider $Q_2[u,\varepsilon g_0+g_1]$ defined in \eqref{newQ2}.
			We claim that there exist $\Theta>1-\sigma$ and $C$, such that for $\varepsilon>0$ small, for $g=\varepsilon g_0+g_1$ and $g_1\in I$, 
			\begin{align}
				\label{Q2-est1}
				\|Q_2[u,g]\|_{**, \partial \Omega_0} &\leq C \varepsilon^{\Theta}\|u\|_{*,\Omega_0}+ C \|u\|_{*,\Omega_0}^2,\quad u\in B,
			\end{align}
			\begin{align}
				\label{Q2-est2}
				\|Q_2[u,g]-Q_2[v,g]\|_{**, \partial \Omega_0 } &\leq C \varepsilon^{\Theta}\|u-v\|_{*,\Omega_0}+ C \varepsilon^{\sigma(1-b)} |\log\varepsilon| \|u-v\|_{*,\Omega_0},
			\end{align}
			for $u,v\in B$.
			We also claim that for $u\in B$ and $g_1,\tilde g_1\in I$,
			\begin{align}
				\label{Q2-est3}
				\|Q_2[u,\varepsilon g_0+g_1]-Q_2[u,\varepsilon g_0+\tilde g_1]\|_{**, \partial \Omega_0 } &\leq C \varepsilon^{\sigma(1-b)} \, |\log\varepsilon| \, |g_1-\tilde g_1|.
			\end{align}
			To prove these claims, we will check that the operators $\hat Q_{2,a},\dots,\hat Q_{2,f}$ satisfy the corresponding versions of \eqref{Q2-est1} \eqref{Q2-est2}, and \eqref{Q2-est3}.
			
			\medskip
			Let us consider $\hat Q_{2,a}(\hat u,h)$ \eqref{Q2a}, and in particular $\hat B_h(\hat u,\hat u)$. The main term in $\hat B_h(\hat u,\hat u)$ is $|\nabla u|^2$. We have
			\begin{alignat*}{2}
				(1+|y|)^\sigma |\nabla u(y)|^2
				& \leq (1+|y|)^{-\sigma} \|u\|_{*,\Omega_0}^2,
                &\quad&\text{for } |y|\leq\frac{\delta_1}{\varepsilon}
                \\
				(\frac{\delta_1}{\varepsilon})^\sigma e^{\mu\varepsilon|y|} |\nabla u(y)|^2
				& \leq (\frac{\varepsilon}{\delta_1})^\sigma e^{-\mu\varepsilon|y|} \|u\|_{*,\Omega_0}^2,
                &\quad&\text{for } |y|\geq\frac{\delta_1}{\varepsilon}.
			\end{alignat*}
			A similar estimate is valid for the H\"older part and we get
			\begin{align*}
				\| |\nabla u|^2 \|_{**, \partial \Omega_0} \leq C \|u\|_{*,\Omega_0}^2.
			\end{align*}
			Similarly
			\begin{align*}
				\| \hat B_h(\hat u,\hat u) \|_{**, \partial \Omega_0}  \leq C \|u\|_{*,\Omega_0}^2,\quad u\in B.
			\end{align*}
			The other terms in $\hat Q_{2,a}(\hat u,h)$ are estimated similarly.
			
			\medskip
			Let us consider $\hat Q_{2,b}(\hat u,h)$ \eqref{Q2b}.
			The main contribution is $B_h^{(1)}( \hat\psi_0[h],\hat\psi_0[h])$.
			We claim that for $u\in B$ and $h=\mathbf{h}(u,s)$, we have
			\begin{align*}
				\|\hat B_h^{(1)}(\hat\psi_0[h],\hat\psi_0[h])\|_{**, \partial \Omega_0}\leq C \varepsilon |\log\varepsilon| \|u\|_{*,\Omega_0}.
			\end{align*}
			Expanding in $h$, the main term is $B_h^{(1)}( \hat\psi_0[0],\hat\psi_0[0])$.
			We recall that 
			\begin{align*}
				\hat B_h^{(1)}( \hat\psi_0[0],  \hat\psi_0[0])
				&=
				2 \kappa h  \partial_s \hat\psi_0[0] \partial_s \hat\psi_0[0]
				- 2 h'  \partial_s \hat\psi_0[0] \partial_t \hat  \psi_0[0].
			\end{align*}
			
			For the estimate, we use \eqref{dpsi0R1}, \eqref{dpsi0R2}, \eqref{dpsi0R3} and summarize the result for $j=1$:
			\begin{align}
				\label{dPsi0}
				|\psi_0^{(1)}(y)| \leq 
				\begin{cases}
					C \varepsilon,&|y|\leq\varepsilon^{-b},
					\\
					C \varepsilon |\log\varepsilon|   , & \varepsilon^{-b} \leq |y| \leq2\varepsilon^{-b},
					\\
					\frac{|\log\varepsilon|}{|y|^{2}} , &  2\varepsilon^{-b}\leq|y|\leq\frac{\delta_1}{\varepsilon},
					\\
					\varepsilon^{2} e^{-\mu \varepsilon |y|},&|y|\geq\frac{\delta_1}{\varepsilon}.
				\end{cases}
			\end{align}
			Using \eqref{est-kappa} and \eqref{dPsi0} we get
            \begin{align*}
                &
    			\begin{aligned}
    				(1+|y|)^\sigma
    				|\kappa h\partial_s \hat\psi_0[0] \partial_s \hat\psi_0[0]|
    				&\leq (1+|y|)^\sigma (\frac{1}{(1+|y|)^2}+\varepsilon)(1+|y|)^{1-\sigma}
    				\|h\|_{*,\partial\Omega_0} \varepsilon^2
    				\\
    				&\leq\varepsilon^{2}\|h\|_{*,\partial\Omega_0}
    				\quad\text{for } |y|\leq\varepsilon^{-\frac{1}{2}}
                    ,
    			\end{aligned}
                \displaybreak[1]\\[1ex]
                &
    			\begin{aligned}
    				(1+|y|)^\sigma
    				|\kappa h\partial_s \hat\psi_0[0] \partial_s \hat\psi_0[0]|
    				&\leq (1+|y|)^\sigma (\frac{1}{(1+|y|)^2}+\varepsilon)(1+|y|)^{1-\sigma}
    				\|h\|_{*,\partial\Omega_0} \varepsilon^2
    				\\
    				&\leq \varepsilon^3 (1+|y|)\|h\|_{*,\partial\Omega_0}
    				\\
    				&\leq \varepsilon^{3-b}\|h\|_{*,\partial\Omega_0}
    				\quad\text{for } \varepsilon^{-\frac{1}{2}}\leq|y|\leq\varepsilon^{-b}
                    ,
    			\end{aligned}
                \displaybreak[1]\\[1ex]
                &
    			\begin{aligned}
    				(1+|y|)^\sigma |\kappa h\partial_s \hat\psi_0 \partial_s \hat\psi_0|
    				&\leq (1+|y|)^\sigma (\frac{1}{(1+|y|)^2}+\varepsilon)(1+|y|)^{1-\sigma}
    				\|h\|_{*,\partial\Omega_0} (\frac{|\log\varepsilon|}{|y|^2})^2
    				\\
    				&\leq\varepsilon\frac{|\log\varepsilon|^2}{|y|^3}\|h\|_{*,\partial\Omega_0}
    				\\
    				&\leq \varepsilon^{1+3b}|\log\varepsilon|^2\|h\|_{*,\partial\Omega_0}
    				\quad\text{for }2\varepsilon^{-b}\leq|y|\leq\frac{\delta_1}{\varepsilon}
                    ,
    			\end{aligned}
                \displaybreak[1]\\[1ex]
                &
    			\begin{aligned}
    				(1+|y|)^\sigma |\kappa h\partial_s \hat\psi_0[0] \partial_s \hat\psi_0[0]|
    				&\leq (1+|y|)^\sigma (\frac{1}{(1+|y|)^2}+\varepsilon)(1+|y|)^{1-\sigma}
    				\|h\|_{*,\partial\Omega_0} (\varepsilon|\log\varepsilon|)^2
    				\\
    				&\leq \varepsilon^{3}|\log\varepsilon|^2 (1+|y|) \|h\|_{*,\partial\Omega_0} 
    				\\
    				&\leq \varepsilon^{3-b}|\log\varepsilon|^2 \|h\|_{*,\partial\Omega_0}
    				\quad\text{for }\varepsilon^{-b}\leq|y|\leq2\varepsilon^{-b}
                    ,
    			\end{aligned}
    			\intertext{and}
                &
    			\begin{aligned}
    				(\frac{\delta_1}{\varepsilon})^\sigma e^{\mu\varepsilon|y|}
    				|\kappa h\partial_s \hat\psi_0[0] \partial_s \hat\psi_0[0]|
    				&\leq (\frac{\delta_1}{\varepsilon})^\sigma  e^{\mu\varepsilon|y|}(\frac{1}{(1+|y|)^2}+\varepsilon)(\frac{\delta_1}{\varepsilon})^{1-\sigma}
    				\|h\|_{*,\partial\Omega_0} (\varepsilon^2 e^{-\mu\varepsilon|y|})^2
    				\\
    				&\leq\varepsilon^4 e^{-\mu\varepsilon|y|}
    				\|h\|_{*,\partial\Omega_0}
    				\\
    				&\leq\varepsilon^4 
    				\|h\|_{*,\partial\Omega_0}
    				\quad\text{for } |y|\geq\frac{\delta_1}{\varepsilon}.
    			\end{aligned}
            \end{align*}
			Proceeding similarly for the H\"older term we get
			\begin{align}
				\label{Q2-L1-1-00}
				\|\kappa h \partial_s \hat\psi_0[0] \partial_s \hat\psi_0[0]\|_{**, \partial \Omega_0} \leq C\varepsilon^2\|h\|_{*,\partial\Omega_0}.
			\end{align}
			
			Analogously,
            \begin{align*}
                &
    			\begin{aligned}
    				(1+|y|)^\sigma
    				|h'\partial_s \hat\psi_0[0] \partial_t \hat\psi_0[0]|
    				&\leq \varepsilon (1+|y|)^\sigma (1+|y|)^{-\sigma}
    				\|h\|_{*,\partial\Omega_0} 
    				\\
    				&\leq\varepsilon\|h\|_{*,\partial\Omega_0} \quad\text{for } |y|\leq\varepsilon^{-b},
    			\end{aligned}
                \displaybreak[1]\\[1ex]
                &
    			\begin{aligned}
    				(1+|y|)^\sigma |h'\partial_s \hat\psi_0 \partial_t \hat\psi_0|
    				&(1+|y|)^\sigma (1+|y|)^{-\sigma}
    				\|h\|_{*,\partial\Omega_0} \varepsilon|\log\varepsilon|
    				\\
    				&\leq \varepsilon|\log\varepsilon| \|h\|_{*,\partial\Omega_0}
    				\quad\text{for }\varepsilon^{-b}\leq|y|\leq2\varepsilon^{-b},
    			\end{aligned}
                \displaybreak[1]\\[1ex]
                &
    			\begin{aligned}
    				(1+|y|)^\sigma |h'\partial_s \hat\psi_0 \partial_t \hat\psi_0|
    				&\leq(1+|y|)^\sigma (1+|y|)^{-\sigma}
    				\|h\|_{*,\partial\Omega_0} \frac{|\log\varepsilon|}{|y|^2}
    				\\
    				&\leq\varepsilon^{2b}|\log\varepsilon|\|h\|_{*,\partial\Omega_0}\quad\text{for } 2\varepsilon^{-b}\leq|y|\leq\frac{\delta_1}{\varepsilon},
                \end{aligned}
                \intertext{and}
                &
    			\begin{aligned}
    				(\frac{\delta_1}{\varepsilon})^\sigma e^{\mu\varepsilon|y|} |h'\partial_s \hat\psi_0 \partial_t \hat\psi_0|
    				&\leq (\frac{\delta_1}{\varepsilon})^\sigma e^{\mu\varepsilon|y|} 
    				(\frac{\varepsilon}{\delta_1})^{1+\sigma}e^{-\mu\varepsilon|y|} \varepsilon^2 e^{-\mu\varepsilon|y|} \|h\|_{*,\partial\Omega_0}
    				\\
    				&\leq \varepsilon^3 \|h\|_{*,\partial\Omega_0}\quad\text{for } |y|\geq\frac{\delta_1}{\varepsilon}.
    			\end{aligned}
            \end{align*} 
			From this and similar estimates for the H\"older part of the norm, we obtain
			\begin{align}
				\label{Q2-L1-2-00}
				\|h'\partial_s \hat\psi_0 \partial_t \hat\psi_0\|_{**, \partial \Omega_0}\leq \varepsilon|\log\varepsilon|  \|h\|_{*,\partial\Omega_0}.
			\end{align}
			Combining \eqref{Q2-L1-1-00} and \eqref{Q2-L1-2-00} we get
			\begin{align*}
				\|\hat B_h^{(1)}( \hat\psi_0[0],  \hat\psi_0[0])\|_{**, \partial \Omega_0}\leq C \varepsilon |\log\varepsilon| \|h\|_{*,\partial\Omega_0}.
			\end{align*}
			Similar computations are valid for the other terms in $\hat Q_{2,b}(\hat u,h)$.
			
			\medskip
			Next we turn to $\hat Q_{2,c}(\hat u,h)$ \eqref{Q2c} for $u\in B$ and $h=\mathbf{h}(u,s)$.
			Consider $\hat B_h^{(1)}(\hat u,\hat\psi_0[h])$, whose main term is $\hat B_h^{(1)}(\hat u,\hat\psi_0[0])$.
			We have
			\begin{align}
				\nonumber
				\hat B_h^{(1)}(\hat u,\hat\psi_0[0])
				=2 \kappa h  \partial_s \hat u \partial_s \hat\psi_0[0]
				- h'  \partial_s \hat u \partial_t\hat\psi_0[0]
				- h' \partial_s\hat\psi_0[0] \partial_t \hat u.
			\end{align}
			Using \eqref{est-kappa} and \eqref{dPsi0}, we get
            \begin{align*}
                &
    			\begin{aligned}
    				&
    				(1+|y|)^\sigma |\kappa h  \partial_s \hat u \partial_s \hat\psi_0[0]|
    				\\
    				&\quad\leq C(1+|y|)^\sigma (\frac{1}{(1+|y|)^2}+\varepsilon)(1+|y|)^{1-2\sigma}
    				\|u\|_{*,\Omega_0}\|h\|_{*,\partial\Omega_0} \varepsilon
    				\\
    				&\quad\leq C\varepsilon
    				\|u\|_{*,\Omega_0}\|h\|_{*,\partial\Omega_0},\quad|y|\leq\varepsilon^{-\frac{1}{2}}
                    ,
    			\end{aligned}
                \displaybreak[1]\\[1ex]  
                &
    			\begin{aligned}
    				&
    				(1+|y|)^\sigma |\kappa h  \partial_s \hat u \partial_s \hat\psi_0[0]|
    				\\
    				&\quad\leq C(1+|y|)^\sigma (\frac{1}{(1+|y|)^2}+\varepsilon)(1+|y|)^{1-2\sigma}
    				\|u\|_{*,\Omega_0}\|h\|_{*,\partial\Omega_0} \varepsilon
    				\\
    				&\quad\leq C\varepsilon^{2-b(1-\sigma)}
    				\|u\|_{*,\Omega_0}\|h\|_{*,\partial\Omega_0},\quad\varepsilon^{-\frac{1}{2}}\leq|y|\leq\varepsilon^{-b}
                    ,
    			\end{aligned}
                \displaybreak[1]\\[1ex]  
                &
    			\begin{aligned}
    				&
    				(1+|y|)^\sigma |\kappa h  \partial_s \hat u \partial_s \hat\psi_0[0]|
    				\\
    				&\quad\leq C(1+|y|)^\sigma (\frac{1}{(1+|y|)^2}+\varepsilon)(1+|y|)^{1-2\sigma}
    				\|u\|_{*,\Omega_0}\|h\|_{*,\partial\Omega_0} \varepsilon|\log\varepsilon|
    				\\
    				&\quad\leq C\varepsilon^{2-b(1-\sigma)}|\log\varepsilon|
    				\|u\|_{*,\Omega_0}\|h\|_{*,\partial\Omega_0},\quad\varepsilon^{-b}\leq|y|\leq2\varepsilon^{-b}
                    ,
    			\end{aligned}
                \displaybreak[1]\\[1ex]  
                &
    			\begin{aligned}
    				&
    				(1+|y|)^\sigma |\kappa h  \partial_s \hat u \partial_s \hat\psi_0[0]|
    				\\
    				&\quad\leq C(1+|y|)^\sigma (\frac{1}{(1+|y|)^2}+\varepsilon)(1+|y|)^{1-2\sigma}
    				\|u\|_{*,\Omega_0}\|h\|_{*,\partial\Omega_0} \frac{|\log\varepsilon|}{|y|^2}
    				\\
    				&\quad\leq C\varepsilon^{2+\sigma}|\log\varepsilon|
    				\|u\|_{*,\Omega_0}\|h\|_{*,\partial\Omega_0},\quad2\varepsilon^{-b}\leq|y|\leq\frac{\delta_1}{\varepsilon}
                    ,
    			\end{aligned}
                \displaybreak[1]\\[1ex]  
                &
    			\begin{aligned}
    				&
    				(\frac{\delta_1}{\varepsilon})^\sigma e^{\mu\varepsilon|y|}|\kappa h  \partial_s \hat u \partial_s \hat\psi_0[0]|
    				\\
    				&\quad\leq C (\frac{\varepsilon}{\delta_1})^{1+\sigma} e^{-\mu\varepsilon|y|} (\frac{1}{(1+|y|)^2}+\varepsilon)
    				\|u\|_{*,\Omega_0}\|h\|_{*,\partial\Omega_0} \varepsilon^2e^{-\mu\varepsilon|y|}
    				\\
    				&\quad\leq C\varepsilon^{4+\sigma}\|u\|_{*,\Omega_0}\|h\|_{*,\partial\Omega_0},\quad|y|\geq\frac{\delta_1}{\varepsilon}
                    ,
    			\end{aligned}
    			\intertext{and}
                &
    			\begin{aligned}
    				(1+|y|)^\sigma |h'\partial_s\hat u\partial_t\hat\psi_0[0]|
    				&\leq C(1+|y|)^\sigma (1+|y|)^{-2\sigma}
    				\|u\|_{*,\Omega_0}\|h\|_{*,\partial\Omega_0} 
    				\\
    				&\leq C
    				\|u\|_{*,\Omega_0}\|h\|_{*,\partial\Omega_0},\quad|y|\leq\frac{\delta_1}{\varepsilon}
                    ,
    			\end{aligned}
                \displaybreak[1]\\[1ex]  
                &
    			\begin{aligned}
    				(\frac{\delta_1}{\varepsilon})^\sigma e^{\mu\varepsilon|y|} |h'\partial_s\hat u\partial_t\hat\psi_0[0]|
    				&\leq C (\frac{\varepsilon}{\delta_1})^{1+\sigma} e^{-\mu\varepsilon|y|} 
    				\|u\|_{*,\Omega_0}\|h\|_{*,\partial\Omega_0} \varepsilon^2e^{-\mu\varepsilon|y|}
    				\\
    				&\leq C\varepsilon^{1+\sigma}\|u\|_{*,\Omega_0}\|h\|_{*,\partial\Omega_0},\quad|y|\geq\frac{\delta_1}{\varepsilon}.
    			\end{aligned}
            \end{align*}
			Arguing similarly for the H\"older part we find that
			\begin{align*}
				\|\hat B_h^{(1)}(\hat u,\hat\psi_0[0])\|_{**, \partial \Omega_0} \leq C \|u\|_{*,\Omega_0}^2,\quad u\in B.
			\end{align*}
			The other terms in $\hat Q_{2,c}(\hat u,h)$ are estimated similarly.
			
			\medskip
			Let us consider $\hat Q_{2,d}(\hat u,h)$ defined in \eqref{Q2d}, in particular $\hat B_h^{(1)}(\hat\Psi[0],\hat\Psi[0])$.
			We claim that for $u\in B$ and $h=\mathbf{h}(u,s)$, 
			\begin{align}
				\label{abd1}
				\|\hat B_h^{(1)}(\hat\Psi[0],\hat\Psi[0]) \|_{**, \partial \Omega_0 }
				\leq C \varepsilon^{2(1-\sigma b)}|\log\varepsilon|^2 \|u\|_{*,\Omega_0}.
			\end{align}
			Indeed, 
			\begin{align*}
				\hat \Psi[0]=u_0+u_1
			\end{align*}
			with $u_0$ and $u_1$ defined in \eqref{equ0} and \eqref{equ1}.
			We use Lemmas~\ref{lem:error-u1} and \ref{lem:error-u0} to get 
            \begin{align*}
                &
    			\begin{aligned}
    				&(1+|y|)^\sigma
    				|\kappa h (\partial_s \hat\Psi[0])^2|
    				\\
    				&\quad\leq (1+|y|)^\sigma (\frac{1}{|y|^2}+\varepsilon) (1+|y|)^{1-\sigma } (\varepsilon^{1-\sigma b} |\log\varepsilon| (1+|y|)^{-\sigma})^2 \|h\|_{*,\partial\Omega_0}
    				\\
    				&\quad\leq \varepsilon^{2(1-\sigma b)}|\log\varepsilon|^2 \|h\|_{*,\partial\Omega_0}
    				\quad\text{for } |y|\leq \varepsilon^{-\frac{1}{2}},
    			\end{aligned}
                \displaybreak[1]\\[1ex] &
    			\begin{aligned}
    				&(1+|y|)^\sigma
    				|\kappa h (\partial_s \Psi[0])^2|
    				\\
    				&\quad\leq (1+|y|)^\sigma (\frac{1}{|y|^2}+\varepsilon) (1+|y|)^{1-\sigma } (\varepsilon^{1-\sigma b} |\log\varepsilon| (1+|y|)^{-\sigma})^2 \|h\|_{*,\partial\Omega_0}
    				\\
    				&\quad\leq \varepsilon^{2(1-\sigma b)}|\log\varepsilon|^2 \|h\|_{*,\partial\Omega_0}
    				\quad\text{for } \varepsilon^{-\frac{1}{2}}\leq|y|\leq\frac{\delta_1}{\varepsilon},
    			\end{aligned}
                \displaybreak[1]\\[1ex] &
    			\begin{aligned}
    				&(\frac{\varepsilon}{\delta_1})^{-\sigma}e^{\mu\varepsilon|y|}
    				|\kappa h (\partial_s \Psi[0])^2|
    				\\
    				&\quad\leq (\frac{\varepsilon}{\delta_1})^{-\sigma}e^{\mu\varepsilon|y|}(\frac{1}{|y|^2}+\varepsilon) (1+|y|)^{1-\sigma } (\varepsilon^{1-\sigma b} |\log\varepsilon| (\frac{\varepsilon}{\delta_1})^{\sigma}e^{-\mu\varepsilon|y|} )^2 \|h\|_{*,\partial\Omega_0}
    				\\
    				&\quad\leq \varepsilon^{2(1-\sigma b)}|\log\varepsilon|^2 \|h\|_{*,\partial\Omega_0}\quad\text{for } |y|\geq\frac{\delta_1}{\varepsilon},
    			\end{aligned}
                \displaybreak[1]\\[1ex] &
    			\begin{aligned}
    				&(1+|y|)^\sigma |h'  \partial_s \hat\Psi[0]\partial_t \hat\Psi[0]|
    				\\
    				&\quad\leq C (1+|y|)^\sigma  (1+|y|)^{-\sigma } (\varepsilon^{1-\sigma b} |\log\varepsilon| (1+|y|)^{-\sigma})^2 \|h\|_{*,\partial\Omega_0}
    				\\
    				&\quad\leq C \varepsilon^{2(1-\sigma b)} |\log\varepsilon|^2\|h\|_{*,\partial\Omega_0}
    				\quad\text{for } |y|\leq \varepsilon^{-\frac{1}{2}},
    			\end{aligned}
                \displaybreak[1]\\[1ex] &
    			\begin{aligned}
    				&(1+|y|)^\sigma| h'  \partial_s \hat\Psi[0] \partial_t \hat\Psi[0]|
    				\\
    				&\quad\leq C\leq (1+|y|)^\sigma (1+|y|)^{-\sigma } (\varepsilon^{1-\sigma b} |\log\varepsilon| (1+|y|)^{-\sigma})^2 \|h\|_{*,\partial\Omega_0}
    				\\
    				&\quad\leq C \varepsilon^{2(1-\sigma b)}|\log\varepsilon|^2 \|h\|_{*,\partial\Omega_0}
    				\quad\text{for } \varepsilon^{-\frac{1}{2}}\leq|y|\leq\frac{\delta_1}{\varepsilon},
    			\end{aligned}
                \displaybreak[1]\\[1ex] &
    			\begin{aligned}
    				&(\frac{\varepsilon}{\delta_1})^{-\sigma}e^{\mu\varepsilon|y|}
    				|h'\partial_s\hat\Psi[0] \partial_t \hat\Psi[0]|
    				\\
    				&\quad\leq 
    				(\varepsilon^{1-\sigma b} |\log\varepsilon| (\frac{\varepsilon}{\delta_1})^{\sigma}e^{-\mu\varepsilon|y|} )^2 \|h\|_{*,\partial\Omega_0}
    				\\
    				&\quad\leq \varepsilon^{2(1-\sigma b)}|\log\varepsilon|^2 \|h\|_{*,\partial\Omega_0}\quad\text{for } |y|\geq\frac{\delta_1}{\varepsilon}.
    			\end{aligned}
            \end{align*}
			Arguing similarly for the H\"older part we obtain \eqref{abd1}.
			The other terms in $\hat Q_{2,d}(\hat u,h)$ are estimated similarly.
			
			\medskip
			Let us consider $\hat Q_{2,e}(\hat u,h,g)$ defined in \eqref{Q2e}, with $u\in B$ and $h=\mathbf{h}(u,y)$. The main term in $\hat B_h^{(1)}( \hat \Psi[0] , \hat\psi_0[h] )$ is:
			\begin{align*}
				\hat B_h^{(1)}(\hat\Psi[0],\hat\psi_0[0]) 
				&=2 \kappa h  \partial_s \hat \Psi[0] \partial_s\hat\psi_0[0] 
				-  h'  \partial_s\hat \Psi[0] \partial_t \hat\psi_0[0] 
				-  h'  \partial_t\hat \Psi[0] \partial_s \hat\psi_0[0] .
			\end{align*}
			By Lemmas~\ref{lem:error-u1} and \ref{lem:error-u0}, and \eqref{dPsi0}, we have 
            \begin{align*}
                &
    			\begin{aligned}
    				&(1+|y|)^\sigma |\kappa h  \partial_s \hat \Psi[0] \partial_s\hat\psi_0[0] |
    				\\
    				&\quad\leq
    				(1+|y|)^\sigma (\frac{1}{|y|^2}+\varepsilon) (1+|y|)^{1-\sigma } \varepsilon^{1-\sigma b} |\log\varepsilon| (1+|y|)^{-\sigma} 
    				\varepsilon\|h\|_{*,\partial\Omega_0}
    				\\
    				&\quad\leq \varepsilon^{2-\sigma b}|\log\varepsilon|\|h\|_{*,\partial\Omega_0}
    				\quad\text{for } |y|\leq \varepsilon^{-\frac{1}{2}},
    			\end{aligned}
                \displaybreak[1]\\[1ex] &
    			\begin{aligned}
    				&(1+|y|)^\sigma |\kappa h  \partial_s \hat \Psi[0] \partial_s\hat\psi_0[0] |
    				\\
    				&\quad\leq  (1+|y|)^\sigma (\frac{1}{|y|^2}+\varepsilon) (1+|y|)^{1-\sigma } \varepsilon^{1-\sigma b} |\log\varepsilon| (1+|y|)^{-\sigma} 
    				\varepsilon\|h\|_{*,\partial\Omega_0}
    				\\
    				&\quad\leq  (1+|y|)^\sigma \varepsilon^2 (1+|y|)^{1-\sigma } \varepsilon^{1-\sigma b} |\log\varepsilon| (1+|y|)^{-\sigma} \|h\|_{*,\partial\Omega_0}
    				\\
    				&\quad\leq \varepsilon^{3-b}|\log\varepsilon|\|h\|_{*,\partial\Omega_0}
    				\quad\text{for } \varepsilon^{-\frac{1}{2}}\leq|y|\leq \varepsilon^{-b},
    			\end{aligned}
                \displaybreak[1]\\[1ex] &
    			\begin{aligned}
    				&(1+|y|)^\sigma |\kappa h  \partial_s \hat \Psi[0] \partial_s\hat\psi_0[0] |
    				\\
    				&\quad\leq  (1+|y|)^\sigma (\frac{1}{|y|^2}+\varepsilon) (1+|y|)^{1-\sigma } \varepsilon^{1-\sigma b} |\log\varepsilon| (1+|y|)^{-\sigma} 
    				\varepsilon|\log\varepsilon| \|h\|_{*,\partial\Omega_0}
    				\\
    				&\quad\leq  (1+|y|)^\sigma \varepsilon^2 (1+|y|)^{1-\sigma } \varepsilon^{1-\sigma b} |\log\varepsilon|^2 (1+|y|)^{-\sigma} \|h\|_{*,\partial\Omega_0}
    				\\
    				&\quad\leq \varepsilon^{3-b}|\log\varepsilon|^2\|h\|_{*,\partial\Omega_0}
    				\quad\text{for } \varepsilon^{-b}\leq|y|\leq 2\varepsilon^{-b},
    			\end{aligned}
                \displaybreak[1]\\[1ex] &
    			\begin{aligned}
    				&(1+|y|)^\sigma |\kappa h  \partial_s \hat \Psi[0] \partial_s\hat\psi_0[0] |
    				\\
    				&\quad\leq  (1+|y|)^\sigma (\frac{1}{|y|^2}+\varepsilon) (1+|y|)^{1-\sigma } \varepsilon^{1-\sigma b} |\log\varepsilon| (1+|y|)^{-\sigma} 
    				\frac{|\log\varepsilon|}{|y|^2} \|h\|_{*,\partial\Omega_0}
    				\\
    				&\quad\leq  (1+|y|)^\sigma \varepsilon^2 (1+|y|)^{1-\sigma } \varepsilon^{1-\sigma b} |\log\varepsilon| (1+|y|)^{-\sigma} \|h\|_{*,\partial\Omega_0}
    				\\
    				&\quad\leq \varepsilon^{2+b}|\log\varepsilon|\|h\|_{*,\partial\Omega_0}
    				\quad\text{for } 2\varepsilon^{-b}\leq|y|\leq \frac{\delta_1}{\varepsilon}
    			\end{aligned}
    			\intertext{and}
                &
    			\begin{aligned}
    				&(\frac{\delta_1}{\varepsilon})^\sigma e^{\mu\varepsilon|y|}
    				|\kappa h  \partial_s \hat \Psi[0] \partial_s\hat\psi_0[0] |
    				\\
    				&\quad\leq 
    				(\frac{\delta_1}{\varepsilon})^\sigma e^{\mu\varepsilon|y|} (\frac{1}{|y|^2}+\varepsilon) (\frac{\delta_1}{\varepsilon})^{1-\sigma} e^{-\mu\varepsilon|y|} \varepsilon^{1-\sigma b}|\log\varepsilon|
    				(\frac{\delta_1}{\varepsilon})^{1-\sigma} e^{-\mu\varepsilon|y|} 
    				\varepsilon^2e^{-\mu\varepsilon|y|}
    				\|h\|_{*,\partial\Omega_0}
    				\\
    				&\quad\leq 
    				C\varepsilon^{3+\sigma(1-b)}
    				\|h\|_{*,\partial\Omega_0}
    				\quad\text{for } |y|\geq \frac{\delta_1}{\varepsilon}.
    			\end{aligned}
            \end{align*}
			With a similar proof for the H\"older part we get 
			\[
			\| \kappa h  \partial_s \hat \Psi[0] \partial_s\hat\psi_0[0] \|_{**, \partial \Omega_0}
			\leq C \varepsilon^{2-\sigma b}\|h\|_{*,\partial\Omega_0}.
			\]
			
			We use Lemmas~\ref{lem:error-u1} and \ref{lem:error-u0} to estimate $ h'  \partial_s\hat \Psi[0] \partial_t \hat\psi_0[0] $:
			\begin{align*}
				(1+|y|)^\sigma |h'  \partial_s\hat \Psi[0] \partial_t \hat\psi_0[0] |
				\leq C \varepsilon^{1-\sigma b}|\log\varepsilon|\|h\|_{*,\partial\Omega_0}
				\quad\text{for } |y|\leq\frac{\delta_1}{\varepsilon}
			\end{align*}
			and
			\begin{align*}
				(\frac{\delta_1}{\varepsilon})^\sigma e^{\mu\varepsilon|y|}
				|h'  \partial_s\hat \Psi[0] \partial_t \hat\psi_0[0] |
				\leq 
				(\frac{\delta_1}{\varepsilon})^\sigma 
				\|h\|_{*,\partial\Omega_0}
				\quad\text{for } |y|\geq\frac{\delta_1}{\varepsilon}.
			\end{align*}
			Estimating similarly the H\"older term in the norm, we get
			\begin{align*}
				\|h'  \partial_s\hat \Psi[0] \partial_t \hat\psi_0[0]\|_{**, \partial \Omega_0}
				\leq C \varepsilon^{1-\sigma b}|\log\varepsilon|\|h\|_{*,\partial\Omega_0}.
			\end{align*}
			The other terms in $\hat Q_{2,e}(\hat u,h,g)$ are estimated similarly.
			
			\medskip
			Let us consider finally $\hat Q_{2,f}(\hat u,h,g)$ \eqref{Q2f}, where, we recall, $g=\varepsilon g_0 + g_1$.
			For $u\in B$ and $h=\mathbf{h}(u,y)$
			\begin{align*}
				(1+|y|)^\sigma |\varepsilon g h|
				&\leq C \varepsilon^2 (1+|y|)^\sigma (1+|y|)^{1-\sigma}\|h\|_{*,\partial\Omega_0}
				\\
				&\leq C \varepsilon \|h\|_{*,\partial\Omega_0} \quad\text{for }|y|\leq\frac{\delta_1}{\varepsilon}
			\end{align*}
			and
			\begin{align*}
				(\frac{\delta_1}{\varepsilon})^\sigma e^{\mu\varepsilon|y|}
				|\varepsilon g h|
				&\leq C \varepsilon^2 (\frac{\delta_1}{\varepsilon})^\sigma (\frac{\delta_1}{\varepsilon})^{1-\sigma}\|h\|_{*,\partial\Omega_0}
				\\
				&\leq C \varepsilon \|h\|_{*,\partial\Omega_0}
				\quad\text{for }|y|\geq\frac{\delta_1}{\varepsilon}.
			\end{align*}
			With a similar estimate for the H\"older part, we get
			\begin{align}
				\label{estQ2f}
				\|\hat Q_{2,f}(\hat u,h,g)\|_{**, \partial\Omega_0}\leq C \varepsilon \|u\|_{*,\Omega_0}.
			\end{align}
			
			\medskip
			Among the other terms in $Q_2[u,g_1]$ we have $( \varepsilon\omega + \Delta \psi_0) \mathbf{h}(u,y)$. Using \eqref{decomp-err-la} and writing $h=\mathbf{h}(u,y)$ we get
			\begin{align*}
				\|( \varepsilon\omega + \Delta \psi_0) \mathbf{h}(u,y)\|_{**, \partial \Omega_0}
				&\leq C\varepsilon|\log\varepsilon|\|h\|_{*,\partial\Omega_0}.
			\end{align*}
			
			The proof of \eqref{Q2-est3} is similar to the estimate \eqref{estQ2f}.
			Let $u\in B$ and $h=\mathbf{h}(u,y)$. 
			Let $g_1,\tilde g_1\in I$. Then 
			\begin{align*}
				Q_2[u,\varepsilon g_0+g_1](y)-Q_2[u,\varepsilon g_0+\tilde g_1](y)
				=-2 \varepsilon \nu(y)\cdot e_2 h(y) (g_1-\tilde g_1).
			\end{align*}
			We have
			\begin{align*}
				(1+|y|)^\sigma \varepsilon |h| |g_1-\tilde g_1|
				&\leq C \varepsilon (1+|y|)^\sigma (1+|y|)^{1-\sigma}\|h\|_{*,\partial\Omega_0} |g_1-\tilde g_1|
				\\
				&\leq C \|h\|_{*,\partial\Omega_0} |g_1-\tilde g_1| \quad\text{for }|y|\leq\frac{\delta_1}{\varepsilon}
			\end{align*}
			and
			\begin{align*}
				(\frac{\delta_1}{\varepsilon})^\sigma e^{\mu\varepsilon|y|}
				\varepsilon |h| |g_1-\tilde g_1|
				&\leq C \varepsilon (\frac{\delta_1}{\varepsilon})^\sigma (\frac{\delta_1}{\varepsilon})^{1-\sigma}\|h\|_{*,\partial\Omega_0}|g_1-\tilde g_1|
				\\
				&\leq C \|h\|_{*,\partial\Omega_0}|g_1-\tilde g_1|
				\quad\text{for }|y|\geq\frac{\delta_1}{\varepsilon}.
			\end{align*}
			With a similar estimate for the H\"older part, we get
			\begin{align}
				\nonumber
				\|Q_2[u,\varepsilon g_0+g_1]-Q_2[u,\varepsilon g_0+\tilde g_1]\|_{**, \partial \Omega_0 } &\leq  C \|h\|_{*,\partial\Omega_0} |g_1-\tilde g_1|
			\end{align}
			and using \eqref{h-from-u} and $u\in B$ we get \eqref{Q2-est3}.
			
			\subsubsection{Existence of a fixed point}
			Now we can prove that $\mathcal{F}$ maps $B\times I$ into itself.
			Indeed, let $(u,g_1)\in B\times I$ and $(\bar u,\bar g_1) = \mathcal{F}(u,g_1)$.
			By Proposition~\ref{prop:main-linear}, 
			\begin{align*}
				\|\bar u\|_{*,\Omega_0} + \varepsilon^{-1} |\bar g_1|&\leq C \varepsilon^{-1+\sigma}( \|Q_1[u]\|_{**, \Omega_0 }+\|E_0(g_1)\|_{**, \partial \Omega_0}+\|Q_2[u,\varepsilon g_0+g_1]\|_{**, \partial \Omega_0 }).
			\end{align*}
			By \eqref{estQ1}, \eqref{E0g1}
			\begin{align*}
				\|Q_1[u]\|_{**, \Omega_0 } & \leq C  \|u\|_{*,\Omega_0}^2 \leq C M^2 \varepsilon^{2\sigma(1-b)}|\log\varepsilon|^2
				\\
				\|E_0(g_1)\|_{**, \partial \Omega_0} &\leq C\varepsilon^{1-\sigma b}|\log\varepsilon|
				\\
				\|Q_2[u,\varepsilon g_0+g_1]\|_{**, \partial \Omega_0 }&\leq C \varepsilon^{\Theta}\|u\|_{*,\Omega_0}+ C \|u\|_{*,\Omega_0}^2
				\\
				&\leq C M \varepsilon^{\Theta+\sigma(1-b)}|\log\varepsilon|+C M^2 \varepsilon^{2\sigma(1-b)}|\log\varepsilon|^2.
			\end{align*}
			Therefore
			\begin{align*}
				&
				\|\bar u\|_{*,\Omega_0} + \varepsilon^{-1} |\bar g_1|
				\\
				&\quad\leq C \varepsilon^{-1+\sigma}( \varepsilon^{1-\sigma b}|\log\varepsilon|+M \varepsilon^{\Theta+\sigma(1-b)}|\log\varepsilon|+ M^2\varepsilon^{2\sigma(1-b)}|\log\varepsilon|^2).
			\end{align*}
			We assume in what follows that
			\begin{align}
				\label{sigma-b}
				\sigma>\frac1{2-b}
			\end{align}
			(note that $\frac1{2-b}<1$ so there is a non-empty range for $\sigma$).
			Then, using that $\Theta>1-\sigma$, we see that there is $M>0$ so that for all $\varepsilon>0$ small,
			\begin{align*}
				\|\bar u\|_{*,\Omega_0} + \varepsilon^{-1} |\bar g_1|
				\leq M \varepsilon^{\sigma(1-b)}|\log\varepsilon|.
			\end{align*}
			Since $0<m<\sigma(1-b)$, this inequality implies that $(\bar u,\bar g_1)\in B\times I$ and so $\mathcal{F}$ maps $(\bar u,\bar g_1)\in B\times I$ into itself.
			
			\medskip
			We also get that  $\mathcal{F}$ is a contraction on $B\times I$.
			Let $(u,g_1), (\bar u,\bar g_1) \in B\times I$ and let us write $(\tilde u,\tilde g_1)=\mathcal{F}(u,g_1)-\mathcal{F}(\bar u,\bar g_1)$.
			Then by \eqref{L1b}, \eqref{L2b}
			\begin{align*}
				\|\tilde u\|_{*,\Omega_0}+\varepsilon^{-1} |g_1-\bar g_1|
				&\leq C \varepsilon^{\sigma-1}
				(\|Q_1[u-\bar u]\|_{**, \Omega_0 }+\|E_0(g_1)-E_0(\bar g_1)\|_{**, \partial \Omega_0 }
				\\
				&\quad
				+\|Q_2[u,\varepsilon g_0+g_1]-Q_2[\bar u,\varepsilon g_0+g_1]\|_{**, \partial \Omega_0 }
				\\
				&\quad
				+\|Q_2[\bar u,\varepsilon g_0+g_1]-Q_2[\bar u,\varepsilon g_0+\bar g_1]\|_{**, \partial \Omega_0 }).
			\end{align*}
			Using \eqref{estQ1-lip}, we find
			\begin{align*}
				\varepsilon^{\sigma-1}\|Q_1[u-\bar u]\|_{**, \Omega_0 }
				\leq C \varepsilon^{\sigma-1 + \sigma(1-b)}|\log\varepsilon \|u-\bar u\|_{*,\Omega_0},
			\end{align*}
			using \eqref{lipE0} we find
			\begin{align*}
				\varepsilon^{\sigma-1}\|E_0(g_1)-E_0(\bar g_1)\|_{**, \partial \Omega_0 }
				\leq C \varepsilon^{\sigma-1+b(1-\sigma)} |\log\varepsilon| \, |g_1-\bar g_1|,
			\end{align*}
			from \eqref{Q2-est2} we get
			\begin{align*}
				\varepsilon^{\sigma-1}\|Q_2[u,g]-Q_2[\bar u,g]\|_{**, \partial \Omega_0 } &\leq C \varepsilon^{\sigma-1+\Theta}\|u-\bar u\|_{*,\Omega_0}
				\\
				&\quad + C \varepsilon^{\sigma-1+\sigma(1-b)}|\log\varepsilon|  \|u-\bar u\|_{*,\Omega_0},
			\end{align*}
			and from \eqref{Q2-est3} we get
			\begin{align*}
				\varepsilon^{\sigma-1}\|Q_2[\bar u,\varepsilon g_0+g_1]-Q_2[\bar u,\varepsilon g_0+\bar g_1]\|_{**, \partial \Omega_0 } &\leq C \varepsilon^{\sigma-1+\sigma(1-b)}|\log\varepsilon| |g_1-\bar g_1| .
			\end{align*}
			Combining the previous estimates we get
			\begin{align*}
				\|\tilde u\|_{*,\Omega_0}+\varepsilon^{-1} |g_1-\bar g_1|
				&\leq C ( \varepsilon^{\sigma-1+\Theta}+\varepsilon^{\sigma-1+\sigma(1-b)}|\log\varepsilon| )\|u-\bar u\|_{*,\Omega_0},
			\end{align*}
			which shows that $\mathcal{F}$ is a contraction on $B\times I$, for $\varepsilon>0$ sufficiently small.
			
			It follows that there exists a fixed point $(u,g_1)$ of $\mathcal{F}$ in $B\times I$, and hence a solution to \eqref{main-problem}.

\subsection{Avoiding self-intersection}
We show here that the solution $u$ to \eqref{main-problem} constructed in Section~\ref{sect:solve}  leads to a solution of the overdetermined problem \eqref{overd0}.
For this we will prove that $F_h$ \eqref{def:Fh}, where $h=\mathbf{h}(u)$,  is a diffeomorphism from $\overline\Omega_0$ onto $\overline\Omega_h$.

As in Section~\ref{sect:operators}, we let $\gamma\colon \R \to \R^2$ be an arc-length parametrization of  $\mathcal{S}_0$ and let $\nu = \gamma'^\perp$ where $(a,b)^\perp = (b,-a)$ so that so $\gamma$ goes ``from right to left''.
We can also assume that $\gamma$ is even with respect to 0. We recall $X_h$ defined in \eqref{Xh}.

Let  $\tilde \rho\in C^\infty(\R)$ be a function with the properties described in  \eqref{properties-tilded}.
We claim that
\begin{align}
\label{no-self}
|h(\gamma(s))|\leq \tilde \rho(s).
\end{align}
To verify \eqref{no-self}, let $s\in\R$ and $y=\gamma(s)$.
If $|y|\leq \varepsilon^{-\frac{1}{2}}|\log\varepsilon|^{\frac{1}{2}}$, then
\begin{align*}
|h(y)|&\leq \|h\|_{*,\partial\Omega_0}(1+|y|)^{1-\sigma}
\\
&\leq 2 M \varepsilon^{\sigma(1-b)}|\log\varepsilon| ( \varepsilon^{-\frac{1}{2}}|\log\varepsilon|^{\frac{1}{2}})^{1-\sigma}
\\
&= 2 M \varepsilon^{\sigma(\frac32-b)-\frac12}|\log\varepsilon|^{1+\frac12(1-\sigma)}.
\end{align*}
We are assuming \eqref{sigma-b}, which implies
\begin{align}
\label{sigma-b2}
\sigma>\frac{1}{3-2b}.
\end{align}
Then \eqref{no-self} holds for $|y|\leq \varepsilon^{-\frac{1}{2}}|\log\varepsilon|^{\frac{1}{2}}$ and $\varepsilon>0$ sufficiently small.

If $ \varepsilon^{-\frac{1}{2}}|\log\varepsilon|^{\frac{1}{2}}\leq|y|\leq\frac{\delta_1}\varepsilon$, then \eqref{no-self} holds if 
\begin{align*}
2 M \varepsilon^{\sigma(1-b)}|\log\varepsilon| (1+|y|)^{1-\sigma}\leq  c_1\varepsilon|y|^2,
\end{align*}
which is true because of \eqref{sigma-b2}.
In the region $|y|\geq\frac{\delta_1}\varepsilon$ a similar argument shows that \eqref{no-self} is also valid.

To prove the injectivity of $F_h$ on $\overline\Omega_0$, we consider the regions 
\begin{align*}
U_1 &= \{ \, X_0(s,t) \, | \, 0\leq t \leq 2 \rho(s) \,\}
\\
U_2 &= \{ \, X_0(s,t) \, | \, 2 \rho(s)\leq t \leq 4 \rho(s) \,\}
\\ 
U_3 &= \Omega_0\setminus \{ \, X_0(s,t) \, | \, -\tilde \rho(s)\leq t < 4 \rho(s) \,\},
\end{align*}
so that $\overline\Omega_0=U_1\cup U_2\cup U_3$. 
As discussed in Section~\ref{sect:operators}, $F_h=\operatorname{id}$ on $U_2\cup U_3$, and $F_h=X_h\circ X_0^{-1}$ on $U_1\cup U_2$.

Let $y_1,y_2\in\overline\Omega_0$ and assume that $F_h(y_1)=F_h(y_2)$. We will show that $y_1=y_2$.

Suppose that $y_1,y_2 \in U_1\cup U_2$. Then
\begin{align*}
y_j = \gamma(s_j) - t_j \nu(s_j)
\end{align*}
for some $s_j\in\R$, $0\leq t_j\leq 2 \rho(s_j)$ and 
\[
F_h(y_j)= \gamma(s_j) - (t_j+h(s_j)\eta(s_j,t_j)) \nu(s_j).
\]
But we are assuming $0\leq t_j \leq 2\rho(s_j)$ and by \eqref{no-self}, we get
\begin{align*}
-\tilde \rho(s_j)\leq t_j+h(s_j)\eta(s_j,t_j) \leq 5\rho(s_j).
\end{align*}
By Lemma~\ref{lem:injective}, $s_1=s_2$ and $t_1+h(s_1)\eta(s_1,t_1)=t_2+h(s_2)\eta(s_2,t_2)$. Using the condition \eqref{kappa-d} with $\bar\delta$ small and \eqref{no-self},
the map $t\mapsto t+h(s)\eta(s,t)$ is strictly increasing. It follows that $t_1=t_2$ and so $y_1=y_2$.

If $y_1,y_2 \in U_2\cup U_3$, then, since $F_h(y_j)=y_j$ we directly get that $y_1=y_2$. 

Finally, suppose that $y_1\in U_1$, $y_2\in U_3$. Then 
\begin{align*}
y_1 = \gamma(s) - t \nu(s)
\end{align*}
for some $s\in\R$, $0\leq t\leq 2 \rho(s)$ and 
\[
F_h(y_1)= \gamma(s) - (t+h(s)\eta(s,t)) \nu(s)=F_h(y_2)=y_2.
\]
By \eqref{no-self},
\begin{align*}
-\tilde \rho(s)\leq t+h(s)\eta(s,t) \leq 3\rho(s)
\end{align*}
Then $y_0=X_0(s,t_1)$ with $t_1=t+h(s)\eta(s,t)$. 
This is not possible since $y_2 \in U_3$.

This proves that $F_h$ is injective on $\overline\Omega_0$. Since $F_h$ is $C^{2,\alpha}$ and a local diffeomorphism, it follows that $\Omega_h=F_h(\Omega_0)$ is a $C^{2,\alpha}$ domain that together with $\psi=\psi_0+\Psi[h]$ solve \eqref{overd0}.

Thus, so far we have shown that if $\varepsilon>0$ is sufficiently small, there exists $g = \varepsilon g_0 + g_1$ \eqref{decomp-g}, and a solution $u$ to \eqref{overd0}, that gives rise to a solution $\Omega_h$, $\psi$ of \eqref{overd}. To finish the proof of  Theorem~\ref{thm} as stated, we need to show that given any $g>0$ small there is $\varepsilon>0$ and a solution $\Omega_h$, $\psi$ of \eqref{overd}. The fixed point formulation \eqref{fixed} and the continuity of the functions involved shows that $g_1(\varepsilon)$ is continuous, and the construction gives $|g_1(\varepsilon)|\leq M \varepsilon^{1+m} |\log \varepsilon|$ for some $m>0$. An argument using the intermediate value theorem yields that for any $g>0$ small there is $\varepsilon>0$ such that $g= \varepsilon g_0 + g_1(\varepsilon)$, and a corresponding solution of $u$ to \eqref{overd0}.
This completes the proof of Theorem~\ref{thm}.

\appendix

\section{Proof of some estimates}
\label{sect-p2}

\subsection{\texorpdfstring{Expansion of $\psi^H$}{Expansion of psiH}}

\begin{proof}[Proof of Lemma~\ref{lemma:estGradPsiH}]
Let us write $f(z) = \cos(w)$ with $z = w+ \sin(w) \in \Omega^{H,+}_\delta$, so that $\psi^H(z) = \Re (f(w))$.
Write $z = z_1 + i z_2$ and $w = w_1 + i w_2$ with $z_1,z_2,w_1,w_2\in \R$. Then
\begin{align}
\label{change0b}
z_1 = w_1  + \sin(w_1) \cosh(w_2)  , \quad
z_2 = w_2 +  \cos(w_1) \sinh(w_2)  .
\end{align}
This implies
\[
(z_1-w_1)^2 + (z_2-w_2)^2 = \tfrac{1}{4} e^{2w_2}+\tfrac{1}{4} e^{-2w_2} - \tfrac{1}{2}\cos(2w_1).
\]
For  $w_1 \in [ -\frac{\pi}{2}-\delta,\frac{\pi}{2}+\delta]$, $w_2\geq 0$ such that $|z|>2$ this implies that
\begin{align}
\label{expansionw2b}
w_2 = \log(2 |z| ) + O\Bigl( \frac{\log |z|}{|z|}\Bigr) .
\end{align}

Using \eqref{change0b} we see that
\[
z_2 = w_2 + \cos(w_1) \sinh(w_2).
\]
Therefore, $\psi^H(z) = \Re(\cos(w))$ and \eqref{expansionw2b} imply that 
\begin{align*}
\psi^H(z) &= \cos(w_1) \cosh(w_2)
\\
&= \cos(w_1) \sinh(w_2) + O(e^{-w_2})
\\
&= z_2-w_2+ O(e^{-w_2})
\\
&= z_2-\log(2|z|) + O\Bigl( \frac{\log |z|}{|z|}\Bigr) .
\end{align*}
This proves \eqref{est0}.

We have
\begin{align*}
\frac{df}{dz} = \partial_{z_1} \psi - i \partial_{z_2} \psi
= \frac{ \frac{df}{dw} }{\frac{dz}{dw}}
= - \frac{\sin(w)}{1+\cos(w)} ,
\end{align*}
or, in terms of real and imaginary parts,
\begin{align*}
\frac{df}{dz} 
&= - \frac{\sin(w_1)\cosh(w_2)+\sin(w_1)\cos(w_1)}{(1+\cos(w_1)\cosh(w_2))^2 + \sin^2(w_1) \sinh^2(w_2)}
\\& \quad 
- i  \frac{  \sinh(w_2)\cosh(w_2) +\cos(w_1) \sinh(w_2)  }{(1+\cos(w_1)\cosh(w_2))^2 + \sin^2(w_1) \sinh^2(w_2)} .
\end{align*}
Note that the denominator is given by 
\begin{align*}
&(1+\cos(w_1)\cosh(w_2))^2 + \sin^2(w_1) \sinh^2(w_2)
\\
&\quad= 
1+2\cos(w_1)\cosh(w_2)+\cos^2(w_1)\cosh^2(w_2) + \sin^2(w_1) \sinh^2(w_2)
\\
&\quad= 
1+2\cos(w_1)\cosh(w_2)
+\frac{e^{2w_2}+e^{-2w_2}}{4}
+\frac{\cos(2w_1)}{2}
\end{align*}
and that
\begin{align*}
|z|^2 &= |w+\sin(w)|^2
\\
&=w_1^2+w_2^2+2w_1\sin(w_1)\cosh(w_2)
+ 2w_2\cos(w_1)\sinh(w_2)
\\
& \quad 
+ \frac{e^{2w_2}+e^{-2w_2}}{4} 
-\frac{\cos(2w_1)}{2}.
\end{align*}

Let us compute an expansion as $w_2\to\infty$.
We have
\begin{align*}
\Re \Bigl(\frac{df}{dz} \Bigr) 
&=- \frac{ \sin(w_1)( \cosh(w_2)+\cos(w_1))}{ 1+2\cos(w_1)\cosh(w_2)
+\frac{e^{2w_2}+e^{-2w_2}}{4}
+\frac{\cos(2w_1)}{2} }
\\
&= - \frac{\sin(w_1)(\frac{1}{2}e^{w_2}+O(1))}{\frac{1}{4}e^{2w_2}+O(e^{w_2})}
\end{align*}
and
\begin{align*}
\frac{|z|^2}{z_1}
&= \frac{w_1^2+w_2^2+2w_1\sin(w_1)\cosh(w_2)
+ 2w_2\cos(w_1)\sinh(w_2) }{w_1+\sin(w_1)\cosh(w_2)}
\\
& \quad +
\frac{ \frac{e^{2w_2}+e^{-2w_2}}{4} 
-\frac{\cos(2w_1)}{2} }{w_1+\sin(w_1)\cosh(w_2)}
\\
&= \frac{\frac{1}{4}e^{2w_2}+O(w_2 e^{w_2})}{\sin(w_1)(\frac{w_1}{\sin(w_1)}+\cosh(w_2))} .
\end{align*}
Therefore
\begin{align*}
\Re \Bigl(\frac{df}{dz} \Bigr) \cdot \frac{|z|^2}{z_1}
&= -\frac{\frac{1}{2}e^{w_2} +O(1)}{\frac{1}{4}e^{2w_2}+O(e^{w_2}) }
\cdot 
\frac{\frac{1}{4}e^{2w_2}+O(w_2 e^{w_2}) }{ \frac{1}{2}e^{w_2}+O(1)}
\\
&= -1 + O(w_2 e^{-w_2})
\\
&= -1 + O\Bigl( \frac{\log|z|}{|z|} \Bigr),
\end{align*}
by \eqref{expansionw2b}.
This shows that 
\begin{align*}
\partial_{z_1} \psi^H(z) = -\frac{z_1}{|z|^2} + O\Bigl(\frac{\log|z|}{|z|^2}\Bigr),
\end{align*}
as $w_2\to\infty$, which is the same as $|z|\to\infty$ with $z_2>0$.
The implied constant is uniform for $w_1  \in [ -\frac{\pi}{2}-\delta,\frac{\pi}{2}+\delta]$ for fixed $0 < \delta < \frac{\pi}{2}$.
This proves \eqref{est1a}.

We claim that
\begin{align}
\label{est20}
|\partial_{z_2} \psi^H(z) - 1|\leq \frac{C}{1+|z|},
\end{align}
for $z\in \Omega^{H,+}_\delta$.
Indeed, we have
\begin{align}
\nonumber
\Im\Bigl( \frac{df}{dz} \Bigr) 
&=  - \frac{  \sinh(w_2)\cosh(w_2) +\cos(w_1) \sinh(w_2)  }{(1+\cos(w_1)\cosh(w_2))^2 + \sin^2(w_1) \sinh^2(w_2)} 
\\
\label{form1}
&=1-\tfrac{1}{2} \cos(w_1) e^{-w_2} + O(e^{-2w_2}).
\end{align}
But from $ z = w+ \sin(w) $ we get
\begin{align}
\nonumber
\cos(w_1) e^{-w_2} 
&= 2 z_2 e^{-2w_2} - 2  w_2 e^{-2w_2} + \cos(w_1)  e^{-3w_2}
\\
\label{form2}
&= O\Bigl( \frac{z_2}{|z|^2} \Bigr) + O \Bigl( \frac{ \log |z|}{|z|^2} \Bigr)
\end{align}
by \eqref{expansionw2b}.
From \eqref{form1} and \eqref{form2} we then obtain 
\begin{align}
\nonumber
|\partial_{z_2} \psi^H(z) - 1|\leq C \frac{\log |z|+z_2}{1+|z|^2}\leq\frac{C}{1+|z|}.
\end{align}

Let us write \eqref{est1a} as 
\begin{align}
\nonumber
\left| \partial_{z_1} (\psi^H(z)+\log|z|) \right|
\leq C \frac{\log|z|}{|z|^2}
\end{align}
in the region $\Omega^{H,+}_\delta$.
Using that $ \partial_{z_1} (\psi^H(z) -\log|z|)$ is a harmonic function, by standard estimates for harmonic functions, for $j=1,2$,
\begin{align}
\nonumber
\left| \partial_{z_j,z_1}^2 (\psi^H(z)+\log|z|) \right| 
\leq C \frac{\log|z|}{|z|^3} ,
\end{align}
in a slightly smaller region $\Omega^{H,+}_{\delta'}$.
From here
\begin{align}
\nonumber
\left| \partial_{z_2}^2 (\psi^H(z)+\log|z|) \right| 
\leq C \frac{\log|z|}{|z|^3}
\end{align}
and hence
\begin{align}
\label{cota-nabla}
\left| \nabla \partial_{z_2} (\psi^H(z)+\log|z|) \right| 
\leq C \frac{\log|z|}{|z|^3}.
\end{align}

Combining \eqref{est1a} with \eqref{est20} we see that 
\begin{align}
\label{nabla2}
| \nabla \psi^H(z) - e_2 | \leq \frac{C}{1+|z|}.
\end{align}
This implies
\begin{align}
\nonumber
| \nabla \psi^H(z) - e_2 - \nabla \log |z|| \leq \frac{C}{1+|z|}.
\end{align}
Integrating over a ray $\{ z t \,|\, t\geq1\}$, using \eqref{cota-nabla} and \eqref{nabla2} we find that 
\begin{align}
\nonumber
\left| \partial_{z_2} (\psi^H(z) -z_2+\log|z|) \right| 
\leq C \frac{\log|z|}{|z|^2}.
\end{align}
Using standard estimates for harmonic functions we also deduce \eqref{D2psiH-1} and \eqref{D2psi}.
\end{proof}

\subsection{Proof of Lemma~\ref{lem:sol-disc2}}

\begin{proof}[Proof of Lemma~\ref{lem:sol-disc2}]
We can assume that $R=1$ and and that the disk is centered at 0. 
We consider the problem
\begin{align}
\label{lne-eq-disc}
\left\{
\begin{aligned}
\Delta u &= 0 \quad \text{in } B_1(0) \\
\frac{\partial u }{\partial \nu} - u &= h , \quad \text{on } \partial B_1(0) ,
\end{aligned}
\right.
\end{align}
where $h\colon \partial B_1(0)\to\R$.
We want to find a solution with $h = \frac{1}{\pi}\delta_P + g(x_2+d)$, where $P=(0,-1)$, $d$ is a constant and $g$ is to be adjusted.

First we discuss \eqref{lne-eq-disc} when $h$ is a sufficiently regular function. Using polar coordinates $r,\theta$ the boundary condition translates into
\begin{align}
\label{cb1}
\partial_r u - u = h \quad \text{at }r=1.
\end{align}
Since $u$ is harmonic, there is then an analytic function $f$ on $B_1(0)$ such that $u = \Re(f)$. 
We have
\begin{align*}
f' 
&= 
\partial_r u e^{-i\theta} - \frac{i}{r}\partial_\theta u e^{-i\theta}  .
\end{align*}
Let $H$ be analytic in $B_1(0)$ such that $\Re(H)=h$ on $\partial B_1(0)$.
Letting $z = r e^{i\theta}$ the boundary condition \eqref{cb1} becomes
\begin{align*}
\Re( z f'(z) - f(z) - H ) = 0 \quad \text{on } \partial B_1(0).
\end{align*}
Assuming that all functions are sufficiently regular, there is a constant $\kappa \in \R$ such that 
\begin{align}
\label{eq-f}
z f'(z) - f(z) = H + i \kappa.
\end{align}
We can assume that $\kappa=0$ because $u$ is the real part of $f$. We see that $z$ is in the kernel of equation \eqref{eq-f}. 
Let us write
\begin{align}
\label{power-s-H}
H(z)=\sum_{n=0}^\infty H_n z^n.
\end{align}
Then for \eqref{eq-f} to be solvable we need $H_1=0$, which is equivalent to the conditions
\begin{align*}
\int_{\partial B_1(0)} h Z_1^D
=
\int_{\partial B_1(0)} h Z_2^D
=0.
\end{align*}
If $H_1=0$ then a solution of \eqref{eq-f} is
\begin{align}
\label{f1}
f(z) = \sum_{n\not=1} \frac{H_n}{n-1} z^n .
\end{align}
There is another way of writing a solution of \eqref{eq-f}. Noting that 
\begin{align*}
z f'(z) - f(z) = z^2 \frac{d}{dz}\Bigl( \frac{1}{z}f(z) \Bigr),
\end{align*}
we get a particular solution to \eqref{eq-f} by the formula
\begin{align}
\label{f2}
f(z) = z \int_0^z \frac{H(\zeta)}{\zeta^2}d\zeta
\end{align}
where the integral is over any path inside $B_1(0)$ from $0$ to $z$. Asking that $H_1=0$ in \eqref{power-s-H} we get a well defined analytic function $f$ in $B_1(0)$ which satisfies \eqref{eq-f}.

We are interested in the particular right hand side $h=\pi\delta_{-i} + g (x_2+d)$ where $d$ is a constant and $g$ is a constant to be determined. Thus the real part of $f$ should be at main order
\begin{align*}
-\log(|z+i|) .
\end{align*}
This leads us to define
\begin{align*}
f_0(z) = -\log(i(z+i)) ,
\end{align*}
where the branch cut of the logarithm is taken along the positive real axis, so that  $f_0$ has the branch cut at  $\{ \, it \, | \, t\leq -1 \, \}$.
The parameter $g \in \C$ is to be adjusted.
We compute
\begin{align*}
zf_0'(z)-f_0(z) = -\frac{z}{z+i}+\log(i(z+i)) .
\end{align*}
The most singular part near $-i$ is 
\begin{align*}
\frac{i}{z+i}
\end{align*}
and we compute 
\begin{align*}
\Re\Bigl( \frac{i}{z+i} \Bigr) = 
 \frac{1}{2}
\quad \text{for } |z|=1.
\end{align*}
This means 
\begin{align*}
\Re\Bigl[ z \Bigl(f_0+\frac{1}{2}\Bigr)'-\Bigl(f_0+\frac{1}{2}\Bigr)\Bigr]
&=
\Re\Bigl[ -\frac{z}{z+i}+\log(i(z+i)) - \frac{i}{z+i}\Bigr]
\\
&=
\Re\left[ \log(i(z+i)) -1\right] .
\end{align*}
Therefore, the problem that we want to solve for $f_1$ is
\begin{align}
\label{eq:f1}
z f_1'(z)-f_1(z) =  H(z)
\end{align}
where
\begin{align*}
H(z) = \log(i(z+i)) -1- iz.
\end{align*}
The last term is added so that $H_1=0$ and corresponds to the choice $g=1$. 
Now we solve for $f_1$ in equation \eqref{eq:f1} using either \eqref{f1} or \eqref{f2}.
Define 
\begin{align*}
f=f_0+\tfrac{1}{2}+f_1 .
\end{align*}
Then $u=\Re(f)-d$ satisfies \eqref{lne-eq-disc} with $h = \pi\delta_P + x_2+d$. 
From the formula for $u$ and using the representation \eqref{f2} we see that $u$ is well defined and harmonic in $\R^2$ minus the half line $L=\{ i t , \ t\leq -1\}$ and $u$ has a continuous limit at points $i t$, $t<-1$ from either side of $L$. Moreover there are constants $a$, $C$ such that 
\begin{align*}
\left| u(z) - \log |z+i| -a \right| \leq C |z+i| \log\Bigl(\frac{1}{|z+i|}\Bigr) , \quad \text{ as } z\to -i.
\end{align*}
Then the function $u - a Z_2^D$ is the desired solution.
\end{proof}

\subsection{Proof of Lemma~\ref{lem:sol-strip2}}

\begin{proof}[Proof of Lemma~\ref{lem:sol-strip2}]
By shifting vertically we consider the problem in the strip $0<x_2<\A$.
We claim that if $0<\omega<1$ then there is a solution $u$ to
\begin{align}
\label{eq:strip2}
\left\{
\begin{aligned}
\Delta u &= 0 && \text{in } \{ 0<x_2<\A\} \\
\frac{\partial u}{\partial x_2} - \omega u &= \pi \delta_{(0,\A)}  && \text{on } x_2=\A
\\
u &=0 && \text{on } x_2=0 ,
\end{aligned}
\right.
\end{align}
with a log singularity at $(0,\A)$ and exponential decay.

Let $u_0(x)=-\log(|x-(0,\A)|) \chi_0(4 (x-(0,\A)))$ where $\chi_0$ is a smooth cut-off in $\R^2$ such that $\chi_0(z)=1$ for $|z|\leq 1$ and $\chi_0(z)=0$ for $|z|\geq 0$.
Let $u_1$ be the solution to 
\begin{align*}
\left\{
\begin{aligned}
\Delta u_1 & =0 && \text{in } x_2<\A
\\
\frac{\partial u_1}{\partial x_2} &= \omega u_0 && \text{on } x_2=\A ,
\end{aligned}
\right.
\end{align*}
obtained using the Neumann Green function of the half plane $\{x_2<\A\}$.
Note that $u_1$ is continuous in $x_2\leq \A$.

To obtain $u$ satisfying \eqref{eq:strip2}, we write $u = u_0 + u_1 \chi_0(4 (x-(0,\A))) + u_2$ and solve
\begin{align}
\label{eq:strip}
\left\{
\begin{aligned}
- \Delta u_2 &= f_1(x) && \text{in } \{ 0<x_2<\A\} \\
\frac{\partial u_2}{\partial x_2} - \omega u_2 &= f_2(x)  && \text{on } x_2=\A
\\
u_2 &=0 && \text{on } x_2=0 ,
\end{aligned}
\right.
\end{align}
where $f_1$, $f_2$ are continuous with compact support.

Let $\lambda>0$ be such that 
\begin{align*}
\omega \Ae < \frac{\lambda \Ae }{\tan(\lambda \Ae)}
\end{align*}
and $0<\mu<\lambda$.
Using the supersolution \eqref{superstrip} we can find a unique solution $u$ to \eqref{eq:strip} satisfying
\begin{align*}
|u_2(x)| \leq C e^{-\mu |x_1|} .
\end{align*}
Hence $u$ satisfies
\begin{align*}
|u(x)| \leq e^{-\mu |x_1|} , \quad |x_1|>1 ,
\end{align*}
and
\begin{align*}
\mathop{\lim_{x\to (0,\A)}}_{0<x_2<\A} ( u(x) + \log(|x-(0,\A)|) ) \quad\text{exists}.
\end{align*}

Next we prove that $u$ has an analytic extension to $\{ (x_1,x_2) \, | \, 0<x_2<2\Ae\}$ minus the half line $\{ (0,t) \, | \, t \geq \A\}$. 
To do this, write $u=\Re(f)$ with $f$ analytic in $\{ 0<\Im(z) <\A\}$.
The boundary condition in \eqref{eq:strip2} for $x_2=\A$, $x_1\not=0$ can be written as 
\begin{align*}
\Im\Big( \frac{d}{dz}\left( e^{-i\omega z}f(z) \right) \Bigr) =0.
\end{align*}
Then by the Schwarz reflection principle applied in $\{ 0<\Im(z) <\A, \ \Re(z)<0\}$ and $\{ 0<\Im(z) <\A, \ \Re(z)>0\}$, we find that $\frac{d}{dz}( e^{-i\omega z}f ) $ and hence $f$ have analytic extensions to $\{ 0<\Im(z) <2\Ae, \ \Re(z)<0\}$ and $\{ 0<\Im(z) <2\Ae, \ \Re(z)>0\}$. From this we get the desired analytic extension of $u$ to  $\{ 0<\Im(z) <2\Ae, \ \Re(z)<0\}$ and $\{ 0<\Im(z) <2\Ae, \ \Re(z)>0\}$.
\end{proof}

\subsection{Estimate for the harmonic conjugate in a strip}
\label{sec:holder-strip}
Let
\[
S_0 = \{ \, (x_1,x_2) \, | \, x_1\in\R,\ 0<x_2<\pi\}. 
\]
Suppose that $u \in C(\overline S_0)$ is harmonic in $S_0$ and satisfies, for some $0<\sigma<1$,
\[
|u(x_1,x_2)|\leq e^{-\sigma|x_1|},\quad (x_1,x_2)\in\overline S_0.
\]

Consider the harmonic conjugate $v$ of $u$.
By standard gradient estimates for harmonic functions, for any $x_2\in(0,\pi)$ there is $C(x_2)$ such that 
\begin{align}
\nonumber
|\nabla u(x_1,x_2)|\leq C(x_2)e^{-\sigma|x_1|},\quad x_1\in\R,
\end{align}
with $C(x_2) \leq C/\min(x_2,\pi-x_2)$.
Hence 
\begin{align}
\label{nabv}
|\nabla v(x_1,x_2)|\leq C(x_2)e^{-\sigma|x_1|},\quad x_1\in\R.
\end{align}
This implies that for all $x_2\in(0,\pi)$
\[
L_\pm = \lim_{x_1\to\pm\infty} v(x_1,x_2)\quad\text{exists}
\]
and is independent of $x_2$. 
\begin{lemma}
\label{lem:hc}
Suppose that there exists $\alpha\in(0,1)$ such that 
\begin{align}
\label{holder-exp}
\begin{aligned}
&\text{for all $x_1\in\R$, $x_2\in\{0,\pi\}$, $x_1',x_1''\in [x_1-1,x_1+1]$:}
\\
&|u(x_1',x_2)-u(x_1'',x_2)|\leq e^{-\sigma|x_1|} |x_1'-x_1''|^\alpha.
\end{aligned}
\end{align}
Then
\begin{align}
\nonumber
& |v(x_1,x_2)-L_+|\leq C e^{-\sigma x_1},\quad (x_1,x_2)\in\overline S_0, \ x_1>0
\\
\nonumber
&|v(x_1,x_2)-L_-|\leq C e^{-\sigma |x_1|},\quad (x_1,x_2)\in\overline S_0, \ x_1<0.
\end{align}
\end{lemma}

\begin{lemma}
\label{lem:grad}
Assume \eqref{holder-exp}. Then
\begin{align}
\nonumber
|\nabla u(x_1,x_2)|\leq C \frac{1}{\min(x_2,\pi-x_2)^{1-\alpha}}e^{-\sigma|x_1|},\quad (x_1,x_2)\in S_0.
\end{align}
\end{lemma}
\begin{proof}
Let $\bar x_1>0$, $\bar x=(\bar x_1,0)$.
Let $\eta\in C^\infty(\R^2)$ be a smooth cut-off function such that $\eta(x)=1$ for $|x-\bar x|\leq1$ and  $\eta(x)=0$ for $|x-\bar x|\geq2$.

Let $h$ be the harmonic extension of $u\eta|_{x_2=0}$ to the upper half plane $\{(x_1,x_2)|x_2>0\}$. From a direct computation using the Poisson kernel we get
\[
|h(x_1,x_2)|\leq C e^{-\sigma \bar x_1},\quad x_1\in\R,\ x_2>0,
\]
\[
|\nabla h(x_1',x_2)|\leq C \frac{e^{-\sigma \bar x_1}}{x_2^{1-\alpha}},\quad x_1\in\R,\ x_2>0.
\]
Let
\[
\hat u = u - h \eta.
\]
Then
\begin{align*}
\begin{aligned}
\Delta\hat u &= -2\nabla h \cdot\nabla\eta - h\Delta\eta&&\text{in }S_0,
\\
\hat u &=u&&\text{on }x_2=\pi
\\
\hat u&=u (1-\eta^2)&&\text{on }x_2=0.
\end{aligned}
\end{align*}
Using barriers of the form
\[
\bar u(x_1,x_2)=e^{-\sigma x_1} \cos\Bigl(\lambda \Bigl(x_2-\frac\pi2\Bigr)\Bigr), \quad \sigma<\lambda<1,
\]
we get
\[
|\hat u(x_1,x_2)|\leq C e^{\sigma |x_1|}.
\]
By estimates for harmonic functions
\[
|\nabla \hat u(\bar x_1,x_2)|\leq  e^{-\sigma\bar x_1},\quad 0<x_2<\frac{1}{2}.
\]
We deduce that 
\[
|\nabla u(\bar x_1,x_2)|\leq C \frac{1}{x_2^{1-\alpha}}e^{-\sigma\bar x_1},\quad0<x_2<\frac{1}{2}.
\qedhere
\]
\end{proof}

\begin{proof}[Proof of Lemma~\ref{lem:hc}]
Let $x_1>0$, $x_2\in [0,\pi/2]$. 
Integrating $\partial_{x_1} v$ on the line $\{(t,\pi/2)\,|\,t>x_1\}$ and using \eqref{nabv}, we get
\[
|v(x_1,\tfrac\pi2)-L_+|\leq C e^{-\sigma x_1}.
\]
Integrating $\partial_{x_2} v$ on the line $\{(x_1,s)\,|\,x_2<s<\pi/2\}$ and using Lemma~\ref{lem:grad}, we get
\[
|v(x_1,\tfrac\pi2)-v(x_1,x_2)|\leq C e^{-\sigma x_1}.
\qedhere
\]
\end{proof}
The estimate of Lemma~\ref{lem:grad} can also be obtained from the results in \cite{craig-sternberg}.


\bigskip\noindent
{\bf Acknowledgements.}
J.~D\'avila has been supported  by  a Royal Society  Wolfson Fellowship, UK.
M.~del Pino has been supported by a Royal Society Research Professorship, UK.
M. Musso has been supported by EPSRC research Grant EP/T008458/1. 

\end{document}